\documentclass[a4]{amsart}

\input xypic 
\input xy 
\xyoption{all} 
\usepackage{amssymb} 
\usepackage{bbm}
\newtheorem{theorem}{Theorem}[section]
\newtheorem{proposition}[theorem]{Proposition}
\newtheorem{lemma}[theorem]{Lemma}
\newtheorem{corollary}[theorem]{Corollary}

\theoremstyle{definition}
\newtheorem{definition}[theorem]{Definition}
\newtheorem{remark}[theorem]{Remark}
\newtheorem{example}[theorem]{Example}

\newtheorem{conjecture}[theorem]{Conjecture} 
\newtheorem{problem}[theorem]{Problem} 

\newcommand{\conn}{\ensuremath{\#}} 
\newcommand{\cp}{\ensuremath{\mathbb{C}P}} 
\newcommand{\uxa}{\ensuremath{(\underline{X},\underline{A})}} 
\newcommand{\ux}{\ensuremath{(\underline{X},\underline{\ast})}} 
\newcommand{\sux}{\ensuremath{(\underline{\Sigma X},\underline{\ast})}} 
\newcommand{\cxx}{\ensuremath{(\underline{CX},\underline{X})}}

\newcommand{\sclxx}{\ensuremath{(\underline{C\Omega\Sigma X},\underline{\Omega\Sigma X})}}

\newcommand{\zk}{\ensuremath{\mathcal{Z}_{K}}}

\newcommand{\dm}{\ensuremath{\mbox{dim}}} 

\newcommand{\hlgy}[1]{\ensuremath{H_{*}(#1)}}
\newcommand{\rhlgy}[1]{\ensuremath{\widetilde{H}_{*}(#1)}}
\newcommand{\cohlgy}[1]{\ensuremath{H^{*}(#1)}}

\newcounter{bean}
\newenvironment{letterlist}{\begin{list}{\rm ({\alph{bean}})}
      {\usecounter{bean}\setlength{\rightmargin}{\leftmargin}}}
      {\end{list}}

\newcommand{\namedright}[3]{\ensuremath{#1\stackrel{#2}
 {\longrightarrow}#3}}
\newcommand{\nameddright}[5]{\ensuremath{#1\stackrel{#2}
 {\longrightarrow}#3\stackrel{#4}{\longrightarrow}#5}}
\newcommand{\namedddright}[7]{\ensuremath{#1\stackrel{#2}
 {\longrightarrow}#3\stackrel{#4}{\longrightarrow}#5
  \stackrel{#6}{\longrightarrow}#7}}

\newcommand{\larrow}{\relbar\!\!\relbar\!\!\rightarrow}
\newcommand{\llarrow}{\relbar\!\!\relbar\!\!\larrow}
\newcommand{\lllarrow}{\relbar\!\!\relbar\!\!\llarrow} 
\newcommand{\llllarrow}{\relbar\!\!\relbar\!\!\lllarrow}

\newcommand{\lnamedright}[3]{\ensuremath{#1\stackrel{#2}
 {\larrow}#3}}
\newcommand{\lnameddright}[5]{\ensuremath{#1\stackrel{#2}
 {\larrow}#3\stackrel{#4}{\larrow}#5}}
\newcommand{\lnamedddright}[7]{\ensuremath{#1\stackrel{#2}
 {\larrow}#3\stackrel{#4}{\larrow}#5
  \stackrel{#6}{\larrow}#7}}
\newcommand{\llnamedright}[3]{\ensuremath{#1\stackrel{#2}
 {\llarrow}#3}}
\newcommand{\llnameddright}[5]{\ensuremath{#1\stackrel{#2}
 {\llarrow}#3\stackrel{#4}{\llarrow}#5}}
\newcommand{\llnamedddright}[7]{\ensuremath{#1\stackrel{#2}
 {\llarrow}#3\stackrel{#4}{\llarrow}#5
  \stackrel{#6}{\llarrow}#7}}
\newcommand{\lllnamedright}[3]{\ensuremath{#1\stackrel{#2}
 {\lllarrow}#3}}
\newcommand{\lllnameddright}[5]{\ensuremath{#1\stackrel{#2}
 {\lllarrow}#3\stackrel{#4}{\lllarrow}#5}}
\newcommand{\lllnamedddright}[7]{\ensuremath{#1\stackrel{#2}
 {\lllarrow}#3\stackrel{#4}{\lllarrow}#5
  \stackrel{#6}{\lllarrow}#7}}

\newcommand{\qqed}{\hfill\Box}

\begin{document}


\title{Homotopy fibrations with a section after looping} 

\author{Stephen Theriault}
\address{School of Mathematical Sciences, University of Southampton, Southampton 
   SO17 1BJ, United Kingdom}
\email{S.D.Theriault@soton.ac.uk}



\begin{abstract} 
We analyze a general family of fibrations which, after looping, have sections. 
Methods are developed to determine the homotopy type of the fibre and the homotopy 
classes of the map from the fibre to the base. The methods are driven by applications to two-cones, 
Poincar\'{e} Duality complexes, the connected sum operation, and polyhedral products. 
\end{abstract}

\maketitle 

\tableofcontents 
\newpage

\section{Introduction} 
\label{sec:intro} 

A fundamental goal in homotopy theory is to determine the homotopy types of spaces 
and the homotopy classes of the maps between them. This paper builds on new 
methods developed in~\cite{BT2} in order to do that in an appropriate context. The applications 
are wide-ranging, informing on the homotopy theory of two-cones, Poincar\'{e} 
Duality complexes, connected sums, and polyhedral products. 

To describe the context, it will be assumed throughout that all spaces are $CW$-complexes 
so that weak homotopy equivalences are homotopy equivalences. 
Suppose that there is a homotopy fibration 
\(\nameddright{E}{p}{Y}{h}{Z}\)  
and a homotopy cofibration 
\(\nameddright{\Sigma A}{f}{Y}{}{Y'}\). 
Suppose that $h$ extends to a map 
\(h'\colon\namedright{Y'}{}{Z}\) 
and let $E'$ be the homotopy fibre of $h'$. This data is assembled into a diagram 
\begin{equation} 
  \label{data} 
  \diagram 
       & E\rto\dto^{p} & E'\dto^{p'} \\ 
       \Sigma A\rto^-{f} & Y\rto\dto^{h} & Y'\dto^{h'} \\ 
       & Z\rdouble & Z. 
  \enddiagram 
\end{equation} 
where the vertical columns and the maps between them form a homotopy fibration diagram. 
Using either Dold and Lashof~\cite{DL} or Mather's Cube Lemma~\cite{M}, there is a 
homotopy pushout 
\begin{equation} 
  \label{DLpo} 
  \diagram 
      \Omega Z\times\Sigma A\rto\dto^{\pi_{1}} & E\dto \\ 
      \Omega Z\rto & E' 
  \enddiagram 
\end{equation}  
where $\pi_{1}$ is the projection. Under favourable circumstances, this homotopy pushout 
may allow for the homotopy type of $E'$ to be determined, and possibly also the homotopy 
class of the map 
\(\namedright{E}{}{E'}\). 
However, much depends on the homotopy class of the map 
\(\namedright{\Omega Z\times\Sigma A}{}{E}\), 
and this can be difficult to identify with sufficient precision. 

Suppose in addition that the map 
\(\namedright{\Omega Y}{\Omega h}{\Omega Z}\) 
has a right homotopy inverse 
\(s\colon\namedright{\Omega Z}{}{\Omega Y}\). 
Then the homotopy pushout~(\ref{DLpo}) simplifies to a homotopy cofibration  
\begin{equation} 
  \label{thetacofintro} 
  \nameddright{\Omega Z\ltimes\Sigma A}{\theta}{E}{}{E'} 
\end{equation}  
for some map $\theta$. In the special case when $Y'=Z$ and $h'$ is the identity map, 
this implies that~$E'$ is contractible so $\theta$ is a homotopy equivalence. But in general 
this cofibration by itself says little about the precision with which $\theta$ 
can be identified. However, as will be explained in Section~\ref{sec:background}, the 
existence of a right homotopy inverse for $\Omega h$ implies that there is a profound 
connection between $\theta$, the homotopy action of $\Omega Z$ on $E$, and Whitehead 
products mapping into $Y$. Specifically, there is a homotopy commutative diagram 
\begin{equation} 
  \label{introBTdgrm} 
  \diagram 
        \Omega Z\ltimes\Sigma A\rto^-{\theta}\dto^{\simeq} & E\dto^{p} \\ 
               (\Omega Z\wedge\Sigma A)\vee\Sigma A\rto^-{[\gamma,f]+f} 
        & Y   
  \enddiagram 
\end{equation} 
where $\gamma$ is the composite 
\(\gamma\colon\nameddright{\Sigma\Omega Z}{\Sigma s}{\Sigma\Omega Y}{ev}{Y}\), 
the map $ev$ is the canonical evaluation map, and $[\gamma,f]$ is the Whitehead 
product of $\gamma$ and $f$. That is, the homotopy class of $\theta$ is identified, 
at least up to composition with $p$, and this gives a measure of control over the 
homotopy cofibration 
\(\nameddright{\Omega Z\ltimes\Sigma A}{\theta}{E}{}{E'}\). 
But the level of control is often not fine enough to precisely describe the homotopy 
type of $E'$ in cases of interest. Obtaining that control is the thrust of this paper. 
\medskip 

We consider, then, homotopy fibrations 
\(\nameddright{E}{p}{Y}{h}{Z}\) 
which have a section after looping. That is, those for which $\Omega h$ has a right homotopy 
inverse. This begins with a simple but foundational case that will play an important role at 
many points later on. We move on to consider different families of examples, each of which 
involves distinctive features that influence how control over $\theta$ is obtained. 
\medskip 

\noindent 
\textbf{A foundational case}. 
Consider the homotopy fibration 
\(\nameddright{E}{p}{\Sigma X\vee\Sigma Y}{q_{1}}{\Sigma X}\) 
where $q_{1}$ is the pinch map to the first wedge summand. Note that $q_{1}$ has a right 
homotopy inverse, so $\Omega q_{1}$ does as well, implying that this is an example of 
a homotopy fibration with a section after looping. 

For $k\geq 1$, let $X^{\wedge k}$ be the $k$-fold smash product of $X$ with itself. 
By~\cite[Theorem 4.3.2]{N3} there is a homotopy equivalence 
\[E\simeq\bigvee_{k=0}^{\infty} X^{\wedge k}\wedge\Sigma Y\] 
 where, by convention, $X^{\wedge 0}\wedge\Sigma Y$ refers to $\Sigma Y$. 
 Further, let 
 \(i_{1}\colon\namedright{\Sigma X}{}{\Sigma X\vee\Sigma Y}\) 
 and 
 \(i_{2}\colon\namedright{\Sigma Y}{}{\Sigma X\vee\Sigma Y}\) 
 be the inclusions of the first and second wedge summands respectively. Let 
 $ad^{\, 0}(i_{1})(i_{2})=i_{2}$ and for $k\geq 1$ let $ad^{k}(i_{1})(i_{2})$ be the 
 Whitehead product $[i_{1},ad^{k-1}(i_{1})(i_{2})]$. Then~\cite[Theorem 4.3.2]{N3} shows 
 that, under the homotopy equivalence for $E$ above, the map $p$ may be identified as 
 $\bigvee_{k=0}^{\infty} ad^{k}(i_{1})(i_{2})$. 
 
 We give an alternative proof of this which has the advantage of having a compatibility with 
 the map $\theta$ in~(\ref{introBTdgrm}). Here, the general homotopy cofibration 
 \(\nameddright{\Sigma A}{f}{Y}{}{Y'}\) 
 specifies to  
 \(\nameddright{\Sigma Y}{i_{2}}{\Sigma X\vee\Sigma Y}{q_{1}}{\Sigma X}\) 
 and $\theta$ takes the form of a map 
 \(\namedright{\Omega\Sigma X\ltimes\Sigma Y}{\theta}{E}\).  
The point to emphasize is that our choice of a homotopy equivalence for $E$ has the 
additional property of respecting the homotopy action of $\Omega\Sigma X$ on $E$. 

\begin{theorem}[appearing in the text as Theorem~\ref{dbard}] 
   \label{dbardintro} 
   Let $X$ and $Y$ be path-connected, pointed spaces and consider the homotopy fibration  
   \(\nameddright{E}{}{\Sigma X\vee\Sigma Y}{q_{1}}{\Sigma X}\). 
   There is a homotopy commutative diagram 
   \[\diagram 
          \bigvee_{k=0}^{\infty} X^{\wedge k}\wedge\Sigma Y 
                 \rto^-{\mathfrak{d}}\drto_{\bigvee_{k=0}^{\infty} ad^{k}(i_{1})(i_{2})} 
               & E\dto^{p} \\ 
           & \Sigma X\vee\Sigma Y 
     \enddiagram\]      
   where: 
   \begin{letterlist} 
      \item $\mathfrak{d}$ is a homotopy equivalence; 
      \item $\Sigma\mathfrak{d}\simeq\Sigma d$, where $d$ is the composite 
               \[\nameddright{\bigvee_{k=0}^{\infty} X^{\wedge k}\wedge\Sigma Y}{c} 
                  {\Omega\Sigma X\ltimes\Sigma Y}{\theta}{E}.\] 
    \end{letterlist}  
\end{theorem} 

Note that the maps $\mathfrak{d}$ and $d$ may not be homotopy equivalent but their 
suspensions are. Consequently, they induce the same map in homology. As $\mathfrak{d}$ 
is a homotopy equivalence, it induces an isomorphism in homology, and therefore so does $d$, 
and hence $d$ is also a homotopy equivalence by Whitehead's Theorem. This ability to 
use one homotopy equivalence to prove the existence of another will be used repeatedly 
throughout. It has the advantage of allowing us to exchange maps that have 
different properties: in this case the maps $\mathfrak{d}$ behaves well with respect to 
Whitehead products while the map $d$ (via $c$) behaves better with respect to the 
multiplication on $\Omega\Sigma X$. 
\medskip 

\noindent 
\textbf{Two-cones}. 
A two-cone is the homotopy cofibre $C$ of a map 
\(\namedright{\Sigma A}{}{\Sigma B}\) 
where $A$ and $B$ are both path-connected. More generally, one could consider 
a map between co-$H$-spaces instead of suspensions, but the latter simplifies the exposition. 
This notion can be iterated: a finite $CW$-complex $X$ has \emph{cone-length} $t$ if $t$ is 
the smallest number such that there is a sequence of homotopy cofibrations 
\(\nameddright{\Sigma A_{k}}{}{C_{k-1}}{}{C_{k}}\) 
for $1\leq k\leq t$ where $C_{0}$ is some initial space $\Sigma A_{0}$ and $C_{t}\simeq X$. 
Cone-length is an upper bound on the Lusternik-Schnirelmann category of $X$. A great deal 
of work has gone into studying cone-length (see~\cite{CLOT} for a comprehensive 
overview). The homotopy theory around two-cones and their based loop spaces has received 
particular attention~\cite{A,FHT,FT2} since they are the nearest neighbour to suspensions, whose based 
loop spaces are well understood through the Bott-Samelson Theorem, the James construction, 
and the Hilton-Milnor Theorem. 

In Theorem~\ref{E'typeI} we prove a general result which lets us consider, as examples, certain 
families of two-cones. One case is the following. Define the two-cone $M_{k}$ by 
the homotopy cofibration 
\[\lllnameddright{\Sigma X^{\wedge k}\wedge\Sigma Y}{ad^{k}(i_{1})(i_{2})}{\Sigma X\vee\Sigma Y}{}{M_{k}}.\] 
We give a homotopy decomposition of $\Omega M_{k}$. Note that as $ad^{k}(i_{1})(i_{2})$ 
is an iterated Whitehead product, it composes trivially with the pinch map 
\(\namedright{\Sigma X\vee\Sigma Y}{q_{1}}{\Sigma X}\), 
implying that $q_{1}$ extends to a map 
\(\namedright{M_{k}}{q'}{\Sigma X}\).  
Define the map $\gamma_{k}$ by the composite 
\[\gamma_{k}\colon\bigvee_{t=0}^{k-1} X^{\wedge t}\wedge\Sigma Y 
      \stackrel{\bigvee_{t=0}^{k-1} ad^{t}(i_{1})(i_{2})}{\llllarrow}\namedright{\Sigma X\vee\Sigma Y}{}{M_{k}}.\] 

\begin{theorem}[appearing in the text as Theorem~\ref{Mtypealt}] 
   \label{Mtypealtintro} 
   For $k\geq 1$, there is a homotopy fibration 
   \[\nameddright{\bigvee_{t=0}^{k-1} X^{\wedge t}\wedge\Sigma Y}{\gamma_{k}}{M_{k}}{q'}{\Sigma X}\] 
   which splits after looping to give a homotopy equivalence 
   \[\Omega M_{k}\simeq\Omega\Sigma X\times\Omega(\bigvee_{t=0}^{k-1} X^{\wedge k}\wedge\Sigma Y).\] 
\end{theorem} 

Particular examples of interest occur when $X$ and $Y$ are both spheres or Moore spaces. 
These are discussed in Section~\ref{sec:2cone}; they give a large family of examples that satisfy Moore's 
conjecture. 
\medskip 

\noindent 
\textbf{Poincar\'{e} Duality complexes}. 
A finite $CW$-complex $X$ is a \emph{Poincar\'{e} Duality complex} if $\cohlgy{X;\mathbb{Z}}$ 
satisfies Poincar\'{e} Duality. These spaces are generalizations of closed, orientable 
manifolds. Poincar\'{e} Duality complexes have a long history in both geometry 
and topology (see the survey by Klein~\cite{K}) and recently there has been progress in analyzing 
their homotopy groups through homotopy decompositions of their loop spaces. In particular, 
Beben and Wu~\cite{BW} studied $(n-1)$-connected $(2n+1)$-dimensional Poincar\'{e} Duality 
complexes $M$ with $n$ odd, $n\geq 6$ and $H_{2n-1}(M;\mathbb{Z})$ consisting only of odd torsion; 
Beben and the author~\cite{BT1} studied all $(n-1)$-connected $2n$-dimensional Poincar\'{e} Duality 
complexes; this case was also considered using different methods by Sa. Basu and So. Basu~\cite{BB}, 
and Sa. Basu~\cite{Ba} went on to consider $(n-1)$-connected $(2n+1)$-dimensional Poincar\'{e} Duality 
complexes $M$ with $H_{n}(M;\mathbb{Z})$ having at least one integral summand. In~\cite{BT2}, 
Beben and the author developed the new methods that are the basis of this paper and used them to 
recover in a unified way the results in~\cite{Ba,BB,BT1}. 

The case of an $(n-1)$-connected $(2n+1)$-dimensional Poincar\'{e} Duality complex $M$ 
when $n$ is even and $H_{n}(M;\mathbb{Z})$ consists only of odd torsion is trickier. The 
methods used in~\cite{BW} do not work. They showed that if $n$ is odd then there is a space $V$ 
and a map 
\(\namedright{M}{h}{V}\) 
where $\Omega h$ has a right homotopy inverse and, for an appropriate prime $p$, 
$\hlgy{V;\mathbb{Z}/p\mathbb{Z}}\cong\Lambda(x,y)$ with $\vert x\vert=n$, $\vert y\vert=n+1$ 
and $x$ and $y$ connected by a Bockstein (possibly of higher order). No such space 
exists when $n$ is even. The problem boils down to the following. For a prime $p$ and integers $r\geq 1$ 
and $m\geq 2$, the \emph{mod-$p^{r}$ Moore space} $P^{m}(p^{r})$ is the homotopy 
cofibre of the degree~$p^{r}$ map on $S^{m-1}$. It is characterized by the fact that 
$\widetilde{H}_{n}(P^{m}(p^{r});\mathbb{Z})$ is $\mathbb{Z}/p^{r}\mathbb{Z}$ if $n=m$ 
and is~$0$ if $n\neq m$. The factor of least connectivity in $\Omega P^{2n}(p^{r})$ is 
the homotopy fibre of the degree~$p^{r}$ map on $S^{2n-1}$, which does retract off 
$\Omega V$ for a certain $3$-cell complex $V$, but the factor of least connectivity 
in $\Omega P^{2n+1}(p^{r})$ is a space constructed by Cohen, Moore and Neisendorfer~\cite{CMN} 
whose mod-$p$ homology is much more complex and is not a factor of $\Omega V$ 
for some $3$-cell complex $V$. So we approach 
the problem from a different perspective. Instead of trying to find a factor of least 
connectivity that is indecomposable, we are content to find a copy of $\Omega P^{n+1}(p^{r})$ 
in $\Omega M$ and aim to identify the complementary factor.  


In doing this we consider a much larger family of examples, most of which are 
not Poincar\'{e} Duality complexes. A general result is proved in Theorem~\ref{E'typeII} which 
is then increasingly specialized. In the case presented below, the attaching map $f$ in a 
homotopy cofibration 
\(\nameddright{S^{2n}}{f}{\bigvee_{i=1}^{m} P^{n+1}(p^{r})}{}{M}\) 
factors through Whitehead products and so composes trivially with the pinch map 
\(\namedright{\bigvee_{i=1}^{m} P^{n+1}(p^{r})}{q_{1}}{P^{n+1}(p^{r})}\) 
to the first wedge summand. Therefore $q_{1}$ extends to a map 
\(\namedright{M}{q'}{P^{n+1}(p^{r})}\). 
For $1\leq k\leq m$, let 
\[i_{k}\colon\namedright{P^{n+1}(p^{r})}{}{\bigvee_{i=1}^{m} P^{n+1}(p^{r})}\] 
be the inclusion of the $k^{th}$-wedge summand. Note that the Whitehead product 
$[i_{j},i_{k}]$ is a map 
\(\namedright{\Sigma P^{n}(p^{r})\wedge P^{n}(p^{r})}{}{\bigvee_{i=1}^{m} P^{n+1}(p^{r})}\). 
There is a map 
\(\namedright{S^{2n}}{v}{\Sigma P^{n}(p^{r})\wedge P^{n}(p^{r})}\) 
which induces an injection in mod-$p$ homology. 

\begin{theorem}[appearing in the text as Theorem~\ref{PDexintro2}]
   \label{PDexintro} 
   Let $p$ be an odd prime, $r\geq 1$ and $n\geq 2$. Suppose that there is a homotopy cofibration 
   \[\nameddright{S^{2n}}{f}{\bigvee_{i=1}^{m} P^{n+1}(p^{r})}{}{M}\] 
   where $f=\sum_{1\leq j<k\leq m} [i_{j},i_{k}]\circ (d_{j,k}\cdot v)$ for $d_{j,k}\in\mathbb{Z}$ 
   and at least one $d_{j,k}$ reduces to a unit mod-$p$. Rearranging the wedge summands 
   $\bigvee_{i=1}^{m} P^{n+1}(p^{r})$ so that some $d_{1,t}$ reduces to a unit mod-$p$, 
   there is a homotopy fibration 
   \[\nameddright{(\Omega P^{n+1}(p^{r})\ltimes\overline{C})\vee(\bigvee_{i=2}^{m} P^{n+1}(p^{r}))} 
        {}{M}{q'}{P^{n+1}(p^{r})}\] 
    where 
    $\overline{C}\simeq \bigg(P^{n}(p^{r})\wedge(\displaystyle\bigvee_{\substack{i=2 \\ i\neq t}}^{m} 
         P^{n+1}(p^{r}))\bigg)\vee \bigg(S^{2n+1}\vee P^{2n}(p^{r})\bigg)$, 
    and this homotopy fibration splits after looping to give a homotopy equivalence 
    \[\Omega M\simeq\Omega P^{n+1}(p^{r})\times
           \Omega\bigg((\Omega P^{n+1}(p^{r})\ltimes\overline{C})\vee(\bigvee_{i=2}^{m} P^{n+1}(p^{r}))\bigg).\] 
\end{theorem} 

The interpretation of Theorem~\ref{PDexintro} requires care. Some of the spaces $M$ 
are Poincar\'{e} Duality complexes while others are not, and not all $(n-1)$-connected 
$(2n+1)$-dimensional Poincar\'{e} Duality complexes with $H_{n}(M;\mathbb{Z})$ consisting 
only of odd torsion have the form described in the theorem. But there are tangible 
results. For example, simply-connected $5$-dimensional Poincar\'{e} Duality complexes 
have been classified by St\"{o}cker~\cite{St}. The classification shows that if $M$ is a Spin manifold 
and $H_{2}(M;\mathbb{Z})$ is a direct sum of $\mathbb{Z}/p^{r}\mathbb{Z}$'s for $p$ odd, 
then the attaching map for its top cell has the form described in Theorem~\ref{PDexintro}. 
If $M$ is either a non-Spin manifold or a Poincar\'{e} Duality complex that is not a manifold, 
then the attaching map for the top cell involves a stable term and the theorem does not apply. 
\medskip 

\noindent 
\textbf{Connected sums}. 
A classical problem in homotopy theory is to determine the effect on 
a $CW$-complex $X$, or its loop space $\Omega X$, by attaching 
a cell. Rational homotopy theory has had some success in this direction 
for certain families of attaching maps. Let 
\(\nameddright{S^{n-1}}{f}{X}{i}{X\cup e^{n}}\) 
be a cofibration where $f$ attaches an $n$-cell to $X$ and $i$ is the 
inclusion. The map $f$ is \emph{inert} if $\Omega i$ induces an epimorphism 
in rational homology. This implies that, rationally, $\Omega i$ has a right 
homotopy inverse. Inert maps have received notable attention, for example, 
in~\cite{FT1, HaL}, as have assorted variants such as nice, lazy and semi-inert 
attaching maps~\cite{Bu, HeL}. 

We consider an integral version of an inert map, and generalize 
from attaching a cell to attaching a cone, that is, to a cofibration 
\(\nameddright{A}{f}{X}{i}{X\cup CA}\).  
Modifiying, we consider a homotopy cofibration 
\(\nameddright{\Sigma A}{f}{X}{i}{X'}\) 
where all spaces are assumed to be simply-connected and have the homotopy type of 
$CW$-complexes. The map $f$ is \emph{inert} if $\Omega i$ has a right homotopy inverse. 
Note there is no localization hypothesis here. 

Let 
\(\nameddright{\Sigma A}{g}{Y}{}{Y'}\) 
be another such cofibration, where $g$ need not be inert. As $\Sigma A$ is a 
suspension we may add to obtain 
\(\namedright{\Sigma A}{f+g}{X\vee Y}\). 
In Theorem~\ref{inertideal} we show that $f+g$ is inert. If $C$ is the homotopy 
cofibre of $f+g$ then we give a homotopy decomposition for $\Omega C$ in terms of~$X$ 
and~$Y'$ and prove additional related statements. The property that $f+g$ is inert, 
regardless of whether~$g$ is inert, is intriguing. 

As a special case we consider the connected sum $M\conn N$ of two Poincar\'{e} Duality spaces $M$ 
and $N$ of the same dimension. It is natural to ask how the homotopy type of $M\conn N$ 
reflects the homotopy types of $M$ and $N$. Theorem~\ref{introconnsum} provides an answer. 

Let $X$ and $Y$ be the $(n-1)$-skeletons of $M$ and $N$ respectively. Then there are 
homotopy cofibrations 
\(\nameddright{S^{n-1}}{f}{X}{h}{M}\) 
and 
\(\nameddright{S^{n-1}}{g}{Y}{k}{N}\) 
that attach the top cells to $M$ and $N$ respectively. The connected sum of $M$ and $N$ 
has $(n-1)$-skeleton $X\vee Y$ and the attaching map for its top cell is $f+g$. We prove, 
among other properties, the following. 

\begin{theorem}[appearing in the text subsumed within Theorem~\ref{connsum}]  
   \label{introconnsum} 
   Let $M$ and $N$ be simply-connected Poincar\'{e} Duality complexes of dimension $n$, 
   where $n\geq 2$. Let $X$ and $Y$ be the $(n-1)$-skeletons of $M$ and $N$ respectively. 
   If the inclusion 
   \(\namedright{X}{h}{M}\) 
   has the property that $\Omega h$ has a right homotopy inverse, then the following hold: 
   \begin{letterlist} 
      \item there is a homotopy equivalence 
              $\Omega(M\conn N)\simeq\Omega M\times\Omega(\Omega M\ltimes Y)$; 
      \item the map 
               \(\namedright{X\vee Y}{}{M\conn N}\) 
               has a right homotopy inverse after looping. 
    \end{letterlist} 
\end{theorem} 

In particular, if 
\(\namedright{X}{h}{M}\) 
has the property that $\Omega h$ has a right homotopy inverse, then so does 
\(\namedright{X\vee Y}{}{M\conn N}\), 
regardless of whether 
\(\namedright{Y}{k}{N}\) 
has that property. The homotopy fibrations 
\(\nameddright{F}{}{X}{h}{M}\) 
and 
\(\nameddright{G}{}{X\vee Y}{}{M\conn N}\) 
then both have sections after looping and fit into the framework of the paper. 

Interesting examples include connected sums of products of two spheres, which 
play an important role in toric topology~\cite{BM,GPTW,GIPS}. Another example 
would take a connected sum of products of two spheres and take its connected sum 
with a complex projective space of the same dimension. Many more examples are 
considered in Section~\ref{sec:connsum}.

\medskip 

\noindent 
\textbf{Polyhedral products}. 
Let $K$ be a simplicial complex on $m$ vertices.  For $1\leq i\leq m$,
let $(X_{i},A_{i})$ be a pair of pointed $CW$-complexes, where $A_{i}$ is 
a pointed subspace of~$X_{i}$. Let $\uxa=\{(X_{i},A_{i})\}_{i=1}^{m}$ be 
an $m$-tuple of $CW$-pairs. For each simplex (face) $\sigma\in K$, let 
$\uxa^{\sigma}$ be the subspace of $\prod_{i=1}^{m} X_{i}$ defined by
\[\uxa^{\sigma}=\prod_{i=1}^{m} Y_{i}\qquad
       \mbox{where}\qquad Y_{i}=\left\{\begin{array}{ll}
                                             X_{i} & \mbox{if $i\in\sigma$} \\
                                             A_{i} & \mbox{if $i\notin\sigma$}.
                                       \end{array}\right.\]
The \emph{polyhedral product} determined by \uxa\ and $K$ is
\[\uxa^{K}=\bigcup_{\sigma\in K}\uxa^{\sigma}\subseteq\prod_{i=1}^{m} X_{i}.\] 
For example, suppose each $A_{i}$ is a point. If $K$ is a disjoint union
of $m$ points then $(\underline{X},\underline{\ast})^{K}$ is the wedge
$X_{1}\vee\cdots\vee X_{m}$, and if $K$ is the standard $(m-1)$-simplex
then $(\underline{X},\underline{\ast})^{K}$ is the product
$X_{1}\times\cdots\times X_{m}$. 

Polyhedral products are currently a subject of intense study. They are at the locus 
of several constructions from disparate areas of mathematics: moment-angle complexes 
in toric topology, complements of complex coordinate subspace arrangements in 
combinatorics, monomial rings with the Golod property in commutative algebra, 
intersections of quadrics in complex geometry, and Bestvina-Brady groups in geometric 
group theory. 

An important problem is to study the connection between Whitehead 
products (and higher Whitehead products) and polyhedral products. There has 
been significant headway on this in the context of the homotopy fibration 
\[\nameddright{\sclxx^{K}}{}{\sux^{K}}{}{\prod_{i=1}^{m}\Sigma X_{i}}.\] 
Here, $\sclxx^{K}$ is the polyhedral product formed from the pairs 
$(C\Omega\Sigma X_{i},\Omega\Sigma X_{i})$ 
where $C\Omega\Sigma X_{i}$ is the reduced cone on $\Omega\Sigma X_{i}$. 
In a sequence of papers~\cite{GT1,GT2,IK1,IK2} leading up to $K$ satisfying 
the combinatorial condition of being totally fillable (this includes shifted complexes 
and Alexander duals of shellable complexes) the space $\sclxx^{K}$ is shown to be homotopy equivalent 
to a wedge of spaces of the form $\Sigma^{t} X_{i_{1}}\wedge\cdots\wedge X_{i_{k}}$ for 
various $t\geq 1$ and $1\leq i_{1}<\cdots<i_{k}\leq m$. In~\cite{GT3,GPTW} for special 
cases and~\cite{AP,IK3} more generally, under such a decomposition the map 
\(\namedright{\sclxx^{K}}{}{\sux^{K}}\) 
is a wedge sum of iterated Whitehead products of the form 
$[v_{i_{k}},[\cdots[v_{i_{1}},w]\ldots]$ where each $v_{i_{k}}$ represents the inclusion 
of $\Sigma X_{i_{k}}$ into $\sux^{K}$ induced by the inclusion of the vertex $i_{k}$ into $K$, 
and $w$ is a higher Whitehead product corresponding to a (minimal) missing face of $K$. 

We go a step further by showing that Whitehead and higher Whitehead products 
are pervasive in the formation of the polyhedral products, regardless of whether 
$K$ is totally fillable. If a set of (minimal) missing faces is attached to $K$ to form 
a new simplicial complex $\overline{K}$, then we show that on the level of polyhedral 
products there is a corresponding homotopy cofibration 
\[\nameddright{\Sigma A}{f}{\sux^{K}}{}{\sux^{\overline{K}}}.\] 
The inclusion 
\(\namedright{\sux^{K}}{}{\prod_{i=1}^{m}\Sigma X_{i}}\) 
has a right homotopy inverse after looping, and so fits into the overall framework of 
the paper. Consider the homotopy cofibration 
\(\nameddright{(\prod_{i=1}^{m}\Omega\Sigma X_{i})\ltimes\Sigma A}{\theta}{E}{}{E'}\) 
from~(\ref{thetacofintro}). On the one hand, in this context $E=\sclxx^{K}$ and $E'=\sclxx^{\overline{K}}$, 
and on the other hand, the James construction implies that there is a homotopy equivalence 
\[(\prod_{i=1}^{m}\Omega\Sigma X_{i})\ltimes\Sigma A\simeq 
        \displaystyle\bigvee_{k=0}^{\infty}\ \bigvee_{1\leq i_{1}\leq\cdots\leq i_{k}\leq m} 
              (X_{i_{1}}\wedge\cdots\wedge X_{i_{k}})\wedge\Sigma A.\] 
This homotopy equivalence can be chosen so the following holds, where the homotopy 
between suspended maps reflects the same feature in Theorem~\ref{dbardintro}.  

\begin{theorem}[appearing in the text as Theorem~\ref{polyWh}]  
   \label{polyWhintro} 
   There is a homotopy cofibration 
   \[\nameddright{\displaystyle\bigvee_{k=0}^{\infty}\ \bigvee_{1\leq i_{1}\leq\cdots\leq i_{k}\leq m} 
              (X_{i_{1}}\wedge\cdots\wedge X_{i_{k}})\wedge\Sigma A}{\zeta}{\sclxx^{K}} 
              {}{\sclxx^{\overline{K}}}\] 
    where the map $\zeta$ has the property that $\Sigma\zeta\simeq\Sigma\zeta'$ for a map $\zeta'$ 
    satisfying a homotopy commutative diagram 
    \[\diagram 
            \displaystyle\bigvee_{k=0}^{\infty}\ \bigvee_{1\leq i_{1}\leq\cdots\leq i_{k}\leq m} 
              (X_{i_{1}}\wedge\cdots\wedge X_{i_{k}})\wedge\Sigma A\rto^-{\zeta'} 
                    \drto_-(0.6){\bigvee_{k=1}^{\infty}\ \bigvee_{1\leq i_{1}\leq\cdots\leq i_{k}\leq m} 
                    [v_{i_{1}},[v_{i_{2}},[\cdots [v_{i_{k}},f]]\cdots ]\hspace{2cm}} 
                 & \sclxx^{K}\dto^{p} \\ 
            & \sux^{K}.   
        \enddiagram\] 
\end{theorem} 
\medskip 

This paper is organized as follows. Section~\ref{sec:background} reviews the results 
in~\cite{BT2} that will be needed later. Theorem~\ref{dbardintro} is proved in 
Section~\ref{sec:ad}. Section~\ref{sec:2cone} 
then considers two-cones, proves Theorem~\ref{Mtypealtintro}, 
and relates the results to Moore's Conjecture. Section~\ref{sec:improve} proves a general 
decomposition result in Theorem~\ref{E'typeII} and in Section~\ref{sec:decompapps} this is 
specialized and applied to certain families of two-cones. Section~\ref{sec:PD} is a modification 
of the results in Section~\ref{sec:decompapps} that leads to the proof of Theorem~\ref{PDexintro} 
and applications to loop space decompositions of $(n-1)$-connected $(2n+1)$-dimensional 
Poincar\'{e} Duality complexes which are rationally copies of $S^{2n+1}$. Section~\ref{sec:inert} 
builds on the notion of an inert map and proves a general decomposition result in 
Theorem~\ref{inertideal}, while Section~\ref{sec:connsum} specializes this to prove 
Theorem~\ref{introconnsum} and give an array of examples. Section~\ref{sec:Hopf} 
turns momentarily to algebra to calculate $\hlgy{\Omega Y}$ as a Hopf algebra, where 
\(\nameddright{E}{p}{Y}{h}{Z}\) 
is a homotopy fibration with a section after looping. In Section~\ref{sec:iteratedad} 
we return to homotopy theory to address a second foundational case involving extensions 
across the inclusion of a wedge into a product, the James construction and Whitehead products 
that leads to the explicit description of Whitehead products in toric topology that is stated in 
Theorem~\ref{polyWhintro} and proved in Section~\ref{sec:polyprod}. 

It is also useful to have a guide on the sections needed to prove each of the main 
theorems. Theorem~\ref{dbardintro} appears in Section~\ref{sec:ad} and depends 
only on Section~\ref{sec:background}. Theorem~\ref{Mtypealtintro} appears in 
Section~\ref{sec:2cone} and depends on Sections~\ref{sec:background} and~\ref{sec:ad}. 
Theorem~\ref{PDexintro} appears in Section~\ref{sec:PD} and depends on 
Sections~\ref{sec:background} through~\ref{sec:decompapps}. Theorem~\ref{introconnsum} 
appears in Section~\ref{sec:connsum} and depends on Sections~\ref{sec:background}, 
\ref{sec:ad} and~\ref{sec:inert}. The homological interlude in Section~\ref{sec:Hopf} 
depends on Sections~\ref{sec:background} and~\ref{sec:ad}. Theorem~\ref{polyWhintro} 
appears in Section~\ref{sec:polyprod} and depends 
only on Sections~\ref{sec:background}, ~\ref{sec:ad} and~\ref{sec:iteratedad}. 

The author would like to thank the referee for a careful reading of the paper, many 
valuable suggestions for improvement, and spotting a gap in the original submission, 
the filling of which led to more interesting mathematics.

\newpage 

\section{Background} 
\label{sec:background} 

This section discusses the ingredients behind Theorem~\ref{GTcofib} as they relate 
to material that will come later in the paper. Recall the blanket assumption that all 
spaces are assumed to be $CW$-complexes so that weak homotopy equivalences 
are homotopy equivalences. We start with a well known lemma and some notation. 

The \emph{left half-smash} of two path-connected spaces $A$ and $B$ is the quotient space 
defined by the cofibration 
\(\nameddright{A}{i_{1}}{A\times B}{}{A\ltimes B}\) 
where $i_{1}$ is the inclusion of the first factor. 

\begin{lemma} 
   \label{halfsmashsusp} 
   Let $A$ and $B$ be pointed, path-connected spaces. Then there is a homotopy equivalence 
   \[A\ltimes\Sigma B\simeq (A\wedge\Sigma B)\vee\Sigma B\] 
   which is natural for maps 
   \(\namedright{A}{}{A'}\) 
   and 
   \(\namedright{B}{}{B'}\). 
   $\qqed$ 
\end{lemma} 


For path-connected spaces $A$ and $B$, let 
\[j\colon\namedright{B}{}{A\ltimes B}\] 
be the inclusion and let 
\[q\colon\namedright{A\ltimes B}{}{A\wedge B}\] 
be the quotient map that collapses $B$ to a point. By Lemma~\ref{halfsmashsusp} 
there is a natural map 
\[i\colon\namedright{A\wedge\Sigma B}{}{A\ltimes\Sigma B}\] 
which is a right homotopy inverse for $\Sigma q$.  

Suppose that  
\(f\colon\namedright{A}{}{Y}\) 
and 
\(g\colon\namedright{B}{}{Y}\) 
are maps. The map from $A\vee B$ to $Y$ determined by $f$ and $g$ is denoted by 
\[f\perp g\colon\namedright{A\vee B}{}{Y}.\] 
In particular, a choice of the homotopy equivalence in Lemma~\ref{halfsmashsusp} is given by $i\perp j$. 

Given maps 
\(a\colon\namedright{\Sigma A}{}{Y}\) 
and 
\(b\colon\namedright{\Sigma B}{}{Y}\) 
let 
\([a,b]\colon\namedright{\Sigma A\wedge B}{}{Y}\) 
be the Whitehead product of $a$ and $b$. In what follows we will consider a map 
\(\namedright{\Omega Z\ltimes\Sigma A}{}{E}\) 
with the property that the composite 
\(\nameddright{\Omega Z\wedge\Sigma A}{i}{\Omega Z\ltimes\Sigma A}{}{E}\) 
is a Whitehead product $[\gamma,f]$ for some maps $\gamma$ and $f$. As such, 
in these cases we prefer to write the Whitehead product with domain $\Omega Z\wedge\Sigma A$ 
rather than $\Sigma\Omega Z\wedge A$ to emphasize the link to the half-smash. 

The following was proved in~\cite{BT2}. Let 
\(ev\colon\namedright{\Sigma\Omega Y}{ev}{Y}\) 
be the canonical evaluation map. 

\begin{theorem} 
   \label{GTcofib} 
   Suppose that there is a homotopy fibration 
   \(\nameddright{E}{p}{Y}{h}{Z}\)  
   and a homotopy cofibration 
   \(\nameddright{\Sigma A}{f}{Y}{}{Y'}\). 
   Suppose that $h$ extends to a map 
   \(h'\colon\namedright{Y'}{}{Z}\) 
   and let $E'$ be the homotopy fibre of $h'$. This data is assembled 
   into a diagram 
   \begin{equation} 
     \label{data1} 
     \diagram 
          & E\rto\dto^{p} & E'\dto^{p'} \\ 
          \Sigma A\rto^-{f} & Y\rto\dto^{h} & Y'\dto^{h'} \\ 
          & Z\rdouble & Z. 
     \enddiagram 
   \end{equation} 
   where the vertical columns and the maps between them form a homotopy fibration diagram. 
   Suppose in addition that the map 
   \(\namedright{\Omega Y}{\Omega h}{\Omega Z}\) 
   has a right homotopy inverse 
   \(s\colon\namedright{\Omega Z}{}{\Omega Y}\). 
   Then there is a map 
   \(\theta\colon\namedright{\Omega Z\ltimes\Sigma A}{}{E}\) 
   such that:  
   \begin{letterlist} 
      \item there is a homotopy cofibration 
               \[\nameddright{\Omega Z\ltimes\Sigma A}{\theta}{E}{}{E'};\] 
      \item there is a homotopy commutative diagram 
               \[\diagram 
                    \Omega Z\ltimes\Sigma A\rto^-{\theta}\dto^{(i\perp j)^{-1}} & E\dto^{p} \\ 
                         (\Omega Z\wedge\Sigma A)\vee\Sigma A\rto^-{[\gamma,f]\perp f} 
                        & Y   
                \enddiagram\] 
              where $\gamma$ is the composite 
              \(\gamma\colon\nameddright{\Sigma\Omega Z}{\Sigma s}{\Sigma\Omega Y}{ev}{Y}\).  
   \end{letterlist} 
\end{theorem} 
\vspace{-1cm}~$\qqed$\medskip 

For the benefit of the reader, and to make explicit parts of the construction that will be 
used later on, a sketch of the proof of Theorem~\ref{GTcofib} will be given. The key 
is a link between two seemingly distinct constructions. First, let 
\(\namedright{\Omega Z}{\partial}{E}\) 
be the connecting map for the homotopy fibration 
\(\nameddright{E}{p}{Y}{h}{Z}\). 
There is a canonical homotopy action  
\[a\colon\namedright{\Omega Z\times E}{}{E}\] 
which extends the map 
\(\namedright{\Omega Z\vee E}{\partial\vee 1}{E}\). 
The composite 
\(\lnameddright{\Omega Y\times E}{\Omega h\times 1}{\Omega Z\times E}{a}{E}\) 
therefore has the property that its restriction to $\Omega Y$ is null homotopic, resulting 
in a quotient map 
\[\Theta\colon\namedright{\Omega Y\ltimes E}{}{E}.\] 
Second, it is well known that the homotopy fibre of the pinch map 
\(\namedright{Y\vee E}{}{Y}\) 
is naturally homotopy equivalent to $\Omega Y\ltimes E$. From this we obtain a homotopy fibration diagram 
\begin{equation} 
  \label{YEGammadgrm} 
  \diagram 
        \Omega Y\ltimes E\rto^-{\Gamma}\dto & E\dto^{p} \\ 
        Y\vee E\rto^-{1\perp p}\dto & Y\dto^{h} \\ 
        Y\rto^-{h} & Z 
  \enddiagram 
\end{equation}  
for some map $\Gamma$. The link between the two constructions is the following, proved in~\cite{Gr}. 

\begin{lemma} 
   \label{Gray} 
   The maps $\Theta$ and $\Gamma$ may be chosen so that they are homotopic..~$\qqed$ 
\end{lemma} 
   
 Assume from now on that $\Theta$ and $\Gamma$ have been chosen so that Lemma~\ref{Gray} 
 holds. We will discuss some general properties through to Proposition~\ref{thetaWh}, and then use 
 the material developed to sketch a proof of Theorem~\ref{GTcofib}. Suppose for some space $B$ 
 there is a map 
\(\namedright{\Sigma B}{\alpha}{E}\). 
One example of this will be $B=A$, where $A$ is as in the data for Theorem~\ref{GTcofib} 
and $\alpha$ will be an appropriate lift for $f$, but other examples are also needed in 
Section~\ref{sec:ad}. The naturality of the homotopy fibration 
\(\nameddright{\Omega Y\ltimes E}{}{Y\vee E}{}{Y}\) 
implies that there is a homotopy commutative diagram  
\begin{equation} 
  \label{Gammadgrm1} 
  \diagram 
        \Omega Y\ltimes\Sigma B\rto^-{1\ltimes\alpha}\dto 
        & \Omega Y\ltimes E\rto^-{\Gamma}\dto & E\dto^{p} \\ 
        Y\vee\Sigma B\rto^-{1\vee\alpha} 
        & Y\vee E\rto^-{1\perp p} & Y.  
\enddiagram 
\end{equation}  
By Lemma~\ref{halfsmashsusp} and its use in defining the map $i$, there is a natural homotopy equivalence 
\[\namedright{(\Omega Y\wedge\Sigma B)\vee\Sigma B}{i\perp j}{\Omega Y\ltimes\Sigma B}.\]  
The composite 
\(\nameddright{\Sigma B}{j}{\Omega Y\ltimes\Sigma B}{}{Y\vee\Sigma B}\) 
is the inclusion $i_{2}$ of the second wedge summand. The composite 
\(\nameddright{\Omega Y\wedge\Sigma B}{i}{\Omega Y\ltimes\Sigma B}{}{Y\vee\Sigma B}\) 
can also be identified; it is a certain Whitehead product. 

To motivate the appearance of Whitehead products, in general suppose 
that $X_{1}$ and $X_{2}$ are pointed, path-connected spaces. For $j=1,2$, let 
\[i_{j}\colon\namedright{X_{j}}{}{X_{1}\vee X_{2}}\] 
be the inclusion of the $j^{th}$-wedge summand, and let $ev_{j}$ be the composite 
\[ev_{j}\colon\nameddright{\Sigma\Omega X_{j}}{ev}{X_{j}}{i_{j}}{X_{1}\vee X_{2}}.\] 
Ganea~\cite{Ga} showed that there is a homotopy fibration 
\[\llnameddright{\Sigma\Omega X_{1}\wedge\Omega X_{2}}{[ev_{1},ev_{2}]} 
     {X_{1}\vee X_{2}}{}{X_{1}\times X_{2}}\] 
where the right map is the inclusion of the wedge into the product and the left map is the 
Whitehead product of $ev_{1}$ and $ev_{2}$. 
Consider the composite 
\[\llnameddright{\Sigma\Omega X_{1}\wedge\Omega X_{2}}{[ev_{1},ev_{2}]} 
     {X_{1}\vee X_{2}}{q_{1}}{X_{1}}\]
where $q_{1}$ is the pinch map to the first wedge summand. The naturality of 
the Whitehead product implies that this composite is homotopic to $[q_{1}\circ ev_{1},q_{1}\circ ev_{2}]$. 
But $q_{1}\circ ev_{2}=q_{1}\circ i_{2}\circ ev$ is null homotopic since $q_{1}\circ i_{2}$ is. 
Therefore $q_{1}\circ [ev_{1},ev_{2}]$ is null homotopic, so there is a lift 
\[\diagram 
      & \Omega X_{1}\ltimes X_{2}\dto \\ 
      \Sigma\Omega X_{1}\wedge\Omega X_{2}\rto^-{[ev_{1},ev_{2}]}\urto^-{\xi} & X_{1}\vee X_{2} 
  \enddiagram\] 
for some map $\xi$. Suppose that $X_{2}$ is a suspension, $X_{2}\simeq\Sigma X'_{2}$, and let 
\(E\colon\namedright{X'_{2}}{}{\Omega\Sigma X'_{2}}\) 
be the suspension map, which is adjoint to the identity map on $\Sigma X'_{2}$. Then  
we may precompose $[ev_{1},ev_{2}]$ with 
\(\namedright{\Sigma\Omega X_{1}\wedge X'_{2}}{\Sigma 1\wedge E} 
      {\Sigma\Omega X_{1}\wedge\Omega\Sigma X'_{2}}\). 
The naturality of the Whitehead product implies that 
$[ev_{1},ev_{2}]\circ(\Sigma 1\wedge E)\simeq [ev_{1},ev_{2}\circ\Sigma E]$. Since 
$\Sigma E$ is a left homotopy inverse for $ev$, we have 
$ev_{2}\circ\Sigma E=i_{2}\circ ev\circ\Sigma E\simeq i_{2}$. Combining this with the lift $\xi$ 
gives a homotopy commutative diagram 
\[\diagram 
      & \Omega X_{1}\ltimes\Sigma X'_{2}\dto \\ 
      \Sigma\Omega X_{1}\wedge X'_{2}\rto^-{[ev_{1},i_{2}]}\urto^-{\xi'} & X_{1}\vee\Sigma X'_{2} 
  \enddiagram\] 
where $\xi'=\xi\circ(\Sigma 1\wedge E)$. Writing $\Sigma\Omega X_{1}\wedge X'_{2}$ as 
$\Omega X_{1}\wedge\Sigma X'_{2}$, in~\cite{BT2} it is shown that $\xi$ can be chosen 
so that $\xi'$ is the map 
\(\namedright{\Omega X_{1}\wedge\Sigma X'_{2}}{i}{\Omega X_{2}\ltimes\Sigma X'_{2}}\). 

Thus, returning to the case of $Y\vee\Sigma B$, there is a homotopy commutative diagram 
\begin{equation} 
  \label{Gammadgrm2} 
  \diagram 
       (\Omega Y\wedge\Sigma B)\vee\Sigma B\rto^-{i\perp j}\drto_{[ev_{1},i_{2}]\perp i_{2}} 
             & \Omega Y\ltimes\Sigma B\dto \\ 
       & Y\vee\Sigma B 
  \enddiagram 
\end{equation}   
and the map $i\perp j$ along the top row is a homotopy equivalence. 
Combining~(\ref{Gammadgrm1}) and~(\ref{Gammadgrm2}) and using the naturality 
of the Whitehead product results in a homotopy commutative diagram 
\begin{equation} 
  \label{Gammadgrm3} 
  \diagram 
      (\Omega Y\wedge\Sigma B)\vee\Sigma B\rto^-{\Psi}\drto_{[ev,p\circ\alpha]\perp(p\circ\alpha)} 
             & E\dto^{p} \\
          & Y  
   \enddiagram 
\end{equation} 
where $\Psi=\Gamma\circ(1\ltimes\alpha)\circ(i\perp j)$. 

Next, suppose that the map 
\(\namedright{Y}{h}{Z}\) 
in Theorem~\ref{GTcofib} has the additional property that $\Omega h$ has a right homotopy inverse 
\(s\colon\namedright{\Omega Z}{}{\Omega Y}\). 
Then the fibration connecting map 
\(\namedright{\Omega Z}{}{E}\) 
is null homotopic, so the homotopy action 
\(\namedright{\Omega Z\times E}{a}{E}\) 
factors as a composite 
\[\nameddright{\Omega Z\times E}{\pi}{\Omega Z\ltimes E}{}{E}\] 
where $\pi$ is the quotient map and the right map is a choice of extension. The next 
lemma shows that an extension may be chosen to be the composite 
\[\overline{a}\colon\nameddright{\Omega Z\ltimes E}{s\ltimes 1}{\Omega Y\ltimes E}{\Gamma}{E}.\] 

\begin{lemma} 
   \label{abarGamma} 
   Suppose that the homotopy fibration  
   \(\nameddright{E}{}{Y}{h}{Z}\) 
   has the property that $\Omega h$ has a right homotopy inverse 
   \(s\colon\namedright{\Omega Z}{}{\Omega Y}\). 
   Then there is a homotopy commutative diagram 
   \[\diagram 
         \Omega Z\times E\rto^-{a}\dto^{\pi} & E\ddouble \\ 
         \Omega Z\ltimes E\rto^-{\overline{a}} & E. 
      \enddiagram\] 
\end{lemma} 
 
\begin{proof} 
Consider the diagram 
\[\diagram 
       \Omega Z\times E\rto^-{s\times 1}\dto^{\pi} 
           & \Omega Y\times E\rto^-{\Omega h\times 1}\dto^{\pi}  
           & \Omega Z\times E\rto^-{a} & E\ddouble \\  
       \Omega Z\ltimes E\rto^-{s\ltimes 1} & \Omega Y\ltimes E\rrto^-{\Theta} & & E.  
  \enddiagram\] 
The left square commutes by the naturality of $\pi$ while the right square commutes by 
definition of~$\Theta$. Since $s$ is a right homotopy inverse of $\Omega h$, the top row 
is homotopic to $a$. Thus the homotopy commutativity of the diagram implies that 
$a\simeq\Theta\circ(s\ltimes 1)\circ\pi$. By Lemma~\ref{Gray}, $\Theta\circ\Gamma$, 
and by definition, $\overline{a}=\Gamma\circ(s\ltimes 1)$. Therefore 
Therefore $a\simeq\Gamma\circ(s\ltimes 1)\circ\pi=\overline{a}\circ\pi$, giving the 
asserted homotopy commutative diagram.  
\end{proof} 

Define $\vartheta$ by the composite 
\[\vartheta\colon\nameddright{\Omega Z\ltimes\Sigma B}{s\ltimes\alpha}{\Omega Y\ltimes E}{\Gamma}{E}.\] 
Note that as $(s\ltimes\alpha)=(1\ltimes\alpha)\circ (s\ltimes 1)$, the definitions of $\theta$ and $\overline{a}$ 
immediately imply the following. 

\begin{lemma} 
   \label{varthetabara} 
   Suppose that the homotopy fibration  
   \(\nameddright{E}{}{Y}{h}{Z}\) 
   has the property that $\Omega h$ has a right homotopy inverse 
   \(s\colon\namedright{\Omega Z}{}{\Omega Y}\). 
   Then for any map 
   \(\namedright{\Sigma B}{\alpha}{E}\) 
   there is a homotopy commutative diagram 
   \[\diagram 
         \Omega Z\ltimes\Sigma B\drto^{\vartheta}\dto^{1\ltimes\alpha} & \\ 
         \Omega Z\ltimes E\rto^-{\overline{a}} & E. 
     \enddiagram\] 
\end{lemma} 
\vspace{-1cm}~$\qqed$\bigskip 

The map $\vartheta$ is the bridge between the homotopy action, in the form 
of $\overline{a}$, and Whitehead products as in~(\ref{Gammadgrm3}). 

\begin{proposition} 
   \label{thetaWh} 
   Suppose that there is a homotopy fibration sequence 
   \(\namedddright{\Omega Z}{\partial}{E}{p}{Y}{h}{Z}\) 
   where $\Omega h$ has a right homotopy inverse 
   \(s\colon\namedright{\Omega Z}{}{\Omega Y}\). 
   Let 
   \(\alpha\colon\namedright{\Sigma B}{}{E}\) 
   be a map. Then there is a homotopy commutative diagram 
   \[\diagram 
         (\Omega Z\wedge\Sigma B)\vee\Sigma B\rto^-{i\perp j}\drrto_{[ev\circ\Sigma s,p\circ\alpha]\perp p\circ\alpha} 
              & \Omega Z\ltimes\Sigma B\rto^-{\vartheta} & E\dto^{p} \\ 
          & & Y. 
      \enddiagram\] 
\end{proposition} 

\begin{proof} 
Consider the diagram 
\[\diagram 
     (\Omega Z\wedge\Sigma B)\vee\Sigma B\rto^-{(s\wedge 1)\vee 1}\dto^{i\perp j} 
          & (\Omega Y\wedge\Sigma B)\vee\Sigma B\dto^{i\perp j} 
                   \xto[drrr]^{[ev,p\circ\alpha]\perp(p\circ\alpha)} & & & \\ 
     \Omega Z\ltimes\Sigma B\rto^-{s\ltimes 1} 
          & \Omega Y\ltimes\Sigma B\rto^-{\Gamma\circ(1\ltimes\alpha)} 
          & E\rrto^{p} & & Y.
  \enddiagram\] 
The left square homotopy commutes by the naturality of $i$ and $j$ while the right triangle 
homotopy commutes by~(\ref{Gammadgrm3}). The composite 
$\Gamma\circ(1\ltimes\alpha)\circ(s\ltimes 1)=\Gamma\circ(s\ltimes\alpha)$ along the bottom row 
is the definition of $\vartheta$. The naturality of the Whitehead product implies that the composite 
in the upper direction around the diagram is $[ev\circ\Sigma s,p\circ\alpha]\perp(p\circ\alpha)$. Thus 
the homotopy commutativity of the diagram implies that 
$p\circ\theta\circ(i\perp j)\simeq [ev\circ\Sigma s,p\circ\alpha]\perp(p\circ\alpha)$, as asserted.  
\end{proof}  

Finally, we justify Theorem~\ref{GTcofib}. In the data for Theorem~\ref{GTcofib}, focus on the map 
\(\namedright{\Sigma A}{f}{Y}\). 
The extension of $h$ to $h'$ implies that $h\circ f$ is null homotopic. Thus $f$ lifts to a map 
\(g\colon\namedright{\Sigma A}{}{E}\). 
However, not every choice of lift $g$ will also make part~(a) of Theorem~\ref{GTcofib} hold. 
For part~(a) to hold, the composite 
\(\nameddright{\Sigma A}{g}{E}{}{E'}\) 
must be null homotopic, or equivalently, $g$ must factor through the homotopy fibre $F$ of 
\(\namedright{E}{}{E'}\). 
In terms of the data in~(\ref{data1}), as the upper square is a homotopy pullback, $F$ is also 
the homotopy fibre of the map 
\(\namedright{Y}{}{Y'}\). 
Since $Y'$ is the homotopy cofibre of 
\(\namedright{\Sigma A}{f}{Y}\), 
the map $f$ lifts to $F$ so for any such lift we may choose $g$ to be the composite 
\(\nameddright{\Sigma A}{}{F}{}{E}\). 
Assume from now on that such a $g$ has been chosen. Then~\cite{BT2} ensures 
that Theorem~\ref{GTcofib}~(a) holds. 

Define $\theta$ by the composite 
\[\theta\colon\nameddright{\Omega Z\ltimes\Sigma A}{s\ltimes g}{\Omega Y\ltimes E}{\Gamma}{E}.\] 
Then $\theta$ is a special case of the map $\vartheta$ in Lemma~\ref{varthetabara}, and as in 
that lemma, there is a homotopy commutative diagram 
\begin{equation} 
  \label{thetabara} 
  \diagram 
      \Omega Z\ltimes\Sigma A\drto^{\theta}\dto^{1\ltimes g} & \\ 
      \Omega Z\ltimes E\rto^-{\overline{a}} & E. 
  \enddiagram 
\end{equation} 
Applying Proposition~\ref{thetaWh} to $g$ and $\theta$ we obtain a homotopy commutative diagram 
\begin{equation} 
  \label{thetaWhnat} 
  \diagram 
      (\Omega Z\wedge\Sigma A)\vee\Sigma A\rto^-{i\perp j}\drrto_{[ev\circ\Sigma s,p\circ g]\perp p\circ g} 
           & \Omega Z\ltimes\Sigma A\rto^-{\theta} & E\dto^{p} \\ 
       & & Y, 
   \enddiagram 
\end{equation} 
which is exactly the statement of Theorem~\ref{GTcofib}~(b).

\begin{remark} 
\label{GTcofibnat} 
Theorem~\ref{GTcofib} also has a naturality property, not explicitly stated in~\cite{BT2}. Suppose 
that there is a map of principal fibrations 
\[\diagram 
      \Omega Z\rto\dto & E\rto^-{p}\dto & Y\dto \\ 
      \Omega\widehat{Z}\rto & \widehat{Z}\rto^-{\widehat{p}} & \widehat{Y}. 
  \enddiagram\] 
The map $\Gamma$ is natural for maps of principal fibrations by~\cite[Proposition 2.9]{BT2} 
and therefore, by its definition, so is $\bar{a}$ provided there is a homotopy commutative 
diagram of right homotopy inverses 
\begin{equation} 
\label{snat} 
  \diagram 
      \Omega Z\rto^-{s}\dto & \Omega Y\dto \\ 
      \Omega\widehat{Z}\rto^-{\widehat{s}} & \Omega\widehat{Y}.  
  \enddiagram 
\end{equation} 
The construction of the lift $g$ in~\cite{BT2} is natural and hence so is the homotopy 
commutative diagram~(\ref{thetabara}). The naturality of the Whitehead product then 
implies that~(\ref{thetaWhnat}) is natural. Thus the homotopy commutative 
diagram in Theorem~\ref{GTcofib}~(b) is natural for maps of principal fibrations with 
compatible right homotopy inverses. Further, the naturality of $\theta\simeq\overline{a}\circ(1\ltimes g)$ 
implies that the homotopy cofibration in Theorem~\ref{GTcofib}~(a) is natural in the 
sense that we obtain a homotopy cofibration diagram  
\[\diagram 
     \Omega Z\ltimes A\rto^-{\theta}\dto & E\rto\dto & E'\ddashed|>\tip \\ 
     \Omega\widehat{Z}\ltimes\widehat{A}\rto^-{\widehat{\theta}} & \widehat{E}\rto & \widehat{E}' 
  \enddiagram\] 
where the left square homotopy commutes by the naturality of $\theta$ and 
the dashed arrow is an induced map of homotopy cofibres that makes the right 
square homotopy commute. 

In conclusion, Theorem~\ref{GTcofib} is natural for maps of principal homotopy fibrations 
with the property that the right homotopy inverses satisfy the homotopy commutative 
diagram~\ref{snat}. 
\end{remark}

\newpage 
\section{The fibre of the pinch map} 
\label{sec:ad} 

In this section we prove Theorem~\ref{dbardintro}. This begins with some information 
on homotopy actions and half-smashes. In Lemma~\ref{halfsmashaction} it is shown that 
the homotopy associativity of the homotopy action 
\(\namedright{\Omega Z\times E}{a}{E}\) 
has a partial analogue in the half-smash case with respect to the map 
\(\namedright{\Omega Z\ltimes E}{\overline{a}}{E}\). 

\begin{lemma} 
   \label{halfsmashquotient} 
   Let $B, C$ and $D$ be pointed, path-connected spaces. Then there is a natural homeomorphism 
   \[\namedright{(B\times C)\ltimes D}{\varphi}{B\ltimes(C\ltimes D)}\] 
   satisfying commutative diagrams 
   \[\diagram 
          B\times C\times D\rto^-{1\times\pi}\dto^{\pi} & B\times(C\ltimes D)\dto^{\pi} \\ 
          (B\times C)\ltimes D\rto^-{\varphi} & B\ltimes(C\ltimes D). 
      \enddiagram\] 
   and 
   \[\diagram 
        (B\times C)\ltimes D\rto^-{\varphi}\dto & B\ltimes(C\ltimes D)\dto \\ 
        B\wedge C\wedge D\rdouble & B\wedge C\wedge D 
     \enddiagram\] 
   where in the second diagram the vertical maps are the quotients to the smash products. 
\end{lemma} 

\begin{proof} 
The map 
\(\namedright{B\times C\times D}{\pi}{(B\times C)\ltimes D}\) 
identifies the subspace $B\times C\times\ast\subseteq B\times C\times D$ to the basepoint. 
On the other hand, the map 
\(\namedright{B\times(C\ltimes D)}{\pi}{B\ltimes(C\ltimes D)}\) 
identifies the subspace $B\times\ast'\subset B\times(C\ltimes D)$ to the basepoint, 
where $\ast'$ is the basepoint of $C\ltimes D$. 
But $\ast'$ is the result of identifying the subspace $C\times\ast\subseteq C\times D$ 
to the basepoint. Thus $\pi\circ(1\times \pi)$ identifies the subspace 
$B\times C\times\ast\subseteq B\times C\times D$ to the basepoint. Therefore 
$(B\times C)\ltimes D$ and $B\ltimes(C\ltimes D)$ are identical as quotient spaces of 
$B\times C\times D$. This identification is natural since each quotient map involved 
is natural. 

The identification of $(B\times C)\ltimes D$ and $B\ltimes (C\ltimes D)$ as identical 
quotient spaces of $B\times C\times D$ implies that the further quotient maps in both 
cases to the smash product $B\wedge C\wedge D$ are also identical, giving the second 
asserted commutative diagram. 
\end{proof} 

Given a homotopy fibration sequence 
\(\namedddright{\Omega Z}{\partial}{E}{p}{Y}{h}{Z}\), 
one property of the homotopy action 
\(\namedright{\Omega Z\times E}{a}{E}\) 
is that it satisfies a homotopy commutative diagram 
\begin{equation} 
  \label{actionassoc} 
  \diagram 
       \Omega Z\times\Omega Z\times E\rto^-{\mu\times 1}\dto^{1\times a} & \Omega Z\times E\dto^{a} \\ 
       \Omega Z\times E\rto^-{a} & E 
  \enddiagram 
\end{equation} 
where $\mu$ is the loop space multiplication. If $\Omega h$ has a right homotopy inverse 
then $a$ factors as the composite 
\[\nameddright{\Omega Z\times E}{\pi}{\Omega Z\ltimes E}{\overline{a}}{E}\] 
where we may take $\overline{a}$ as in Lemma~\ref{abarGamma}. Ideally, we would 
like~(\ref{actionassoc}) to descend to a homotopy commutative diagram 
\[\diagram 
     (\Omega Z\times\Omega Z)\ltimes E\rrto^-{\mu\ltimes 1}\dto^{\varphi} 
          & & \Omega Z\ltimes E\dto^{\overline{a}} \\ 
     \Omega Z\ltimes(\Omega Z\ltimes E)\rto^-{1\ltimes\overline{a}} 
          & \Omega Z\ltimes E\rto^-{\overline{a}} & E.  
   \enddiagram\] 
However, it is not clear if this diagram does in fact homotopy commute, as will be explained momentarily. 
As it would be useful to have such a diagram in what is to come, we will discuss to 
what extent valid information can be extracted from the diagram. 

The issue is a general one. Let $B$, $C$ and $W$ be pointed, path-connected spaces. 
The homotopy cofibration sequence 
\(\namedddright{B}{}{B\times C}{\pi}{B\ltimes C}{\delta}{\Sigma B}\) 
induces an exact sequence of pointed sets 
\begin{equation} 
  \label{setexact} 
  \namedddright{[\Sigma B,W]}{\delta^{\ast}}{[B\ltimes C,W]}{\pi^{\ast}}{[B\times C,W]} 
        {}{[B,W]}.  
\end{equation} 
The inclusion of $B$ into $B\times C$ has a left inverse given by the projection 
\(\namedright{B\times C}{}{B}\). 
Therefore~$\delta$ is null homotopic, implying that $\delta^{\ast}$ is the zero map. However, 
as~(\ref{setexact}) is only an exact sequence of pointed sets the fact that $\delta^{\ast}=0$ 
does not in general imply that $\pi^{\ast}$ is a monomorphism. More precisely, the homotopy coaction 
\(\namedright{B\ltimes C}{\psi}{(B\ltimes C)\vee\Sigma B}\) 
induces an action 
\(\namedright{[B\ltimes C,W]\times [\Sigma B,W]}{}{[B\ltimes C,W]}\) 
of the group $[\Sigma B,W]$ on the set $[B\ltimes C,W]$. The orbits of this action 
have the property that if $a,b\in [B\ltimes C,W]$ are in different orbits then $\pi^{\ast}(a)\neq\pi^{\ast}(b)$. 
However, it may be the case that distinct homotopy classes are in the same orbit and both are sent 
by $\pi^{\ast}$ to the same element. The fact that $\delta^{\ast}=0$ only implies that 
the orbit of the trivial map consists of a single homotopy class. 

In our case, we may not be able to use the fact that 
$a\circ(1\times a)\simeq a\circ(\mu\times 1)$ to show that the quotient map 
\(\namedright{(\Omega Z\times\Omega Z)\times E}{\pi}{(\Omega Z\times\Omega Z)\ltimes E}\) 
induces a homotopy 
$\overline{a}\circ(1\times\overline{a})\circ\varphi\simeq\overline{a}\circ(\mu\ltimes 1)$. 
However, we are able to prove the following, which will suffice for our purposes. 

\begin{lemma} 
   \label{halfsmashaction} 
   Suppose that there is a homotopy fibration sequence 
   \(\namedddright{\Omega Z}{\partial}{E}{}{Y}{h}{Z}\) 
   where~$\Omega h$ has a right homotopy inverse. Then the diagram 
   \[\diagram 
        (\Omega Z\times\Omega Z)\ltimes E\rrto^-{\mu\ltimes 1}\dto^{\varphi} 
             & & \Omega Z\ltimes E\dto^{\overline{a}} \\ 
        \Omega Z\ltimes(\Omega Z\ltimes E)\rto^-{1\ltimes\overline{a}} 
             & \Omega Z\ltimes E\rto^-{\overline{a}} & E  
      \enddiagram\] 
   has the following properties: 
   \begin{letterlist} 
      \item it homotopy commutes when precomposed with the map 
               \(\namedright{(\Omega Z\times\Omega Z)\times E}{\pi}{(\Omega Z\times\Omega Z)\ltimes E}\); 
      \item it homotopy commutes after suspending; 
      \item it commutes in homology. 
   \end{letterlist} 
\end{lemma} 

\begin{proof} 
First consider the diagram 
\[\diagram 
      \Omega Z\times\Omega Z\times E\rto^-{\mu\times 1}\dto^{\pi} & \Omega Z\times E\dto^{\pi}\drto^{a} & \\ 
      (\Omega Z\times\Omega Z)\ltimes E\rto^-{\mu\ltimes 1} & \Omega Z\ltimes E\rto^-{\overline{a}} & E.  
  \enddiagram\] 
The left square commutes by the naturality of the quotient map $\pi$ and the right side 
commutes by Lemma~\ref{abarGamma}. Thus the homotopy commutativity of the diagram 
implies that $a\circ(\mu\times 1)\simeq\overline{a}\circ(\mu\ltimes 1)\circ \pi$. Next consider 
the diagram 
\[\diagram 
      \Omega Z\times\Omega Z\times E\rto^-{1\times\pi}\dto^{\pi} 
           & \Omega Z\times(\Omega Z\ltimes E)\rto^-{1\times \overline{a}}\dto^{\pi} 
           & \Omega Z\times E\dto^{\pi}\drto^{a} & \\ 
      (\Omega Z\times\Omega Z)\ltimes E\rto^-{\varphi} 
           & \Omega Z\ltimes(\Omega Z\ltimes E)\rto^-{1\ltimes\overline{a}} 
           & \Omega Z\ltimes E\rto^-{\overline{a}} & E. 
  \enddiagram\] 
The left square commutes by Lemma~\ref{halfsmashquotient}, the middle square commutes by the naturality 
of the quotient map $\pi$, and the right triangle commutes by Lemma~\ref{abarGamma}. Notice 
that Lemma~\ref{abarGamma} also implies that the top row is homotopic to $1\times a$. 
Thus the homotopy commutativity of the diagram implies that 
$a\circ(1\times a)\simeq\overline{a}\circ(1\ltimes\overline{a})\circ\varphi\circ\pi$. 
Hence, from the property $a\circ(1\times a)\simeq a\circ(\mu\times 1)$ of a homotopy action, 
we obtain $\overline{a}\circ(1\ltimes\overline{a})\circ\varphi\circ\pi\simeq\overline{a}\circ(\mu\ltimes 1)\circ\pi$, 
proving part~(a). 

Since the map $\Sigma\pi$ has a right homotopy inverse, part~(b) follows from part~(a). Part~(c) 
then follows from part~(b) since suspending induces an isomorphism in homology. 
\end{proof} 

Lemma~\ref{halfsmashaction} motivates a definition, which appeared in a different 
context in~\cite{Gr}. 

\begin{definition} 
Let 
\(f,g\colon\namedright{X}{}{Y}\) 
be maps of pointed, path-connected spaces. Then $f$ and $g$ are \emph{congruent} if 
$\Sigma f\simeq\Sigma g$. 
\end{definition} 

Note that if $f$ and $g$ are homotopic then they are congruent but the converse need not 
hold. Note also that $f$ congruent to $g$ implies that, in homology, $f_{\ast}=g_{\ast}$. 
For our purposes, congruence is often shorthand for saying two maps induces the same 
map in homology. This will often be used in the context of homotopy equivalences when 
both $X$ and $Y$ are simply-connected. In this case, if $f$ is a homotopy equivalence 
then it induces an isomorphism in homology, so $g$ also induces an isomorphism in 
homology by the congruence property, and hence is also a homotopy equivalence by 
Whitehead's Theorem. Other properties of congruent maps are discussed in~\cite{Gr} 
but they will not be used here. 

\begin{remark} 
Lemma~\ref{halfsmashaction} may now be rephrased as saying the composites 
$\overline{a}\circ(\mu\ltimes 1)$ and $\overline{a}\circ(1\ltimes\overline{a})\circ\varphi$ are congruent. 
\end{remark} 

Now specialize to the homotopy fibration sequence 
\[\namedddright{\Omega\Sigma X}{\partial}{E}{p}{\Sigma X\vee\Sigma Y}{q_{1}}{\Sigma X}\] 
where~$q_{1}$ is the pinch map to the first wedge summand. It is known (see, 
for example, \cite[4.3.2]{N4}) that there is a homotopy equivalence 
\[E\simeq\bigvee_{k=0}^{\infty} X^{\wedge k}\wedge\Sigma Y\] 
where $X^{0}\wedge\Sigma Y$ is regarded as $\Sigma Y$, and the map from $p$ can be 
identified in terms of iterated Whitehead products based on the inclusions of $\Sigma X$ 
and $\Sigma Y$ into $\Sigma X\vee\Sigma Y$. We show in Theorem~\ref{dbard} that such 
an equivalence can be chosen so that it also inherits properties of the homotopy action 
of $\Omega\Sigma X$ on $E$. 

This begins with an initial homotopy equivalence for $E$ that depends on a special case 
of Theorem~\ref{GTcofib} proved in~\cite{BT2}. 

\begin{theorem} 
   \label{BTfibinclusion} 
   Suppose that there is a homotopy cofibration 
   \(\nameddright{\Sigma A}{f}{Y}{h}{Y'}\). 
   Let $E$ be the homotopy fibre of $h$ and let 
   \(g\colon\namedright{\Sigma A}{}{E}\) 
   be a lift of $f$. If $\Omega h$ has a right homotopy inverse 
   \(s\colon\namedright{\Omega Y'}{}{\Omega Y}\) 
   then the the lift $g$ may be chosen so that: 
   \begin{letterlist} 
      \item the composite 
               \(\nameddright{\Omega Y'\ltimes\Sigma A}{1\ltimes g}{\Omega Y'\ltimes E} 
                    {\overline{a}}{E}\) 
               is a homotopy equivalence; 
      \item there is a homotopy fibration 
               \[\nameddright{\Omega Y'\ltimes\Sigma A}{\chi}{Y}{h}{Y'}\] 
               where $\chi$ is the sum of the maps 
               \(\nameddright{\Omega Y'\ltimes\Sigma A}{\pi}{\Sigma A}{f}{Y}\) 
              and 
              \(\namedright{\Omega Y'\ltimes\Sigma A}{q}{\Omega Y'\wedge\Sigma A}
                  \stackrel{[ev\circ s,f]}{\llarrow} Y\). 
    \end{letterlist} 
\end{theorem} 

\begin{proof} 
While proved in~\cite{BT2}, the proof is included to make the assertions transparent. 
Taking $Z=Y'$ in Theorem~\ref{GTcofib} gives a diagram of data 
\[\diagram 
      & E\rto\dto^{p} & E'\dto^{p'} \\ 
      \Sigma A\rto^-{f} & Y\rto^-{h}\dto^{h} & Y'\ddouble \\ 
      & Y'\rdouble & Y'. 
 \enddiagram\] 
where the vertical columns and the maps between them form a homotopy fibration diagram. 
Since~$\Omega h$ has a right homotopy inverse, by Theorem~\ref{GTcofib} there is a 
homotopy cofibration 
\[\nameddright{\Omega Y'\ltimes\Sigma A}{\theta}{E}{}{E'}.\] 
As $E'$ is contractible, $\theta$ is a homotopy equivalence. By~(\ref{thetabara}), 
$\theta$ is homotopic to the composite $\overline{a}\circ(1\ltimes g)$ for an appropriate 
choice of a lift $g$ of $f$. This proves part~(a). Defining $\chi$ as $p\circ\theta$, part~(b) 
follows immediately from the statement of Theorem~\ref{GTcofib}~(b). 
\end{proof} 

\begin{example} 
\label{specialGTadexample} 
Start with the homotopy cofibration 
\[\nameddright{\Sigma Y}{i_{2}}{\Sigma X\vee\Sigma Y}{q_{1}}{\Sigma X}\] 
where $i_{2}$ is the inclusion of the second wedge summand. Let $E$ be the 
homotopy fibre of $q_{1}$ and let 
\(g\colon\namedright{\Sigma Y}{}{E}\) 
be a lift of $i_{2}$. Since the inclusion $i_{1}$ of the first wedge summand is a right homotopy inverse 
for $q_{1}$, Theorem~\ref{BTfibinclusion} applies to show that $g$ can be chosen so there is 
a homotopy equivalence 
\[\nameddright{\Omega\Sigma X\ltimes\Sigma Y}{1\ltimes g}{\Omega\Sigma X\ltimes E}{\overline{a}}{E}\] 
and there is a homotopy fibration 
\[\nameddright{\Omega\Sigma X\ltimes\Sigma Y}{\chi}{\Sigma X\vee\Sigma Y}{q_{1}}{\Sigma X}\] 
where $\chi$ is the wedge sum of the maps 
\(\nameddright{\Omega\Sigma X\ltimes\Sigma Y}{\pi}{\Sigma Y}{i_{2}}{\Sigma X\vee\Sigma Y}\) 
and 
\(\lllnameddright{\Omega\Sigma X\ltimes\Sigma Y}{q}{\Omega\Sigma X\wedge\Sigma Y}{[ev\circ\Omega i_{1},i_{2}]} 
     {\Sigma X\vee\Sigma Y}\).  
\end{example} 

\begin{remark} 
Example~\ref{specialGTadexample} works equally well with $\Sigma X$ replaced by a space $X'$ 
that is not necessarily a suspension, but the use of $\Sigma X$ is to align with later examples 
and applications. 
\end{remark} 

Next, we build towards Theorem~\ref{dbard}. In general, the James construction gives 
a homotopy equivalence $\Sigma\Omega\Sigma X\simeq\bigvee_{k=1}^{\infty}\Sigma X^{\wedge k}$ 
which is natural for maps 
\(\namedright{X}{}{Y}\). 
There are different choices of such an equivalence and it will help if we fix one. Focus 
on the suspension  
\(\namedright{X}{E}{\Omega\Sigma X}\). 
For $k\geq 1$, let $e_{k}$ be the composite 
\[e_{k}\colon\nameddright{X^{\times k}}{E^{\times k}}{(\Omega\Sigma X)^{\times k}} 
     {\mu}{\Omega\Sigma X}\] 
where $\mu$ is the standard loop multiplication. There is a natural homotopy equivalence 
$\Sigma(A\times B)\simeq\Sigma A\vee\Sigma B\vee(\Sigma A\wedge B)$. Iterating 
this we obtain a natural map 
\(\namedright{\Sigma X_{1}\wedge\cdots\wedge X_{k}}{}{\Sigma(X_{1}\times\cdots\times X_{k})}\). 
Let $\phi_{k}$ be the composite 
\[\phi_{k}\colon\nameddright{\Sigma X^{\wedge k}}{}{\Sigma(X^{\times k})}{\Sigma e_{k}} 
     {\Sigma\Omega\Sigma X}.\] 
Let 
\[\phi\colon\namedright{\bigvee_{k=1}^{\infty}\Sigma X^{\wedge k}}{}{\Sigma\Omega\Sigma X}\] 
be the wedge sum of the maps $\phi_{k}$ for $k\geq 1$. 

\begin{lemma} 
   \label{SigmaJsplitting} 
   The map 
   \(\namedright{\bigvee_{k=1}^{\infty}\Sigma X^{\wedge k}}{\phi}{\Sigma\Omega\Sigma X}\) 
   is a homotopy equivalence that is natural for maps 
   \(\namedright{X}{}{Y}\). 
\end{lemma} 

\begin{proof} 
By the Bott-Samelson Theorem there is an algebra 
isomorphism $\hlgy{\Omega\Sigma X;\mathbf{k}}\cong T(\rhlgy{X;\mathbf{k}})$ where $T(\ \ )$ is the 
free tensor algebra functor and $\mathbf{k}$ is a field. By construction, the map $\phi_{k}$ 
induces an isomorphism onto the suspension of the submodule of tensors of length $k$. 
Thus $\phi_{\ast}$ is an isomorphism. As this is true for homology with mod-$p$ coefficients 
for any prime $p$ and for rational coefficients, $\phi$ induces an isomorphism in integral 
homology. Thus $\phi$ is a homotopy equivalence by Whitehead's Theorem. 

The naturality of $\phi$ follows by the naturality of $e_{k}$ and the map 
\(\namedright{\Sigma X^{\wedge k}}{}{\Sigma(X^{\times k})}\). 
\end{proof} 

We now turn to specifying a homotopy equivalence 
$\Omega\Sigma X\ltimes\Sigma Y\simeq\bigvee_{k=0}^{\infty} X^{\wedge k}\wedge\Sigma Y$ 
that will be needed to prove Theorem~\ref{dbardintro}. Let 
\[b_{1}\colon\namedright{X\wedge\Sigma Y}{}{X\ltimes\Sigma Y}\] 
be the inclusion $i$. For $k\geq 2$, define 
\[b_{k}\colon\namedright{X^{\wedge k}\wedge\Sigma Y}{}{(X^{\times k})\ltimes\Sigma Y}\] 
recursively by the composite 
\[\llnameddright{X^{\wedge k}\wedge\Sigma Y}{=}{X\wedge(X^{\wedge k-1}\wedge\Sigma Y)} 
      {i}{X\ltimes(X^{\wedge k-1}\wedge\Sigma Y)}\] 
\[\hspace{4cm}\stackrel{1\ltimes b_{k-1}}{\llarrow}\llnameddright{X\ltimes(X^{\times k-1}\ltimes\Sigma Y)} 
        {\varphi^{-1}}{(X\times X^{k-1})\ltimes\Sigma Y}{=}{(X^{\times k})\ltimes\Sigma Y}\]  
where $\varphi$ is the homeomorphism from Lemma~\ref{halfsmashquotient}. 

For $k\geq 1$, let 
\[q_{k}\colon\namedright{(X^{\times k})\ltimes\Sigma Y}{}{X^{\wedge k}\wedge\Sigma Y}\] 
be the quotient map to the smash product. 

\begin{lemma} 
   \label{bkquotient} 
   For $k\geq 1$ the composite 
   \[\nameddright{X^{\wedge k}\wedge\Sigma Y}{b_{k}}{(X^{\times k})\ltimes\Sigma Y} 
          {q_{k}}{X^{\wedge k}\wedge\Sigma Y}\] 
   is homotopic to the identity map. 
\end{lemma} 

\begin{proof} 
The proof is by induction on $k$. The $k=1$ case holds because $b_{1}$ is defined as the inclusion 
\(\namedright{X\wedge\Sigma Y}{i}{X\ltimes\Sigma Y}\),  
the map $i$ is a left homotopy inverse of the quotient map 
\(\namedright{X\ltimes\Sigma Y}{q}{X\ltimes\Sigma Y}\), 
and by definition $q_{1}=q$. 

For $k\geq 2$, suppose that $q_{k-1}\circ b_{k-1}$ is homotopic to the identity map. 
Consider the diagram 
\[\spreaddiagramcolumns{-0.5pc}\diagram 
     X\wedge(X^{\wedge k-1}\wedge\Sigma Y)\rto^-{i}\drdouble 
        & X\ltimes(X^{\wedge k-1}\wedge\Sigma Y)\rto^-{1\ltimes b_{k-1}}\dto^{q} 
        & X\ltimes(X^{\times k-1}\ltimes\Sigma Y)\rto^-{\varphi^{-1}}\dto 
        & (X\times X^{k-1})\ltimes\Sigma Y\dto  \\ 
     & X\wedge(X^{\wedge k-1}\wedge\Sigma Y)\rdouble 
        & X\wedge(X^{\wedge k-1}\wedge\Sigma Y)\rdouble 
        & (X\wedge X^{\wedge k-1})\wedge\Sigma Y 
   \enddiagram\] 
where the vertical maps are quotient maps to the smash product. The left triangle 
homotopy commutes since $i$ is a right homotopy inverse of $q$, the middle square 
homotopy commutes by inductive hypothesis, and the right square commutes by 
Lemma~\ref{halfsmashquotient}. By definition, $b_{k}$ is the top row (up to identification 
of $X\wedge(X^{\wedge k-1}\wedge\Sigma Y)$ as $X^{\wedge k}\wedge\Sigma Y$ 
and $(X\times X^{k-1})\ltimes\Sigma Y$ as $(X^{\times k})\ltimes\Sigma Y$) and 
the right vertical map can be identified with $q_{k}$. Thus $q_{k}\circ b_{k}$ is 
homotopic to the identity map, proving the inductive step. 
\end{proof} 

Let 
\[c_{0}\colon\namedright{\Sigma Y}{}{\Omega\Sigma X\ltimes\Sigma Y}\] 
be the inclusion $j$ and for $k\geq 1$ define $c_{k}$ by the composite 
\[c_{k}\colon\lnameddright{X^{\wedge k}\wedge\Sigma Y}{b_{k}}{(X^{\times k})\ltimes\Sigma Y} 
      {e_{k}\ltimes 1}{\Omega\Sigma X\ltimes\Sigma Y}.\] 
Let 
\[c\colon\namedright{\bigvee_{k=0}^{\infty} X^{\wedge k}\wedge\Sigma Y}{} 
     {\Omega\Sigma X\ltimes\Sigma Y}\] 
be the wedge sum of the maps $c_{k}$ for $k\geq 0$, where $X^{\wedge 0}\wedge\Sigma Y$ 
is understood to be $\Sigma Y$.       
      
\begin{lemma} 
   \label{cequiv} 
   The map 
   \(\namedright{\bigvee_{k=0}^{\infty} X^{\wedge k}\wedge\Sigma Y}{c} 
       {\Omega\Sigma X\ltimes\Sigma Y}\)
    is a homotopy equivalence. 
\end{lemma} 

\begin{proof} 
Take homology with field coefficients. By the Bott-Samelson Theorem there is an algebra 
isomorphism $\hlgy{\Omega\Sigma X}\cong T(V)$ where $V=\rhlgy{X}$. The homotopy 
equivalence $\Omega\Sigma X\ltimes\Sigma Y\simeq(\Omega\Sigma X\wedge\Sigma Y)\vee\Sigma Y$ 
therefore implies that there is a module isomorphism 
\[\hlgy{\Omega\Sigma X\ltimes\Sigma Y}\cong \bigg(T(V)\otimes\rhlgy{\Sigma Y}\bigg)\oplus\hlgy{\Sigma Y}.\] 

By Lemma~\ref{bkquotient}, the map 
\(\namedright{X^{\wedge k}\wedge\Sigma Y}{b_{k}}{(X^{\times k})\ltimes\Sigma Y}\) 
is a left homotopy inverse for the quotient map 
\(\namedright{(X^{\times k})\ltimes\Sigma Y}{q_{k}}{X^{\wedge k}\wedge\Sigma Y}\). 
Therefore the image of $(b_{k})_{\ast}$ is isomorphic to the submodule 
$\rhlgy{X}^{\otimes k}\otimes\rhlgy{\Sigma Y}$. The map $(e_{k}\ltimes 1)_{\ast}$ maps 
this submodule isomorphically onto $V^{\otimes k}\otimes\rhlgy{\Sigma Y}$. Thus, if $k\geq 1$,  
as $c_{k}$ is defined as $(e_{k}\ltimes 1)\circ b_{k}$, the image of $(c_{k})_{\ast}$ 
is isomorphic to the submodule $V^{\otimes k}\otimes\rhlgy{\Sigma Y}$. Since $c_{0}$ 
is the inclusion of $\Sigma Y$ into $\Omega\Sigma X\ltimes\Sigma Y$, its image in 
homology is $\hlgy{\Sigma Y}$. Thus $c_{\ast}$ is an isomorphism. As this is true for 
homology with any field coefficients, $c$ induces an isomorphism in integral homology 
so $c$ is a homotopy equivalence by Whitehead's Theorem. 
\end{proof} 

We now define two maps 
\(\namedright{\bigvee_{k=0}^{\infty} X^{\wedge k}\wedge\Sigma Y}{}{E}\), 
both of which will be homotopy equivalences, and which are congruent. First, let 
\[d_{0}\colon\namedright{\Sigma Y}{}{E}\] 
be $g$. For $k\geq 1$, let $d_{k}$ be the composite 
\[d_{k}\colon\namedddright{X^{\wedge k}\wedge\Sigma Y}{c_{k}}{\Omega\Sigma X\ltimes\Sigma Y} 
      {1\ltimes g}{\Omega\Sigma X\ltimes E}{\overline{a}}{E}.\] 
Let 
\[d\colon\namedright{\bigvee_{k=0}^{\infty} X^{\wedge k}\wedge\Sigma Y}{}{E}\] 
be the wedge sum of the maps $d_{k}$ for $k\geq 0$. Since $c$ is the wedge sum 
of the maps $c_{k}$ for $k\geq 0$, the map $d$ may equivalently be written as the composite 
\[\namedddright{\bigvee_{k=0}^{\infty} X^{\wedge k}\wedge\Sigma Y}{c}{\Omega\Sigma X\ltimes\Sigma Y} 
      {1\ltimes g}{\Omega\Sigma X\ltimes E}{\overline{a}}{E}.\] 

\begin{lemma} 
   \label{dequiv} 
   The map $d$ is a homotopy equivalence. 
\end{lemma} 

\begin{proof} 
This follows immediately since $d=\overline{a}\circ(1\ltimes g)\circ c$ and, by 
Example~\ref{specialGTadexample} and Lemma~\ref{cequiv} respectively, 
both $\overline{a}\circ(1\ltimes g)$ and $c$ are homotopy equivalences. 
\end{proof} 

Next, let 
\[\mathfrak{d}_{0}\colon\namedright{\Sigma Y}{}{E}\] 
be $g$. For $k\geq 1$, let $\mathfrak{d}_{k}$ be the composite 
\[\mathfrak{d}_{k}\colon\llnamedddright{X\wedge(X^{\wedge k-1}\wedge\Sigma Y)}{i} 
      {X\ltimes(X^{\wedge k-1}\wedge\Sigma Y)} 
      {E\ltimes\mathfrak{d}_{k-1}}{\Omega\Sigma X\ltimes E}{\overline{a}}{E}.\] 
Let 
\[\mathfrak{d}\colon\namedright{\bigvee_{k=0}^{\infty} X^{\wedge k}\wedge\Sigma Y}{}{E}\] 
be the wedge sum of the maps $\mathfrak{d}_{k}$ for $k\geq 0$. We will show that $\mathfrak{d}$ 
is congruent to $d$, that $\mathfrak{d}$ is a homotopy equivalence, and that it 
lifts a certain wedge sum of Whitehead products on $\Sigma X\vee\Sigma Y$. 

\begin{lemma} 
   \label{dcongruent} 
   If $k=0$ or $k=1$ then $\mathfrak{d}_{k}=d_{k}$. If $k\geq 2$ then $\mathfrak{d}_{k}$ 
   is congruent to $d_{k}$. 
\end{lemma} 

\begin{proof} 
First, observe that $\mathfrak{d}_{0}=d_{0}$ since, by their definitions, both equal $g$. 
Next, by definition, $\mathfrak{d}_{1}=\overline{a}\circ(E\ltimes d_{0})\circ i=\overline{a}\circ(E\ltimes g)\circ i$. 
On the other hand, by definition of $d_{1}$ and $c_{1}$ we have 
$d_{1}=\overline{a}\circ(1\ltimes g)\circ c_{1}=\overline{a}\circ(1\ltimes g)\circ(e_{1}\ltimes 1)\circ b_{1}$. 
By definition, $e_{1}=E$ and $b_{1}=i$, so we obtain $d_{1}=\overline{a}\circ(E\ltimes g)\circ i$. Thus 
$\mathfrak{d}_{1}=d_{1}$. 

Now suppose that $k\geq 2$ and assume inductively that $\mathfrak{d}_{k-1}$ is congruent 
to $d_{k-1}$. By definition, $d_{k}=\overline{a}\circ(1\ltimes g)\circ c_{k}$ and 
$c_{k}=(e_{k}\ltimes 1)\circ b_{k}$, so $d_{k}$ is equivalently described by the composite 
\begin{equation} 
  \label{dkalt} 
  \lnamedddright{X^{\wedge k}\wedge Y}{b_{k}}{(X^{\times k})\ltimes\Sigma Y}{e_{k}\ltimes g} 
       {\Omega\Sigma X\ltimes E}{\overline{a}}{E}. 
\end{equation} 
Consider the diagram 
\[\diagram 
     X\ltimes(X^{\times k-1}\ltimes\Sigma Y)\rrto^-{E\ltimes(e_{k-1}\ltimes g)}\dto^{\varphi^{-1}} 
        & & \Omega\Sigma X\ltimes(\Omega\Sigma X\ltimes E)\rrto^-{1\ltimes\overline{a}}\dto^{\varphi^{-1}} 
        & & \Omega\Sigma X\ltimes E\dto^-{\overline{a}} \\ 
     (X\times X^{\times k-1})\ltimes\Sigma Y\rrto^-{(E\times e_{k-1})\ltimes g} 
        & & (\Omega\Sigma X\times\Omega\Sigma X)\ltimes E\rto^-{\mu\ltimes 1} 
        & \Omega\Sigma X\ltimes E\rto^-{\overline{a}} & E 
  \enddiagram\] 
where $\varphi$ is the homeomorphism in Lemma~\ref{halfsmashquotient}. 
The left square commutes by the naturality of $\varphi$. The right rectangle may not 
homotopy commute but the two ways around the diagram are congruent 
by Lemma~\ref{halfsmashaction}. By definition, $e_{k}=\mu\circ E^{\times k}$ 
where $\mu$ is an iterated loop multiplication on $\Omega\Sigma X$. Thus the composite 
\[\lllnamedddright{X\times X^{\times k-1}}{E\times E^{\times k-1}} 
      {\Omega\Sigma X\times(\Omega\Sigma X)^{\times k-1}}{1\times\mu} 
      {\Omega\Sigma X\times\Omega\Sigma X}{\mu}{\Omega\Sigma X}\] 
is, on the one hand, $\mu\circ(E\times e_{k-1})$, and on the other hand, $e_{k}$. 
Therefore the bottom row in the diagram is $\overline{a}\circ(e_{k}\ltimes g)$. 

Next, by definition of $b_{k}$, there is a commutative diagram 
\[\diagram 
     X\wedge(X^{\wedge k-1}\wedge\Sigma Y)\rto^-{i}\drrto_{b_{k}} 
          & X\ltimes(X^{\wedge k-1}\wedge\Sigma Y)\rto^-{1\ltimes b_{k-1}} 
          & X\ltimes(X^{\times k-1}\ltimes\Sigma Y)\dto^{\varphi^{-1}} \\ 
      & & (X\times X^{\times k-1})\ltimes\Sigma Y. 
  \enddiagram\] 
Juxtapose the two previous diagrams. In the lower direction we obtain 
$\overline{a}\circ(e_{k}\ltimes g)\circ b_{k}$ which, by~(\ref{dkalt}), is $d_{k}$. 
On the other hand, the upper direction is  
$\overline{a}\circ(1\ltimes\overline{a})\circ(E\ltimes(e_{k-1}\ltimes g))\circ(1\ltimes b_{k-1})\circ i$  
which may be rewritten as 
$\overline{a}\circ(E\ltimes(\overline{a}\circ (e_{k-1}\ltimes g)\circ b_{k-1}))\circ i$.  
By~(\ref{dkalt}) this is the same as $\overline{a}\circ(E\ltimes d_{k-1})\circ i$.  
Hence $d_{k}$ is congruent to $\overline{a}\circ(E\ltimes d_{k-1})\circ i$. 

By the inductive hypothesis, $d_{k-1}$ is congruent to $\mathfrak{d}_{k-1}$ so 
$\overline{a}\circ(E\ltimes d_{k-1})\circ i$ is congruent to $\overline{a}\circ(E\ltimes\mathfrak{d}_{k-1})\circ i$. 
By definition, $\mathfrak{d}_{k}=\overline{a}\circ(E\ltimes\mathfrak{d}_{k-1})\circ i$. Hence 
$d_{k}$ is congruent to $\mathfrak{d}_{k}$. 
\end{proof} 

\begin{lemma} 
   \label{bardequiv} 
   The map $\mathfrak{d}$ is congruent to $d$. Consequently, $\mathfrak{d}$ is a homotopy equivalence. 
\end{lemma} 

\begin{proof} 
By Lemma~\ref{dcongruent}, $d_{k}$ is congruent to $\mathfrak{d}_{k}$ for all $k\geq 0$. 
Since $d$ and $\mathfrak{d}$ are the wedge sums of the maps $d_{k}$ and $\mathfrak{d}_{k}$ 
respectively, we obtain that $d$ is congruent to $\mathfrak{d}$. Congruence implies that 
$d_{\ast}=\mathfrak{d}_{\ast}$. By Lemma~\ref{dequiv}, $d$ is a homotopy equivalence so $d_{\ast}$ is 
an isomorphism. Hence $\mathfrak{d}_{\ast}$ is an isomorphism and so is also a homotopy 
equivalence by Whitehead's Theorem. 
\end{proof} 

We next show that $\mathfrak{d}$ lifts certain Whitehead products. In general, given maps 
\(u\colon\namedright{\Sigma A}{}{Z}\) 
and 
\(v\colon\namedright{\Sigma B}{}{Z}\) 
define the iterated Whitehead product 
\[ad^{k}(u)(v)\colon\namedright{A^{\wedge k}\wedge\Sigma B}{}{Z}\] 
recursively as follows. If $k=0$ then $ad^{0}(u)(v)=v$. If $k>0$ then 
$ad^{k}(u)(v)=[u,ad^{k-1}(u)(v)]$. In our case the roles of $u$ and $v$ will be played by 
the inclusions 
\[i_{1}\colon\namedright{\Sigma X}{}{\Sigma X\vee\Sigma Y}\qquad 
     i_{2}\colon\namedright{\Sigma Y}{}{\Sigma X\vee\Sigma Y}\] 
of the left and right wedge summands respectively. 

\begin{lemma} 
   \label{dbarkaction} 
    For each $k\geq 1$ there is a homotopy commutative diagram 
   \[\diagram 
          X^{\wedge k}\wedge\Sigma Y 
                 \rto^-{\mathfrak{d}_{k}}\drto_{ad^{k}(i_{1})(i_{2})} 
               & E\dto^{p} \\ 
           & \Sigma X\vee\Sigma Y. 
     \enddiagram\]  
\end{lemma} 

\begin{proof} 
The proof is by induction on $k$. For the base case when $k=1$ 
we want to show that $p\circ\mathfrak{d}_{1}\simeq [i_{1},i_{2}]$. By Lemma~\ref{dcongruent}, 
$\mathfrak{d}_{1}=d_{1}$ so it is equivalent to show that $p\circ d_{1}\simeq [i_{1},i_{2}]$. 
Consider the diagram 
\begin{equation} 
  \label{i1i2dgrm} 
  \diagram 
     X\wedge\Sigma Y\rto^-{i}\dto^{E\wedge 1} 
         & X\ltimes\Sigma Y\dto^{E\ltimes 1} & & \\ 
     \Omega\Sigma X\wedge\Sigma Y
            \rto^-{i}\xto[drrr]_{[ev\circ\Sigma\Omega i_{1},p\circ g]} 
         & \Omega\Sigma X\ltimes\Sigma Y\rto^-{1\ltimes g} 
         & \Omega\Sigma X\ltimes E\rto^-{\overline{a}} & E\dto^{p} \\ 
      & & & \Sigma X\vee\Sigma Y. 
  \enddiagram 
\end{equation} 
The top left square homotopy commutes by the naturality of $i$. The lower triangle 
homotopy commutes by Proposition~\ref{thetaWh} with $B=Y$ and $\alpha=g$. 
Observe that the composite $\overline{a}\circ(1\ltimes g)\circ(E\ltimes 1)\circ i$ along 
the top direction around the diagram is the definition of $d_{1}$. 

Now consider the composite in the lower direction around~(\ref{i1i2dgrm}). Write the 
identity map on $\Sigma Y$ as the suspension of the identity map on $Y$. So we 
are considering $[ev\circ\Sigma\Omega i_{1},p\circ g]\circ(E\wedge\Sigma 1)$.  
Observe that 
\(\namedright{\Sigma X}{i_{1}}{\Sigma X\vee\Sigma Y}\) 
is a suspension, say $i_{1}\simeq\Sigma i'_{1}$, so the naturality of $E$ implies that 
$\Omega i_{1}\circ E\simeq\Omega\Sigma i'_{1}\circ E\simeq E\circ i'_{1}$. As $ev$ is 
a right homotopy inverse for $\Sigma E$ we obtain 
$ev\circ\Sigma\Omega i_{1}\circ\Sigma E\simeq ev\circ\Sigma E\circ\Sigma i'_{1}\simeq 
    \Sigma i'_{1}\simeq i_{1}$. 
Therefore  the naturality of the Whitehead product and the fact that $p\circ g=i_{2}$ imply that 
\[[ev\circ\Sigma\Omega i_{1},p\circ g]\circ(E\wedge\Sigma 1)\simeq  
       [ev\circ\Sigma\Omega i_{1}\circ\Sigma E,p\circ g]\simeq [i_{1}, i_{2}].\] 
Thus the homotopy commutativity of~(\ref{i1i2dgrm}) implies that $p\circ d_{1}\simeq [i_{1},i_{2}]$.  

Assume inductively that $p\circ\mathfrak{d}_{k-1}\simeq ad^{k-1}(i_{1})(i_{2})$. Consider the diagram 
\[\diagram 
     X\wedge (X^{\wedge k-1}\wedge\Sigma Y)\rto^-{i}\dto^{E\wedge 1} 
         & X\ltimes\Sigma(X^{\wedge k-1}\wedge Y)\dto^{E\ltimes 1} & & \\ 
     \Omega\Sigma X\wedge (X^{\wedge k-1}\wedge\Sigma Y)
            \rto^-{i}\xto[drrr]_{[ev\circ\Sigma\Omega i_{1},p\circ\mathfrak{d}_{k-1}]} 
         & \Omega\Sigma X\ltimes\Sigma(X^{\wedge k-1}\wedge Y)\rto^-{1\ltimes\mathfrak{d}_{k-1}} 
         & \Omega\Sigma X\ltimes E\rto^-{\overline{a}} & E\dto^{p} \\ 
      & & & \Sigma X\vee\Sigma Y. 
  \enddiagram\] 
The top left square homotopy commutes by the naturality of $i$. The lower triangle 
homotopy commutes by Proposition~\ref{thetaWh} with $B=X^{\wedge k-1}\wedge Y$, $Z=\Sigma X$ and 
$\alpha=\mathfrak{d}_{k-1}$. The composite 
$\overline{a}\circ(1\ltimes\mathfrak{d}_{k-1})\circ(E\ltimes 1)\circ i$ 
along the top direction around the diagram is the definition of $\mathfrak{d}_{k}$. 
Again writing the identity map on $\Sigma Y$ as the suspension of the identity 
map on~$Y$, the naturality of the Whitehead product and the inductive hypothesis 
$p\circ\mathfrak{d}_{k-1}\simeq ad^{k-1}(i_{1})(i_{2})$ then imply that  
\[[ev\circ\Sigma\Omega i_{1},p\circ\mathfrak{d}_{k-1}]\circ(E\wedge\Sigma 1)\simeq  
       [ev\circ\Sigma\Omega i_{1}\circ\Sigma E,p\circ\mathfrak{d}_{k-1}]\simeq  
       [i_{1}, ad^{k-1}(i_{1})(i_{2})]=ad^{k}(i_{1})(i_{2}).\] 
Thus the diagram implies that $p\circ\mathfrak{d}_{k}\simeq ad^{k}(i_{1})(i_{2})$, completing 
the induction. 
\end{proof} 

Putting all this together we are able to prove Theorem~\ref{dbardintro}, re-stated as follows. 
 
\begin{theorem} 
   \label{dbard} 
   Let $X$ and $Y$ be path-connected, pointed spaces and consider the homotopy fibration  
   \(\nameddright{E}{}{\Sigma X\vee\Sigma Y}{q_{1}}{\Sigma X}\). 
   There is a homotopy commutative diagram 
   \[\diagram 
          \bigvee_{k=0}^{\infty} X^{\wedge k}\wedge\Sigma Y 
                 \rto^-{\mathfrak{d}}\drto_{\bigvee_{k=0}^{\infty} ad^{k}(i_{1})(i_{2})} 
               & E\dto^{p} \\ 
           & \Sigma X\vee\Sigma Y 
     \enddiagram\]      
   where: 
   \begin{letterlist} 
      \item $\mathfrak{d}$ is a homotopy equivalence; 
      \item $\mathfrak{d}$ is congruent to $d$, where $d$ is the composite 
               \[\namedddright{\bigvee_{k=0}^{\infty} X^{\wedge k}\wedge\Sigma Y}{c} 
                  {\Omega\Sigma X\ltimes\Sigma Y}{1\ltimes g}{\Omega\Sigma X\ltimes E}{\overline{a}}{E}.\] 
    \end{letterlist}  
\end{theorem} 

\begin{proof} 
By definition, $\mathfrak{d}$ is the wedge sum of the maps $\mathfrak{d}_{k}$ for $k\geq 0$. 
When $k=0$ we have $\mathfrak{d}_{0}=g$ where $p\circ g=i_{2}$ while $ad^{0}(i_{1})(i_{2})=i_{2}$. 
When $k\geq 1$, by Lemma~\ref{dbarkaction} we have $p\circ\mathfrak{d}_{k}\simeq ad^{k}(i_{1})(i_{2})$. 
Thus $p\circ\mathfrak{d}\simeq\bigvee_{k=0}^{\infty} ad^{k}(i_{1})(i_{2})$, implying that the 
asserted diagram homotopy commutes. 

Parts~(a) and~(b) are proved by Lemma~\ref{bardequiv}. 
\end{proof}

\newpage 

\section{Based loops on certain $2$-cones} 
\label{sec:2cone} 

The main result in this section is Theorem~\ref{E'typeI}, which will then be specialized 
to prove Theorem~\ref{Mtypealtintro}. We go on to give applications to Moore's conjecture. 

In general, start with the data 
\[\diagram 
       & E\rto\dto^{p} & E'\dto \\ 
       \Sigma A\rto^-{f} & Y\rto\dto^{h} & Y'\dto \\ 
       & Z\rdouble & Z 
  \enddiagram\] 
and suppose that $\Omega h$ has a right homotopy inverse. 
By Theorem~\ref{GTcofib}~(a) there is a homotopy cofibration 
\[\nameddright{\Omega Z\ltimes\Sigma A}{\theta}{E}{}{E'}\] 
where the restriction of $\theta$ to $\Sigma A$ is a map 
\[g\colon\namedright{\Sigma A}{}{E}\] 
which lifts $f$ through $p$. The goal is to determine the homotopy type of $E'$ by knowing 
properties of the space $E$ and the map~$\theta$. In Theorem~\ref{E'typeI} several 
hypotheses are given which will let us do this. 

One hypothesis is that $Z$ is a suspension. Rewriting the data, we have 
\[\diagram 
       & E\rto\dto^{p} & E'\dto \\ 
       \Sigma A\rto^-{f} & Y\rto\dto^{h} & Y'\dto \\ 
       & \Sigma X\rdouble & \Sigma X 
  \enddiagram\] 
where $\Omega h$ has a right homotopy inverse, there is a homotopy cofibration 
\[\nameddright{\Omega\Sigma X\ltimes\Sigma A}{\theta}{E}{}{E'}\] 
and the restriction of $\theta$ to $\Sigma A$ is a map 
\(g\colon\namedright{\Sigma A}{}{E}\) 
that lifts $f$ through $p$. The appearance of $\Omega\Sigma X$ lets us take advantage 
of the James construction. 

For $k\geq 0$, let $J_{k}(X)$ be the $k^{th}$-stage of the James construction. 
Explicitly, $J_{0}(X)=\ast$ and if~$k\geq 1$ then $J_{k}(X)=X^{\times k}/\sim$ where 
$(x_{1},\ldots,x_{t},\ast,x_{t+1},\ldots,x_{k})\sim (x_{1},\ldots,\ast,x_{t},x_{t+1},\ldots,x_{k})$. 
There is an inclusion 
\(\namedright{J_{k}(X)}{}{J_{k+1}(X)}\) 
given by sending $(x_{1},\ldots,x_{k})$ to $(x_{1},\ldots,x_{k},\ast)$. Taking a direct limit 
gives the space $J(X)$, and James~\cite{J1} showed that there is a homotopy equivalence 
of $H$-spaces $J(X)\simeq\Omega\Sigma X$. Let $j_{k}$ be the composite 
\[j_{k}\colon\nameddright{J_{k}(X)}{}{J(X)}{\simeq}{\Omega\Sigma X}.\] 

Let $D$ be any pointed, path-connected space. For $k\geq 1$ let $I_{k}$ be the composite 
\[I_{k}\colon\llnameddright{\Omega\Sigma X\ltimes(X^{\wedge k}\wedge\Sigma D)} 
      {1\ltimes c_{k}}{\Omega\Sigma X\ltimes(\Omega\Sigma X\ltimes\Sigma D)}{\varphi^{-1}} 
      {(\Omega\Sigma X\times\Omega\Sigma X)\ltimes\Sigma D}\] 
where $c_{k}$ was defined in Section~\ref{sec:ad} and $\varphi$ is the homeomorphism 
in Lemma~\ref{halfsmashquotient}. Recall that, generically, the map 
\(\namedright{B}{j}{A\ltimes B}\) 
is the inclusion. For $k\geq 1$ let $J_{k}$ be the composite 
\[J_{k}\colon\llnamedright{J_{k-1}(X)\ltimes\Sigma D}{j_{k-1}\ltimes 1} 
       {\Omega\Sigma X\ltimes\Sigma D}\stackrel{j}{\longrightarrow}
       \namedright{\Omega\Sigma X\ltimes(\Omega\Sigma X\ltimes\Sigma D)} 
       {\varphi^{-1}}{(\Omega\Sigma X\times\Omega\Sigma X)\ltimes\Sigma D}.\] 
     
\begin{lemma} 
   \label{IkJkequiv} 
   The composite 
   \[\llnameddright{\Omega\Sigma X\ltimes(X^{\wedge k}\wedge\Sigma D)\vee (J_{k-1}(X)\ltimes\Sigma D)} 
         {I_{k}\perp J_{k}}{(\Omega\Sigma X\times\Omega\Sigma X)\ltimes\Sigma D}{\mu\ltimes 1} 
         {\Omega\Sigma X\ltimes\Sigma D}\] 
   is a homotopy equivalence. 
\end{lemma} 

\begin{proof} 
Take homology with field coefficients. Let $V=\rhlgy{X}$. By the Bott-Samelson Theorem 
there is an algebra isomorphism $\hlgy{\Omega\Sigma X}\cong T(V)$. Let us rewrite this 
as a module isomorphism $\rhlgy{\Omega\Sigma X}\cong\bigoplus_{t=1}^{\infty} V^{\otimes t}$. 
Therefore there is a module isomorphism 
\[\rhlgy{\Omega\Sigma X\ltimes\Sigma D}\cong\bigoplus_{t=0}^{\infty} V^{\otimes t}\otimes\rhlgy{\Sigma D}\] 
where we regard $V^{0}\otimes\rhlgy{\Sigma D}$ as $\rhlgy{\Sigma D}$. 

In homology, the map 
\(\namedright{J_{k-1}(X)}{j_{k-1}}{\Omega\Sigma X}\) 
induces the injection 
\(\namedright{\bigotimes_{t=1}^{k-1} V^{\otimes t}}{}{\bigotimes_{t=1}^{\infty} V^{\otimes t}}\). 
Observe that the composite 
\[\namedddright{\Omega\Sigma X\ltimes\Sigma D}{j}{\Omega\Sigma X\ltimes(\Omega\Sigma X\ltimes\Sigma D)} 
      {\varphi^{-1}}{(\Omega\Sigma X\times\Omega\Sigma X)\ltimes\Sigma D}{\mu\ltimes 1} 
      {\Omega\Sigma X\ltimes\Sigma D}\] 
is homotopic to the identity map: for if the domain $\Omega\Sigma X\ltimes\Sigma D$ is 
regarded as $\ast\ltimes(\Omega\Sigma X\ltimes\Sigma D)$ then~$j$ can be regarded as 
$b\ltimes(1\ltimes 1)$ where $b$ is the inclusion of the basepoint, so the naturality of~$\varphi$ 
implies that $\varphi^{-1}\circ j$ is equal 
to the composite 
\(\llnameddright{\ast\ltimes(\Omega\Sigma X\ltimes\Sigma D)}{\varphi^{-1}} 
     {(\ast\times\Omega\Sigma X)\ltimes\Sigma D}{(b\times 1)\ltimes 1} 
     {(\Omega\Sigma X\times\Omega\Sigma X)\ltimes\Sigma D}\),      
implying that $(\mu\ltimes 1)\circ\varphi^{-1}\circ j$ is homotopic to the identity map.       
Therefore in homology the map $(\mu\ltimes 1)\circ J_{k}$ induces the same map as 
$j_{k-1}\ltimes 1$, which is the injection 
\(\namedright{\bigotimes_{t=1}^{k-1} V^{\otimes t}\otimes\rhlgy{\Sigma D}}{} 
       {\bigotimes_{t=1}^{\infty} V^{\otimes t}\otimes\rhlgy{\Sigma D}}\). 
On the other hand, as in the proof of Lemma~\ref{cequiv}, in homology the 
map $c_{k}$ induces the inclusion 
\(\namedright{V^{\otimes k}\otimes\rhlgy{\Sigma D}}{} 
      {\bigotimes_{t=0}^{\infty} V^{\otimes t}\otimes\rhlgy{\Sigma D}}\). 
Therefore $(\mu\ltimes 1)\circ I_{k}$ induces the inclusion 
\(\namedright{\bigotimes_{t=k}^{\infty} V^{\otimes k}\otimes\rhlgy{\Sigma D}}{} 
      {\bigotimes_{t=0}^{\infty} V^{\otimes t}\otimes\rhlgy{\Sigma D}}\). 
Hence $I_{k}\perp J_{k}$ induces an isomorphism in homology. As this is true for 
homology with mod-$p$ coefficients for all primes $p$ and rational coefficients, $I_{k}\perp J_{k}$ 
therefore induces an isomorphism in integral homology, and so is a homotopy equivalence 
by Whitehead's Theorem. 
\end{proof} 

Next, suppose that there is a map 
\[\delta\colon\namedright{\Sigma D}{}{E}.\] 
For $k\geq 0$, let $\overline{c}_{k}$ be the composite 
\[\overline{c}_{k}\colon\nameddright{X^{\wedge k}\wedge\Sigma D}{c_{k}} 
      {\Omega\Sigma X\ltimes\Sigma D}{1\ltimes\delta}{\Omega\Sigma X\ltimes E}\] 
and let $\overline{J}_{k}$ be the composite 
\[\overline{J}_{k}\colon\llnameddright{J_{k-1}(X)\ltimes\Sigma D}{j_{k-1}\ltimes\delta} 
       {\Omega\Sigma X\ltimes E}{j}{\Omega\Sigma X\ltimes(\Omega\Sigma X\ltimes E)}.\] 

\begin{lemma} 
   \label{IkJklemma} 
   There is a homotopy commutative diagram 
   \[\diagram 
        (\Omega\Sigma X\ltimes(X^{\wedge k}\wedge\Sigma D))\vee (J_{k-1}(X)\ltimes\Sigma D) 
                  \rto^-{I_{k}\perp J_{k}}\ddrto_-{(1\ltimes\overline{c}_{k})\perp\overline{J}_{k}\ } 
             & (\Omega\Sigma X\times\Omega\Sigma X)\ltimes\Sigma D\dto^{(1\times 1)\ltimes\delta} \\ 
        & (\Omega\Sigma X\times\Omega\Sigma X)\ltimes E\dto^{\varphi} \\ 
        & \Omega\Sigma X\ltimes(\Omega\Sigma X\ltimes E). 
   \enddiagram\]  
\end{lemma} 

\begin{proof} 
It suffices to show that the diagram homotopy commutes when restricted to each wedge 
summand in the domain. Observe that the definition of $I_{k}$ as $\varphi^{-1}\circ (1\ltimes c_{k})$ 
and the naturality of $\varphi$ give a homotopy commutative diagram 
\[\spreaddiagramcolumns{0.4pc}\diagram 
        (\Omega\Sigma X\ltimes(X^{\wedge k}\wedge\Sigma D))\rto^-{I_{k}}\drto_{1\ltimes c_{k}} 
             & (\Omega\Sigma X\times\Omega\Sigma X)\ltimes\Sigma D\rto^{(1\times 1)\ltimes\delta}\dto^{\varphi}   
             & (\Omega\Sigma X\times\Omega\Sigma X)\ltimes E\dto^{\varphi} \\ 
        & \Omega\Sigma X\ltimes(\Omega\Sigma X\ltimes\Sigma D)\rto^-{1\ltimes (1\ltimes\delta)} 
        & \Omega\Sigma X\ltimes(\Omega\Sigma X\ltimes E). 
  \enddiagram\]    
By definition, $\overline{c}_{k}=(1\ltimes\delta)\circ c_{k}$ so the composite in the lower 
direction around the diagram is $1\ltimes\overline{c}_{k}$. Thus 
$\varphi\circ((1\times 1)\ltimes\delta)\circ I_{k}\simeq 1\ltimes\overline{c}_{k}$, as asserted. 

Next, the naturality of $\varphi$ and $j$ give a homotopy commutative diagram 
\[\diagram 
     J_{k-1}(X)\ltimes\Sigma D\rto^-{j_{k-1}\ltimes 1}\drto_{j_{k-1}\ltimes\delta} 
         & \Omega\Sigma X\ltimes\Sigma D\rto^-{j}\dto^{1\ltimes\delta}  
         & \Omega\Sigma X\ltimes(\Omega\Sigma X\ltimes\Sigma D)\rto^-{\varphi^{-1}}\dto^{1\ltimes(1\ltimes\delta)} 
         & (\Omega\Sigma X\times\Omega\Sigma X)\ltimes\Sigma D\dto^{(1\times 1)\ltimes\delta} \\ 
      & \Omega\Sigma X\ltimes E\rto^-{j} & \Omega\Sigma X\ltimes(\Omega\Sigma X\ltimes E)\rto^-{\varphi^{-1}} 
      & (\Omega\Sigma X\times\Omega\Sigma X)\ltimes E. 
  \enddiagram\] 
Observe that the top row is the definition of $J_{k}$ while along the bottom row the composite 
$j\circ(j_{k-1}\ltimes\delta)$ is the definition of $\overline{J}_{k}$. Therefore the diagram implies that 
$((1\times 1)\ltimes\delta)\circ J_{k}\simeq\varphi^{-1}\circ\overline{J}_{k}$. Thus 
$\varphi\circ((1\times 1)\ltimes\delta)\circ J_{k}\simeq\varphi\circ\varphi^{-1}\circ\overline{J}_{k}\simeq 
     \overline{J}_{k}$, 
as asserted. 
\end{proof} 

\begin{proposition} 
   \label{JkEequiv} 
   Suppose that there is a homotopy fibration sequence 
   \(\namedddright{\Omega\Sigma X}{\partial}{E}{p}{Y}{h}{\Sigma X}\) 
   where $\Omega h$ has a right homotopy inverse. Suppose that there is a map 
   \(\delta\colon\namedright{\Sigma D}{}{E}\) 
   such that the composite 
   \[\nameddright{\Omega\Sigma X\ltimes\Sigma D}{1\ltimes\delta}{\Omega\Sigma X\ltimes E}{\overline{a}}{E}\] 
   is a homotopy equivalence. Then the composite 
   \[\lllnamedright{\Omega\Sigma X\ltimes(X^{\wedge k}\wedge\Sigma D)\vee(J_{k-1}(X)\ltimes\Sigma D)} 
         {(1\ltimes\overline{c}_{k})\perp\overline{J}_{k}}{\Omega\Sigma X\ltimes(\Omega\Sigma X\ltimes E)} 
         \stackrel{1\ltimes\overline{a}}{\longrightarrow}\namedright{\Omega\Sigma X\ltimes E}{\overline{a}}{E}\] 
   is a homotopy equivalence. 
\end{proposition} 

\begin{proof} 
Consider the diagram 
\[\small\spreaddiagramcolumns{-0.5pc}\diagram 
     (\Omega\Sigma X\ltimes(X^{\wedge k}\wedge\Sigma D))\vee (J_{k-1}(X)\ltimes\Sigma D) 
                  \rto^-{I_{k}\perp J_{k}}\ddrto_-{(1\ltimes\overline{c}_{k})\perp\overline{J}_{k}\ } 
          & (\Omega\Sigma X\times\Omega\Sigma X)\ltimes\Sigma D\rrto^-{\mu\ltimes 1}\dto^{(1\times 1)\ltimes\delta} 
          & & \Omega\Sigma X\ltimes\Sigma D\dto^{1\ltimes\delta} \\ 
     & (\Omega\Sigma X\times\Omega\Sigma X)\ltimes E\rrto^-{\mu\ltimes 1}\dto^{\varphi} 
          & & \Omega\Sigma X\ltimes E\dto^{\overline{a}} \\ 
     & \Omega\Sigma X\ltimes(\Omega\Sigma X\ltimes E)\rto^-{1\ltimes\overline{a}} 
          & \Omega\Sigma X\ltimes E\rto^-{\overline{a}} & E. 
  \enddiagram\normalsize\]  
The left triangle homotopy commutes by Lemma~\ref{IkJklemma} and the upper right square 
clearly commutes. The lower right square may not homotopy commute but by Lemma~\ref{halfsmashaction} 
the two ways around the square are congruent.  The top row is a homotopy equivalence by 
Lemma~\ref{IkJkequiv} and the right column is a homotopy equivalence by hypothesis. 
Therefore the upper direction around the diagram is a homotopy equivalence. In particular, 
it induces an isomorphism in homology. As congruent maps induce isomorphisms in homology, 
this implies that the entire diagram commutes in homology, and therefore the lower direction around the 
diagram also induces an isomorphism in homology. Thus the lower direction around the 
diagram is a homotopy equivalence by Whitehead's Theorem, proving the lemma. 
\end{proof}   

Now we begin a process of altering the homotopy equivalence in Proposition~\ref{JkEequiv} by 
one in which the composite $\overline{a}\circ\overline{c}_{k}$ has been replaced by a congruent 
composite $\overline{a}\circ\overline{\mathfrak{c}}_{k}$. As in Section~\ref{sec:ad}, the point 
is that $\overline{c}_{k}$ behaves well with respect to multiplications, making Proposition~\ref{JkEequiv} 
easy to prove, while~$\overline{\mathfrak{c}}_{k}$ behaves well with respect to Whitehead products, 
which is where the applications lie.  
 
Return to the starting map 
\(\delta\colon\namedright{\Sigma D}{}{E}\). 
Let 
\[\overline{\mathfrak{c}}_{0}\colon\namedright{\Sigma D}{}{E}\] 
be $\delta$. For $k\geq 1$, let $\overline{\mathfrak{c}}_{k}$ be the composite 
\[\overline{\mathfrak{c}}_{k}\colon\llnameddright
      {X^{\wedge k}\wedge\Sigma D\simeq X\wedge (X^{\wedge k-1}\wedge\Sigma D)}{i} 
      {X\ltimes(X^{\wedge k-1}\wedge\Sigma D)}{E\ltimes\overline{\mathfrak{c}}_{k-1}} 
      {\Omega\Sigma X\ltimes E}.\] 
Define $\overline{d}_{0}=\overline{c}_{0}$ and $\overline{\mathfrak{d}}_{0}=\overline{\mathfrak{c}_{0}}$ 
(so $\overline{d}_{0}=\overline{\mathfrak{d}}_{0}=\delta$), and for $k\geq 1$ define 
$\overline{d}_{k}$ and $\overline{\mathfrak{d}}_{k}$ by the composites 
\[\overline{d}_{k}\colon\nameddright{X^{\wedge k}\wedge D}{\overline{c}_{k}}{\Omega\Sigma X\ltimes E} 
      {\overline{a}}{E}\] 
\[\overline{\mathfrak{d}}_{k}\colon\nameddright{X^{\wedge k}\wedge D}{\overline{\mathfrak{c}}_{k}} 
      {\Omega\Sigma X\ltimes E}{\overline{a}}{E}\] 

\begin{lemma} 
   \label{frakdcongruent} 
   If $k=0$ or $k=1$ then $\overline{d}_{k}=\overline{\mathfrak{d}}_{k}$. If $k\geq 2$ then 
   $\overline{d}_{k}$ is congruent to $\overline{\mathfrak{d}}_{k}$. 
\end{lemma} 

\begin{proof} 
Argue just as for Lemma~\ref{dcongruent}, replacing 
\(\namedright{\Sigma Y}{g}{E}\) 
by 
\(\namedright{\Sigma D}{\delta}{E}\). 
\end{proof} 

The congruence between $\overline{d}_{k}$ and $\overline{\mathfrak{d}}_{k}$ 
lets us alter Proposition~\ref{JkEequiv} to the following. 

\begin{proposition} 
   \label{JkEequivfrak} 
   Suppose that there is a homotopy fibration sequence 
   \(\namedddright{\Omega\Sigma X}{\partial}{E}{p}{Y}{h}{\Sigma X}\) 
   where $\Omega h$ has a right homotopy inverse. Suppose that there is a map 
   \(\delta\colon\namedright{\Sigma D}{}{E}\) 
   such that the composite 
   \[\nameddright{\Omega\Sigma X\ltimes\Sigma D}{1\ltimes\delta}{\Omega\Sigma X\ltimes E}{\overline{a}}{E}\] 
   is a homotopy equivalence. Then the composite 
   \[\lllnamedright{\Omega\Sigma X\ltimes(X^{\wedge k}\wedge\Sigma D)\vee(J_{k-1}(X)\ltimes\Sigma D)} 
      {(1\ltimes\overline{\mathfrak{c}}_{k})\perp\overline{J}_{k}}{\Omega\Sigma X\ltimes(\Omega\Sigma X\ltimes E)} 
      \stackrel{1\ltimes\overline{a}}{\longrightarrow}\namedright{\Omega\Sigma X\ltimes E}{\overline{a}}{E}\] 
   is a homotopy equivalence. 
\end{proposition} 
   
\begin{proof} 
By Proposition~\ref{JkEequiv} there is a homotopy equivalence   
\begin{equation} 
  \label{JkEequivhe} 
  \lllnamedright{\Omega\Sigma X\ltimes(X^{\wedge k}\wedge\Sigma D)\vee(J_{k-1}(X)\ltimes\Sigma D)} 
      {(1\ltimes\overline{c}_{k})\perp\overline{J}_{k}}{\Omega\Sigma X\ltimes(\Omega\Sigma X\ltimes E)} 
      \stackrel{1\ltimes\overline{a}}{\longrightarrow}\namedright{\Omega\Sigma X\ltimes E}{\overline{a}}{E}. 
 \end{equation}  
The restriction of this homotopy equivalence to $\Omega\Sigma X\ltimes(X^{\wedge k}\wedge\Sigma D)$ is 
$\overline{a}\circ(1\ltimes\overline{a})\circ(1\ltimes\overline{c}_{k})$. By definition of $\overline{d}_{k}$, 
this equals $\overline{a}\circ(1\ltimes\overline{d}_{k})$. 
By Lemma~\ref{frakdcongruent}, $\overline{d}_{k}$ is congruent to $\overline{\mathfrak{d}}_{k}$, 
so $\overline{a}\circ(1\ltimes\overline{d}_{k})$ is congruent to 
$\overline{a}\circ(1\ltimes\overline{\mathfrak{d}}_{k})$, which by definition of $\overline{\mathfrak{d}}_{k}$ 
equals $\overline{a}\circ(1\ltimes\overline{a})\circ(1\ltimes\overline{\mathfrak{c}}_{k})$. 
Thus the composite 
\begin{equation} 
  \label{JkEequivhe2} 
  \lllnamedright{\Omega\Sigma X\ltimes(X^{\wedge k}\wedge\Sigma D)\vee(J_{k-1}(X)\ltimes\Sigma D)} 
      {(1\ltimes\overline{\mathfrak{c}}_{k})\perp\overline{J}_{k}}{\Omega\Sigma X\ltimes(\Omega\Sigma X\ltimes E)} 
      \stackrel{1\ltimes\overline{a}}{\longrightarrow}\namedright{\Omega\Sigma X\ltimes E}{\overline{a}}{E} 
\end{equation}  
is congruent to~(\ref{JkEequivhe}). Consequently, both~(\ref{JkEequivhe}) and~(\ref{JkEequivhe2}) 
induce the same map in homology. Since~(\ref{JkEequivhe}) is a homotopy equivalence 
it induces an isomorphism in homology. Thus~(\ref{JkEequivhe2}) also induces an isomorphism 
in homology. By hypothesis, $E\simeq\Omega\Sigma X\ltimes\Sigma D$ so $E$ is simply-connected. 
Thus the domain and range of~(\ref{JkEequivhe2}) are simply-connected so the fact that it induces 
an isomorphism in homology implies that it is a homotopy equivalence by Whitehead's Theorem. 
\end{proof} 

Recall from the setup at the beginning of the section that the restriction of 
\(\namedright{\Omega\Sigma X\ltimes\Sigma A}{\theta}{E}\) 
to $\Sigma A$ is a map 
\(g\colon\namedright{\Sigma A}{}{E}\)  
that lifts $f$ through $p$. 
      
\begin{theorem} 
   \label{E'typeI} 
   Suppose that there is a homotopy fibration sequence 
   \(\namedddright{\Omega\Sigma X}{\partial}{E}{p}{Y}{h}{\Sigma X}\) 
   with the following properties: 
   \begin{itemize} 
      \item[(a)] $\Omega h$ has a right homotopy inverse; 
      \item[(b)] there is a map 
                     \(\delta\colon\namedright{\Sigma D}{}{E}\) 
                     such that the composite 
                      \[\nameddright{\Omega\Sigma X\ltimes\Sigma D}{1\ltimes\delta} 
                              {\Omega\Sigma X\ltimes E}{\overline{a}}{E}\] 
                     is a homotopy equivalence; 
     \item[(c)] $g$ can be chosen to factor as a composite 
                     \[g\colon\namedddright{\Sigma A}{\ell}{X^{\wedge k}\wedge\Sigma D}{\overline{\mathfrak{c}}_{k}} 
                           {\Omega\Sigma X\ltimes E}{\overline{a}}{E}\] 
                     for some map $\ell$. 
    \end{itemize} 
   Let $C$ be the homotopy cofibre of $\ell$. Then there is a homotopy equivalence 
   \[\namedright{(\Omega\Sigma X\ltimes C)\vee (J_{k-1}(X)\ltimes\Sigma D)}{}{E'}.\]  
\end{theorem} 

\begin{proof} 
Consider the diagram 
\[\spreaddiagramcolumns{-0.5pc}\diagram 
     \Omega\Sigma X\ltimes\Sigma A\xdouble[rrr]\dto^{(1\ltimes\ell)+\ast} 
          & & & \Omega\Sigma X\ltimes\Sigma A\dto^{1\ltimes g}\drto^{\theta} & \\ 
     \Omega\Sigma X\ltimes(X^{\wedge k}\wedge\Sigma D)\vee
            (J_{k-1}(X)\ltimes\Sigma D)\rrto^-{(1\ltimes\overline{\mathfrak{c}}_{k})\perp\overline{J}_{k}} 
          & & \Omega\Sigma X\ltimes(\Omega\Sigma X\ltimes E)\rto^-{1\ltimes\overline{a}} 
          & \Omega\Sigma X\ltimes E\rto^-{\overline{a}} & E.  
  \enddiagram\] 
Since $\Omega\Sigma X\ltimes\Sigma A$ maps trivially to $J_{k-1}(X)\ltimes\Sigma D$, 
to show the rectangle homotopy commutes it suffices to show that 
$(1\ltimes\overline{a})\circ(1\ltimes\overline{\mathfrak{c}}_{k})\circ(1\ltimes\ell)\simeq 1\ltimes g$. 
But this holds by hypothesis~(c). The right triangle 
homotopy commutes by~(\ref{thetabara}). As hypotheses~(a) and~(b) hold, 
Proposition~\ref{JkEequivfrak} implies that the composite along the bottom row is a homotopy 
equivalence. Rewriting, there is a homotopy commutative square 
\begin{equation} 
  \label{genellthetadgrm} 
  \diagram 
       \Omega\Sigma X\ltimes\Sigma A\rdouble\dto^{(1\ltimes \ell)+\ast} 
            & \Omega\Sigma X\ltimes\Sigma A\dto^{\theta} \\
       \Omega\Sigma X\ltimes(X^{\wedge k}\wedge\Sigma D)\vee(J_{k-1}(X)\ltimes\Sigma D)\rto^-{\simeq} 
            & E. 
   \enddiagram 
\end{equation}     
As the homotopy cofibre of $\ell$ is $C$, the homotopy cofibre of $(1\ltimes\ell)+\ast$ 
in~(\ref{genellthetadgrm}) is 
\[(\Omega\Sigma X\ltimes C)\vee(J_{k-1}(X)\ltimes\Sigma D).\] 
As the homotopy 
cofibre of $\theta$ is $E'$, the homotopy commutativity of~(\ref{genellthetadgrm}) implies that 
there is an induced map of cofibres 
\[\alpha\colon\namedright{(\Omega\Sigma X\ltimes C)\vee(J_{k-1}(X)\ltimes\Sigma D)}{}{E'}.\] 
The homotopy equivalence in~(\ref{genellthetadgrm}) and the five-lemma then imply that $\alpha$ 
induces an isomorphism in homology and so is a homotopy equivalence by Whitehead's Theorem.  
\end{proof} 

\begin{remark} 
\label{E'Jmap} 
Note from the proof of Theorem~\ref{E'typeI} that the restriction of the 
homotopy equivalence 
\(\namedright{(\Omega\Sigma X\ltimes C)\vee (J_{k-1}(X)\ltimes\Sigma D)}{}{E'}\) 
to $J_{k-1}(X)\ltimes\Sigma D$ is homotopic to the composite 
\begin{equation} 
  \label{E'Jfirst} 
  \namedddright{J_{k-1}(X)\ltimes\Sigma D}{\overline{J}_{k}}{\Omega\Sigma X\ltimes(\Omega\Sigma X\ltimes E)} 
      {1\ltimes\overline{a}}{\Omega\Sigma X\ltimes E}{\overline{a}}{E}\longrightarrow E'. 
\end{equation}  
It will be useful for later to give an alternative description of this composite. Consider the diagram 
\[\diagram 
      J_{k-1}(X)\ltimes\Sigma D\rto^-{j_{k-1}\ltimes\delta}\ddouble 
           & \Omega\Sigma X\ltimes E\rto^-{\overline{a}}\dto^{j} 
           & E\dto^{j}\drdouble & \\ 
      J_{k-1}(X)\ltimes\Sigma D\rto^-{\overline{J}_{k}} 
           & \Omega\Sigma X\ltimes(\Omega\Sigma X\ltimes E)\rto^-{1\ltimes\overline{a}} 
           & \Omega\Sigma X\ltimes E\rto^-{\overline{a}} & E. 
   \enddiagram\] 
The left square commutes by definition of $\overline{J}_{k}$, the middle square 
commutes by the naturality of~$j$, and the right square homotopy commutes since $\overline{a}$ 
is a quotient of the action $a$ and $a$ restricts to the identity map on $E$. The diagram 
therefore implies that~(\ref{E'Jfirst}) is homotopic to the composite 
$\overline{a}\circ(j_{k-1}\ltimes\delta)$. Writing $j_{k-1}\ltimes\delta$ as 
$(1\ltimes\delta)\circ(j_{k-1}\ltimes 1)$, we conclude that~(\ref{E'Jfirst}) is homotopic to 
the composite 
\[\llnamedddright{J_{k-1}(X)\ltimes\Sigma D}{j_{k-1}\ltimes 1}{\Omega\Sigma X\ltimes\Sigma D} 
         {\overline{a}\circ(1\ltimes\delta)}{E}{}{E'}\] 
where $\overline{a}\circ(1\ltimes\delta)$ is assumed to be a homotopy equivalence 
in hypothesis~(c) of Theorem~\ref{E'typeI}. 
\end{remark} 

An interesting general example of Theorem~\ref{E'typeI} is the following. Start with the 
homotopy fibration sequence 
\[\namedddright{\Omega\Sigma X}{}{E}{p}{\Sigma X\vee\Sigma Y}{q_{1}}{\Sigma X}.\] 
Fix a positive integer $k$ and take $A=X^{\wedge k}\wedge Y$. 
Recall the definition of $M_{k}$ as the homotopy cofibration 
\[\lllnameddright{X^{\wedge k}\wedge\Sigma Y}{ad^{k}(i_{1})(i_{2})}{\Sigma X\vee\Sigma Y}{}{M_{k}}.\] 
Observe that $q_{1}$ extends to a map 
\(q'_{k}\colon\namedright{M_{k}}{}{\Sigma X}\). 
By Example~\ref{specialGTadexample} there is a map  
\(\namedright{\Sigma Y}{\delta}{E}\) 
lifting~$i_{2}$ through $p$ and for which there is a homotopy equivalence 
\[\nameddright{\Omega\Sigma X\ltimes\Sigma Y}{1\ltimes\delta}{\Omega\Sigma X\ltimes E} 
        {\overline{a}}{E}.\] 
(The map 
\(\namedright{\Sigma Y}{}{E}\) 
was called ``$g$" in Example~\ref{specialGTadexample} for use in Theorem~\ref{BTfibinclusion} 
but in the setup for Theorem~\ref{E'typeI} we are about to consider this map will play the role of $\delta$  
and the Whitehead product $ad^{k-1}(i_{1})(i_{2})$ will play the role of $g$.) Let $p'_{k}$ be the composite 
\[p'_{k}\colon\llnamedddright{J_{k-1}(X)\ltimes\Sigma Y}{j_{k-1}\ltimes 1}{\Omega\Sigma X\ltimes\Sigma Y} 
      {\overline{a}\circ(1\ltimes\delta)}{E}{p}{\Sigma X\vee\Sigma Y}\longrightarrow M_{k}.\] 

\begin{lemma} 
   \label{Mtype} 
   For $k\geq 1$, there is a homotopy fibration 
   \[\nameddright{J_{k-1}(X)\ltimes\Sigma Y}{p'_{k}}{M_{k}}{q'_{k}}{\Sigma X}\] 
   which splits after looping to give a homotopy equivalence 
   \[\Omega M_{k}\simeq\Omega\Sigma X\times\Omega(J_{k-1}(X)\ltimes\Sigma Y).\] 
\end{lemma} 

\begin{proof} 
We wish to apply Theorem~\ref{E'typeI} to the homotopy fibration sequence 
\[\namedddright{\Omega\Sigma X}{}{E}{p}{\Sigma X\vee\Sigma Y}{q_{1}}{\Sigma X}\] 
with appropriate choices of the maps $\delta$ and $g$. To do so, hypotheses~(a) to (c) 
for the theorem need to be checked. As~$q_{1}$ has a right homotopy inverse so 
does $\Omega q_{1}$, and therefore hypothesis~(a) holds. The map $\delta$ has 
the property that the composite 
\(\nameddright{\Omega\Sigma X\ltimes\Sigma Y}{1\ltimes\delta}{\Omega\Sigma X\ltimes E} 
        {\overline{a}}{E}\) 
is a homotopy equivalence so hypothesis~(b) holds. Notice that with this map $\delta$ 
the definition of $\overline{\mathfrak{d}}_{k}$ identifies with the definition of $\mathfrak{d}_{k}$ 
in Section~\ref{sec:ad}. Thus, relabelling $\mathfrak{d}_{k}$ in Lemma~\ref{dbarkaction} 
by $\overline{\mathfrak{d}}_{k}$, and using the definition of $\overline{\mathfrak{d}}_{k}$ 
as $\overline{a}\circ\overline{\mathfrak{c}}_{k}$, by  
Lemma~\ref{dbarkaction} there is a homotopy commutative diagram  
\[\diagram 
       X^{\wedge k}\wedge\Sigma Y 
              \rto^-{\overline{\mathfrak{c}}_{k}}\drrto_{ad^{k}(i_{1})(i_{2})} 
            & \Omega\Sigma X\ltimes E\rto^-{\overline{a}} & E\dto^{p} \\ 
        & & \Sigma X\vee\Sigma Y. 
  \enddiagram\] 
Thus a map $g$ lifting $ad^{k}(i_{1})(i_{2})$ through~$p$ is~$\overline{a}\circ\overline{\mathfrak{c}}_{k}$. 
Taking $\ell$ to be the identity map on $X^{\wedge k}\wedge\Sigma Y$, the factorization 
of $g$ as $\overline{a}\circ\overline{\mathfrak{c}}_{k}\circ\ell$ satisfies hypothesis~(c). Therefore 
all the hypotheses of Theorem~\ref{E'typeI} hold. 
Since $\ell$ is the identity map on $X^{\wedge k}\wedge\Sigma Y$, its homotopy cofibre $C$ is 
a point. Thus, by Theorem~\ref{E'typeI}, if $E'$ is the homotopy fibre of 
\(\namedright{M_{k}}{q'_{k}}{\Sigma X}\), 
then there is a homotopy equivalence $J_{k-1}(X)\ltimes\Sigma Y\simeq E'$. The 
fact that $\Omega q_{1}$ has a right homotopy inverse then implies that there are 
homotopy equivalences 
\[\Omega M\simeq\Omega\Sigma X\times\Omega E'\simeq 
      \Omega\Sigma X\times\Omega(J_{k-1}(X)\ltimes\Sigma Y).\] 
   
It remains to identify the composite 
\(\nameddright{J_{k-1}(X)\ltimes\Sigma Y}{\simeq}{E'}{}{M_{k}}\) 
as $p'_{k}$. By Remark~\ref{E'Jmap}, the homotopy equivalence 
$J_{k-1}(X)\ltimes\Sigma Y\simeq E'$ is realized by the composite 
\[\llnamedddright{J_{k-1}(X)\ltimes\Sigma Y}{j_{k-1}\ltimes 1} 
        {\Omega\Sigma X\ltimes\Sigma Y}{\overline{a}\circ(1\ltimes\delta)}{E}{}{E'}.\]  
Composing with 
\(\namedright{E'}{}{M_{k}}\) 
and using the fact that 
\(\nameddright{E}{}{E'}{}{M_{k}}\) 
is homotopic to 
\(\nameddright{E}{p}{\Sigma X\vee\Sigma Y}{}{M_{k}}\) 
then shows that 
\(\nameddright{J_{k-1}(X)\ltimes\Sigma Y}{\simeq}{E'}{}{M_{k}}\) 
is homotopic to the composite 
\[\llnamedddright{J_{k-1}(X)\ltimes\Sigma Y}{j_{k-1}\ltimes 1}{\Omega\Sigma X\ltimes\Sigma Y} 
      {\overline{a}\circ(1\ltimes\delta)}{E}{p}{\Sigma X\vee\Sigma Y}\longrightarrow M_{k}\] 
which is the definition of $p'_{k}$. 
\end{proof} 

In particular, observe that if $k=1$ then we have attached $ad(i_{1})(i_{2})=[i_{1},i_{2}]$ so 
$M_{1}=\Sigma X\times\Sigma Y$, and $\bigvee_{k=0}^{0} X^{k}\wedge\Sigma Y=\Sigma Y$, 
so we recover the usual homotopy fibration 
\(\nameddright{\Sigma Y}{}{\Sigma X\times\Sigma Y}{}{\Sigma X}\). 

Lemma~\ref{Mtype} identifies the homotopy fibre of $q'_{k}$ as $J_{k-1}(X)\ltimes\Sigma Y$ 
and the map from the fibre to the total space as $p'_{k}$, but we would like an alternate 
description of the fibre that identifies the map from it to $M_{k}$ as a wedge sum of 
Whitehead products. Since $p'_{k}$ depends on the inclusion 
\(\namedright{J_{k-1}(X)}{j_{k-1}}{\Omega\Sigma X}\) 
it blends well with the multiplication on $\Omega\Sigma X$, so we will make use of 
congruence again to make the conversion from multiplication to Whitehead products. 

Define the map $\gamma_{k}$ by the composite 
\[\gamma_{k}\colon\bigvee_{t=0}^{k-1} X^{\wedge t}\wedge\Sigma Y 
      \stackrel{\bigvee_{t=0}^{k-1} ad^{t}(i_{1})(i_{2})}{\llllarrow}\namedright{\Sigma X\vee\Sigma Y}{}{M_{k}}.\] 
We now prove Theorem~\ref{Mtypealtintro}, restated verbatim.  

\begin{theorem} 
   \label{Mtypealt} 
   For $k\geq 1$, there is a homotopy fibration 
   \[\nameddright{\bigvee_{t=0}^{k-1} X^{\wedge t}\wedge\Sigma Y}{\gamma_{k}}{M_{k}}{q'_{k}}{\Sigma X}\] 
   which splits after looping to give a homotopy equivalence 
   \[\Omega M_{k}\simeq\Omega\Sigma X\times\Omega(\bigvee_{t=0}^{k-1} X^{\wedge k}\wedge\Sigma Y).\] 
\end{theorem} 

\begin{proof} 
By Lemma~\ref{cequiv} the map 
\(\namedright{\bigvee_{t=0}^{\infty} X^{\wedge t}\wedge\Sigma Y}{c}{\Omega\Sigma X\ltimes\Sigma Y}\) 
is a homotopy equivalence. By definition, $c$ is the wedge sum of maps 
$c_{t}$ for $0\leq t<\infty$, and the definition of $c_{t}$ via the mulitiplication on $\Omega\Sigma X$ 
implies that it factors through 
\(\namedright{J_{k-1}(X)}{j_{k-1}}{\Omega\Sigma X}\) 
if $t\leq k-1$. Thus there is a homotopy commutative square 
\begin{equation} 
  \label{ctoc'} 
  \diagram 
       \bigvee_{t=0}^{k-1} X^{\wedge t}\wedge\Sigma Y\rto^-{c'_{k-1}}\dto^{I} 
            & J_{k-1}(X)\ltimes\Sigma Y\dto^{j_{k-1}\ltimes 1} \\ 
       \bigvee_{t=0}^{\infty} X^{\wedge k}\wedge\Sigma Y\rto^-{c}
            & \Omega\Sigma X\ltimes\Sigma Y 
  \enddiagram 
\end{equation} 
where $I$ is the inclusion and $c'_{t}$ lifts $\bigvee_{t=0}^{k-1} c_{t}$ through $j_{k-1}$. 
The same argument in Lemma~\ref{cequiv} that shows $c$ is a homotopy equivalence 
also shows that $c'_{k-1}$ is also a homotopy equivalence. 

Next, since $q'_{k}$ factors through $q_{1}$ there is a homotopy fibration diagram 
\begin{equation} 
  \label{Ep'k} 
  \diagram 
       E\rto^-{\lambda}\dto^{p} & J_{k-1}(X)\ltimes\Sigma Y\dto^{p'_{k}} \\ 
       \Sigma X\vee\Sigma Y\rto\dto^{q_{1}} & M_{k}\dto^{q'_{k}} \\ 
       \Sigma X\rdouble & \Sigma X 
  \enddiagram 
\end{equation} 
for some induced map $\lambda$ of fibres. Consider the composite 
\[\kappa\colon\llnamedddright{J_{k-1}(X)\ltimes\Sigma Y}{j_{k-1}\ltimes 1}{\Omega\Sigma X\ltimes\Sigma Y} 
      {\overline{a}\circ(1\ltimes\delta)}{E}{\lambda}{J_{k-1}(X)\ltimes\Sigma Y}.\] 
We claim that $\kappa$ is homotopic to the identity map. To see this, observe that by~(\ref{Ep'k}), 
$p'_{k}\circ\kappa$ is homotopic to the composite 
\(\llnamedddright{J_{k-1}(X)\ltimes\Sigma Y}{j_{k-1}\ltimes 1}{\Omega\Sigma X\ltimes\Sigma Y} 
      {\overline{a}\circ(1\ltimes\delta)}{E}{p}{\Sigma X\vee\Sigma Y}\longrightarrow M_{k}\), 
which is the definition of $p'_{k}$. Thus $p'_{k}\circ\kappa\simeq p'_{k}$. Since 
$J_{k-1}(X)\ltimes\Sigma Y$ is a suspension, we may subtract in order to get 
$p'_{k}\circ(1-\kappa)\simeq\ast$. There is a homotopy fibration 
\(\nameddright{\Omega\Sigma X}{s}{J_{k-1}(X)\ltimes\Sigma Y}{p'_{k}}{M_{k}}\) 
where $s$ is null homotopic by Lemma~\ref{Mtype}. The null homotopy for $p'_{k}\circ(1-\kappa)$  
implies that $1-\kappa$ lifts through $s$, and hence is null homotopic. Thus $\kappa$ is 
homotopic to the identity map. 

Putting~(\ref{ctoc'}) and~(\ref{Ep'k}) together gives a homotopy commutative diagram 
\[\diagram 
       \bigvee_{t=0}^{k-1} X^{\wedge t}\wedge\Sigma Y\rto^-{c'_{k-1}}\dto^{I} 
            & J_{k-1}(X)\ltimes\Sigma Y\rrto^-{\overline{a}\circ(j_{k-1}\ltimes\delta)}\dto^{j_{k-1}\ltimes 1}
            & & E\rdouble\ddouble & E\rto^-{\lambda}\dto^{p} & J_{k-1}(X)\ltimes\Sigma Y\dto^{p'_{k}} \\ 
       \bigvee_{t=0}^{\infty} X^{\wedge k}\wedge\Sigma Y\rto^-{c}
            & \Omega\Sigma X\ltimes\Sigma Y\rrto^-{\overline{a}\circ(1\ltimes\delta)} 
            & & E\rto^-{p} & \Sigma X\vee\Sigma Y\rto & M_{k}. 
  \enddiagram\]
Note that the middle squares commute. Along the top row both $c'_{k-1}$ and 
$\kappa=\lambda\circ\overline{a}\circ(j_{k-1}\ltimes\delta)$ are homotopy equivalences. 
In particular, in homology $\lambda_{\ast}\circ(\overline{a}\circ(j_{k-1}\ltimes\delta)\circ c'_{k-1})_{\ast}$ 
is an isomorphism. Along the bottom row the composite $\overline{a}\circ(1\ltimes\delta)\circ c$ is the 
definition of the map $d$ in Section~\ref{sec:ad}. By Lemma~\ref{bardequiv}, $d$ is congruent 
to the map $\mathfrak{d}$. In particular, in homology $d_{\ast}=\mathfrak{d}_{\ast}$. Therefore 
$(\mathfrak{d}\circ I)_{\ast}=(d\circ I)_{\ast}$, which by the previous diagram equals 
$(\overline{a}\circ(j_{k-1}\ltimes\delta)\circ c'_{k-1})_{\ast}$. Hence 
$\lambda_{\ast}\circ(\mathfrak{d}\circ I)_{\ast}$ is an isomorphism, implying that 
$\lambda\circ\mathfrak{d}\circ I$ is a homotopy equivalence by Whitehead's Theorem. 

By Theorem~\ref{dbard} there is a homotopy commutative diagram 
\[\diagram 
        \bigvee_{t=0}^{k-1} X^{\wedge t}\wedge\Sigma Y\rto^-{I} 
             & \bigvee_{t=0}^{\infty} X^{\wedge t}\wedge\Sigma Y  
                    \rto^-{\mathfrak{d}}\drto_{\bigvee_{t=0}^{\infty} ad^{t}(i_{1})(i_{2})} 
             & E\rto^-{\lambda}\dto^{p} & J_{k-1}(X)\ltimes\Sigma Y\dto^{p'_{k}} \\ 
        & & \Sigma X\vee\Sigma Y\rto & M_{k}. 
  \enddiagram\] 
As the top row of this diagram is a homotoopy equivalence and the bottom 
direction around the diagram matches the definition of $\gamma_{k}$,  
the homotopy fibration 
\(\nameddright{J_{k-1}(X)\ltimes\Sigma Y}{p'_{k}}{M_{k}}{q'_{k}}{\Sigma X}\) 
may be replaced up to equivalence by a homotopy fibration 
\(\nameddright{\bigvee_{t=0}^{k-1} X^{\wedge t}\wedge\Sigma Y}{\gamma_{k}}{M_{k}}{q'_{k}}{\Sigma X}\). 
This proves the first assertion of the lemma, the splitting of the fibration after looping 
follows from the existence of a right homotopy inverse of $\Omega q'_{k}$. 
\end{proof} 

Interesting specific examples occur when $X=S^{m}$ and $Y=S^{n}$. Note then that 
for each $k\geq 1$ we have $X^{\wedge k}\wedge\Sigma Y\simeq S^{km+n+1}$. 

\begin{example} 
\label{Mkexample} 
For $k\geq 1$, define $M_{k}$ by the homotopy cofibration 
\[\lllnameddright{S^{km+n+1}}{ad^{k}(i_{1})(i_{2})}{S^{m+1}\vee S^{n+1}}{}{M_{k}}.\] 
Then there is a homotopy fibration 
\[\nameddright{\bigvee_{t=0}^{k-1} S^{tm+n+1}}{\gamma_{k}}{M_{k}}{}{S^{m+1}}\] 
where $\gamma_{k}$ is the composite 
\[\bigvee_{t=0}^{k-1} S^{tm+n+1}\stackrel{\bigvee_{t=0}^{k-1} ad^{t}(i_{1})(i_{2})}{\llllarrow} 
      \namedright{S^{m+1}\vee S^{n+1}}{}{M_{k}},\]  
and after looping this homotopy fibration splits to give a homotopy equivalence 
\[\Omega M_{k}\simeq\Omega S^{m+1}\times\Omega\bigg(\bigvee_{t=0}^{k-1} S^{tm+n+1}\bigg).\] 
In particular, if $k=2$ then $M_{2}$ is the homotopy cofibre of $[i_{1},[i_{1},i_{2}]]$ 
and there is a homotopy equivalence 
$\Omega M_{2}\simeq\Omega S^{m+1}\times\Omega(S^{n+1}\vee S^{m+n+1})$. 
\end{example} 

A bit more generally, take $X=S^{m}$ and let $Y$ be arbitrary. Note then 
that for each $k\geq 1$ we have $X^{\wedge k}\wedge\Sigma Y\simeq\Sigma^{km+1} Y$. 

\begin{example} 
\label{YMkexample} 
For $k\geq 1$, define $M_{k}$ by the homotopy cofibration 
\[\lllnameddright{\Sigma^{km+1} Y}{ad^{k}(i_{1})(i_{2})}{S^{m+1}\vee\Sigma Y}{}{M_{k}}.\] 
Then there is a homotopy fibration 
\[\nameddright{\bigvee_{t=0}^{k-1}\Sigma^{tm+1} Y}{\gamma_{k}}{M_{k}}{}{S^{m+1}}\] 
where $\gamma_{k}$ is the composite 
\[\bigvee_{t=0}^{k-1}\Sigma^{tm+1} Y\stackrel{\bigvee_{t=0}^{k-1} ad^{t}(i_{1})(i_{2})}{\llllarrow}  
      \namedright{S^{m+1}\vee\Sigma Y}{}{M_{k}},\] 
and after looping this homotopy fibration splits to give a homotopy equivalence 
\[\Omega M_{k}\simeq\Omega S^{m+1}\times\Omega\bigg(\bigvee_{t=0}^{k-1}\Sigma^{tm+1} Y\bigg).\] 
In particular, if $k=2$ then $M_{2}$ is the homotopy cofibre of $[i_{1},[i_{1},i_{2}]]$ 
and there is a homotopy equivalence 
$\Omega M_{2}\simeq\Omega S^{m+1}\times\Omega(\Sigma Y\vee\Sigma^{m+1} Y)$. 
\end{example} 

Even more can be said about the homotopy theory of the spaces $M_{k}$. Return to the 
general case, of a homotopy cofibration 
\(\llnameddright{X^{\wedge k}\wedge\Sigma Y}{ad^{k}(i_{1})(i_{2})}{\Sigma X\vee\Sigma Y}{m_{k}}{M_{k}}\), 
where the inclusion of $\Sigma X\vee\Sigma Y$ into $M_{k}$ is now labelled $m_{k}$.   
In Corollary~\ref{adinvcor} we identify the homotopy fibre of $m_{k}$. 
To compress notation, write $ad^{k}$ for $ad^{k}(i_{1})(i_{2})$. 

\begin{lemma} 
   \label{adinv} 
   The homotopy cofibration 
   \(\nameddright{X^{\wedge k}\wedge\Sigma Y}{ad^{k}}{\Sigma X\vee\Sigma Y}{m_{k}}{M_{k}}\) 
   has the property that the map $\Omega m_{k}$ has a right homotopy inverse. 
\end{lemma} 

\begin{proof}   
By Theorem~\ref{Mtypealt}, there is a homotopy fibration 
\[\nameddright{\bigvee_{t=0}^{k-1} X^{\wedge t}\wedge\Sigma Y}{\gamma_{k}}{M_{k}}{q'_{k}}{\Sigma X}\] 
which splits after looping to give a homotopy equivalence 
\[\Omega M_{k}\simeq\Omega\Sigma X\times\Omega(\bigvee_{t=0}^{k-1} X^{\wedge k}\wedge\Sigma Y).\] 
Here, $\gamma_{k}$ is the composite 
\[\llnameddright{\bigvee_{t=0}^{k-1} X^{\wedge t}\wedge\Sigma Y}{\bigvee_{t=0}^{k-1} ad^{t}} 
        {\Sigma X\vee\Sigma Y}{m_{k}}{M_{k}},\] 
a right homotopy inverse for $q'_{k}$ is the composite 
\[i'_{1}\colon\nameddright{\Sigma X}{i_{1}}{\Sigma X\vee\Sigma Y}{m_{k}}{M_{k}},\] 
and the homotopy equivalence is given by the composite 
\begin{equation} 
  \label{Mkequiv} 
  \lllnamedright{\Omega\Sigma X\times\Omega(\bigvee_{i=0}^{k-1} X^{\wedge k}\wedge\Sigma Y)} 
     {\Omega i'_{1}\times\Omega\gamma_{k}}{\Omega M_{k}\times\Omega M_{k}} 
     \stackrel{\mu}{\longrightarrow}\Omega M_{k} 
\end{equation} 
where $\mu$ is the standard loop multiplication. Observe that both $i'_{1}$ and $\gamma_{k}$ 
factor through $\Sigma X\vee\Sigma Y$, so as $\Omega m_{k}$ is an $H$-map the homotopy 
equivalence in~(\ref{Mkequiv}) is homotopic to the composite 
\[\lllnamedright{\Omega\Sigma X\times\Omega(\bigvee_{i=0}^{k-1} X^{\wedge k}\wedge\Sigma Y)} 
     {\Omega i_{1}\times\Omega a_{k}}{\Omega(\Sigma X\vee\Sigma Y)\times\Omega(\Sigma X\vee\Sigma Y)} 
     \stackrel{\mu}{\longrightarrow}\namedright{\Omega(\Sigma X\vee\Sigma Y)}{\Omega m_{k}}{\Omega M_{k}}\] 
where $a_{k}=\bigvee_{i=0}^{k-1} ad^{i}(i_{1})(i_{2})$. In particular, this implies that 
$\Omega m_{k}$ has a right homotopy inverse. 
\end{proof} 

By Lemma~\ref{adinv}, $\Omega m_{k}$ has a right homotopy inverse 
\(s\colon\namedright{\Omega M_{k}}{}{\Omega(\Sigma X\vee\Sigma Y)}\).  
The existence of $s$ implies that the hypotheses of Theorem~\ref{BTfibinclusion} are satisfied 
when applied to the homotopy cofibration 
\(\nameddright{X^{\wedge k}\wedge\Sigma Y}{ad^{k}}{\Sigma X\vee\Sigma Y}{ m_{k}}{M_{k}}\). 
Therefore we immediately obtain the following. 

\begin{corollary} 
   \label{adinvcor} 
   There is a homotopy fibration 
   \[\nameddright{\Omega M_{k}\ltimes(X^{\wedge k}\wedge\Sigma Y)}{\chi}{\Sigma X\vee\Sigma Y} 
         {m_{k}}{M_{k}}\] 
   where $\chi$ is the sum of the maps 
   \[\nameddright{\Omega M_{k}\ltimes(X^{\wedge k}\wedge\Sigma Y)}{\pi}{X^{\wedge k}\wedge\Sigma Y} 
          {ad^{k}}{\Sigma X\vee\Sigma Y}\] 
   and  
   \[\lllnameddright{\Omega M_{k}\ltimes(X^{\wedge k}\wedge\Sigma Y)}{q} 
           {\Omega M_{k}\wedge (X^{\wedge k}\wedge\Sigma Y)}{[ev\circ s, ad^{k}]}{\Sigma X\vee\Sigma Y}.\]  
\end{corollary} 
\vspace{-0.8cm}~$\qqed$\bigskip

\begin{example} 
Consider the homotopy cofibration 
\(\lllnamedright{S^{km+n+1}}{ad^{k}(i_{1})(i_{2})}{S^{m+1}\vee S^{n+1}} 
     \stackrel{m_{k}}{\longrightarrow} M_{k}\) 
from Example~\ref{Mkexample}. By Lemma~\ref{adinv} the map $\Omega m_{k}$ 
has a right homotopy inverse and by Theorem~\ref{BTfibinclusion} there is a homotopy fibration 
\[\nameddright{\Omega M_{k}\ltimes S^{km+n+1}}{\chi}{S^{m+1}\vee S^{n+1}}{m_{k}}{M_{k}}\] 
where $\chi$ is the sum of 
\(\nameddright{\Omega M_{k}\ltimes S^{km+n+1}}{\pi}{S^{km+n+1}}{ad^{k}}{S^{m+1}\vee S^{n+1}}\) 
and 
\(\namedright{\Omega M_{k}\ltimes S^{km+n+1}}{}{\Omega M_{k}\wedge S^{km+n+1}} 
      \stackrel{[ev\circ s, ad^{k}]}{\lllarrow} S^{m+1}\vee S^{n+1}\). 
\end{example} 

\noindent 
\textbf{Relation to Moore's Conjecture}. 
A few definitions are necessary to state the conjecture. 

\begin{definition} 
Let $X$ be a simply-connected $CW$-complex and let $p$ be a prime. The 
\emph{homotopy exponent} of $X$ at $p$ is the least power of $p$ that annihilates 
the $p$-torsion in $\pi_{\ast}(X)$. 
\end{definition} 

Write $\exp_{p}(X)=p^{r}$ if $p^{r}$ is this least power of $p$. If the prime is understood 
this may be shortened to $\exp(X)=p^{r}$. If $\pi_{\ast}(X)$ has torsion of all orders, 
write $\exp_{p}(X)=\infty$. 

\begin{definition} 
Let $X$ be a simply-connected $CW$-complex. If there are finitely many $\mathbb{Z}$ 
summands in $\pi_{\ast}(X)$ then $X$ is \emph{elliptic}, otherwise $X$ is \emph{hyperbolic}. 
\end{definition} 

\begin{conjecture}[Moore] 
Let $X$ be a simply-connected finite $CW$-complex. Then the following are equivalent: 
\begin{itemize} 
   \item $X$ is elliptic; 
   \item $\exp_{p}(X)$ is finite for some prime $p$; 
   \item $\exp_{p}(X)$ is finite for all primes $p$. 
\end{itemize} 
\end{conjecture} 

Moore's Conjecture posits a remarkable relationship between the rational homotopy groups 
of~$X$ and its torsion homotopy groups. The rational homotopy groups deeply influence 
the torsion homotopy groups, and torsion at any one prime deeply influences the torsion 
that occurs at any prime. The conjecture has been shown to hold in a wide variety of 
cases: for $H$-spaces~\cite{L}, for torsion-free suspensions~\cite{Se}, for various 
$2$ and $3$-cell complexes~\cite{NS}, for generalized moment-angle complexes~\cite{HST}, 
and for families of highly-connected Poincar\'{e} duality complexes~\cite{Ba,BB,BT1,BT2}.  

More examples of spaces for which Moore's Conjecture holds may be extracted from 
Proposition~\ref{Mtypealt}. We give three examples. 

\begin{example} 
\label{MkexampleMoore}
Return to Example~\ref{Mkexample}. The $k=1$ case has $M_{k}\simeq S^{m+1}\times S^{n+1}$. 
A sphere was shown to have a finite homotopy exponent for any prime $p$ by James~\cite{J2} 
for $p=2$ and Toda~\cite{To} for any odd prime. Therefore the product of a finite number of 
spheres also has an exponent for any prime $p$, and of course, this product is elliptic. 
The case when the attaching map $[i_{1},i_{2}]$ for the top cell in $S^{m+1}\times S^{n+1}$ 
is replaced by $q\cdot [i_{1},i_{2}]$ for a prime $q$ a nonzero integer was considered by 
Neisendorfer and Selick~\cite{NS} and shown to also be elliptic and have an exponent 
at every prime~$p$. Now suppose that $k\geq 2$. An argument using the Hilton-Milnor 
Theorem shows that a wedge of spheres is hyperbolic and has no exponent at any prime  
(see, for example,~\cite{NS}). Example~\ref{Mkexample} shows that 
$\Omega M_{k}\simeq\Omega S^{m+1}\times\Omega(\bigvee_{t=0}^{k-1} S^{tm+n+1})$. 
In particular, $\Omega(S^{n+1}\vee S^{m+n+1})$ is a retract of $\Omega M_{k}$, and so 
$M_{k}$ must be hyperbolic and have no exponent at any prime. 
\end{example} 

\begin{remark} 
Using different methods, Anick~\cite[Lemma 2.3 and Theorem 3.7]{A} showed that 
any space obtained by attaching a sphere to a wedge of two or more spheres by a 
linear combination of Whitehead products satisfies Moore's Conjecture for all primes 
except possibly $2$ and $3$. In particular, this is true of the spaces $M_{k}$ for $k\geq 2$. 
Example~\ref{MkexampleMoore} is an improvement in this case as there is no restriction 
on the primes. Further, Example~\ref{Mkexample} goes much further by giving an explicit 
integral homotopy decomposition of $\Omega M_{k}$. 
\end{remark} 

Recall from the Introduction that for $n\geq 2$, $p$ a prime and $r\geq 1$, 
the \emph{mod-$p^{r}$ Moore space} $P^{n}(p^{r})$ 
is defined as the homotopy cofibre of the degree~$p^{r}$ map on $S^{n-1}$. Note that 
$\Sigma P^{n}(p^{r})\simeq P^{n+1}(p^{r})$. A useful property that will be needed at several 
points is a result of Neisendorfer~\cite[Corollary 6.6]{N1} that describes the homotopy type 
of the smash product of two mod-$p^{r}$ Moore spaces. 

\begin{lemma} 
   \label{mooresmash} 
   Let $p$ be a prime, $r$ a nonnegative integer, and assume that $p^{r}\neq 2$. 
   If $s,t\geq 2$ then there is a homotopy equivalence 
   \[P^{s}(p^{r})\wedge P^{t}(p^{r})\simeq P^{s+t}(p^{r})\vee P^{s+t-1}(p^{r}).\] 
\end{lemma} 
\vspace{-0.8cm}~$\qqed$\bigskip 

The $p^{r}=2$ case is very different: the smash product $P^{s}(2)\wedge P^{t}(2)$ is known 
to be indecomposable. 

By~\cite{N3}, if $n\geq 3$ and $p$ is odd then $\exp_{p}(P^{n}(p^{r}))=p^{r+1}$; by~\cite{Th}, 
if $n\geq 4$, $p=2$ and $r\geq 6$ then $\exp_{2}(P^{n}(2^{r}))=2^{r+1}$, and by~\cite{C}, 
if $n\geq 3$, $p=2$ and $r\geq 2$ then $P^{n}(2^{r})$ has a finite $2$-primary exponent. 
In all cases, as $P^{n}(p^{r})$ is contractible when localized at a prime $q\neq p$, or rationally, 
we see that $P^{n}(p^{r})$ is elliptic and has a finite homotopy exponent at every prime $p$ 
and so satisfies Moore's Conjecture. 

\begin{example} 
Return to Example~\ref{YMkexample}. Take $Y=P^{n}(p^{r})$ for $n\geq 3$ and $p^{r}\neq 2$. 
The example shows that 
\[\Omega M_{k}\simeq\Omega S^{m+1}\times\Omega(\bigvee_{t=0}^{k-1} P^{tm+n+1}(p^{r})).\] 
Using the Hilton-Milnor Theorem and Lemma~\ref{mooresmash} iteratively shows that the loops 
on a wedge of mod-$p^{r}$ Moore spaces with $p^{r}\neq 2$ is homotopy equivalent 
to a finite type infinite product of mod-$p^{r}$ Moore spaces. Consequently, $M_{k}$ is elliptic 
and has a finite exponent at every prime $p$. Hence $M_{k}$ satisfies Moore's Conjecture. 
\end{example} 

\begin{example} 
In Theorem~\ref{Mtypealt} take $X=P^{m}(p^{r})$ and $Y=P^{n}(p^{r})$ for $m,n\geq 3$ 
and $p^{r}\neq 2$. Then by Lemma~\ref{mooresmash} there is a homotopy equivalence 
\[\Omega M_{k}\simeq\Omega P^{m+1}(p^{r})\times\Omega W\] 
where $W$ is a finite type wedge of mod-$p^{r}$ Moore spaces. Arguing as in the 
previous example then shows that $M_{k}$ is elliptic (it is rationally trivial) and 
has a finite exponent at every prime $p$. Hence it satisfies Moore's Conjecture. 
\end{example}

\newpage 
\section{An improvement} 
\label{sec:improve} 

The factorization of 
\(\namedright{\Sigma A}{g}{E}\) 
through $X^{\wedge k}\wedge\Sigma D$ in Theorem~\ref{E'typeI} gives a condition which lets 
us find the homotopy type of $E'$, but it does not apply to many  of the cases we are interested in. 
The construction behind that theorem needs as input Whitehead products of the form $[i_{1},f]$ 
for some map $f$, but if we try to attach $S^{2n+1}$ to $S^{n+1}\vee S^{n+1}\vee S^{n+1}$ 
(with $X=S^{n}$ and $D=Y=S^{n}\vee S^{n}$) by $[i_{1},i_{2}]+[i_{2},i_{3}]$ then the map does not 
have the right form so the theorem does not apply. To handle the latter case we have to allow the 
attaching map to have some component in $\Sigma D$ as well as $X^{\wedge k}\wedge\Sigma D$, 
in other words, we need to consider the case where $g$ factors through $X^{\wedge k}\ltimes\Sigma D$. 

This will require some modifications to the strategy behind proving Theorem~\ref{E'typeI}. 
There were two key steps: first, take the half-smash of $\Omega\Sigma X$ with the factorization 
of $g$ as 
\(\namedddright{\Sigma A}{\ell}{X^{\wedge k}\wedge\Sigma D}{\overline{\mathfrak{c}}_{k}} 
       {\Omega\Sigma X\ltimes\Sigma D}{\overline{a}}{E}\) 
to obtain a homotopy commutative diagram      
\[\diagram 
       \Omega\Sigma X\ltimes\Sigma A\xdouble[rrr]\dto^{(1\ltimes \ell)} 
            & & & \Omega\Sigma X\ltimes\Sigma A\dto^{1\ltimes g}\drto^{\theta} & \\
       \Omega\Sigma X\ltimes(X^{\wedge k}\wedge\Sigma D))\rrto^-{1\ltimes\overline{\mathfrak{c}}_{k}} 
            & & \Omega\Sigma X\ltimes(\Omega\Sigma X\ltimes E)\rto^-{1\ltimes\overline{a}} 
            & \Omega\Sigma X\ltimes E\rto^-{\overline{a}} & E.  
   \enddiagram\] 
Second, adjust the bottom row by inserting the wedge $J_{k-1}\ltimes\Sigma D$ via the 
map $\overline{J}_{k}$ in order to obtain a homotopy equivalence along the bottom row. 

To modify this, first take the half-smash of $\Omega\Sigma X$ with a factorization of $g$ as 
\(\namedddright{\Sigma A}{\ell}{X^{\wedge k}\ltimes\Sigma D}{\overline{\mathfrak{c}}'_{k}} 
       {\Omega\Sigma X\ltimes\Sigma D}{\overline{a}}{E}\) 
for an appropriate map $\overline{\mathfrak{c}}'_{k}$ to obtain a homotopy commutative diagram      
\[\diagram 
       \Omega\Sigma X\ltimes\Sigma A\xdouble[rrr]\dto^{(1\ltimes \ell)} 
            & & & \Omega\Sigma X\ltimes\Sigma A\dto^{1\ltimes g}\drto^{\theta} & \\
       \Omega\Sigma X\ltimes(X^{\wedge k}\ltimes\Sigma D))\rrto^-{1\ltimes\overline{\mathfrak{c}}'_{k}} 
            & & \Omega\Sigma X\ltimes(\Omega\Sigma X\ltimes E)\rto^-{1\ltimes\overline{a}} 
            & \Omega\Sigma X\ltimes E\rto^-{\overline{a}} & E.  
   \enddiagram\] 
Then the bottom row has to be adjusted to obtain a homotopy equivalence. This adjustment 
involves more than just inserting an extra space, it also involves removing part of 
$\Omega\Sigma X\ltimes(X^{\wedge k}\ltimes\Sigma D)$, and this requires some 
extra hypotheses. The precise statement generalizing Theorem~\ref{E'typeI} is 
Theorem~\ref{E'typeII}. 
    
In general, let $B$, $C$ and $D$ be path-connected, pointed spaces. 
Define $e_{1}$, $e_{2}$ and $e_{3}$ by the composites 
\[e_{1}\colon\namedright{B\ltimes(C\wedge\Sigma D)}{1\ltimes i}{B\ltimes(C\ltimes\Sigma D)}\] 
\[e_{2}\colon\nameddright{\Sigma D}{j}{C\ltimes\Sigma D}{j}{B\ltimes(C\ltimes\Sigma D)}\] 
\[e_{3}\colon\nameddright{B\wedge\Sigma D}{i}{B\ltimes\Sigma D}{1\ltimes j}{B\ltimes(C\ltimes\Sigma D)}.\] 
Define maps $f_{1}$, $f_{2}$ and $f_{3}$ by the composites 
\[f_{1}\colon\namedright{B\ltimes(C\ltimes\Sigma D)}{1\ltimes q}{B\ltimes(C\wedge\Sigma D)}\] 
\[f_{2}\colon\nameddright{B\ltimes(C\ltimes\Sigma D)}{\pi}{C\ltimes\Sigma D}{\pi}{\Sigma D}\] 
\[f_{3}\colon\nameddright{B\ltimes(C\ltimes\Sigma D)}{1\ltimes\pi}{B\ltimes\Sigma D} 
       {q}{B\wedge\Sigma D}.\] 

\begin{lemma} 
   \label{efcomps} 
   For $1\leq i,j\leq 3$, the composite $f_{i}\circ e_{i}$ is homotopic to the identity map 
   while if $i\neq j$ then $f_{j}\circ e_{i}$ is null homotopic. 
\end{lemma} 

\begin{proof} 
In general, the composites 
\(\nameddright{\Sigma D}{j}{B\ltimes\Sigma D}{\pi}{\Sigma D}\) 
and 
\(\nameddright{B\wedge\Sigma D}{i}{B\ltimes\Sigma D}{q}{B\wedge\Sigma D}\) 
are homotopic to the identity maps while the composites 
\(\nameddright{\Sigma D}{j}{B\ltimes\Sigma D}{q}{B\wedge\Sigma D}\) 
and 
\(\nameddright{B\wedge\Sigma D}{i}{B\ltimes\Sigma D}{\pi}{\Sigma D}\) 
are null homotopic. The assertions now follow from the definitions of the 
maps $e_{i}$ and $f_{i}$ for $1\leq i\leq 3$. 
\end{proof} 

The wedge sum of $e_{1}$, $e_{2}$ and $e_{3}$ is a map 
\[e\colon\namedright{(B\ltimes(C\wedge\Sigma D))\vee\Sigma D\vee(B\wedge\Sigma D)} 
         {}{B\ltimes(C\ltimes\Sigma D)}.\]    
    
\begin{lemma} 
   \label{eequiv} 
   The map $e$ is a homotopy equivalence. 
\end{lemma} 

\begin{proof} 
The wedge sum of 
\(\namedright{C\wedge\Sigma D}{i}{C\ltimes\Sigma D}\) 
and 
\(\namedright{\Sigma D}{j}{C\ltimes\Sigma D}\) 
is a homotopy equivalence. Therefore, taking half-smashes with $B$, the wedge sum of 
\(\namedright{B\ltimes(C\wedge\Sigma D)}{1\ltimes i}{B\ltimes(C\ltimes\Sigma D)}\) 
and 
\(\namedright{B\ltimes\Sigma D}{1\ltimes j}{B\ltimes(C\ltimes\Sigma D)}\) 
is a homotopy equivalence. Notice that $1\ltimes i$ is the definition of $e_{1}$. Next, 
consider $1\ltimes j$. The wedge sum of 
\(\namedright{\Sigma D}{j}{B\ltimes\Sigma D}\) 
and 
\(\namedright{B\wedge\Sigma D}{i}{B\ltimes\Sigma D}\) 
is a homotopy equivalence. So $1\ltimes j$ may be rewritten as the wedge sum 
of the composites 
\(\nameddright{\Sigma D}{j}{B\ltimes\Sigma D}{j}{B\ltimes(C\ltimes\Sigma D)}\) 
and 
\(\nameddright{B\wedge\Sigma D}{i}{B\ltimes\Sigma D}{1\ltimes j}{B\ltimes(C\ltimes\Sigma D)}\), 
that is, $1\ltimes j$ may be rewritten as the wedge sum of~$e_{2}$ and $e_{3}$. 
Therefore the wedge sum of $e_{1}$, $e_{2}$ and $e_{3}$ (that is, $e$) is a homotopy equivalence. 
\end{proof} 

In our case, we start as in Theorem~\ref{E'typeI} with a homotopy fibration sequence 
\(\namedddright{\Omega\Sigma X}{\partial}{E}{p}{Y}{h}{\Sigma X}\) 
such that $\Omega h$ has a right homotopy inverse and there is a map 
\(\delta\colon\namedright{\Sigma D}{}{E}\) 
such that the composite 
\[\nameddright{\Omega\Sigma X\ltimes\Sigma D}{1\ltimes\delta} 
            {\Omega\Sigma X\ltimes E}{\overline{a}}{E}\] 
is a homotopy equivalence. Assume there is a map 
\(\namedright{\Sigma A}{}{Y}\) 
that lifts through $p$ to 
\(\namedright{\Sigma A}{g}{Y}\). 

The construction of the maps $e_{i}$ and $f_{i}$ above in this case are composites 
\[e_{1}\colon\lnamedright{\Omega\Sigma X\ltimes(X^{\wedge k}\wedge\Sigma D)}{1\ltimes i}  
         {\Omega\Sigma X\ltimes(X^{\wedge k}\ltimes\Sigma D)}\] 
\[e_{2}\colon\nameddright{\Sigma D}{j}{X^{\wedge k}\ltimes\Sigma D}{j} 
      {\Omega\Sigma X\ltimes(X^{\wedge k}\ltimes\Sigma D)}\]  
\[e_{3}\colon\nameddright{\Omega\Sigma X\wedge\Sigma D}{i}{\Omega\Sigma X\ltimes\Sigma D} 
      {1\ltimes j}{\Omega\Sigma X\ltimes(X^{\wedge k}\ltimes\Sigma D)}\]  
and 
\[f_{1}\colon\namedright{\Omega\Sigma X\ltimes(X^{\wedge k}\ltimes\Sigma D)}{1\ltimes q} 
       {\Omega\Sigma X\ltimes(X^{\wedge k}\wedge\Sigma D)}\] 
\[f_{2}\colon\nameddright{\Omega\Sigma X\ltimes(X^{\wedge k}\ltimes\Sigma D)}{\pi} 
       {X^{\wedge k}\ltimes\Sigma D}{\pi}{\Sigma D}\] 
 \[f_{3}\colon\nameddright{\Omega\Sigma X\ltimes(X^{\wedge k}\ltimes\Sigma D)}{1\ltimes\pi} 
       {\Omega\Sigma X\ltimes\Sigma D}{q}{\Omega\Sigma X\wedge\Sigma D}.\] 
By Lemma~\ref{efcomps}, for $1\leq i,j\leq 3$ the composite $f_{i}\circ e_{i}$ is homotopic 
to the identity map and if $i\neq j$ the composite $f_{j}\circ e_{i}$ is null 
homotopic. By Lemma~\ref{eequiv}, the wedge sum of $e_{1}$, $e_{2}$ and $e_{3}$ 
gives a homotopy equivalence 
\[e\colon\namedright{(\Omega\Sigma X\ltimes(X^{\wedge k}\wedge\Sigma D))\vee\Sigma D\vee 
      (\Omega\Sigma X\wedge\Sigma D)}{}{\Omega\Sigma X\ltimes(X^{\wedge k}\ltimes\Sigma D)}.\] 

Given a map 
\(\namedright{\Sigma A}{\ell}{X^{\wedge k}\ltimes\Sigma D}\), 
let $\kappa$ be the composite 
\[\kappa\colon\nameddright{\Omega\Sigma X\ltimes\Sigma A}{1\ltimes\ell} 
      {\Omega\Sigma X\ltimes(X^{\wedge k}\ltimes\Sigma D)}{e^{-1}} 
      {(\Omega\Sigma X\ltimes(X^{\wedge k}\wedge\Sigma D))\vee\Sigma D\vee(\Omega\Sigma X\wedge\Sigma D)}.\] 
By definition of $\kappa$ there is a commutative square 
\begin{equation} 
  \label{kappadgrm} 
  \diagram 
       \Omega\Sigma X\ltimes\Sigma A\rrdouble\dto^{\kappa} 
           & & \Omega\Sigma X\ltimes\Sigma A\dto^{1\ltimes\ell} \\ 
       (\Omega\Sigma X\ltimes(X^{\wedge k}\wedge\Sigma D))\vee\Sigma D\vee(\Omega\Sigma X\wedge\Sigma D) 
               \rrto^-{e} 
           & & \Omega\Sigma X\ltimes(X^{\wedge k}\ltimes\Sigma D). 
  \enddiagram 
\end{equation} 
By the Hilton-Milnor Theorem, $\kappa=\kappa_{1}+\kappa_{2}+\kappa_{3}+W$ 
where $\kappa_{1}$, $\kappa_{2}$ and $\kappa_{3}$ are obtained by composing~$\kappa$ 
with the pinch maps $p_{1}$, $p_{2}$ and $p_{3}$ to $\Omega\Sigma X\ltimes(X^{\wedge k}\wedge\Sigma D)$, 
$\Sigma D$ and $\Omega\Sigma X\wedge\Sigma D$ respectively, and~$W$ factors through 
a wedge sum of iterated Whitehead products. 

We identify $\kappa_{1}$, $\kappa_{2}$ and $\kappa_{3}$. Since $e$ is the wedge 
sum of $e_{1}$, $e_{2}$ and $e_{3}$, the fact that for $1\leq i,j\leq 3$ the composite 
$f_{i}\circ e_{i}$ is homotopic to the identity map while if $i\neq j$ the composite 
$f_{j}\circ e_{i}$ is null homotopic implies that $f_{i}\circ e\simeq p_{i}$. 
Thus using $e\circ\kappa=1\ltimes\ell$ in~(\ref{kappadgrm}) we obtain, for $1\leq i\leq 3$,  
\begin{equation} 
  \label{kappas} 
  \kappa_{i}=p_{i}\circ\kappa\simeq f_{i}\circ e\circ\kappa = f_{i}\circ(1\ltimes\ell). 
\end{equation} 

\begin{lemma} 
   \label{Sigmakappa3} 
   Suppose that the composite 
   \(\nameddright{\Sigma^{2} A}{\Sigma\ell}{\Sigma (X^{\wedge k}\ltimes\Sigma D)}{\Sigma\pi}{\Sigma^{2} D}\) 
   is null homotopic. Then the maps $\Sigma\kappa_{2}$ and $\Sigma\kappa_{3}$ are null homotopic. 
\end{lemma} 

\begin{proof} 
By~(\ref{kappas}), $\kappa_{2}\simeq f_{2}\circ(1\ltimes\ell)$. Consider the diagram 
\[\diagram 
      \Omega\Sigma X\ltimes\Sigma A\rto^-{\pi}\dto^{1\ltimes\ell} & \Sigma A\dto^{\ell} & \\ 
      \Omega\Sigma X\ltimes(X^{\wedge k}\ltimes\Sigma D)\rto^-{\pi} 
      & X^{\wedge k}\ltimes\Sigma D\rto^-{\pi} & \Sigma D. 
  \enddiagram\] 
The square commutes by the naturality of $\pi$. As the bottom row is the definition of $f_{2}$, 
the lower direction around the diagram is $f_{2}\circ(1\ltimes\ell)$, that is, $\kappa_{2}$. 
This equals the upper direction around the diagram, which is null homotopic after 
suspending since $\Sigma\pi\circ\Sigma\ell$ is null homotopic by hypothesis. Thus 
$\Sigma\kappa_{2}$ is null homotopic. 

By~(\ref{kappas}), $\kappa_{3}\simeq f_{3}\circ(1\ltimes\ell)$. By definition, 
$f_{3}=q\circ(1\ltimes\pi)$. Thus $\kappa_{3}$ factors through $1\ltimes(\pi\circ\ell)$. 
Suspending, $\Sigma\kappa_{3}$ factors through 
$\Sigma(1\ltimes(\pi\circ\ell))\simeq 1\ltimes\Sigma(\pi\circ\ell)$.  
By hypothesis $\Sigma(\pi\circ\ell)$ is null homotopic, and therefore $\Sigma\kappa_{3}$ 
is null homotopic.  
\end{proof} 

\begin{corollary} 
   \label{Sigmakappa} 
   Suppose that the composite 
   \(\nameddright{\Sigma^{2} A}{\Sigma\ell}{\Sigma (X^{\wedge k}\ltimes\Sigma D)}{\Sigma\pi}{\Sigma^{2} D}\) 
   is null homotopic. Then there is a homotopy commutative square 
   \[\diagram 
         \Sigma(\Omega\Sigma X\ltimes\Sigma A)\rrdouble\dto^{\Sigma\kappa_{1}} 
             & & \Sigma(\Omega\Sigma X\ltimes\Sigma A)\dto^{\Sigma(1\ltimes\ell)} \\ 
         \Sigma(\Omega\Sigma X\ltimes(X^{\wedge k}\wedge\Sigma D))\rrto^-{\Sigma e_{1}} 
             & & \Sigma(\Omega\Sigma X\ltimes(X^{\wedge k}\ltimes\Sigma D)). 
    \enddiagram\] 
\end{corollary} 

\begin{proof} 
Following~(\ref{kappadgrm}) we saw that $\kappa=\kappa_{1}+\kappa_{2}+\kappa_{3}+W$ 
where $W$ factors through a wedge sum of Whitehead products. In particular, as Whitehead 
products suspend trivially, $\Sigma W$ is null homotopic. By Lemma~\ref{Sigmakappa3}, 
$\Sigma\kappa_{2}$ and $\Sigma\kappa_{3}$ are also null homotopic. Therefore 
$\Sigma\kappa\simeq\Sigma\kappa_{1}$ and the homotopy commutativity 
of the asserted diagram follows. 
\end{proof} 

Some maps need to be defined that modify the maps $\mathfrak{\overline{c}}_{k}$ and $J_{k}$ 
from Section~\ref{sec:2cone}. For $k=1$ define the map $\overline{\mathfrak{c}}'_{1}$ by  
\[\overline{\mathfrak{c}}'_{1}\colon\namedright{X\ltimes\Sigma D}{E\ltimes\delta}{\Omega\Sigma X\ltimes E}.\] 
Note that $\overline{\mathfrak{c}}'_{1}\circ i=\overline{\mathfrak{c}}_{1}$. For $k\geq 2$ a 
recursive definition is used: define the maps $\overline{\mathfrak{c}}'_{k}$ by the composite 
\[\overline{\mathfrak{c}}'_{k}\colon\llnameddright{X^{\wedge k}\ltimes\Sigma D} 
      {q+\pi}{(X^{\wedge k}\wedge\Sigma D)\vee\Sigma D}{\overline{\mathfrak{c}}_{k}\perp j\circ\delta} 
      {\Omega\Sigma X\ltimes E}.\] 
Also, define the map $J'_{k}$ by the composite 
\[J'_{k}\colon\llnamedright{J_{k-1}(X)\ltimes\Sigma D}{j_{k-1}\ltimes j} 
          {\Omega\Sigma X\ltimes(X^{\wedge k}\ltimes\Sigma D)}.\] 
Recall that $\overline{J}_{k}$ was defined in Section~\ref{sec:2cone} as the composite 
\[\overline{J}_{k}\colon\llnameddright{J_{k-1}(X)\ltimes\Sigma D} 
      {j_{k-1}\ltimes\delta}{\Omega\Sigma X\ltimes E}{j}{\Omega\Sigma X\ltimes(\Omega\Sigma X\ltimes E)}.\]  

\begin{lemma} 
   \label{c'kprops} 
   The following hold: 
   \begin{letterlist} 
      \item the composite 
               \(\nameddright{\Sigma X^{\wedge k}\wedge\Sigma D}{\Sigma i}{\Sigma X^{\wedge k}\ltimes\Sigma D} 
                    {\Sigma\overline{\mathfrak{c}}'_{k}}{\Sigma\Omega\Sigma X\ltimes E}\) 
               is homotopic to~$\Sigma\overline{\mathfrak{c}}_{k}$; 
      \item the composite 
               \[\nameddright{\Sigma D}{j}{X^{\wedge k}\ltimes\Sigma D}{\overline{\mathfrak{c}}'_{k}} 
                      {\Omega\Sigma X\ltimes E}\] 
               is homotopic to $j\circ\delta$; 
      \item the composite 
               \(\nameddright{J_{k-1}(X)\ltimes\Sigma D}{J'_{k}} 
                       {\Omega\Sigma X\ltimes(X^{\wedge k}\ltimes\Sigma D)} 
                       {1\ltimes\overline{\mathfrak{c}}'_{k}}{\Omega\Sigma X\ltimes(\Omega\Sigma X\ltimes E)}\) 
               is homotopic to $j_{k-1}\ltimes(j\circ\delta)$; 
      \item the composite 
               \[\llnamedddright{J_{k-1}(X)\ltimes\Sigma D}{J'_{k}} 
                       {\Omega\Sigma X\ltimes(X^{\wedge k}\ltimes\Sigma D)} 
                       {1\ltimes\overline{\mathfrak{c}}'_{k}}{\Omega\Sigma X\ltimes(\Omega\Sigma X\ltimes E)} 
                       {\overline{a}\circ(1\ltimes\overline{a})}{E}\]
               is homotopic to $\overline{a}\circ(1\ltimes\overline{a})\circ\overline{J}_{k}$. 
   \end{letterlist} 
\end{lemma} 

\begin{proof} 
First, if $k=1$ then it has already been observe that observe that 
$\overline{\mathfrak{c}}'_{1}\circ i=\overline{\mathfrak{c}}_{1}$. For $k\geq 2$ observe that 
$\overline{\mathfrak{c}}'_{k}=(\overline{\mathfrak{c}}_{k}\perp j\circ\delta)\circ(q+\pi)= 
      (\overline{\mathfrak{c}}_{k}\circ q)+(j\circ\delta\circ\pi)$. 
Observe also that 
$q\circ i$ is homotopic to the identity map on $X^{\wedge k}\wedge\Sigma D$ while  
$\pi\circ i$ is null homotopic. As $\Sigma i$ is a suspension it distributes on the right so 
we therefore obtain 
\[\Sigma\overline{\mathfrak{c}}'_{k}\circ\Sigma i\simeq
      \Sigma\overline{\mathfrak{c}}_{k}\circ\Sigma q\circ\Sigma i+ 
      \Sigma j\circ\Sigma\delta\circ\Sigma\pi\circ\Sigma i\simeq  
      \simeq\Sigma\overline{\mathfrak{c}}_{k}+\ast\simeq\Sigma\overline{\mathfrak{c}}_{k},\] 
proving part~(a). 

Second, if $k=1$ the definition of $\overline{\mathfrak{c}}'_{1}$ as $E\ltimes\delta$ and 
the naturality of $j$ implies that $\overline{\mathfrak{c}}'_{1}\circ j=(E\ltimes\delta)\circ j\simeq j\circ\delta$.  
For $k\geq 2$ observe that $q\circ j$ is null homotopic while $\pi\circ j$ is homotopic to the identity map 
on $\Sigma D$. As 
\(\namedright{\Sigma D}{j}{X^{\wedge k}\ltimes\Sigma D}\) 
is a suspension it distributes on the right so we obtain 
\[\overline{\mathfrak{c}}'_{k}\circ j=((\overline{\mathfrak{c}}_{k}\circ q)+(j\circ\delta\circ \pi))\circ j\simeq 
       (\overline{\mathfrak{c}}_{k}\circ q\circ j)+(j\circ\delta\circ\pi\circ j)\simeq\ast + j\circ\delta,\] 
proving part~(b). 

Third, by definition, $J'_{k}=j_{k-1}\ltimes j$ so by part~(b) we have 
\[(1\ltimes\overline{\mathfrak{c}}'_{k})\circ J'_{k}= 
     (1\ltimes\overline{\mathfrak{c}}'_{k})\circ(j_{k-1}\ltimes j)= 
     (j_{k-1}\ltimes(\overline{\mathfrak{c}}'_{k}\circ j))=j_{k-1}\ltimes(j\circ\delta),\] 
proving part~(c). 

Fourth, since $\overline{a}$ is the quotient of the homotopy action 
\(\namedright{\Omega\Sigma X\times E}{a}{E}\), 
which restricts to the identity map on $E$, we obtain that the composite 
\(\nameddright{E}{j}{\Omega\Sigma X\ltimes E}{\overline{a}}{E}\) 
is homotopic to the identity map on $E$. Using this and part~(c) we obtain  
\[\overline{a}\circ(1\ltimes\overline{a})\circ(1\ltimes\overline{\mathfrak{c}}'_{k})\circ J'_{k}\simeq 
     \overline{a}\circ(1\ltimes\overline{a})\circ(j_{k-1}\ltimes(j\circ\delta))\simeq 
     \overline{a}\circ(j_{k-1}\ltimes\delta).\] 
On the other hand, by definition, $\overline{J}_{k}=j\circ(j_{k-1}\ltimes\delta)$, so using 
the naturality of $j$ we have 
\[\overline{a}\circ(1\ltimes\overline{a})\circ\overline{J}_{k}= 
       \overline{a}\circ(1\ltimes\overline{a})\circ j\circ (j_{k-1}\ltimes\delta)\simeq 
       \overline{a}\circ j\circ\overline{a}\circ(j_{k-1}\ltimes\delta)\simeq 
       \overline{a}\circ(j_{k-1}\ltimes\delta).\] 
Thus 
$\overline{a}\circ(1\ltimes\overline{a})\circ(1\ltimes\overline{\mathfrak{c}}'_{k})\circ J'_{k}\simeq 
      \overline{a}\circ(1\ltimes\overline{a})\circ\overline{J}_{k}$, 
proving part~(d). 
\end{proof} 

Recall from Section~\ref{sec:ad} that two maps $f$ and $g$ are congruent if $\Sigma f\simeq\Sigma g$. 

\begin{lemma} 
   \label{e1J'k} 
   The composite 
   \[\llnamedright{(\Omega\Sigma X\ltimes(X^{\wedge k}\wedge\Sigma D))\vee 
          J_{k-1}(X)\ltimes\Sigma D}{e_{1}\perp J'_{k}} {\Omega\Sigma X\ltimes(X^{\wedge k}\ltimes\Sigma D)} 
          \stackrel{1\ltimes\overline{c}'_{k}}{\llarrow}\hspace{2cm}\] 
   \[\hspace{8cm}\llnamedright{\Omega\Sigma X\ltimes(\Omega\Sigma X\ltimes E)} 
        {\overline{a}\circ(1\ltimes\overline{a})}{E}\]  
   is congruent to 
   $\overline{a}\circ(1\ltimes\overline{a})\circ((1\ltimes\overline{\mathfrak{c}}_{k})\perp\overline{J}_{k})$ 
   (the map appearing in Proposition~\ref{JkEequivfrak}).  
\end{lemma} 

\begin{proof} 
It is equivalent to prove the statement when restricted to each of the wedge summands 
$\Omega\Sigma X\ltimes(X^{\wedge k}\wedge\Sigma D)$ and $J_{k-1}(X)\ltimes\Sigma D$. 
By definition, $e_{1}$ is the map 
\(\namedright{\Omega\Sigma X\ltimes(X^{\wedge k}\wedge\Sigma D)}{1\ltimes i} 
    {\Omega\Sigma X\ltimes(X^{\wedge k}\ltimes\Sigma D)}\)  
and by Lemma~\ref{c'kprops}~(a), $\overline{\mathfrak{c}}'_{k}\circ i$ is congruent to $\overline{\mathfrak{c}}_{k}$. 
Therefore $(1\ltimes\overline{\mathfrak{c}}'_{k})\circ e_{1}=(1\ltimes\overline{\mathfrak{c}}'_{k})(1\ltimes i)$ 
is congruent to $1\ltimes\overline{\mathfrak{c}}_{k}$. Thus 
$\overline{a}\circ(1\ltimes\overline{a})\circ(1\ltimes\overline{\mathfrak{c}}'_{k})\circ e_{1}$ is congruent 
to $\overline{a}\circ(1\ltimes\overline{a})\circ(1\ltimes\overline{\mathfrak{c}}_{k})$. 
     
On the other hand, by Lemma~\ref{c'kprops}~(d), 
$\overline{a}\circ(1\ltimes\overline{a})\circ(1\ltimes\overline{\mathfrak{c}}'_{k})\circ J'_{k}\simeq 
      \overline{a}\circ(1\ltimes\overline{a})\circ\overline{J}_{k}$. 
As homotopic maps are congruent, the lemma follows. 
\end{proof} 

Putting things together to this point gives the following. 
  
\begin{proposition} 
   \label{Sigmakappatheta} 
   Suppose that there is a homotopy fibration sequence 
   \(\namedddright{\Omega\Sigma X}{\partial}{E}{p}{Y}{h}{\Sigma X}\) 
   and a map 
   \(\namedright{\Sigma A}{f}{Y}\) 
   that lifts through $p$ to 
   \(\namedright{\Sigma A}{g}{E}\),  
   together satisfying the following properties: 
   \begin{itemize} 
      \item[(a)] $\Omega h$ has a right homotopy inverse; 
      \item[(b)] there is a map 
                     \(\delta\colon\namedright{\Sigma D}{}{E}\) 
                     such that the composite 
                      \[\nameddright{\Omega\Sigma X\ltimes\Sigma D}{1\ltimes\delta} 
                              {\Omega\Sigma X\ltimes E}{\overline{a}}{E}\] 
                     is a homotopy equivalence; 
      \item[(c)] $g$ factors as a composite 
                     \[g\colon\namedddright{\Sigma A}{\ell}{X^{\wedge k}\ltimes\Sigma D}{\overline{\mathfrak{c}}'_{k}} 
                          {\Omega\Sigma X\ltimes E}{\overline{a}}{E}\] 
                      for some map $\ell$; 
       \item[(d)] the composite 
                      \(\nameddright{\Sigma^{2} A}{\Sigma\ell}{\Sigma (X^{\wedge k}\ltimes\Sigma D)} 
                           {\Sigma\pi}{\Sigma D}\) 
                      is null homotopic. 
   \end{itemize} 
   Then there is a homotopy commutative square 
   \[\diagram 
         \Sigma(\Omega\Sigma X\ltimes\Sigma A)\rdouble\dto^{\Sigma\kappa_{1}+\ast} 
              & \Sigma(\Omega\Sigma X\ltimes\Sigma A)\dto^{\Sigma\theta} \\ 
         \Sigma(\Omega\Sigma X\ltimes(X^{\wedge k}\wedge\Sigma D))\vee 
                 \Sigma J_{k-1}(X)\ltimes\Sigma D\rto^-{\Sigma\epsilon} 
              & \Sigma E 
       \enddiagram\] 
    where $\epsilon$ is a homotopy equivalence. 
\end{proposition} 

\begin{proof} 
By hypothesis~(c), $g$ factors as 
\(\namedddright{\Sigma A}{\ell}{X^{\wedge k}\ltimes\Sigma D}{\overline{\mathfrak{c}}'_{k}} 
               {\Omega\Sigma X\ltimes E}{\overline{a}}{E}\). 
Taking the half-smash with the identity map on $\Omega\Sigma X$ then gives 
the commutativity of the left rectangle in the diagram 
\[\diagram 
       \Omega\Sigma X\ltimes\Sigma A\xdouble[rrr]\dto^{(1\ltimes \ell)} 
            & & & \Omega\Sigma X\ltimes\Sigma A\dto^{1\ltimes g}\drto^{\theta} & \\
       \Omega\Sigma X\ltimes(X^{\wedge k}\ltimes\Sigma D))\rrto^-{1\ltimes\overline{\mathfrak{c}}'_{k}} 
            & & \Omega\Sigma X\ltimes(\Omega\Sigma X\ltimes E)\rto^-{1\ltimes\overline{a}} 
            & \Omega\Sigma X\ltimes E\rto^-{\overline{a}} & E.  
   \enddiagram\] 
The right triangle homotopy commutes by~(\ref{thetabara}). Consider the diagram 
\[\spreaddiagramcolumns{-1pc}\diagram 
      \Sigma(\Omega\Sigma X\ltimes\Sigma A)\rrdouble\dto^{\Sigma\kappa_{1}+\ast} 
            & &\Sigma(\Omega\Sigma X\ltimes\Sigma A)\rdouble\dto^{\Sigma(1\ltimes\ell)} 
            & \Sigma(\Omega\Sigma X\ltimes\Sigma A)\dto^{\Sigma\theta} \\ 
      \Sigma(\Omega\Sigma X\ltimes(X^{\wedge k}\wedge\Sigma D))\vee 
             \Sigma J_{k-1}(X)\ltimes\Sigma D\rrto^-{\Sigma e_{1}\perp\Sigma J'_{k}} 
            & & \Sigma(\Omega\Sigma X\ltimes(X^{\wedge k}\ltimes\Sigma D))\rto^-{\Sigma\alpha} & \Sigma E 
   \enddiagram\] 
where $\alpha=\overline{a}\circ(1\ltimes\overline{a})\circ(1\ltimes\overline{\mathfrak{c}}'_{k})$. 
The left square homotopy commutes by Lemma~\ref{Sigmakappa} and the right square is the 
suspension of the previous diagram. Consider the composite along the bottom row and the string 
of identifications:  
\[\begin{split} 
     \Sigma\alpha\circ(\Sigma e_{1}\perp\Sigma J'_{k}) 
          & =\Sigma\overline{a}\circ\Sigma(1\ltimes\overline{a})\circ\Sigma(1\ltimes\overline{\mathfrak{c}}'_{k})\circ 
                   (\Sigma e_{1}\perp\Sigma J'_{k}) \\ 
          & \simeq\Sigma\overline{a}\circ\Sigma(1\ltimes\overline{a})\circ 
                   (\Sigma(1\ltimes\overline{\mathfrak{c}}_{k})\perp\Sigma\overline{J}_{k}) \\ 
          & =\Sigma\overline{a}\circ(1\ltimes\overline{a})\circ(1\ltimes\overline{\mathfrak{c}}_{k}\perp\overline{J}_{k}).  
  \end{split}\]  
The first equality is from the definition of $\alpha$, the second is from Lemma~\ref{e1J'k} 
and the third is just pulling out a suspension coordinate. By Proposition~\ref{JkEequivfrak}, 
$\overline{a}\circ(1\ltimes\overline{a})\circ(1\ltimes\overline{\mathfrak{c}}_{k}\perp\overline{J}_{k})$ 
is a homotopy equivalence. Therefore $\Sigma\alpha\circ(\Sigma e_{1}\perp\Sigma J'_{k})$ is 
a homotopy equivalence. Taking $\epsilon=\alpha\circ(e_{1}\perp J'_{k})$ then gives the 
asserted homotopy commutative diagram and homotopy equivalence. 
\end{proof} 

The homotopy commutative diagram in Proposition~\ref{Sigmakappatheta} is the 
suspension of the diagram obtained in the proof of Theorem~\ref{E'typeI}. Let 
$\overline{C}$ be the homotopy cofibre of the composite 
\[\nameddright{\Sigma A}{\ell}{X^{\wedge k}\ltimes\Sigma D}{q}{X^{\wedge k}\wedge\Sigma D}.\] 
Then the homotopy cofibre of the map $\Sigma\kappa_{1}+\ast$ in 
Proposition~\ref{Sigmakappatheta} is 
$\Sigma(\Omega\Sigma X\ltimes\overline{C})\vee\Sigma J_{k-1}(X)\ltimes\Sigma D$. 
The homotopy cofibre of the map $\theta$ in Proposition~\ref{Sigmakappatheta} is $E'$. 
Therefore the homotopy commutativity of the diagram in the proposition implies that 
there is an induced map of cofibres 
\[\psi\colon\namedright{\Sigma(\Omega\Sigma X\ltimes\overline{C})\vee\Sigma J_{k-1}(X)\ltimes\Sigma D} 
      {}{\Sigma E'}\] 
and the fact that $\Sigma\epsilon$ is a homotopy equivalence implies that $\psi$ induces 
an isomorphism in homology by the five-lemma and so is a homotopy equivalence by 
Whitehead's Theorem. This gives a description of the homotopy type of $\Sigma E'$. 
However, we want to identify the homotopy type of $E'$. To do this an extra hypothesis 
is necessary.  

\begin{theorem} 
   \label{E'typeII} 
   Suppose that there is a homotopy fibration sequence 
   \(\namedddright{\Omega\Sigma X}{\partial}{E}{p}{Y}{h}{\Sigma X}\) 
   and a map 
   \(\namedright{\Sigma A}{f}{Y}\) 
   that lifts through $p$ to 
   \(\namedright{\Sigma A}{g}{E}\),  
   together satisfying the following properties: 
   \begin{itemize} 
      \item[(a)] $\Omega h$ has a right homotopy inverse; 
      \item[(b)] there is a map 
                     \(\delta\colon\namedright{\Sigma D}{}{E}\) 
                     such that the composite 
                      \[\nameddright{\Omega\Sigma X\ltimes\Sigma D}{1\ltimes\delta} 
                              {\Omega\Sigma X\ltimes E}{\overline{a}}{E}\] 
                     is a homotopy equivalence; 
      \item[(c)] $g$ factors as a composite 
                     \[g\colon\namedddright{\Sigma A}{\ell}{X^{\wedge k}\ltimes\Sigma D}{\overline{\mathfrak{c}}'_{k}} 
                          {\Omega\Sigma X\ltimes E}{\overline{a}}{E}\] 
                      for some map $\ell$; 
       \item[(d)] the composite 
                      \(\nameddright{\Sigma^{2} A}{\Sigma\ell}{\Sigma (X^{\wedge k}\ltimes\Sigma D)} 
                           {\Sigma\pi}{\Sigma D}\) 
                      is null homotopic;  
      \item[(e)] the composite 
                     \(\nameddright{\Sigma A}{\ell}{X^{\wedge k}\ltimes\Sigma D}{q}{X^{\wedge k}\wedge\Sigma D}\) 
                     has a left homotopy inverse.  
   \end{itemize} 
   Then if $\overline{C}$ is the homotopy cofibre of $q\circ\ell$ there is a homotopy equivalence 
   \[E'\simeq(\Omega\Sigma X\ltimes\overline{C})\vee(J_{k-1}(X)\ltimes\Sigma D).\] 
\end{theorem} 

\begin{proof} 
By Proposition~\ref{Sigmakappatheta}, hypotheses~(a) through~(d) imply that there is a 
homotopy cofibration diagram 
\begin{equation} 
  \label{epsilonpsi} 
  \diagram 
         \Sigma(\Omega\Sigma X\ltimes\Sigma A)\rdouble\dto^{\Sigma\kappa_{1}+\ast} 
              & \Sigma(\Omega\Sigma X\ltimes\Sigma A)\dto^{\Sigma\theta} \\ 
         \Sigma(\Omega\Sigma X\ltimes(X^{\wedge k}\wedge\Sigma D))\vee 
                 (\Sigma J_{k-1}(X)\ltimes\Sigma D\rto^-{\Sigma\epsilon})\dto^{\Sigma\lambda\vee 1}  
              & \Sigma E\dto^{\Sigma\eta} \\  
      \Sigma(\Omega\Sigma X\ltimes\overline{C})\vee(\Sigma J_{k-1}(X)\ltimes\Sigma D)\rto^-{\psi} 
              & \Sigma E' 
  \enddiagram 
\end{equation} 
where $\epsilon$ is a homotopy equivalence, $\lambda$ is the map to the homotopy 
cofibre of $\kappa_{1}$, $\eta$ is the map to the homotopy cofibre of $\theta$, 
and $\psi$ is an induced map of cofibres. As $\epsilon$ is a homotopy equivalence the 
five-lemma implies that $\psi$ induces an isomorphism in homology and so is a homotopy 
equivalence by Whitehead's Theorem. 

We wish to show that $\psi\simeq\Sigma\psi'$. Write $\psi=\psi_{1}\perp\psi_{2}$ where 
$\psi_{1}$ and $\psi_{2}$ are the restrictions of $\psi$ to $\Sigma(\Omega\Sigma X\ltimes\overline{C})$ 
and $\Sigma J_{k-1}(X)\ltimes\Sigma D$ respectively. Similarly write 
$\epsilon=\epsilon_{1}\perp\epsilon_{2}$ where $\epsilon_{1}$ and $\epsilon_{2}$ are 
the restrictions of $\epsilon$ to $\Omega\Sigma X\ltimes(X^{\wedge k}\wedge\Sigma D)$ 
and $J_{k-1}(X)\ltimes\Sigma D$ respectively. Observe that the bottom square 
in~(\ref{epsilonpsi}) implies that $\psi_{2}=\Sigma\eta\circ\Sigma\epsilon_{2}$. In particular, 
$\psi_{2}$ is a suspension. Next, consider the homotopy cofibration 
\(\nameddright{\Sigma A}{q\circ\ell}{X^{\wedge k}\wedge\Sigma D}{\mu}{\overline{C}}\). 
By hypothesis~(e), $q\circ\ell$ has a left homotopy inverse. As $X^{\wedge k}\wedge\Sigma D$ 
is a suspension this implies that the homotopy cofibration splits to give a homotopy equivalence 
\[X^{\wedge k}\wedge\Sigma D\simeq\Sigma A\vee\overline{C}.\] 
In particular, $\mu$ has a right homotopy inverse 
\(\nu\colon\namedright{\overline{C}}{}{X^{\wedge k}\wedge\Sigma D}\).  
Observe that by~(\ref{kappas}) and the definition of $f_{1}$ we have 
$\kappa_{1}=f_{1}\circ(1\ltimes\ell)=(1\ltimes q)\circ(1\ltimes\ell)$. Therefore $\lambda$ can be 
chosen to be $1\ltimes\mu$ and so has $1\ltimes\nu$ as a right homotopy inverse. Hence 
the bottom square in~(\ref{epsilonpsi}) implies that $\psi_{1}$ is homotopic to 
$\Sigma\eta\circ\Sigma\epsilon_{1}\circ\Sigma\nu$. In particular, $\psi_{1}$ is a suspension. 
Hence $\psi$ is a suspension, $\psi=\Sigma\psi'$. Since $\psi$ is a homotopy equivalence, 
it induces an isomorphism in homology. Therefore so does $\psi'$. Since $\Sigma D$ and 
$\overline{C}$ are simply-connected, so is 
$(\Omega\Sigma X\ltimes\overline{C})\vee(J_{k-1}(X)\ltimes\Sigma D)$. Thus $\psi'$ induces 
an isomorphism in homology between simply-connected spaces and so is 
a homotopy equivalence by Whitehead's Theorem. That is, there is a homotopy equivalence 
$E'\simeq(\Omega\Sigma X\ltimes\overline{C})\vee(J_{k-1}(X)\ltimes\Sigma D)$. 
\end{proof}

\newpage 
\section{Applying Theorem~\ref{E'typeII}} 
\label{sec:decompapps} 

In this section examples are given of Theorem~\ref{E'typeII} in action. This begins 
with a general example in Proposition~\ref{wedgeex} which will then lead to several 
more specific families of examples. We first need a general lemma. 

For a space $X$, let 
\[E\colon\namedright{X}{}{\Omega\Sigma X}\] 
be the suspension, which is adjoint to the identity map on $\Sigma X$. Given pointed, 
path-connected spaces $X_{1},\ldots,X_{m}$, for $1\leq s\leq m$ let 
\[i_{s}\colon\namedright{\Sigma X_{s}}{}{\bigvee_{i=1}^{m}\Sigma X_{i}}\] 
be the inclusion of the $s^{th}$-wedge summand. Let 
\[I\colon\namedright{\bigvee_{i=2}^{m}\Sigma X_{i}}{}{\bigvee_{i=1}^{m}\Sigma X_{i}}\] 
be the inclusion, and note that $I=i_{2}\perp\cdots\perp i_{m}$. 

\begin{lemma} 
   \label{Whlift} 
   Let $X_{1},\ldots,X_{m}$ be pointed, path-connected spaces. Let 
   \(q_{1}\colon\namedright{\bigvee_{i=1}^{m}\Sigma X_{i}}{}{\Sigma X_{1}}\) 
   be the pinch map onto the first wedge summand and let $E$ be its homotopy fibre. 
   Then the following hold: 
   \begin{letterlist} 
      \item there is a map 
               \(g\colon\namedright{\bigvee_{i=2}^{m}\Sigma X_{i}}{}{E}\) 
               which lifts $I$ through $p$; 
      \item the composite 
               \[\nameddright{\Omega\Sigma X_{1}\ltimes(\bigvee_{i=2}^{m}\Sigma X_{i})}{1\ltimes g} 
                    {\Omega\Sigma X_{1}\ltimes E}{\overline{a}}{E}\] 
               is a homotopy equivalence; 
     \item  the composite 
               \[\namedddright{X_{1}\wedge(\bigvee_{i=2}^{m}\Sigma X_{i})}{i} 
                  {X_{1}\ltimes(\bigvee_{i=2}^{m}\Sigma X_{i})}{E\ltimes 1} 
                  {\Omega\Sigma X_{1}\ltimes(\bigvee_{i=2}^{m}\Sigma X_{i})}{1\ltimes g} 
                  {\Omega\Sigma X_{1}\ltimes E}\stackrel{\overline{a}}{\longrightarrow} E\] 
               is a lift of the Whitehead product $[i_{1},i_{2}]\perp\cdots\perp [i_{1},i_{m}]$ through $p$. 
   \end{letterlist} 
\end{lemma} 

\begin{proof} 
Let $X=X_{1}$ and $Y=\bigvee_{i=2}^{m} X_{i}$, so that 
$\bigvee_{i=1}^{m}\Sigma X_{i}=\Sigma X\vee\Sigma Y$. To avoid overlapping notation, let  
\(i_{L}\colon\namedright{\Sigma X}{}{\Sigma X\vee\Sigma Y}\) 
and 
\(i_{R}\colon\namedright{\Sigma Y}{}{\Sigma X\vee\Sigma Y}\) 
be the inclusions of the left and right wedge summands respectively. 
Since $q_{1}\circ i_{R}$ is null homotopic, there is a map 
\(g\colon\namedright{\Sigma Y}{}{E}\) 
that lifts the inclusion through $p$. By Example~\ref{specialGTadexample} the composite 
\[\nameddright{\Omega\Sigma X\ltimes\Sigma Y}{1\ltimes g}{\Omega\Sigma X\ltimes E}{\overline{a}}{E}\] 
is a homotopy equivalence. This proves parts~(a) and~(b). By Lemma~\ref{dbarkaction} 
there is a homotopy commutative diagram 
\[\diagram 
      X\wedge\Sigma Y\rto^{\mathfrak{d}_{1}}\drto_{[i_{L},i_{R}]} & E\dto^{p} \\ 
      & \Sigma X\vee\Sigma Y. 
  \enddiagram\] 
By Lemma~\ref{dcongruent}, $\mathfrak{d}_{1}=d_{1}$, where $d_{1}$ is the composite 
\(\namedddright{X\wedge\Sigma Y}{c_{1}}{\Omega\Sigma X\ltimes\Sigma Y}{1\ltimes g} 
     {\Omega\Sigma X\ltimes E}{\overline{a}}{E}\) 
and $c_{1}$ is the composite 
\(\nameddright{X\wedge\Sigma Y}{i}{X\ltimes\Sigma Y}{E\ltimes 1}{\Omega\Sigma X\ltimes\Sigma Y}\). 
Remembeing that $X=X_{1}$ and $Y=\bigvee_{i=2}^{m} X_{i}$, we have 
$i_{R}=i_{1}$ and $i_{L}=i_{2}\perp\cdots\perp i_{m}$. The linearity of the Whitehead product 
therefore implies that 
\[[i_{L},i_{R}]\simeq [i_{1},i_{2}]\perp\cdots\perp [i_{1},i_{m}].\] 
Thus $\overline{a}\circ(1\ltimes g)\circ(E\ltimes 1)\circ i$ is a lift of 
$[i_{1},i_{2}]\perp\cdots\perp [i_{1},i_{m}]$ through $p$, proving part~(c). 
\end{proof} 

Parts~(b) and~(c) of Lemma~\ref{Whlift} have the following corollaries. 

\begin{corollary} 
   \label{Whliftcor1} 
   Let   
   \(\namedright{B}{\alpha}{\bigvee_{i=2}^{m}\Sigma X_{i}}\) 
   be a map. A lift of the composite 
    \(\nameddright{B}{\alpha}{\bigvee_{i=2}^{m}\Sigma X_{i}}{I}{\bigvee_{i=1}^{m}\Sigma X_{i}}\) 
    through $p$ is given by 
    \[\namedddright{B}{\alpha}{\bigvee_{i=2}^{m}\Sigma X_{i}}{j}{X_{1}\ltimes(\bigvee_{i=2}^{m}\Sigma X_{i})} 
         {E\ltimes 1}{\Omega\Sigma X_{1}\ltimes(\bigvee_{i=2}^{m}\Sigma X_{i})} 
         \stackrel{1\ltimes g}{\longrightarrow}\namedright{\Omega\Sigma X\ltimes E}{\overline{a}}{E}.\]  
\end{corollary} 
\vspace{-1cm}~$\qqed$\bigskip 

\begin{corollary} 
   \label{Whliftcor2} 
   The restriction of the composite in Lemma~\ref{Whlift}~(c) to $X_{1}\wedge\Sigma X_{t}$ 
   for some $2\leq t\leq m$ is a lift of the Whitehead product $[i_{1},i_{t}]$ through $p$.~$\qqed$ 
\end{corollary}

\begin{proposition} 
   \label{wedgeex} 
   Let $X_{1},\ldots,X_{m}$ be pointed, path-connected spaces and suppose that there 
   is a homotopy cofibration 
   \(\nameddright{\Sigma A}{f}{\bigvee_{i=1}^{m}\Sigma X_{i}}{}{M}\). 
   Suppose that $f=f_{1}+f_{2}$ where:  
   \begin{itemize} 
      \item $f_{1}=\sum_{j=2}^{m} [i_{1},i_{j}]\circ h_{1,j}$ for some maps 
               \(h_{1,j}\colon\namedright{\Sigma A}{}{\Sigma X_{1}\wedge X_{j}}\); 
      \item there is at least one $t\in\{2,\ldots,m\}$ such that 
               \(\namedright{\Sigma A}{h_{1,t}}{\Sigma X_{1}\wedge X_{t}}\) 
               has a left homotopy inverse; 
      \item $f_{2}$ factors as 
               \(\nameddright{\Sigma A}{\gamma}{\bigvee_{i=2}^{m}\Sigma X_{i}}{I} 
                     {\bigvee_{i=1}^{m}\Sigma X_{i}}\) 
                for some map $\gamma$; 
      \item $\Sigma\gamma$ is null homotopic. 
   \end{itemize} 
   Let $h=\sum_{j=2}^{m} h_{1,j}$ and let $\overline{C}$ be the homotopy cofibre of 
   \(\namedright{\Sigma A}{h}{X_{1}\wedge(\bigvee_{i=2}^{m}\Sigma X_{i})}\). 
   Then the following hold: 
   \begin{letterlist} 
      \item there is a map 
               \(q'\colon\namedright{M}{}{\Sigma X_{1}}\) 
               extending $q_{1}$;  
      \item there is a homotopy fibration 
              \[\nameddright{(\Omega\Sigma X_{1}\ltimes\overline{C})\vee(\bigvee_{i=2}^{m}\Sigma X_{i})} 
                    {}{M}{q'}{\Sigma X_{1}};\]  
      \item the homotopy fibration in part~(b) splits after looping to give a homotopy equivalence 
              \[\Omega M\simeq\Omega\Sigma X_{1}\times 
                   \Omega\bigg((\Omega\Sigma X_{1}\ltimes\overline{C})\vee(\bigvee_{i=2}^{m}\Sigma X_{i})\bigg).\] 
     \end{letterlist} 
\end{proposition} 

\begin{proof} 
First observe that as $f_{1}$ factors through the Whitehead products $[i_{1},i_{j}]$ and 
$f_{2}$ factors through~$I$, the composite 
\(\nameddright{\Sigma A}{f}{\bigvee_{i=1}^{m}\Sigma X_{i}}{q_{1}}{\Sigma X_{1}}\) 
is null homotopic, so $q_{1}$ extends to a map 
\(q'\colon\namedright{M}{}{\Sigma X_{1}}\). 
This proves part~(a). 

To prove parts~(b) and~(c), Theorem~\ref{E'typeII} will be applied to the homotopy fibration 
\(\nameddright{E}{p}{\bigvee_{i=1}^{m}\Sigma X_{i}}{q_{1}}{\Sigma X_{1}}\) 
and the attaching map $f$ for $M$. The hypotheses for that theorem need to be checked. 
In the notation of Proposition~\ref{E'typeII}, let $D=\bigvee_{i=2}^{m} X_{i}$ and let  
\(\namedright{\Sigma D}{\delta}{E}\) 
be a lift of $I$ through $p$. 
\medskip 

\noindent 
\textit{Step 1}: The map $i_{1}$ is a right homotopy inverse for $q_{1}$, so hypothesis~(a) 
in Theorem~\ref{E'typeII} is satisfied. With $D$ and $\delta$ as above, the homotopy equivalence 
in Lemma~\ref{Whlift}~(b) implies that hypothesis~(b) of Theorem~\ref{E'typeII} is satisfied. 
\medskip 

\noindent 
\textit{Step 2}: For hypothesis~(c) of Theorem~\ref{E'typeII} we need to choose a
 lift $g$ of $f$ through $p$. Let $\ell_{1,j}$ be the composite 
\[\ell_{1,j}\colon\namedright{\Sigma A}{h_{1,j}}{\Sigma X_{1}\wedge X_{j}} 
      \hookrightarrow\namedright{X_{1}\wedge(\bigvee_{i=2}^{m}\Sigma X_{i})}{i} 
      {X_{1}\ltimes(\bigvee_{i=2}^{m}\Sigma X_{i})}.\] 
Then by Corollary~\ref{Whliftcor2} the composite 
\[\namedddright{\Sigma A}{\ell_{1,j}}{X_{1}\ltimes(\bigvee_{i=2}^{m}\Sigma X_{i})} 
     {E\ltimes\delta}{\Omega\Sigma X_{1}\ltimes E}{\overline{a}}{E}\] 
is a lift of $[i_{1},i_{j}]\circ h_{1,j}$ through $p$. Let 
$\ell_{1}=\sum_{j=2}^{m}\ell_{1,j}$. Then the composite 
\[g_{1}\colon\namedddright{\Sigma A}{\ell_{1}}{X_{1}\ltimes(\bigvee_{i=2}^{m}\Sigma X_{i})} 
     {E\ltimes\delta}{\Omega\Sigma X_{1}\ltimes E}{\overline{a}}{E}\] 
is a lift of $f_{1}$ through $p$. Let $\ell_{2}$ be the composite 
\[\ell_{2}\colon\nameddright{\Sigma A}{\gamma}{\bigvee_{i=2}^{m}\Sigma X_{i}}{j} 
       {X_{1}\ltimes(\bigvee_{i=2}^{m}\Sigma X_{i})}.\] 
Then by Corollary~\ref{Whliftcor1} the composite 
\[g_{2}\colon\namedddright{\Sigma A}{\ell_{2}}{X_{1}\ltimes(\bigvee_{i=2}^{m}\Sigma X_{i})} 
        {E\ltimes\delta}{\Omega\Sigma X_{1}\ltimes E}{\overline{a}}{E}\] 
is a lift of $f_{2}$ through $p$. Thus if 
\[\ell\colon\namedright{\Sigma A}{}{X_{1}\ltimes(\bigvee_{i=2}^{m}\Sigma X_{i})}\] 
is $\ell_{1}+\ell_{2}$ and 
\[g\colon\namedddright{\Sigma A}{\ell}{X_{1}\ltimes(\bigvee_{i=2}^{m}\Sigma X_{i})} 
        {E\ltimes\delta}{\Omega\Sigma X_{1}\ltimes E}{\overline{a}}{E}\] 
is $g_{1}+g_{2}$ then $g$ is a lift of $f$ through $p$. By definition, 
the map $\overline{\mathfrak{c}}'_{1}$ in Section~\ref{sec:improve} equals $E\ltimes 1$, 
so the map $g$ satisfies hypothesis~(c) of Theorem~\ref{E'typeII}. 
\medskip 

\noindent 
\textit{Step 3}: Consider the composite 
\[\nameddright{\Sigma A}{\ell}{X_{1}\ltimes(\bigvee_{i=2}^{m}\Sigma X_{i})}{\pi} 
      {\bigvee_{i=2}^{m}\Sigma X_{i}}.\] 
By definition, $\ell_{1,j}$ factors through $i$ and $\pi\circ i$ is null homotopic. 
Therefore $\pi\circ\ell_{1,j}$ is null homotopic for each $2\leq j\leq m$. As 
$\ell_{1}=\sum_{j=2}^{m}\ell_{1,j}$, we obtain a null homotopy for $\pi\circ\ell_{1}$. 
By definition, $\ell_{2}=j\circ\gamma$ and $\pi\circ j$ is the identity map, so 
$\pi\circ\ell_{2}=\gamma$. By hypothesis, $\Sigma\gamma$ is null homotopic, 
and therefore so is $\Sigma(\pi\circ\ell_{2})$. As $\ell=\ell_{1}+\ell_{2}$, we obtain 
a null homotopy for $\Sigma(\pi\circ\ell)$. This fulfils hypothesis~(d) of Theorem~\ref{E'typeII}. 
\medskip 

\noindent 
\textit{Step 4}: Consider the composition 
\[\nameddright{\Sigma A}{\ell}{X_{1}\ltimes(\bigvee_{i=2}^{m}\Sigma X_{i})}{q} 
      {X_{1}\wedge(\bigvee_{i=2}^{m}\Sigma X_{i})}.\] 
By definition, $\ell_{2}=j\circ\gamma$ and $q\circ j$ is null homotopic, so 
$q\circ\ell_{2}$ is null homotopic. Therefore, as $\ell=\ell_{1}+\ell_{2}$, we have 
$q\circ\ell\simeq q\circ\ell_{1}$. On the other hand, by definition, $\ell_{1,t}$ factors 
through $i$ and $q\circ i$ is homotopic to the identity map. Therefore 
$q\circ\ell_{1,j}$ is homotopic to the composite 
\begin{equation} 
  \label{h1j} 
  \namedright{\Sigma A}{h_{1,j}}{\Sigma X_{1}\wedge X_{j}} 
      \hookrightarrow X_{1}\wedge(\bigvee_{i=2}^{m}\Sigma X_{i}). 
\end{equation}  
The sum of the inclusions 
\(\namedright{\Sigma X_{1}\wedge X_{j}}{}{X_{1}\wedge(\bigvee_{i=2}^{m}\Sigma X_{i})}\) 
for $2\leq j\leq m$ is homotopic to the identity map, so as $h=\sum_{j=2}^{m} h_{1,j}$ 
and $\ell=\sum_{j=2}^{m}\ell_{1,j}$ we have $q\circ\ell_{1}$ homotopic to $h$. 
Hence $q\circ\ell\simeq h$. 
\medskip 

\noindent 
\textit{Step 5}: By hypothesis, there is a $t\in\{2,\ldots,m\}$ such that 
\(\namedright{\Sigma A}{h_{1,t}}{\Sigma X_{1}\wedge X_{t}}\) 
has a left homotopy inverse 
\(r\colon\namedright{\Sigma X_{1}\wedge X_{t}}{}{\Sigma A}\). 
Consider the composite 
\begin{equation} 
  \label{qh} 
  \namedddright{\Sigma A}{h}{X_{1}\wedge(\bigvee_{i=2}^{m}\Sigma X_{i})}{1\wedge q_{t}} 
       {X_{1}\wedge\Sigma X_{t}}{r}{\Sigma A} 
\end{equation}  
where $q_{t}$ is the pinch map to the $t^{th}$-wedge summand. 
Observe that~(\ref{h1j}) composed with $1\wedge q_{t}$ is null homotopic if $j\neq t$ and 
is homotopic to $h_{1,t}$ if $j=t$. Thus in~(\ref{qh}) the composite 
$(1\wedge q_{t})\circ h$ is homotopic to $h_{1,t}$. Hence as 
$r$ is a left homotopy inverse for $h_{1,t}$, the composite~(\ref{qh}) 
is homotopic to the identity map. In particular, $h$ has a left homotopy inverse. 
That is, by Step~$4$, $q\circ\ell$ has a left homotopy inverse. This fulfils hypothesis~(e) 
of Theorem~\ref{E'typeII}. 
\medskip 

\noindent 
\textit{Step 6}: As hypotheses~(a) to~(e) of Theorem~\ref{E'typeII} hold, 
applying the proposition immediately implies assertions~(b) and~(c), noting that by Step 3, 
$h=q\circ\ell$. 
\end{proof} 

The homotopy decomposition of $\Omega M$ in Proposition~\ref{wedgeex} 
can be made more precise by identifying the homotopy type of $\overline{C}$. 
One hypothesis is that for some $t\in\{2,\ldots,m\}$ the map 
\(\namedright{\Sigma A}{h_{1,t}}{\Sigma X_{1}\wedge X_{t}}\) 
has a left homotopy inverse. Let $B$ be the homotopy cofibre of $h_{1,t}$. The 
left homotopy inverse for $h_{1,t}$ and the fact that $\Sigma X_{1}\wedge X_{t}$ 
is a suspension implies that there is a homotopy equivalence 
\[\Sigma X_{1}\wedge X_{t}\simeq\Sigma A\vee B.\]

\begin{lemma} 
   \label{Cbartype} 
   In Proposition~\ref{wedgeex}, there is a homotopy equivalence 
   $\overline{C}\simeq\bigg(X_{1}\wedge(\displaystyle\bigvee_{\substack{i=2 \\ i\neq t}}\Sigma X_{i})\bigg)\vee B$. 
\end{lemma} 

\begin{proof} 
Let 
\(q_{t}\colon\namedright{\bigvee_{i=1}^{m}\Sigma X_{i}}{}{X_{t}}\) 
be the pinch map to the $t^{th}$-wedge summand. Then $q_{t}\circ h=h_{1,t}$, so there 
is a homotopy cofibration diagram 
\begin{equation} 
  \label{barCpo} 
  \diagram 
       & X_{1}\wedge(\displaystyle\bigvee_{\substack{i=2 \\ i\neq t}}\Sigma X_{i})\rdouble\dto 
            & X_{1}\wedge(\displaystyle\bigvee_{\substack{i=2 \\ i\neq t}}\Sigma X_{i})\dto \\ 
       \Sigma A\rto^-{h}\ddouble & X_{1}\wedge(\bigvee_{i=1}^{m}\Sigma X_{i})\rto\dto^{1\wedge q_{t}} 
            & \overline{C}\dto \\ 
       \Sigma A\rto^-{h_{1,t}} & X_{1}\wedge\Sigma X_{t}\rto & B. 
  \enddiagram 
\end{equation} 
The homotopy equivalence $\Sigma X_{1}\wedge X_{t}\simeq\Sigma A\vee B$ splitting 
the homotopy cofibration along the bottom row implies that the map 
\(\namedright{\Sigma X_{1}\wedge X_{t}}{}{B}\) 
has a right homotopy inverse 
\(b\colon\namedright{B}{}{X_{1}\wedge\Sigma X_{t}}\). 
As $i_{t}$ is a right homotopy inverse for $q_{t}$, we obtain a composite 
\[\lnamedddright{B}{b}{X_{1}\wedge\Sigma X_{t}}{1\wedge i_{t}}{X_{1}\wedge(\bigvee_{i=1}^{m}\Sigma X_{i})} 
       {}{\overline{C}}\] 
which, by the homotopy commutativity of the lower right square in~(\ref{barCpo}), is a 
right homotopy inverse for the map 
\(\namedright{\overline{C}}{}{B}\). 
Thus the right column in~(\ref{barCpo}) splits to give the asserted homotopy equivalence. 
\end{proof} 

A family of examples satisfying Proposition~\ref{wedgeex} is the following. In words it says 
that if there is a homotopy cofibration  
\(\nameddright{S^{2n-1}}{f}{\bigvee_{i=1}^{m} S^{n}}{}{M}\) 
where the attaching map $f$ is: (i) a sum of Whitehead products, at least one of which is $\pm [i_{1},i_{t}]$ 
for some $t\in\{2,\ldots,m\}$, and (ii) a map factoring through $\bigvee_{i=2}^{m} S^{n}$ that suspends 
trivially, then the homotopy type of $\Omega M$ can be precisely determined. 
           
\begin{proposition} 
   \label{sphereex} 
   Suppose that there is a homotopy cofibration 
   \[\nameddright{S^{2n-1}}{f}{\bigvee_{i=1}^{m} S^{n}}{}{M}.\] 
   Suppose that $f=f_{1}+f_{2}$ where: 
   \begin{itemize} 
      \item $f_{1}=\sum_{j=2}^{m} d_{j}\cdot [i_{1},i_{j}]$ for $d_{j}\in\mathbb{Z}$; 
      \item there is at least one $t\in\{2,\ldots,m\}$ such that $d_{t}=\pm 1$; 
      \item $f_{2}$ factors as 
               \(\nameddright{S^{2n-1}}{\gamma}{\bigvee_{i=2}^{m} S^{n}}{I} 
                     {\bigvee_{i=1}^{m} S^{n}}\) 
                for some map $\gamma$; 
      \item $\Sigma\gamma$ is null homotopic. 
   \end{itemize}    
   Then there is a homotopy fibration 
   \[\nameddright{(\Omega S^{n}\ltimes\overline{C})\vee(\bigvee_{i=2}^{m} S^{n})} 
        {}{M}{q'}{S^{n}}\] 
    where $\overline{C}\simeq S^{n-1}\wedge(\displaystyle\bigvee_{\substack{i=2 \\ i\neq t}}^{m} S^{n})$, 
    and this homotopy fibration splits after looping to give a homotopy equivalence 
    \[\Omega M\simeq\Omega S^{n}\times
           \Omega\bigg((\Omega S^{n}\ltimes\overline{C})\vee(\bigvee_{i=2}^{m} S^{n})\bigg).\] 
\end{proposition} 

\begin{proof} 
The existence of the homotopy fibration and the decomposition for $\Omega M$ will follow 
from Proposition~\ref{wedgeex} once the hypotheses on the attaching map $f$ are shown 
to imply the hypotheses in the proposition. Observe that the map 
\(\namedright{\Sigma A}{h_{1,j}}{\Sigma X_{1}\wedge X_{j}}\) 
in Proposition~\ref{wedgeex} in our case is of the form 
\(\namedright{S^{2n-1}}{}{S^{2n-1}}\) 
and so is a degree map, which has been labelled $d_{j}$. The condition that $d_{t}=\pm 1$ 
for some $t\in\{2,\ldots,m\}$ implies that the map $1\wedge h_{1,t}\simeq 1\wedge d_{t}$ is a homotopy 
equivalence, and so has a right homotopy inverse. The conditions on $f_{2}$ and 
$\gamma$ are the same as in Proposition~\ref{wedgeex}. The homotopy type of $\overline{C}$ 
follows from Lemma~\ref{Cbartype}, noting that as $h_{1,t}$ is a homotopy equivalence its 
homotopy cofibre $B$ is contractible. 
\end{proof}     

\begin{remark} 
Observe that the homotopy type of $\Omega M$ in Proposition~\ref{sphereex} depends only on 
$m$ and $n$. In particular, the map $\gamma$ has no influence on the homotopy type. 
\end{remark} 

\begin{corollary} 
   \label{sphereexcor} 
   In Proposition~\ref{sphereex} the space 
   $(\Omega S^{n}\ltimes\overline{C})\vee(\bigvee_{i=2}^{m} S^{n_{i}+1}))$ 
   is homotopy equivalent to a wedge $W$ of spheres. In particular, 
   \[\Omega M\simeq\Omega S^{n_{1}+1}\times\Omega W.\] 
\end{corollary} 

\begin{proof} 
It suffices to show that $\Omega S^{n_{1}+1}\ltimes\overline{C}$ is homotopy equivalent to a wedge 
of spheres. By Proposition~\ref{sphereex}, $\overline{C}$ is homotopy equivalent to a wedge of 
simply-connected spheres. In particular, $\overline{C}\simeq\Sigma C'$ where $C'$ is a wedge 
of connected spheres. Therefore 
\[\Omega S^{n}\ltimes C\simeq\Omega S^{n}\ltimes\Sigma C'\simeq 
       (\Sigma\Omega S^{n}\wedge C')\vee \Sigma C'.\] 
The James construction implies that $\Sigma\Omega S^{n}$ is homotopy equivalent 
to a wedge of spheres, and therefore so is $(\Sigma\Omega S^{n})\wedge C'$. Hence 
$\Omega S^{n}\ltimes\overline{C}$ is homotopy equivalent to a wedge of spheres. 
\end{proof} 

We give two examples of Proposition~\ref{sphereex}. The first is not new, as it can be 
derived from the results in any one of~\cite{BT1,BT2,BB}. The second is new in general. 

\begin{example} 
In Proposition~\ref{sphereex}, if the cofibration takes the form 
\[\nameddright{S^{2n-1}}{f}{\bigvee_{i=1}^{2m} S^{n}}{}{M}\] 
where $f=[i_{1},i_{2}]+[i_{3},i_{4}]+\cdots + [i_{2m-1},i_{2m}]$ then $M$ is an 
$(n-1)$-connected $2n$-dimensional Poincar\'{e} Duality complex. In fact, it is 
the $m$-fold connected sum $(S^{n}\times S^{n})^{\conn m}$. 
Proposition~\ref{sphereex} then gives a homotopy decomposition of $\Omega M$. 
\end{example} 

\begin{example} 
Modifying the previous example, consider a homotopy cofibration 
\[\nameddright{S^{2n-1}}{f'}{\bigvee_{i=1}^{2m} S^{n}}{}{M'}\] 
where $f'=[i_{1},i_{2}]+[i_{3},i_{4}]+\cdots + [i_{2m-1},i_{2m}]+f''$. 
Here, $f''$ is a composite  
\(f''\colon\nameddright{S^{2n-1}}{\gamma}{\bigvee_{i=2}^{2m} S^{n}}{I} 
       {\bigvee_{i=1}^{2m} S^{n}}\) 
with the property that $\Sigma\gamma$ is null homotopic. Possibly $\gamma$ is a 
sum of more Whitehead products, possibly it is a class of finite order, or some combination 
of the two. Then $M'$ may or may not be a Poincar\'{e} Duality complex but 
Proposition~\ref{sphereex} still applies, giving a homotopy decomposition of $\Omega M'$. 
Note that the decompositions for $\Omega M'$ and the space $\Omega M$ in the previous 
example are identical. That is, while $f''$ may mean $M\not\simeq M'$, after looping 
we nevertheless have $\Omega M\simeq\Omega M'$. 
\end{example} 

Next, we consider the case when the spaces $X_{i}$ in Proposition~\ref{wedgeex} 
are mod-$p^{r}$ Moore spaces. The analogue of Proposition~\ref{sphereex} involves 
mod-$p^{r}$ Whitehead products rather than ordinary Whitehead products. Let
\(a\colon\namedright{P^{m+1}(p^{r})}{}{Z}\) 
and 
\(b\colon\namedright{P^{n+1}(p^{r})}{}{Z}\) 
be maps. If $p^{r}\neq 2$, by Lemma~\ref{mooresmash} there is a map 
\[u\colon\namedright{P^{m+n+1}(p^{r})}{}{\Sigma P^{m}(p^{r})\wedge P^{n}(p^{r})}\] 
which has a left homotopy inverse. The \emph{mod-$p^{r}$ Whitehead product} is the composite 
\[[a,b]_{r}\colon\nameddright{P^{m+n+1}(p^{r})}{u}{\Sigma P^{m}(p^{r})\wedge P^{n}(p^{r})} 
       {[a,b]}{Z}\] 
where $[a,b]$ is the usual Whitehead product. 

\begin{lemma} 
   \label{mooreex} 
   Suppose that for $n\geq 2$ there is a homotopy cofibration 
   \[\nameddright{P^{2n+1}(p^{r})}{f}{\bigvee_{i=1}^{m} P^{n+1}(p^{r})}{}{M}.\] 
   Suppose that $f=f_{1}+f_{2}$ where: 
   \begin{itemize} 
      \item $f_{1}=\sum_{j=2}^{m} d_{j}\cdot [i_{1},i_{j}]_{r}$ for $d_{j}\in\mathbb{Z}/p^{r}\mathbb{Z}$; 
      \item there is at least one $t\in\{2,\ldots,m\}$ such that $d_{t}$ is a unit in $\mathbb{Z}/p^{r}\mathbb{Z}$; 
      \item $f_{2}$ factors as 
               \(\nameddright{P^{2n+1}(p^{r})}{\gamma}{\bigvee_{i=2}^{m} P^{n+1}(p^{r})}{I} 
                     {\bigvee_{i=1}^{m} P^{n+1}(p^{r})}\) 
                for some map $\gamma$; 
      \item $\Sigma\gamma$ is null homotopic. 
   \end{itemize}    
   Then there is a homotopy fibration 
   \[\nameddright{(\Omega P^{n+1}(p^{r})\ltimes\overline{C})\vee(\bigvee_{i=2}^{m} P^{n+1}(p^{r}))} 
        {}{M}{q'}{P^{n+1}(p^{r})}\] 
    where 
    $\overline{C}\simeq \bigg(P^{n}(p^{r})\wedge(\displaystyle\bigvee_{\substack{i=2 \\ i\neq t}}^{m} 
         P^{n+1}(p^{r}))\bigg)\vee P^{2n}(p^{r})$, 
    and this homotopy fibration splits after looping to give a homotopy equivalence 
    \[\Omega M\simeq\Omega P^{n+1}(p^{r})\times
           \Omega\bigg((\Omega P^{n+1}(p^{r})\ltimes\overline{C})\vee(\bigvee_{i=2}^{m} P^{n+1}(p^{r}))\bigg).\] 
\end{lemma} 

\begin{proof} 
The argument is just as for Proposition~\ref{sphereex}, but with the map 
\(\namedright{\Sigma A}{h_{1,j}}{\Sigma X_{1}\wedge X_{j}}\) 
in Proposition~\ref{wedgeex} in this case being of the form  
\(\namedright{P^{2n+1}(p^{r})}{d_{j}\cdot u}{\Sigma P^{n}(p^{r})\wedge P^{n}(p^{r})}\). 
\end{proof}    

\begin{remark} 
As for Proposition~\ref{sphereex}, the homotopy type of $\Omega M$ in Lemma~\ref{mooreex} 
depends only on $m$ and $n$, with the map $\gamma$ having no influence on the homotopy type. 
\end{remark} 

\begin{corollary} 
   \label{mooreexcor} 
   In Lemma~\ref{mooreex} the space 
   $(\Omega P^{n+1}(p^{r})\ltimes\overline{C})\vee(\bigvee_{i=2}^{m} P^{n+1}(p^{r})$ 
   is homotopy equivalent to a wedge $W'$ of mod-$p^{r}$ Moore spaces. In particular, 
   \[\Omega M\simeq\Omega P^{n+1}(p^{r})\times\Omega W'.\] 
\end{corollary} 

\begin{proof} 
The argument is just as for Corollary~\ref{sphereexcor} with appeals to 
Lemma~\ref{mooresmash} in order to decompose iterated smash products 
of mod-$p^{r}$ Moore spaces into wedges of mod-$p^{r}$ Moore spaces. 
\end{proof}

\newpage 
\section{An application to Poincar\'{e} Duality spaces} 
\label{sec:PD} 

In Lemma~\ref{mooreex} a mod-$p^{r}$ Moore space was attached to a wedge 
of mod-$p^{r}$ Moore spaces. We want to next consider attaching a sphere to a 
wedge of mod-$p^{r}$ Moore spaces. That is, we consider a homotopy cofibration of 
the form 
\[\nameddright{S^{2n}}{f}{\bigvee_{i=1}^{m} P^{n+1}(p^{r})}{}{M}\] 
and will assume throughout the section that $p$ is odd and $n\geq 2$. 
Such cofibrations are particularly interesting because for certain maps $f$ the space $M$ is an 
$(n-1)$-connected $(2n+1)$-dimensional Poincar\'{e} Duality complex that is rationally 
equivalent to $S^{2n+1}$. A highlight of this section is the proof of Theorem~\ref{PDexintro}. 

The distinction between attaching a Moore space and attaching a sphere is large 
in the sense that we can no longer appeal to Proposition~\ref{wedgeex} or even 
to Theorem~\ref{E'typeII}. Instead, we have to go back to the inner workings of 
the proof of Theorem~\ref{E'typeII} and make a modification that is specific to 
this case. This requires some initial lemmas. 

\begin{lemma} 
   \label{XAretract} 
   Suppose that there is a map 
   \(\namedright{A}{v}{X^{\wedge k}\wedge D}\) 
   such that 
   \(\namedright{X\wedge\Sigma A}{1\wedge\Sigma v}{X\wedge (X^{\wedge k}\wedge\Sigma D)}\) 
   has a left homotopy inverse. Then 
   \(\namedright{\Omega\Sigma X\wedge\Sigma A}{1\wedge\Sigma v} 
         {\Omega\Sigma X\wedge (X^{\wedge k}\wedge\Sigma D)}\) 
   has a left homotopy inverse.  
\end{lemma} 

\begin{proof} 
It will be convenient to write the identity map on a space $Z$ as $1_{Z}$. Let 
\(u\colon\namedright{X\wedge(X^{\wedge k}\wedge\Sigma D)}{}{X\wedge\Sigma A}\) 
be a left homotopy inverse for the map $1_{X}\wedge\Sigma v$. 
Then for each $t\geq 1$ the composite 
\[\lllnameddright{X^{\wedge t}\wedge\Sigma A}{1_{X^{\wedge t}}\wedge\Sigma v} 
     {X^{\wedge t}\wedge (X^{\wedge k}\wedge\Sigma D)} 
     {1_{X^{\wedge t}}\wedge u}{X^{\wedge t-1}\wedge\Sigma A}\] 
is homotopic to the identity map. Consider the diagram 
\[\diagram 
      \Omega\Sigma X\wedge\Sigma A\rrto^-{1_{\Omega\Sigma X}\wedge\Sigma v}\dto^{\simeq} 
          & & \Omega\Sigma X\wedge\Sigma D\dto^{\simeq} & & \\ 
       \bigvee_{t=1}^{\infty} X^{\wedge t}\wedge\Sigma A 
            \rrto^-{\bigvee_{t=1}^{\infty} 1_{X^{\wedge t}}\wedge\Sigma v} 
           & & \bigvee_{t=1}^{\infty} X^{\wedge t}\wedge (X^{\wedge k}\wedge\Sigma D) 
                   \rrto^-{\bigvee_{t=1}^{\infty} 1_{X^{\wedge t}}\wedge u} 
           & & \bigvee_{t=1}^{\infty} X^{\wedge t}\wedge\Sigma A. 
  \enddiagram\] 
The right square homotopy commutes by the naturality of the James splitting of 
$\Sigma\Omega\Sigma X$, where we have used the fact that $\Sigma v$ is a suspension 
to rewrite $1_{\Omega\Sigma X}\wedge\Sigma v$ as $1_{\Sigma\Omega\Sigma X}\wedge v$. 
The bottom row is homotopic to the identity map since 
$(1_{X^{\wedge t}}\wedge u)\circ(1_{X^{\wedge t}}\wedge\Sigma v)$ is homotopic to the identity 
map for each $t\geq 1$. Therefore the homotopy commutativity of the diagram implies 
that $1_{\Omega\Sigma X}\wedge\Sigma v$ has a left homotopy inverse. 
\end{proof}    

\begin{example} 
\label{spheremooresplit} 
The relevance of Lemma~\ref{XAretract} is as follows. Let $v$ be the composite 
\[v\colon\nameddright{S^{2n-1}}{}{P^{2n}(p^{r})}{u}{P^{n}(p^{r})\wedge P^{n}(p^{r})}\] 
where the left map is the inclusion of the bottom cell and $u$ is the inclusion of the top 
dimensional Moore space in the homotopy decomposition of $P^{n}(p^{r})\wedge P^{n}(p^{r})$ 
in Lemma~\ref{mooresmash}. In particular, $\Sigma v$ does not have a left homotopy inverse. 
However, Lemma~\ref{mooresmash} implies that 
\[\namedright{P^{n}(p^{r})\wedge S^{2n}}{1\wedge\Sigma v} 
       {P^{n}(p^{r})\wedge(\Sigma P^{n}(p^{r})\wedge P^{n}(p^{r}))}\] 
does have a left homotopy inverse. Lemma~\ref{XAretract} then implies that 
\[\namedright{\Omega P^{n+1}(p^{r})\wedge S^{2n}}{1\wedge\Sigma v} 
       {\Omega P^{n+1}(p^{r})\wedge(\Sigma P^{n}(p^{r})\wedge P^{n}(p^{r}))}\] 
has a left homotopy inverse. 
\end{example} 

Next, we consider what will be the analogue of the map $\ell$ in Theorem~\ref{E'typeII}. 

\begin{lemma} 
   \label{mooreell} 
   Suppose that there is a map 
   \(\namedright{S^{2n}}{\ell}{P^{n}(p^{r})\ltimes(\bigvee_{i=2}^{m} P^{n+1}(p^{r}))}\) 
   which induces an inclusion in mod-$p$ homology. If $p$ is odd then the order of $\ell$ is $p^{r}$ 
   and $\ell$ factors as a composite 
   \[\nameddright{S^{2n}}{}{P^{2n+1}(p^{r})}{\ell'}{P^{n}(p^{r})\ltimes(\bigvee_{i=2}^{m} P^{n+1}(p^{r}))}\] 
   for some map $\ell'$. 
\end{lemma} 

\begin{proof} 
By Lemmas~\ref{halfsmashsusp} and~\ref{mooresmash} there are homotopy equivalences 
\[\begin{split} 
   P^{n}(p^{r})\ltimes(\bigvee_{i=2}^{m} P^{n+1}(p^{r})) & \simeq 
     \bigg(\bigvee_{i=2}^{m} P^{n}(p^{r})\wedge P^{n+1}(p^{r})\bigg)\vee(\bigvee_{i=2}^{m} P^{n+1}(p^{r})) \\ 
     & \simeq\bigg(\bigvee_{i=2}^{m} P^{2n+1}(p^{r})\vee P^{2n}(p^{r})\bigg)\vee(\bigvee_{i=2}^{m} P^{n+1}(p^{r})). 
\end{split}\]  
Since $\ell$ induces an inclusion in mod-$p$ homology, it must map into at least one of the 
$P^{2n+1}(p^{r})$ wedge summands as the inclusion of the bottom cell (up to multiplication 
by a unit in $\mathbb{Z}/p^{r}\mathbb{Z}$). Therefore the order of $\ell$ is at least $p^{r}$. 
On the other hand, by~\cite{CMN}, if $p$ is odd then $\pi_{k}(P^{s}(p^{r}))$ is annihilated 
by $p^{r}$ for any $k\leq 2s$. The Hilton-Milnor Theorem implies that any wedge 
$\bigvee_{j=1}^{t} P^{s_{j}}(p^{r})$ with $n+1\leq s_{j}\leq 2n+1$ for all $1\leq j\leq t$ has the 
property that $\pi_{k}(\bigvee_{j=1}^{t} P^{s_{j}}(p^{r}))$ is annihilated by~$p^{r}$ for all $k\leq 2n+2$. 
Thus, in our case, $\pi_{2n}(P^{n}(p^{r})\ltimes(\bigvee_{i=2}^{m} P^{n+1}(p^{r})))$ is 
annihilated by $p^{r}$, so the order of $\ell$ is at most $p^{r}$. Hence the order of $\ell$ 
is exactly $p^{r}$. Consequently, $\ell$ extends to 
\(\namedright{P^{2n+1}(p^{r})}{\ell'}{P^{n}(p^{r})\ltimes(\bigvee_{i=2}^{m} P^{n+1}(p^{r}))}\) 
for some map $\ell'$. 
\end{proof} 

Define the spaces $\widetilde{C}$ and $\overline{C}$ by the homotopy cofibration diagram 
\begin{equation} 
  \label{CtildeCdgrm} 
  \diagram 
      & \bigvee_{i=2}^{m} P^{n+1}(p^{r})\rdouble\dto & \bigvee_{i=2}^{m} P^{n+1}(p^{r})\dto \\ 
      S^{2n}\rto^-{\ell}\ddouble & P^{n}(p^{r})\ltimes(\bigvee_{i=2}^{m} P^{n+1}(p^{r}))\rto\dto^{q} 
           & \widetilde{C}\dto \\ 
       S^{2n}\rto^-{q\circ\ell} & P^{n}(p^{r})\wedge(\bigvee_{i=2}^{m} P^{n+1}(p^{r}))\rto 
           & \overline{C}. 
  \enddiagram 
\end{equation} 

Since $\ell$ induces an inclusion in mod-$p$ homology so does $q\circ\ell$. In particular, 
there is a $t\in\{2,\ldots,m\}$ such that the composite 
\[\nameddright{S^{2n}}{q\circ\ell}{P^{n}(p^{r})\wedge(\bigvee_{i=2}^{m} P^{n+1}(p^{r}))} 
         {1\wedge q_{t}}{P^{n}(p^{r})\wedge P^{n+1}(p^{r})}\] 
induces an injection in mod-$p$ homology. 

\begin{lemma} 
   \label{ell'retract} 
   The composite 
   \[\namedddright{P^{2n+1}(p^{r})}{\ell'}{P^{n}(p^{r})\ltimes(\bigvee_{i=2}^{m} P^{n+1}(p^{r}))} 
          {q}{P^{n}(p^{r})\wedge(\bigvee_{i=2}^{m} P^{n+1}(p^{r}))}{1\wedge q_{t}} 
          {P^{n}(p^{r})\wedge P^{n+1}(p^{r})}\] 
   has a left homotopy inverse. 
\end{lemma} 

\begin{proof} 
The restriction of $(1\wedge q_{t})\circ q\circ\ell'$ to $S^{2n}$ is $(1\wedge q_{t})\circ q\circ\ell$, 
which is an injection in mod-$p$ homology. The action of the Bockstein then implies that 
$(1\wedge q_{t})\circ q\circ\ell'$ induces an injection in mod-$p$ homology. By 
Lemma~\ref{mooresmash}, $P^{n}(p^{r})\wedge P^{n+1}(p^{r})\simeq P^{2n+1}(p^{r})\vee P^{2n}(p^{r})$, 
so composing $(1\wedge q_{t})\circ q\circ\ell'$ with the pinch map to $P^{2n+1}(p^{r})$ 
gives a self-map of $P^{2n+1}(p^{r})$ which induces an isomorphism in mod-$p$ homology, 
and hence in integral homology, and so is a homotopy equivalence by Whitehead's Theorem. 
\end{proof} 

\begin{lemma} 
   \label{CtildeCtype} 
   Suppose that there is a map 
   \(\namedright{S^{2n}}{\ell}{P^{n}(p^{r})\ltimes(\bigvee_{i=2}^{m} P^{n+1}(p^{r}))}\) 
   which induces an inclusion in mod-$p$ homology. If $p$ is odd then the following hold: 
   \begin{letterlist} 
      \item the homotopy cofibration 
              \(\nameddright{\bigvee_{i=2}^{m} P^{n+1}(p^{r})}{}{\widetilde{C}}{}{\overline{C}}\) 
              in~(\ref{CtildeCdgrm}) splits to give a homotopy equivalence 
              \[\widetilde{C}\simeq(\bigvee_{i=2}^{m} P^{n+1}(p^{r}))\vee\overline{C};\] 
    \item there is a homotopy equivalence 
             \[\namedright{\bigg(P^{n}(p^{r})\wedge(\displaystyle\bigvee_{\substack{i=2 \\ i\neq t}}^{m} 
                  P^{n+1}(p^{r}))\bigg)\vee \bigg(S^{2n+1}\vee P^{2n}(p^{r})\bigg)}{\simeq}{\overline{C}}\] 
            where the map 
            \(\namedright{S^{2n+1}}{}{\overline{C}}\) 
            factors through the map 
            \(\namedright{\widetilde{C}}{}{\overline{C}}\). 
   \end{letterlist} 
\end{lemma} 

\begin{proof} 
The hypotheses imply that Lemma~\ref{mooreell} holds. The factorization of $\ell$ through $\ell'$ in Lemma~\ref{mooreell} implies that there is a homotopy cofibration diagram  
\begin{equation} 
  \label{CtildeGdgrm} 
  \diagram 
     S^{2n}\rto\ddouble & P^{2n+1}(p^{r})\rto\dto^{\ell'} & S^{2n+1}\dto \\ 
     S^{2n}\rto^-{\ell} & P^{n}(p^{r})\ltimes(\bigvee_{i=2}^{m} P^{n+1}(p^{r}))\rto\dto^{a} 
         & \widetilde{C}\dto^{b} \\ 
     & G\rdouble & G 
  \enddiagram 
\end{equation} 
which defines the space $G$ and the maps $a$ and $b$. By Lemma~\ref{ell'retract}, 
$\ell'$ has a left homotopy inverse. Therefore as 
$P^{n}(p^{r})\ltimes(\bigvee_{i=2}^{m} P^{n+1}(p^{r}))$ is a suspension, there is a 
homotopy equivalence 
\[P^{n}(p^{r})\ltimes(\bigvee_{i=2}^{m} P^{n+1}(p^{r}))\simeq P^{2n+1}(p^{r})\vee G.\] 
In particular, the map $a$ in~(\ref{CtildeGdgrm}) has a right homotopy inverse. 
The homotopy commutativity of the bottom right square in~(\ref{CtildeGdgrm}) 
then implies that $b$ also has a right homotopy inverse. Thus 
\[\widetilde{C}\simeq S^{2n+1}\vee G.\] 
Since $G$ is the homotopy cofibre of $\ell'$, the left homotopy inverse of $(1\wedge q_{t})\circ q\circ\ell'$ 
in Lemma~\ref{ell'retract} implies that 
\[G\simeq \bigvee_{i=2}^{m} P^{n+1}(p^{r})\vee 
      \bigg(P^{n}(p^{r})\wedge(\displaystyle\bigvee_{\substack{i=2 \\ i\neq t}}^{m} P^{n+1}(p^{r}))\bigg) 
      \vee P^{2n}(p^{r}).\] 
In particular, the upper right vertical arrow in~(\ref{CtildeCdgrm}) composed with 
\(\namedright{\widetilde{C}}{}{G}\) 
is the inclusion of $\bigvee_{i=2}^{m} P^{n+1}(p^{r})$ into $G$. Therefore, the map 
\(\namedright{\bigvee_{i=2}^{m} P^{n+1}(p^{r})}{}{\widetilde{C}}\) 
in~(\ref{CtildeCdgrm}) has a left homotopy inverse, proving part~(a). Further, this implies 
that its homotopy cofibre $\overline{C}$ satisfies 
\[\overline{C}\simeq\bigg(P^{n}(p^{r})\wedge(\displaystyle\bigvee_{\substack{i=2 \\ i\neq t}}^{m} P^{n+1}(p^{r}))\bigg) 
      \vee (S^{2n+1}\vee P^{2n}(p^{r})).\] 
Finally, note that the inclusion of $S^{2n+1}$ in $\overline{C}$ is via the composite 
\(\nameddright{S^{2n+1}}{}{\widetilde{C}}{}{\overline{C}}\), 
completing the proof of part~(b).  
\end{proof} 

Recall that for $1\leq k\leq m$ the map 
\(i_{k}\colon\namedright{P^{n+1}(p^{r})}{}{\bigvee_{i=1}^{m} P^{n+1}(p^{r})}\) 
is the inclusion of the $k^{th}$-wedge summand. The Whitehead product $[i_{j},i_{k}]$ is a map 
\(\namedright{\Sigma P^{n}(p^{r})\wedge P^{n}(p^{r})}{}{\bigvee_{i=1}^{m} P^{n+1}(p^{r})}\). 
Combining this with the map $v$ in Example~\ref{spheremooresplit} gives a composite 
\[\lnameddright{S^{2n}}{\Sigma v}{\Sigma P^{n}(p^{r})\wedge P^{n}(p^{r})}{[i_{j},i_{k}]} 
     {\bigvee_{i=1}^{m} P^{n+1}(p^{r})}\] 
where $\Sigma v$ induces an inclusion in mod-$p$ homology. Note that as $v$ factors 
as the composite 
\(\nameddright{S^{2n}}{}{P^{2n+1}(p^{r})}{}{\Sigma P^{n}(p^{r})\wedge P^{n}(p^{r})}\), 
where the left map is the inclusion of the bottom cell and the right map is the inclusion 
of the top dimensional Moore space, the map $[i_{j},i_{k}]\circ v$ can alternatively 
be regarded as 
\(\llnameddright{S^{2n}}{}{P^{2n+1}(p^{r})}{[i_{j},i_{k}]_{r}}{\bigvee_{i=1}^{k} P^{n+1}(p^{r})}\), 
where $[i_{j},i_{k}]_{r}$ is the mod-$p^{r}$ Whitehead product defined before 
Lemma~\ref{mooreex}. In what follows the notation will be formulated in terms 
of $[i_{j},i_{k}]\circ\Sigma v$. 

\begin{theorem} 
   \label{PDex} 
   Let $p$ be an odd prime, $r\geq 1$ and $n\geq 2$. Suppose that there is a homotopy cofibration 
   \[\nameddright{S^{2n}}{f}{\bigvee_{i=1}^{m} P^{n+1}(p^{r})}{}{M}.\] 
   Suppose also that $f=f_{1}+f_{2}$ where: 
   \begin{itemize} 
      \item $f_{1}=\sum_{j=2}^{m} [i_{1},i_{j}]\circ (d_{j}\cdot\Sigma v)$ for $d_{j}\in\mathbb{Z}$; 
      \item there is at least one $t\in\{2,\ldots,m\}$ such that the mod-$p$ reduction of $d_{t}$ is a unit; 
      \item $f_{2}$ factors as 
               \(\nameddright{S^{2n}}{\gamma}{\bigvee_{i=2}^{m} P^{n+1}(p^{r})}{I} 
                     {\bigvee_{i=1}^{m} P^{n+1}(p^{r})}\) 
                for some map $\gamma$; 
      \item $\Sigma\gamma$ is null homotopic. 
   \end{itemize}    
   Then there is a homotopy fibration 
   \[\nameddright{(\Omega P^{n+1}(p^{r})\ltimes\overline{C})\vee(\bigvee_{i=2}^{m} P^{n+1}(p^{r}))} 
        {}{M}{q'}{P^{n+1}(p^{r})}\] 
    where 
    $\overline{C}\simeq \bigg(P^{n}(p^{r})\wedge(\displaystyle\bigvee_{\substack{i=2 \\ i\neq t}}^{m} 
         P^{n+1}(p^{r}))\bigg)\vee \bigg(S^{2n+1}\vee P^{2n}(p^{r})\bigg)$, 
    and this homotopy fibration splits after looping to give a homotopy equivalence 
    \[\Omega M\simeq\Omega P^{n+1}(p^{r})\times
           \Omega\bigg((\Omega P^{n+1}(p^{r})\ltimes\overline{C})\vee(\bigvee_{i=2}^{m} P^{n+1}(p^{r}))\bigg).\] 
\end{theorem} 

\begin{proof} 
The proof proceeds in several steps. 
\medskip 

\noindent 
\textit{Step 1: An observation.} 
We are assuming that $f_{1}=\sum_{j=2}^{m} [i_{1},i_{j}]\circ (d_{j}\cdot\Sigma v)$ so the map 
\(\namedright{\Sigma A}{h_{1,j}}{\Sigma X_{1}\wedge X_{j}}\) 
in Proposition~\ref{wedgeex} in this case is 
\(\lnamedright{S^{2n}}{d_{j}\cdot\Sigma v}{\Sigma P^{n}(p^{r})\wedge P^{n}(p^{r})}\). 
This does not have a left homotopy inverse, regardless of the value of $d_{j}$. 
Therefore the hypotheses of Proposition~\ref{wedgeex} do not apply. However, by 
Lemma~\ref{mooresmash}, the map 
\(\lllnamedright{P^{n}(p^{r})\wedge S^{2n}}{1\wedge d_{j}\cdot\Sigma v} 
     {P^{n}(p^{r})\wedge\Sigma P^{n}(p^{r})\wedge P^{n}(p^{r})}\) 
does have a left homotopy inverse if the mod-$p$ reduction of $d_{t}$ is a unit. 
Therefore, by Lemma~\ref{XAretract}, the map 
\[\lllnamedright{\Omega P^{n+1}(p^{r})\ltimes S^{2n}}{1\wedge d_{j}\cdot\Sigma v} 
     {\Omega P^{n+1}(p^{r})\wedge(\Sigma P^{n}(p^{r})\wedge P^{n}(p^{r}))}\] 
has a left homotopy inverse. 
\medskip 

\noindent 
\textit{Step 2: The approach.} 
In light of Step 1, the approach is to modify the proof of Theorem~\ref{E'typeII}, on which the proof 
of Proposition~\ref{wedgeex} relied. Consider the data 
\[\diagram 
       & E\rto\dto & E'\dto \\ 
       S^{2n}\rto^-{f} & \bigvee_{i=1}^{m} P^{n+1}(p^{r})\rto\dto^{q_{1}} & M\dto^{q'} \\ 
       & P^{n+1}(p^{r})\rdouble & P^{n+1}(p^{r}). 
  \enddiagram\] 
The map $q'$ exists since $f=f_{1}+f_{2}$ where $f_{1}$ factors through the sum of the Whitehead 
products~$[\iota_{i},\iota_{j}]$ and so composes trivially with $q_{1}$, while $f_{2}$ factors 
through $I$ and so composes trivially with~$q_{1}$. Arguing just as for Steps 1 through 3 in 
the proof of Proposition~\ref{wedgeex} shows that 
the hypotheses of the present lemma imply parts~(a) through~(d) of Theorem~\ref{E'typeII} - 
which are identical to parts~(a) through~(d) of Proposition~\ref{Sigmakappatheta}. Therefore 
the $k=1$ case of Proposition~\ref{Sigmakappatheta} implies that there is a homotopy cofibration diagram 
\begin{equation} 
  \label{epsilonpsi2} 
  \diagram 
         \Sigma(\Omega P^{n+1}(p^{r})\ltimes S^{2n})\rdouble\dto^{\Sigma\kappa_{1}+\ast} 
              & \Sigma(\Omega P^{n+1}(p^{r})\ltimes S^{2n})\dto^{\Sigma\theta} \\ 
         \Sigma(\Omega P^{n+1}(p^{r})\ltimes(P^{n}(p^{r})\wedge\Sigma D))\vee 
                 \Sigma^{2} D\rto^-{\Sigma\epsilon}\dto^{\Sigma\lambda\vee 1}  
              & \Sigma E\dto^{\Sigma\eta} \\  
      \Sigma(\Omega P^{n+1}(p^{r})\ltimes\overline{C})\vee\Sigma^{2} D\rto^-{\psi} 
              & \Sigma E' 
  \enddiagram 
\end{equation} 
where $D=\bigvee_{i=2}^{m} P^{n+1}(p^{r})$, $\epsilon$ is a homotopy equivalence, 
$\lambda$ is the map to the homotopy cofibre of~$\kappa_{1}$, $\eta$ is the map to the 
homotopy cofibre of $\theta$, and $\psi$ is an induced map of cofibres. As $\epsilon$ is a 
homotopy equivalence the five-lemma implies that $\psi$ induces an isomorphism in homology 
and so is a homotopy equivalence by Whitehead's Theorem. 
\medskip 

\noindent 
\textit{Step 3: The homotopy type of $E'$, setting up}. 
Write $\psi=\psi_{1}\perp\psi_{2}\perp\psi_{3}$ where 
$\psi_{1}$, $\psi_{2}$ and $\psi_{3}$ are the restrictions of $\psi$ to 
$\Sigma\overline{C}$, $\Sigma(\Omega P^{n+1}(p^{r})\wedge\overline{C})$ and $\Sigma^{2} D$ respectively. 
Similarly write $\epsilon=\epsilon_{1}\perp\epsilon_{2}\perp\epsilon_{3}$ where $\epsilon_{1}$, 
$\epsilon_{2}$ and $\epsilon_{3}$ are the restrictions of $\epsilon$ to 
$P^{n}(p^{r})\wedge\Sigma D$, $\Omega P^{n+1}(p^{r})\wedge(P^{n}(p^{r})\wedge\Sigma D)$ 
and $\Sigma^{2} D$ respectively. First, observe that the bottom square 
in~(\ref{epsilonpsi2}) implies that $\psi_{3}=\Sigma\eta\circ\Sigma\epsilon_{3}$. In particular, 
if $\psi'_{3}=\eta\circ\epsilon_{3}$ then $\psi_{3}=\Sigma\psi'_{3}$. Second, consider the homotopy cofibration 
\(\nameddright{S^{2n}}{q\circ\ell}{P^{n}(p^{r})\wedge\Sigma D}{a}{\overline{C}}\), 
which defines the map $a$. This does not split but, by Step 1, the homotopy cofibration 
\(\llnameddright{\Omega P^{n+1}(p^{r})\wedge S^{2n}}{1\wedge (q\circ\ell)} 
    {\Omega P^{n+1}(p^{r})\wedge(P^{n}(p^{r})\wedge\Sigma D)}{1\wedge a} 
    {\Omega P^{n+1}(p^{r})\wedge\overline{C}}\) 
does split. If~$b$ is a right homotopy inverse of $1\wedge a$ then, since $\kappa_{3}=1\ltimes (q\circ\ell)$ 
by~(\ref{kappas}), the bottom square in~(\ref{epsilonpsi2}) implies that 
$\psi_{2}\simeq\Sigma\eta\circ\Sigma\epsilon_{2}\circ\Sigma b$. In particular, 
if $\psi'_{2}=\eta\circ\epsilon_{2}\circ b$ then $\psi_{2}\simeq\Sigma\psi'_{2}$. 
Third, consider the homotopy cofibration diagram (in which columns are homotopy cofibrations) 
\begin{equation} 
  \label{jelltheta} 
  \diagram 
       S^{2n}\rto^-{j}\dto^{\ell} & \Omega P^{2n+1}(p^{r})\ltimes S^{2n}\rdouble\dto^{1\ltimes\ell} 
              &  \Omega P^{2n+1}(p^{r})\ltimes S^{2n}\dto^{\theta} \\ 
       P^{n}(p^{r})\ltimes\Sigma D\rto^-{j}\dto 
              & \Omega P^{n+1}(p^{r})\ltimes(P^{n}(p^{r})\ltimes\Sigma D)\rto^-{e}\dto & E\dto^{\eta} \\ 
       \widetilde{C}\rto^-{j} & \Omega P^{n+1}(p^{r})\ltimes\widetilde{C}\rto^-{\xi} & E'   
  \enddiagram 
\end{equation}  
where $\xi$ is an induced map of cofibres. The splitting $\widetilde{C}\simeq\Sigma D\vee\overline{C}$ 
in Lemma~\ref{CtildeCtype}~(a) implies that there is a map 
\[\alpha\colon\nameddright{\overline{C}}{}{\widetilde{C}}{\xi\circ j}{E'}.\] 
Note that the map 
\[\beta\colon\nameddright{\Sigma D}{}{\widetilde{C}}{\xi\circ j}{E'}\] 
is the same as $\eta\circ e\circ j$ restricted to $\Sigma D$. 
\medskip 

\noindent 
\textit{Step 4: The homotopy type of $E'$.} 
We claim that the map 
\[\lllnamedright{\overline{C}\vee\Sigma(\Omega P^{n+1}(p^{r})\wedge\overline{C})\vee\Sigma^{2} D} 
      {\alpha\perp\psi'_{2}\perp\psi'_{3}}{E'}\] 
is a homotopy equivalence. It suffices to show that it induces an isomorphism in homology. 
To do this, it suffices to show that it induces an isomorphism in mod-$q$ homology for 
every prime $q$ and in rational homology. In mod-$p$ homology, since $\ell$ induces 
an injection, the map 
\(\namedright{P^{n}(p^{r})\ltimes\Sigma D}{}{\widetilde{C}}\) 
is a surjection. Thus the image of $(\xi\circ j)_{\ast}$ is determined by the image of 
$(\eta\circ e\circ j)_{\ast}$. That is, the image of $(\alpha\perp\beta)_{\ast}$ is determined 
by the image of $(\eta\circ e\circ j)_{\ast}$. By definition of $\epsilon$, 
$e\circ j\simeq\epsilon_{1}\perp\epsilon_{3}$. Thus, after suspending, the bottom square in~(\ref{epsilonpsi2}) 
implies that the image of $(\Sigma\eta\circ(\Sigma\epsilon_{1}\perp\Sigma\epsilon_{3}))_{\ast}$ 
equals the image of $(\psi_{1}\perp\psi_{3})_{\ast}$. Since $\beta=\psi'_{3}$, we obtain that 
the image of $(\Sigma\alpha\perp\Sigma\beta)_{\ast}$ equals the image of 
$(\psi_{1}\perp\psi_{3})_{\ast}$. Hence as $\psi=\psi_{1}\perp\psi_{2}\perp\psi_{3}$ is a 
homotopy equivalence, it induces an isomorphism in mod-$p$ homology, and therefore so does 
$\Sigma\alpha\perp\psi_{2}\perp\psi_{3}$. Hence, desuspending, $\alpha\perp\psi'_{2}\perp\psi'_{3}$ 
must induce an isomorphism in mod-$p$ homology. Localizing at a prime $q$ for $q\neq p$ or rationally, 
since $P^{n}(p^{r})$ is contractible, the homotopy cofibration diagram~(\ref{jelltheta}) reduces to 
\[\diagram 
       S^{2n}\rdouble\dto & S^{2n}\rdouble\dto & S^{2n}\dto \\ 
       \ast\rto\dto & \ast\rto\dto & \ast\dto \\ 
       S^{2n+1}\rdouble & S^{2n+1}\rdouble & S^{2n+1}. 
  \enddiagram\] 
In particular, $\widetilde{C}\simeq S^{2n+1}$ and $\xi\circ j$ is a homotopy equivalence. 
Observe also that $\alpha$ is a homotopy equivalence. Thus 
$\alpha\perp\psi'_{2}\perp\psi'_{3}$ is a homotopy equivalence and so induces an 
isomorphism in mod-$q$, or respectively rational, homology. Hence 
$\alpha\perp\psi'_{2}\perp\psi'_{3}$ induces an isomorphism in integral homology, as required. 
\medskip 

\noindent 
\textit{Step 5: The homotopy type of $\overline{C}$.} 
As $f_{1}=\sum_{j=2}^{m} [i_{1},i_{j}]\circ (d_{j}\cdot\Sigma v)$, the map 
\(\namedright{\Sigma A}{h_{1,j}}{\Sigma X_{1}\wedge X_{j}}\) 
in Proposition~\ref{wedgeex} in this case is 
\(\lnamedright{S^{2n}}{d_{j}\cdot\Sigma v}{\Sigma P^{n}(p^{r})\wedge P^{n}(p^{r})}\). 
As there is at least one $t\in\{2,\ldots,m\}$ such that the mod-$p$ reduction of $d_{j}$ is 
a unit, the map 
\(\lnamedright{S^{2n}}{d_{t}\cdot\Sigma v}{\Sigma P^{n}(p^{r})\wedge P^{n}(p^{r})}\) 
induces an inclusion in mod-$p$ homology. Therefore, by Lemma~\ref{CtildeCtype}, 
there is a homotopy equivalence 
$\overline{C}\simeq \bigg(P^{n}(p^{r})\wedge(\displaystyle\bigvee_{\substack{i=2 \\ i\neq t}}^{m} 
         P^{n+1}(p^{r}))\bigg)\vee \bigg(S^{2n+1}\vee P^{2n}(p^{r})\bigg)$. 
\end{proof}    

As a special case of Theorem~\ref{PDex} we prove Theorem~\ref{PDexintro}, restated verbatim. 

\begin{theorem} 
   \label{PDexintro2} 
   Let $p$ be an odd prime, $r\geq 1$ and $n\geq 2$. Suppose that there is a homotopy cofibration 
   \[\nameddright{S^{2n}}{f}{\bigvee_{i=1}^{m} P^{n+1}(p^{r})}{}{M}\] 
   where $f=\sum_{1\leq j<k\leq m} [i_{j},i_{k}]\circ (d_{j,k}\cdot v)$ for $d_{j,k}\in\mathbb{Z}$ 
   and at least one $d_{j,k}$ reduces to a unit mod-$p$. Rearranging the wedge summands 
   $\bigvee_{i=1}^{m} P^{n+1}(p^{r})$ so that some $d_{1,t}$ reduces to a unit mod-$p$, 
   there is a homotopy fibration 
   \[\nameddright{(\Omega P^{n+1}(p^{r})\ltimes\overline{C})\vee(\bigvee_{i=2}^{m} P^{n+1}(p^{r}))} 
        {}{M}{q'}{P^{n+1}(p^{r})}\] 
    where 
    $\overline{C}\simeq \bigg(P^{n}(p^{r})\wedge(\displaystyle\bigvee_{\substack{i=2 \\ i\neq t}}^{m} 
         P^{n+1}(p^{r}))\bigg)\vee \bigg(S^{2n+1}\vee P^{2n}(p^{r})\bigg)$, 
    and this homotopy fibration splits after looping to give a homotopy equivalence 
    \[\Omega M\simeq\Omega P^{n+1}(p^{r})\times
           \Omega\bigg((\Omega P^{n+1}(p^{r})\ltimes\overline{C})\vee(\bigvee_{i=2}^{m} P^{n+1}(p^{r}))\bigg).\] 
\end{theorem} 

\begin{proof} 
By hypothesis, there is a homotopy cofibration 
\[\nameddright{S^{2n}}{f}{\bigvee_{i=1}^{m} P^{n+1}(p^{r})}{}{M}\] 
where $f=\sum_{1\leq j<k\leq m} [i_{j},i_{k}]\circ (d_{j,k}\cdot v)$ for $d_{j,k}\in\mathbb{Z}$ 
and at least one $d_{j,k}$ reduces to a unit mod-$p$. Rearrange the wedge summands 
$\bigvee_{i=1}^{m} P^{n+1}(p^{r})$ so that at least one $d_{1,t}$ reduces to a unit mod-$p$. 
Let $f_{1}=\sum_{k=2}^{m} [i_{1},i_{k}]\circ. (d_{1,k}\cdot v)$ and 
$f_{2}=\sum_{2\leq j<k\leq m} [i_{j},i_{k}]\circ (d_{j,k}\cdot v)$. Then $f=f_{1}+f_{2}$, 
there is a $t\in\{2,\ldots,m\}$ such that the mod-$p$ reduction of $d_{1,t}$ is a unit, 
as $f_{2}$ does not involve $i_{1}$ it factors as a composite 
\(\nameddright{S^{2n}}{\gamma}{\bigvee_{i=2}^{m} P^{n+1}(p^{r})}{I}{\bigvee_{i=1}^{m} P^{n+1}(p^{r})}\) 
where $\gamma$ is the same sum of Whitehead products as in $f_{2}$ but with each $i_{j}$ 
for $2\leq j\leq k$ throught of as having range $\bigvee_{i=2}^{m} P^{n+1}(p^{r})$ rather 
than $\bigvee_{i=1}^{m} P^{n=1}(p^{r})$, and $\Sigma\gamma$ is null homotopic since it is 
a sum of Whitehead products. Thus the attaching map $f$ satisfies the hypotheses of 
Theorem~\ref{PDex}, implying that there is a homotopy fibration 
\[\nameddright{(\Omega P^{n+1}(p^{r})\ltimes\overline{C})\vee(\bigvee_{i=2}^{m} P^{n+1}(p^{r}))} 
       {}{M}{q'}{P^{n+1}(p^{r})}\] 
where 
$\overline{C}\simeq \bigg(P^{n}(p^{r})\wedge(\displaystyle\bigvee_{\substack{i=2 \\ i\neq t}}^{m} 
      P^{n+1}(p^{r}))\bigg)\vee \bigg(S^{2n+1}\vee P^{2n}(p^{r})\bigg)$, 
and this homotopy fibration splits after looping to give a homotopy equivalence 
\[\Omega M\simeq\Omega P^{n+1}(p^{r})\times
       \Omega\bigg((\Omega P^{n+1}(p^{r})\ltimes\overline{C})\vee(\bigvee_{i=2}^{m} P^{n+1}(p^{r}))\bigg).\] 
\end{proof} 

\begin{example} 
In Theorem~\ref{PDex}, or Theorem~\ref{PDexintro}, if the cofibration takes the form 
\[\nameddright{S^{2n}}{f}{\bigvee_{i=1}^{2m} P^{n+1}(p^{r})}{}{M}\] 
where $f=[i_{1},i_{2}]\circ v+[i_{3},i_{4}]\circ v+\cdots + [i_{2m-1},i_{2m}]\circ v$ then $M$ is an 
$(n-1)$-connected \mbox{$(2n+1)$}-dimensional Poincar\'{e} Duality complex. In fact, it is 
the $m$-fold connected sum $N\conn\cdots\conn N$ of the Poincar\'{e} Duality 
complex $N$ defined by the homotopy cofibration 
\[\llnameddright{S^{2n}}{[i_{1},i_{2}]\circ v}{P^{n+1}(p^{r})\vee P^{n+1}(p^{r})}{}{N}.\]  
Theorem~\ref{PDex} then gives a homotopy decomposition of $\Omega M$. 
\end{example} 

\begin{example} 
Modifying the previous example, consider a homotopy cofibration 
\[\nameddright{S^{2n}}{f'}{\bigvee_{i=1}^{2m} P^{n+1}(p^{r})}{}{M'}\] 
where $f'=[i_{1},i_{2}]\circ v+[i_{3},i_{4}]\circ v+\cdots + [i_{2m-1},i_{2m}]\circ v+f''$. 
Here, $f''$ is a composite  
\(f''\colon\nameddright{S^{2n}}{\gamma}{\bigvee_{i=2}^{2m} P^{n+1}(p^{r})}{I} 
       {\bigvee_{i=1}^{2m} P^{n+1}(p^{r})}\) 
with the property that $\Sigma\gamma$ is null homotopic. Possibly $\gamma$ is a 
sum of more Whitehead products, possibly it is a class of finite order, or some combination 
of the two. Then $M'$ may or may not be a Poincar\'{e} Duality complex but 
Theorem~\ref{PDex} still applies, giving a homotopy decomposition of $\Omega M'$. 
Note that the decompositions for $\Omega M'$ and the space $\Omega M$ in the previous 
example are identical. That is, while $f''$ may mean $M\not\simeq M'$, after looping 
we nevertheless have $\Omega M\simeq\Omega M'$. 
\end{example}

\newpage 

\section{Inert maps} 
\label{sec:inert} 

Recall from the Introduction that if 
\(\nameddright{\Sigma A}{f}{Y}{h}{Y'}\) 
is a homotopy cofibration then the map~$f$ is \emph{inert} if $\Omega h$ has a right 
homotopy inverse. An interesting example we have already seen is the homotopy cofibration 
\(\lllnameddright{X^{\wedge k}\wedge\Sigma Y}{ad^{k}(i_{1})(i_{2})}{\Sigma X\vee\Sigma Y}{m_{k}}{M_{k}}\) 
in Section~\ref{sec:2cone}. By Lemma~\ref{adinv}, $\Omega m_{k}$ has a right homotopy 
inverse, and hence $ad^{k}(i_{1})(i_{2})$ is inert. 

The inert property is exactly one of the main hypotheses of Theorem~\ref{BTfibinclusion}, 
and that theorem will play a key role in what follows. As such, it is useful to recall what it says, 
compressed slightly to only what will be needed subsequently. Suppose that 
\(\nameddright{\Sigma A}{f}{Y}{h}{Y'}\) 
is a homotopy cofibration and $E$ is the homotopy fibre of $h$. If $\Omega h$ has a 
right homotopy inverse then there is a homotopy equivalence 
\[\namedright{\Omega Y'\ltimes\Sigma A}{\theta}{E}\] 
and a homotopy fibration 
\[\llnameddright{(\Omega Y'\wedge\Sigma A)\vee\Sigma A}{[\gamma,f]+f}{Y}{h}{Y'}\] 
where $\gamma$ is the composite 
\(\nameddright{\Sigma\Omega Y'}{\Sigma s}{\Sigma\Omega Y}{ev}{Y}\). 
\medskip 

Suppose that there are homotopy cofibrations 
\[\nameddright{\Sigma A}{f}{X}{j}{M}\] 
and 
\[\nameddright{\Sigma A}{g}{Y}{k}{N}.\] 
Since $\Sigma A$ is a suspension, $f$ and $g$ can be added: $f+g$ is the composite 
\[f+g\colon\nameddright{\Sigma A}{\sigma}{\Sigma A\vee\Sigma A}{f\vee g}{X\vee Y}\] 
where $\sigma$ is the comultiplication on $\Sigma A$. Define $C$ by the homotopy cofibration  
\[\nameddright{\Sigma A}{f+g}{X\vee Y}{}{C}.\] 
Let 
\(q_{1}\colon\namedright{X\vee Y}{}{X}\) 
be the pinch map to the first wedge summand. Then there is a homotopy cofibration diagram 
\begin{equation} 
  \label{WX} 
  \diagram 
        \Sigma A\rto^-{f+g}\ddouble & X\vee Y\rto\dto^{q_{1}} & C\dto^{\varphi} \\ 
        \Sigma A\rto^-{f} & X\rto^-{j} & M 
  \enddiagram 
\end{equation}  
that defines the map $\varphi$. Let $h$ be the composite 
\[h\colon\nameddright{X\vee Y}{q_{1}}{X}{j}{M}.\] 
Note that by~(\ref{WX}), $h$ is homotopic to the composite 
\(\nameddright{X\vee Y}{}{C}{\varphi}{M}\). 
Let $E$ and $E'$ be the homotopy fibres of $h$ and $\varphi$ respectively. 
Then we obtain the following diagram of spaces and maps that collects the 
data that will go into Theorem~\ref{GTcofib}:  
\begin{equation} 
  \label{EE'} 
  \diagram 
       & E\rto\dto & E'\dto \\ 
       \Sigma A\rto^-{f+g} & X\vee Y\rto\dto^{h} & C\dto^{\varphi} \\ 
       & M\rdouble & M. 
   \enddiagram 
\end{equation} 

\begin{lemma} 
   \label{inverses} 
   Suppose that $f$ is inert. Then both $\Omega h$ and $\Omega\varphi$ have 
   right homotopy inverses. 
\end{lemma} 

\begin{proof} 
As $f$ is inert there is a map 
\(t\colon\namedright{\Omega M}{}{\Omega X}\) 
such that $\Omega j\circ t$ is homotopic to the identity map on $\Omega M$. 
Consider the composite 
\[\namedddright{\Omega M}{t}{\Omega X}{\Omega i_{1}}{\Omega (X\vee Y)}{\Omega h}{\Omega M}\] 
where $i_{1}$ is the inclusion of the left wedge summand. By definition, $h= j\circ q_{1}$, 
so as $q_{1}\circ i_{1}$ is the identity map on $X$, we obtain 
\[\Omega h\circ\Omega i_{1}\circ t\simeq\Omega j\circ\Omega q_{1}\circ\Omega i_{1}\circ t\simeq 
       \Omega j\circ t\simeq id_{\Omega M}.\] 
Thus $\Omega h$ has a right homotopy inverse. The homotopy commutativity of 
the bottom square in~(\ref{EE'}) then implies that $\Omega\varphi$ also has a right 
homotopy inverse. 
\end{proof} 

Since $\Omega h$ has a right homotopy inverse, 
applying Theorem~\ref{GTcofib} to~(\ref{EE'}) gives a homotopy cofibration 
\begin{equation} 
  \label{Efcofib} 
  \lnameddright{\Omega M\ltimes\Sigma A}{\theta_{f+g}}{E}{}{E'}. 
\end{equation}  

Next consider the homotopy cofibration 
\(\nameddright{\Sigma A}{f}{X}{j}{M}\). 
Let $F$ be the homotopy fibre of $j$. Since $f$ is inert the map $\Omega j$ has a right 
homotopy inverse so by Theorem~\ref{BTfibinclusion} there is a homotopy equivalence  
\begin{equation} 
  \label{Fequiv} 
  \namedright{\Omega M\ltimes\Sigma A}{\theta_{f}}{F}. 
\end{equation}  

The homotopy cofibrations~(\ref{Efcofib}) and~(\ref{Fequiv}) can be put together. 
By definition, $h=j\circ q_{1}$, so there is a homotopy fibration diagram 
\begin{equation} 
  \label{kqdgrm} 
  \diagram 
           E\rto^-{\ell}\dto & F\dto \\ 
           X\vee Y\rto^-{q_{1}}\dto^{h} & X\dto^{j} \\ 
           M\rdouble & M. 
  \enddiagram 
\end{equation} 
where $\ell$ is the induced map of fibres. The right homotopy inverse 
\(s\colon\namedright{\Omega M}{}{\Omega(X\vee Y)}\) 
for $\Omega h$ implies that $\Omega q_{1}\circ s$ is a right homotopy inverse 
for $\Omega j$. The naturality property in Remark~\ref{GTcofibnat} implies that there  
is a homotopy commutative diagram of cofibrations 
\begin{equation} 
  \label{Thetacompat} 
  \diagram 
        \Omega M\ltimes\Sigma A\rto^-{\theta_{f+g}}\ddouble & E\rto\dto^{\ell} & E'\dto \\ 
        \Omega M\ltimes\Sigma A\rto^-{\theta_{f}} & F\rto & \ast  
  \enddiagram 
\end{equation} 
Note that $\theta_{f}$ being a homotopy equivalence implies that the 
map $\theta_{f+g}$ has a left homotopy inverse. Moreover, this inverse is 
independent of $g$. We record this for future reference. 

\begin{lemma} 
   \label{Thetainv} 
   The map 
   \(\namedright{\Omega M\ltimes\Sigma A}{\theta_{f+g}}{E}\) 
   has a left homotopy inverse that is independent of the map~$g$.~$\qqed$ 
\end{lemma} 

Next, consider the special case when $g$ is the trivial map. In~(\ref{EE'}) the 
homotopy cofibration 
\(\nameddright{\Sigma A}{f+g}{X\vee Y}{}{C}\) 
becomes 
\(\nameddright{\Sigma A}{f+\ast}{X\vee Y}{j\vee 1}{M\vee Y}\) 
and the map 
\(\namedright{C}{\varphi}{M}\) 
can be chosen to be the pinch map 
\(\namedright{M\vee Y}{q_{1}}{M}\).   
Therefore the homotopy fibre $E'$ of $\varphi$ becomes $\Omega M\ltimes Y$. 
Hence the homotopy cofibration~(\ref{Efcofib}) becomes 
\begin{equation} 
  \label{Eastcofib} 
  \llnameddright{\Omega M\ltimes\Sigma A}{\theta_{f+\ast}}{E}{}{\Omega M\ltimes Y}. 
\end{equation}  
In this case we show that the cofibration~(\ref{Eastcofib}) splits in a way that 
behaves well with respect to the map $\ell$ in~(\ref{Thetacompat}). 

\begin{lemma} 
   \label{Eastinv} 
   The map 
   \(\namedright{E}{}{\Omega M\ltimes Y}\) 
   in~(\ref{Eastcofib}) has a right homotopy inverse 
   \(r\colon\namedright{\Omega M\ltimes Y}{}{E}\) 
   such that $\ell\circ r$ is null homotopic. 
\end{lemma} 

\begin{proof} 
The identifications in~(\ref{EE'}) when $g=\ast$ imply that there is a homotopy 
fibration diagram 
\[\diagram 
       E\rto\dto & \Omega M\ltimes Y\dto \\ 
       X\vee Y\rto^-{j\vee 1}\dto^{h} & M\vee Y\dto^{q_{1}} \\ 
       M\rdouble & M.   
   \enddiagram\]  
In particular, the upper square is a homotopy pullback. From the naturality of 
the pinch map $q_{1}$ we obtain a pullback map 
\[\xymatrix{
       \Omega X\ltimes Y\ar@/_/[ddr]\ar@/^/[drr]^{\Omega j\ltimes 1}\ar@{.>}[dr]^(0.6){\overline{r}}  & & \\ 
       & E\ar[r]\ar[d] & \Omega M\ltimes Y\ar[d] \\ 
        & X\vee Y\ar[r]^-{j\vee 1} & M\vee Y }\]
that defines $\overline{r}$. Since $f$ is inert, $\Omega j$ has a right homotopy inverse 
\(t\colon\namedright{\Omega M}{}{\Omega X}\). 
Let $r$ be the composite 
\[r\colon\nameddright{\Omega M\ltimes Y}{t\ltimes 1}{\Omega X\ltimes Y}{\overline{r}}{E}.\] 
Then the previous diagram implies that the composite 
\(\nameddright{\Omega M\ltimes Y}{r}{E}{}{\Omega M\ltimes Y}\) 
is homotopic to $(\Omega j\ltimes 1)\circ(t\ltimes 1)$, which is homotopic to the identity map. 
Thus 
\(\namedright{E}{}{\Omega M\ltimes Y}\) 
has a right homotopy inverse. 

It remains to show that $\ell\circ r$ is null homotopic. Consider the diagram 
\[\diagram 
      \Omega M\ltimes Y\rto^-{t\ltimes 1} & \Omega X\ltimes Y\rto^-{\overline{r}} 
           & E\rto^-{\ell}\dto & F\dto \\ 
      & & X\vee Y\rto^-{q_{1}} & X.  
  \enddiagram\] 
The square homotopy commutes by~(\ref{kqdgrm}). The definition of $\overline{r}$ 
as a pullback map implies that the composite 
\(\nameddright{\Omega X\ltimes Y}{\overline{r}}{E}{}{X\vee Y}\) 
is the map from the homotopy fibre of $q_{1}$ to the total space. Therefore 
composing it with $q_{1}$ is null homotopic so the lower direction around 
the diagram is null homotopic. Hence the upper direction around the diagram 
is null homotopic. By definition, $r=\overline{r}\circ (t\ltimes 1)$, implying that $\ell\circ r$ 
is null homotopic when composed with 
\(\namedright{F}{}{X}\). 
Now consider the homotopy fibration sequence 
\(\namedddright{\Omega M}{\partial}{F}{}{X}{j}{M}\), 
where $\partial$ is the connecting map. On the one hand, we have just seen 
that $\ell\circ r$ must lift through $\partial$. On the other hand, since $\Omega j$ 
has a right homotopy inverse, $\partial$ is null homotopic. Therefore $\ell\circ r$ 
is null homotopic, as asserted. 
\end{proof} 

In general, suppose that  
\(\nameddright{U}{s}{V}{t}{W}\) 
is a homotopy cofibration where $t$ has a right homotopy inverse $r$. Then the composite 
\(e\colon\nameddright{U\vee W}{s\vee r}{V\vee V}{\nabla}{V}\) 
induces an isomorphism in homology, where $\nabla$ is the fold map. Thus if $U$, $V$ 
and $W$ are simply-connected then $e$ is a homotopy equivalence by Whitehead's 
Theorem. In our case, we are assuming that all spaces are simply-connected, so 
the existence of the right homotopy inverse in Lemma~\ref{Eastinv} implies the following. 

\begin{corollary} 
  \label{Eastsplit} 
  From the homotopy cofibration 
  \(\lnameddright{\Omega M\ltimes\Sigma A}{\theta_{f+\ast}}{E}{}{\Omega M\ltimes Y}\) 
  we obtain a homotopy equivalence 
   \[\lllnameddright{(\Omega M\ltimes\Sigma A)\vee(\Omega M\ltimes Y)}{\theta_{f+\ast}\vee\,r} 
        {E\vee E}{\nabla}{E}\] 
   where $\nabla$ is the fold map.~$\qqed$  
\end{corollary} 

Now return to the general case of the homotopy cofibration 
\(\nameddright{\Omega M\ltimes\Sigma A}{\theta_{f+g}}{E}{}{E'}\).  
We will use the special case when $g=\ast$ to show a splitting in the 
general case, and identify the homotopy type of $E'$. This requires 
a preliminary lemma, which is stated abstractly. To distinguish identity maps 
on different spaces, for a space $V$ let 
\(1_{V}\colon\namedright{V}{}{V}\) 
be the identity map on $V$. 

\begin{lemma} 
   \label{2cofibsplit} 
   Suppose that there are homotopy cofibrations 
   \(\nameddright{P}{p}{Q}{j_{p}}{R_{p}}\) 
   and 
   \(\nameddright{P}{q}{Q}{j_{q}}{R_{q}}\) 
   where all spaces are simply-connected. Also suppose that there are maps 
   \(\namedright{Q}{k}{P}\) 
   and 
   \(\namedright{R_{q}}{s}{Q}\) 
   such that $k\circ p\simeq 1_{P}$, $k\circ q\simeq 1_{P}$, $j_{q}\circ s\simeq 1_{R_{q}}$ 
   and $k\circ s\simeq\ast$. Then the composite 
   \(\nameddright{R_{q}}{s}{Q}{j_{p}}{R_{p}}\) 
   is a homotopy equivalence. 
\end{lemma} 

\begin{proof} 
Start with the homotopy cofibration 
\(\nameddright{P}{q}{Q}{j_{q}}{R_{q}}\).  
Since all spaces are simply-connected, the fact that $j_{q}\circ s\simeq 1_{R_{q}}$ 
implies that the composite 
\[e\colon\nameddright{P\vee R_{q}}{p\vee s}{Q\vee Q}{\nabla}{Q}\] 
is a homotopy equivalence, where $\nabla$ is the fold map. Since 
$k\circ q\simeq 1_{P}$, $k\circ s\simeq\ast$, and the fold map is natural, 
we obtain a homotopy commutative square 
\[\diagram 
       P\vee R_{q}\rto^-{e}\dto^{q_{1}} & Q\dto^{k} \\ 
       P\rdouble & P 
  \enddiagram\] 
where $q_{1}$ is the pinch map to the first wedge summand. 
Restricting to $R_{q}$ we therefore obtain a homotopy cofibration 
\[\nameddright{R_{q}}{s}{Q}{k}{P}.\] 
Now the fact that $k\circ p\simeq 1_{P}$ implies that we obtain a 
homotopy pushout diagram 
\[\diagram 
     & R_{q}\rdouble\dto^{s} & R_{q}\dto \\ 
     P\rto^-{p}\ddouble & Q\rto^-{j_{p}}\dto^{k} & R_{p}\dto \\ 
     P\rto^-{=} & P\rto & \ast. 
  \enddiagram\] 
To be clear, $k\circ p\simeq 1_{P}$ implies that the 
homotopy cofibre along the bottom row is trivial, and therefore the 
homotopy cofibre of $j_{p}\circ s$ is trivial. Hence $j_{p}\circ s$ induces an 
isomorphism in homology and so, as spaces are simply-connected, it is a  
homotopy equivalence by Whitehead's Theorem. 
\end{proof} 

\begin{theorem} 
   \label{inertideal} 
   Suppose that there are homotopy cofibrations 
   \(\nameddright{\Sigma A}{f}{X}{}{M}\) 
   and 
   \(\nameddright{\Sigma A}{g}{Y}{}{N}\) 
   where $f$ is inert. Define the homotopy cofibration 
   \(\nameddright{\Sigma A}{f+g}{X\vee Y}{}{C}\) 
   and the homotopy fibration 
   \(\nameddright{E'}{}{C}{\varphi}{M}\) 
   as in~(\ref{EE'}). Then the following hold: 
   \begin{letterlist} 
      \item the composite 
               \(\nameddright{\Omega M\ltimes Y}{r}{E}{}{E'}\) 
               is a homotopy equivalence, implying that there is a homotopy fibration 
               \(\nameddright{\Omega M\ltimes Y}{}{C}{\varphi}{M}\); 
      \item there is a homotopy equivalence 
               $\Omega C\simeq\Omega M\times\Omega (\Omega M\ltimes Y)$; 
      \item $f+g$ is inert, that is, the map 
               \(\namedright{\Omega(X\vee Y)}{}{\Omega(M\conn N)}\) 
               has a right homotopy inverse;  
      \item there is a homotopy fibration 
               \[\nameddright{(\Omega C\wedge\Sigma A)\vee\Sigma A}{\Psi} 
                   {X\vee Y}{}{C}\] 
               where $\Psi=[\gamma,f+g]+(f+g)$. 
   \end{letterlist} 
\end{theorem} 

\begin{proof} 
By~(\ref{Efcofib}) and~(\ref{Eastcofib}) there are homotopy cofibrations 
\[\lnameddright{\Omega M\ltimes\Sigma A}{\theta_{f+g}}{E}{}{E'}\] 
and 
\[\lnameddright{\Omega M\ltimes\Sigma A}{\theta_{f+\ast}}{E}{}{\Omega M\ltimes Y}.\]  
By Lemma~\ref{Thetainv} there is a map 
\(t\colon\namedright{E}{}{\Omega M\ltimes\Sigma A}\) 
such that $t\circ\vartheta_{f+g}$ and $t\circ\vartheta_{f+\ast}$ 
are both homotopic to the identity map on $\Omega M\ltimes\Sigma A$.  
By Lemma~\ref{Eastinv} there is a map 
\(\namedright{\Omega M\ltimes Y}{r}{E}\) 
such that the composite 
\(\nameddright{\Omega M\ltimes Y}{r}{E}{}{\Omega M\ltimes Y}\) 
is homotopic to the identity map and $t\circ r$ is null homotopic. Therefore, by 
Lemma~\ref{2cofibsplit} the composite 
\(\nameddright{\Omega M\ltimes Y}{r}{E}{}{E'}\) 
is a homotopy equivalence. This proves part~(a). 

For part~(b), consider the homotopy fibration 
\(\nameddright{E'}{}{C}{\varphi}{M}\). 
By Lemma~\ref{inverses}, $\Omega\varphi$ has a right homotopy inverse. This 
immediately implies that there is a homotopy equivalence 
$\Omega C\simeq\Omega M\times\Omega E'$. 
Now substitute in the homotopy equivalence for $E'$ in part~(a) to obtain 
the asserted homotopy equivalence. 

For part~(c), let 
\(i\colon\namedright{X\vee Y}{}{C}\) 
denote the map to the cofibre of $f+g$. To say that $f+g$ is inert means that $\Omega i$ 
has a right homotopy inverse. To see this is the case, consider the loops on 
the homotopy pullback diagram in~(\ref{EE'}), 
\begin{equation} 
  \label{loopEE'} 
  \diagram 
        \Omega E\rto\dto & \Omega E'\dto \\ 
        \Omega(X\vee Y)\rto^-{\Omega i}\dto^{\Omega h} & \Omega C\dto^{\Omega\varphi} \\ 
        \Omega M\rdouble & \Omega M. 
  \enddiagram 
\end{equation}  
We check that the homotopy equivalence in part~(b) can be chosen to factor through $\Omega i$. 
First, by Lemma~\ref{inverses}, the map $\Omega h$ has a right homotopy inverse 
\(s\colon\namedright{\Omega M}{}{\Omega(X\vee Y)}\). 
Thus $\Omega i\circ s$ is a right homotopy inverse for $\Omega\varphi$. Second, 
by part~(a) the composite 
\(\nameddright{\Omega M\ltimes Y}{r}{E}{}{E'}\) 
is a homotopy equivalence. Let $r'$ be the composite 
\(\nameddright{\Omega M\ltimes Y}{r}{E}{}{X\vee Y}\). 
Then the homotopy commutativity of~(\ref{loopEE'}) and the fact that $\Omega i$ is 
an $H$-map implies that the composite 
\[\lnamedddright{\Omega M\times\Omega(\Omega M\ltimes Y)}{s\times\Omega r'} 
      {\Omega(X\vee Y)\times\Omega(X\vee Y)}{\mu}{\Omega(X\vee Y)}{\Omega i}{\Omega C}\] 
is a homotopy equivalence, where $\mu$ is the loop multiplication.  

Finally, now knowing that $f+g$ is inert by~(c), part~(d) is an immediate consequence of 
Theorem~\ref{BTfibinclusion} applied to the homotopy cofibration 
\(\nameddright{A}{f+g}{X\vee Y}{}{C}\). 
\end{proof} 

\begin{remark} 
Theorem~\ref{inertideal} says something notable. The fact that $f$ 
is inert implies that $f+g$ is inert, regardless of what $g$ is. 
\end{remark}

\newpage 

\section{Based loops on connected sums} 
\label{sec:connsum} 

In this section we apply Theorem~\ref{inertideal} to analyze the based loops on a 
connected sum of simply-connected Poincar\'{e} Duality complexes and prove 
Theorem~\ref{introconnsum}. Suppose that $M$ and $N$ 
are simply-connected Poincar\'{e} Duality complexes of dimension $n$, where~$n\geq 3$. Let $X$ 
and $Y$ be the $(n-1)$-skeletons of $M$ and~$N$ respectively. Then there are 
homotopy cofibrations 
\[\nameddright{S^{n-1}}{f}{X}{}{M}\] 
\[\nameddright{S^{n-1}}{g}{Y}{}{N}\] 
where $f$ and $g$ are the attaching maps for the top cells of $M$ and $N$ 
respectively. The connected sum $M\conn N$ is given by the homotopy cofibration 
\[\nameddright{S^{n-1}}{f+g}{X\vee Y}{}{M\conn N}.\] 
This is exactly the situation considered in the previous section, taking $A=S^{n-2}$ 
and $C=M\conn N$. Note that as $n\geq 3$ the space $S^{n-1}$ is a simply-connected suspension. 
As in Section~\ref{sec:inert}, there is a map 
\(\namedright{M\conn N}{\varphi}{M}\), 
where explicitly in this case it is the map given by collapsing $Y\subseteq M\conn N$ 
to a point. So from Theorem~\ref{inertideal} we immediately obtain the following, 
which is a more comprehensive version of Theorem~\ref{introconnsum}. 

\begin{theorem} 
   \label{connsum} 
   Let $M$ and $N$ be simply-connected Poincar\'{e} Duality complexes of dimension $n$, 
   where $n\geq 3$. If the attaching map $f$ for the top cell of $M$ is inert then the following hold: 
   \begin{letterlist} 
      \item there is a homotopy fibration 
               \(\nameddright{\Omega M\ltimes Y}{}{M\conn N}{\varphi}{M}\); 
      \item there is a homotopy equivalence 
              $\Omega(M\conn N)\simeq\Omega M\times\Omega(\Omega M\ltimes Y)$; 
      \item the attaching map $f+g$ for the top cell of $M\conn N$ is inert, that is, 
               the loop map 
               \(\namedright{\Omega(X\vee Y)}{}{\Omega(M\conn N)}\) 
               has a right homotopy inverse; 
      \item there is a homotopy fibration 
               \[\nameddright{(\Sigma\Omega(M\conn N)\wedge S^{n-1})\vee S^{n-1}}{\Psi} 
                   {X\vee Y}{}{M\conn N}\] 
               where $\Psi=[\gamma,f+g]+(f+g)$. 
    \end{letterlist} 
\end{theorem} 
\vspace{-1cm}~$\qqed$\bigskip 

We now give several examples of Theorem~\ref{connsum}. 
First, we consider taking the connected sum with an $(n-1)$-connected 
$2n$-dimensional Poincar\'{e} Duality complex. 

\begin{proposition} 
   \label{connsum4mnfld} 
   Let $M$ be an $(n-1)$-connected $2n$-dimensional Poincar\'{e} Duality 
   complex such that $n\geq 2$ and the attaching map for the top cell of $M$ 
   is inert. Let $N$ be an $(n-k)$-connected, $2n$-dimensional 
   Poincar\'{e} Duality complex with $n-k\geq 1$ and $3k-2\leq n$. Let 
   $Y=N-\ast$. Then~$Y$ is a suspension and there is a homotopy equivalence 
   \[\Omega(M\conn N)\simeq\Omega M\times\Omega((\Omega M\wedge Y)\vee Y).\] 
\end{proposition} 

\begin{proof} 
It is well known that if $m\geq 2$ and $V$ is an $(m-1)$-connected $CW$-complex of 
dimension at most $2m-1$ then~$V$ is homotopy equivalent to a suspension. In our case, $Y$ 
is $(m-1)$-connected for $m=n-k+1$, the condition $n-k\geq 1$ implies $Y$ is 
simply-connected, and the condition that $3k-2\leq n$ implies that~$Y$ is of 
dimension~$\leq 2m-1$. Therefore $Y$ is a suspension. 

Since the attaching map for the top cell of $M$ is inert, by Theroem~\ref{connsum}~(b), there is a 
homotopy equivalence $\Omega(M\conn N)\simeq\Omega M\times\Omega(\Omega M\ltimes Y)$. 
Since $Y$ is a suspension, there is a homotopy equivalence 
$\Omega M\ltimes Y\simeq(\Omega M\wedge Y)\vee Y$, and the assertion follows. 
\end{proof} 

The hypotheses of Proposition~\ref{connsum4mnfld} hold in a wide variety of cases. 
By~\cite{BT2}, if $n\notin\{4,8\}$ then the attaching map for the top cell of an 
$(n-1)$-connected $2n$-dimensional Poincar\'{e} Duality complex~$M$ is inert. 
If $n\in\{4,8\}$ then the attaching map for the top cell is not known to be inert in all 
cases but it may be inert for specific cases: for example, the attaching maps for the 
top cells in $S^{4}\times S^{4}$ and $S^{8}\times S^{8}$ are both inert. 

Observe that if $n=2$ or $n=3$ then the condition $3k-2\leq n$ 
implies $k=1$, and $N$ is then either a simply-connected four-manifold if $n=2$ 
or a $2$-connected $6$-manifold if $n=3$; both cases are then simply repeating 
known decompositions from ~\cite{BT1} or~\cite{BB}. However, if $n=4$ 
then $k=2$ is valid, so we obtain a homotopy decomposition for 
$\Omega(M\conn N)$ when $M=S^{4}\times S^{4}$ and $N$ is 
any $2$-connected $8$-manifold. This is new - in~\cite{BT1} it was shown that 
if $\cohlgy{N;\mathbb{Z}}$ is torsion-free then such a decomposition exists but 
Proposition~\ref{connsum4mnfld} dispenses with the torsion-free cohomology condition. 
More generally, in~\cite{BT1} it was shown that if $M=S^{m}\times S^{2n-m}$ 
and $\cohlgy{N;\mathbb{Z}}$ is torsion-free then $\Omega(M\conn N)$ decomposes.  
Proposition~\ref{connsum4mnfld} significantly generalizes this to $M$ being 
any $(n-1)$-connected $2n$-dimensional Poincar\'{e} Duality complex with $n\geq 2$, 
and $N$ not having a torsion-free cohomology condition but some control over 
the dimensional range in which the middle cells appear. 

Next, we prove a general result in Proposition~\ref{prodexample} about taking 
the connected sum with a product and then increasingly specialize it. 

\begin{lemma} 
   \label{wedgeinv} 
   Let $X_{1},\ldots,X_{k}$ be simply-connected spaces and let 
   \(j\colon\namedright{\bigvee_{i=1}^{k} X_{i}}{}{\prod_{i=1}^{k} X_{i}}\) 
   be the inclusion of the wedge into the product. Then $\Omega j$ has 
   a right homotopy inverse. 
\end{lemma} 

\begin{proof} 
This is well known. Let 
\(j_{i}\colon\namedright{X_{i}}{}{\bigvee_{i=1}^{k} X_{i}}\) 
be the inclusion. Then $j\circ j_{i}$ is the inclusion of the $i^{th}$ factor 
in $\prod_{i=1}^{k} X_{i}$. Looping to multiply, the product of the maps $\Omega j_{i}$ 
for $1\leq i\leq k$ is a right homotopy inverse for $\Omega j$. 
\end{proof} 

The next lemma gives one source of Poincar\'{e} Duality complexes for which the right homotopy 
inverse hypothesis of Theorem~\ref{connsum} holds. 

\begin{lemma} 
   \label{prodM} 
   Let $k\geq 2$ and suppose that $M_{1},\ldots,M_{k}$ are nontrivial simply-connected 
   finite dimensional Poincar\'{e} Duality complexes. Let $M=\prod_{i=1}^{k} M_{i}$ and let 
   \(J\colon\namedright{M-\ast}{}{M}\) 
   be the inclusion. Then $\Omega J$ has a right homotopy inverse. 
\end{lemma} 

\begin{proof} 
As each $M_{i}$ is simply-connected, it may be approximated by a $CW$-complex. 
Doing so, let~$d_{i}$ be the dimension of $M_{i}$. Then $D=\sum_{i=1}^{k} d_{i}$ is the 
dimension of $M$. As $M_{i}$ is nontrivial, we have $d_{i}\geq 1$. Therefore 
as $k\geq 2$ we obtain $d_{i}<D$. Thus the inclusion 
\(\namedright{M_{i}}{}{M}\) 
factors through the $(D-1)$-skeleton of $M$, which is homotopy equivalent to $M-\ast$. 
Hence the inclusion 
\(\namedright{\bigvee_{i=1}^{k} M_{i}}{j}{\prod_{i=1}^{k} M_{i}=M}\) 
of the wedge into the product factors as a composite 
\(\nameddright{\bigvee_{i=1}^{k} M_{i}}{}{M-\ast}{J}{M}\). 
By Lemma~\ref{wedgeinv}, $\Omega j$ has a right homotopy inverse. 
Therefore, so does $\Omega J$. 
\end{proof} 

Theorem~\ref{connsum} and Lemma~\ref{prodM} immediately imply the following. 

\begin{proposition} 
   \label{prodexample} 
   Suppose that $M=\prod_{i=1}^{k} M_{i}$ for $k\geq 2$ and each $M_{i}$ is a nontrivial 
   simply-connected Poincar\'{e} Duality complex of dimension $n$. Let $N$ be any other 
   simply-connected Poincar\'{e} Duality complex of dimension $n$ 
   and let $Y=N-\ast$. Then there is a homotopy equivalence 
   \[\Omega(M\conn N)\simeq\Omega M\times\Omega(\Omega M\ltimes Y).\]  
\end{proposition} 

We consider special cases of Proposition~\ref{prodexample} in which 
the decomposition of $\Omega(M\conn N)$ can be further refined. 

\begin{example} 
\label{Wall} 
Suppose that the product $M$ in Proposition~\ref{prodexample} has dimension $2n$ for $n\geq 2$.  
Let $N$ be an $(n-1)$-connected $2n$-dimensional Poincar\'{e} Duality 
complex. Then Poincar\'{e} Duality implies that $Y=N-\ast$ is homotopy equivalent 
to a wedge of~$d$ copies of $S^{n}$, where $d$ is the rank of $\cohlgy{N;\mathbb{Z}}$.  
If $d\geq 1$ then $Y$ is a suspension, so 
$\Omega M\ltimes Y\simeq(\Omega M\wedge Y)\vee Y$. 
Siimilarly, if $M$ is $(2n+1)$-diimensional for $n\geq 2$ and $N$ is an $(n-1)$-connected 
$(2n+1)$-dimensional Poincar\'{e} Duality complex then $Y=N-\ast$ is homotopy equivalent 
to a wedge of some number of copies of $S^{n}$, $S^{n+1}$ and Moore spaces $P^{n+1}(m)$ 
for various values of $m$. Again, if $Y$ is nontrivial then it is a suspension. Therefore, 
in both cases we obtain a homotopy equivalence 
\[\Omega(M\conn N)\simeq\Omega M\times((\Omega M\wedge Y)\vee Y).\] 
\end{example} 

\begin{example} 
\label{sphereexample} 
Suppose that $M=\prod_{i=1}^{k} S^{m_{i}}$ for $k\geq 2$, each sphere is simply-connected, 
and $N$ is as in Example~\ref{Wall}. Since $\Omega M\simeq\prod_{i=1}^{k}\Omega S^{n_{i}}$, 
iterating the fact that $\Sigma(X\times Y)\simeq\Sigma X\vee\Sigma Y\vee(\Sigma X\wedge Y)$ 
and iterating the fact from~\cite{J1} that 
\[\Sigma\Omega S^{m+1}\simeq\bigvee_{r=1}^{\infty} S^{rm+1}\] 
shows that $\Sigma\Omega M$ is homotopy equivalent to a wedge of spheres. 
If $M$ has dimension $2n$ and $N$ is an $(n-1)$-connected $2n$-dimensional Poincar\'{e} Duality 
complex with $Y=N-\ast$ nontrivial, then~$Y$ is homotopy equivalent to a wedge of copies 
of $S^{n}$, implying that $(\Omega M\wedge Y)\vee Y$ is homotopy equivalent to a 
wedge $W$ of spheres. Thus $\Omega(M\conn N)\simeq\Omega M\times\Omega W$. 
If $M$ has dimension $2n+1$ and $N$ is an $(n-1)$-connected $(2n+1)$-dimensional Poincar\'{e} 
Duality complex with $Y=N-\ast$ nontrivial, then $Y$ is homotopy equivalent to a wedge of spheres 
and Moore spaces, implying that $(\Omega M\wedge Y)\vee Y$ is also homotopy equivalent 
to a wedge $W'$ of spheres and Moore spaces. Thus 
$\Omega(M\conn N)\simeq\Omega M\times\Omega W'$. 
\end{example}

\begin{example} 
\label{cpnexample} 
Suppose that $M=\prod_{i=1}^{k}\cp^{m_{i}}$ for $k\geq 2$, $M$ has dimension $2n$, 
and $N$ is an $(n-1)$-connected $2n$-dimensional Poincar\'{e} Duality complex. 
Since $\Omega\cp^{r}\simeq S^{1}\times\Omega S^{2r+1}$, arguing as in the 
Example~\ref{sphereexample} shows that $\Sigma\Omega\cp^{r}$ is homotopy equivalent 
to a wedge of spheres, as is $\Sigma\Omega M$. Therefore, as in Example~\ref{sphereexample}, 
we obtain a homotopy decomposition of $\Omega(M\conn N)$ in terms of $\Omega M$ 
and the loops on a wedge of spheres. 
\end{example} 

\begin{example} 
\label{Whex} 
In Example~\ref{sphereexample}, suppose that $M=S^{m_{1}}\times S^{m_{2}}$, where   
$m_{1}+m_{2}=2n$, and $N$ is an $(n-1)$-connected $2n$-dimensional Poincar\'{e} 
Duality complex. The decomposition $\Omega(M\conn N)\simeq\Omega M\times\Omega W$ 
in Example~\ref{sphereexample} implies that $\Sigma\Omega(M\conn N)\wedge S^{2n-1}$ is homotopy 
equivalent to a wedge $U$ of spheres. Theorem~\ref{connsum}~(d) then implies that 
there is a homotopy fibration 
\[\nameddright{U\vee S^{2n-1}}{\Psi}{(S^{m_{1}}\vee S^{m_{2}})\vee Y}{} 
        {(S^{m_{1}}\times S^{m_{2}})\conn N}\] 
where the restriction of $\Psi$ to $U$ is a Whitehead product and the restriction 
of $\Psi$ to $S^{2n-1}$ is the attaching map for the top cell of $(S^{m_{1}}\times S^{m_{2}})\conn N$.        
Similarly, if $m_{1}+m_{2}=2n+1$ and $N$ is an $(n-1)$-connected $(2n+1)$-dimensional 
Poincar\'{e} Duality complex then the decomposition 
$\Omega(M\conn N)\simeq\Omega M\times\Omega W'$ in Example~\ref{sphereexample} 
implies that $\Sigma\Omega(M\conn N)\wedge S^{2n}$ is homotopy equivalent to a wedge $U'$ of 
spheres and Moore spaces. Theorem~\ref{connsum}~(d) then implies that 
there is a homotopy fibration 
\[\nameddright{U'\vee S^{2n}}{\Psi}{(S^{m_{1}}\vee S^{m_{2}})\vee Y}{} 
        {(S^{m_{1}}\times S^{m_{2}})\conn N}\] 
where the restriction of $\Psi$ to $U'$ is a Whitehead product and the restriction 
of $\Psi$ to $S^{2n}$ is the attaching map for the top cell of $(S^{m_{1}}\times S^{m_{2}})\conn N$. 
\end{example} 

\begin{example} 
We finish with an interesting specific example. Let $X$ be the Wu manifold, which 
is a $1$-connected $5$-manifold whose mod-$2$ cohomology satisfies 
$\cohlgy{X;\mathbb{Z}/2\mathbb{Z}}\cong\Lambda(x,Sq^{1}(x))$, 
where $\Lambda$ is the free exterior algebra functor and $Sq^{1}$ is the 
first Steenrod operation. As a $CW$-complex, $X=P^{3}(2)\cup e^{5}$. By 
Examples~\ref{sphereexample} and~\ref{Whex} we obtain: (i) a homotopy equivalence 
\[\Omega((S^{2}\times S^{3})\conn X)\simeq\Omega S^{2}\times\Omega S^{3}\times\Omega W\] 
where 
\[W=((\Omega S^{2}\times\Omega S^{3})\wedge P^{3}(2))\vee P^{3}(2)\] 
is homotopy equivalent to a wedge of mod-$2$ Moore spaces; and (ii) a homotopy 
fibration 
\[\nameddright{U'\vee S^{4}}{\Psi}{(S^{2}\vee S^{3})\vee P^{3}(2)}{} 
        {(S^{2}\times S^{3})\conn X}\] 
where 
\[U'=\Sigma^{5}\Omega((S^{2}\times S^{3})\conn X)\] 
is a wedge of spheres and mod-$2$ Moore spaces, the restriction of $\Psi$ to $U'$ is 
a Whitehead product, and the restriction of $\Psi$ to $S^{4}$ is the attaching 
map for the top cell of the connected sum. 
\end{example}

\newpage 

\section{Hopf algebras and one-relator algebras} 
\label{sec:Hopf} 

Now that we have many examples of inert maps we take a homological time-out in 
order to consider the effect an inert map has in homology. To set the stage, consider  
a homotopy cofibration 
\(\nameddright{\Sigma A}{f}{\Sigma Y}{h}{Y'}\) 
with the property that $\Omega h$ has a right homotopy inverse. Our aim is to calculate 
the homology of $\Omega Y'$. Take homology with field coefficients. By the Bott-Samelson 
Theorem there is an algebra isomorphism $\hlgy{\Omega\Sigma Y}\cong T(\rhlgy{Y})$, 
where $T(\ \ )$ is the free tensor algebra functor. 

\begin{proposition} 
   \label{OmegaChlgy} 
   Suppose that there is a homotopy cofibration 
   \(\nameddright{\Sigma A}{f}{\Sigma Y}{h}{Y'}\)  
   where $\Omega h$ has a right homotopy inverse. Let 
   \(\widetilde{f}\colon\namedright{A}{}{\Omega\Sigma Y}\) 
   be the adjoint of $f$ and let $R=\mbox{Im}(\widetilde{f}_{\ast})$. Then 
   there is an algebra isomorphism 
   \[\hlgy{\Omega Y'}\cong T(\rhlgy{Y})/(R)\] 
   where $(R)$ is the two-sided ideal generated by $R$.  
\end{proposition}

\begin{proof} 
First observe that there is an algebra map 
\(\namedright{T(\rhlgy{Y})}{(\Omega h)_{\ast}}{\hlgy{\Omega Y'}}\). 
Since $\Omega h$ has a right homotopy inverse, $(\Omega h)_{\ast}$ is a surjection. 
Since $\widetilde{f}$ is homotopic to the composite 
\(\nameddright{A}{E}{\Omega\Sigma A}{\Omega f}{\Omega\Sigma Y}\), 
where $E$ is the suspension, the composite $\Omega h\circ\widetilde{f}$ is null homotopic. 
Therefore $(\Omega h)_{\ast}(R)=0$. As~$(\Omega h)_{\ast}$ is an algebra map, we obtain  
a factorization 
\[\diagram 
       T(\rhlgy{Y})\rto^-{(\Omega h)_{\ast}}\dto^{a} & \hlgy{\Omega Y'} \\ 
       T(\rhlgy{Y})/(R)\urto_-{b} & 
   \enddiagram\]  
where $a$ is the quotient map and $b$ is an induced algebra homomorphism. Since 
$(\Omega h)_{\ast}$ is surjective, so is $b$. 

On the other hand, by Theorem~\ref{BTfibinclusion} there is a homotopy fibration 
\[\nameddright{\Omega Y'\ltimes\Sigma A}{\chi}{\Sigma Y}{h}{Y'}\] 
where $\chi$ is the sum of the maps 
\(\nameddright{\Omega Y'\ltimes\Sigma A}{\pi}{\Sigma A}{f}{\Sigma Y}\) 
and 
\(\namedright{\Omega Y'\ltimes\Sigma A}{q}{\Omega Y'\wedge\Sigma A}
     \stackrel{[ev\circ s,f]}{\llarrow}\Sigma Y\) 
for 
\(s\colon\namedright{\Omega Y'}{}{\Omega Y}\) 
a right homotopy inverse of $\Omega h$. Consider the composite 
\[\nameddright{\hlgy{\Omega(\Omega Y'\ltimes\Sigma A)}}{(\Omega\chi)_{\ast}} 
      {\hlgy{\Omega\Sigma Y}\cong T(\rhlgy{Y})}{a}{T(\rhlgy{Y})/R}.\] 
Notice that the maps $\pi$ and $q$ are suspensions, so the adjoint of $f\circ\pi$ is 
homotopic to 
\(\alpha\colon\nameddright{\Omega Y'\ltimes A}{}{A}{\widetilde{f}}{\Omega\Sigma Y}\) 
and the adjoint of $[ev\circ s,f]\circ q$ is homotopic to 
\(\beta\colon\nameddright{\Omega Y'\ltimes A}{}{\Omega Y'\wedge A} 
      {\langle\widetilde{ev\circ s},\widetilde{f}\rangle}{\Omega\Sigma Y}\) 
where the right map is the Samelson product of $\widetilde{ev\circ s}$ (the adjoint 
of $ev\circ s$) and $\widetilde{f}$. The James construction implies that $\Omega\chi$ 
is homotopic to the multiplicative extension of $\alpha\perp\beta$. Therefore, 
as $a$ is an algebra map, $a\circ(\Omega\chi)_{\ast}$ is determined by its 
restriction to $a\circ(\alpha\perp\beta)_{\ast}$. By definition, $a$ sends the image of 
$\widetilde{f}_{\ast}$ to the identity element. Therefore $a\circ\alpha_{\ast}$ is trivial. 
Also, the Samelson product commutes with homology in the sense that 
$(\langle\widetilde{ev\circ s},\widetilde{f}\rangle)_{\ast}= 
     \langle(\widetilde{ev\circ s})_{\ast},\widetilde{f}_{\ast}\rangle$, 
where the bracket on the right is the commutator in $T(\rhlgy{Y})$. The triviality of  
$a\circ\widetilde{f}_{\ast}$ therefore implies that 
$a\circ (\langle\widetilde{ev\circ s},\widetilde{f}\rangle)_{\ast}$ is also trivial. 
Thus $a\circ\beta_{\ast}$ is trivial, implying that $a\circ(\alpha\perp\beta)_{\ast}$ is 
trivial, and hence $a\circ(\Omega\chi)_{\ast}$ is trivial. 

Further, as homotopy fibration 
\(\nameddright{\Omega Y'\ltimes\Sigma A}{\chi}{\Sigma Y}{h}{Y'}\) 
has the property that $\Omega h$ has a right homotopy inverse, there is an isomorphism 
$T(\rhlgy{Y})\cong\hlgy{\Omega Y'}\otimes\hlgy{\Omega(\Omega Y'\ltimes\Sigma A)}$ 
of right $\hlgy{\Omega(\Omega Y'\ltimes\Sigma A)}$-modules. Since $a$ is an 
algebra map and $a\circ(\Omega\chi)_{\ast}$ is trivial, we obtain a factorization 
\[\diagram 
      T(\rhlgy{Y})\rto^-{(\Omega h)_{\ast}}\dto^{a} & \hlgy{\Omega Y'}\dlto^{c} \\ 
      T(\rhlgy{Y})/(R) & 
  \enddiagram\] 
where $c$ is an algebra map and a surjection. 

Finally, consider the composite 
\[\namedddright{T(\rhlgy{Y})/(R)}{b}{\hlgy{\Omega Y'}}{c}{T(\rhlgy{Y})/(R)}{b}{\hlgy{\Omega Y'}}.\] 
As $b$ and $c$ are surjections, so are $c\circ b$ and $b\circ c$. Therefore $c\circ b$ 
and $b\circ c$ are surjective self-maps of $T(\rhlgy{Y})/(R)$ and $\hlgy{\Omega Y'}$ 
respectively. Any surjective self-map of a graded finite type module is an isomorphism, 
so both $c\circ b$ and $b\circ c$ are isomorphisms. As $b$ and $c$ are algebra maps, 
these isomorphisms are as algebras. 
\end{proof} 

\begin{remark} 
\label{Hopfalgremark} 
There is an improvement to Proposition~\ref{OmegaChlgy} if $Y$ is a suspension. 
In that case the Bott-Samelson Theorem improves to a Hopf algebra isomorphism 
$\hlgy{\Omega\Sigma Y}\cong T(\rhlgy{Y})$, where the tensor algebra is primitively generated. 
The quotient maps $b$ and $c$ in the proof are then Hopf algebra maps, and we obtain 
an isomorphism of Hopf algebras $\hlgy{\Omega Y'}\cong T(\rhlgy{Y})/(R)$. 
\end{remark} 

\begin{example} 
\label{loopMkhlgy}
Consider the homotopy cofibration 
\(\lllnamedright{X^{\wedge k}\wedge\Sigma Y}{ad^{k}(i_{1})(i_{2})}{\Sigma X\vee\Sigma Y} 
        \stackrel{m_{k}}{\longrightarrow} M_{k}\). 
By Lemma~\ref{adinv}, $\Omega m_{k}$ has a right homotopy inverse. Therefore 
Proposition~\ref{OmegaChlgy} applies and we obtain an algebra isomorphism 
\[\hlgy{\Omega M_{k}}\cong T(\rhlgy{X\vee Y})/(R)\] 
where $R$ is the image in homology of the adjoint of $ad^{k}(i_{1})(i_{2})$. 

A specific  case of interest is the homotopy cofibration 
\(\lllnamedright{S^{km+n+1}}{ad^{k}(i_{1})(i_{2})}{S^{m+1}\vee S^{n+1}} 
     \stackrel{m_{k}}{\longrightarrow} M_{k}\).  
We have $T(\rhlgy{S^{m}\vee S^{n}})=T(x,y)$ where $\vert x\vert=m$ and $\vert y\vert=n$. 
The adjoint of the iterated Whitehead product $ad^{k}(i_{1})(i_{2})$ is an iterated Samelson product, 
and its image in homology is the iterated commutator $ad^{k}(x)(y)$. If $m,n\geq 1$ then by 
Remark~\ref{Hopfalgremark} there is an isomorphism of Hopf algebras 
\[\hlgy{\Omega M_{k}}\cong T(x,y)/(ad^{k}(x)(y)).\] 
\end{example} 

The special case of Example~\ref{loopMkhlgy} is an example of the notion of a one-relator algebra. 
In general, an algebra is a \emph{one-relator algebra} if it is not free and can be written as the 
quotient of a free associative algebra by a two-sided ideal generated by a single element. 
There are many other examples of one-relator algebras that can be obtained from 
Proposition~\ref{OmegaChlgy}. 

\begin{example} 
Let $M$ be an $(n-1)$-connected $2n$-dimensional Poincar\'{e} Duality complex 
where $n\geq 2$. By Poincar\'{e} Duality, as a $CW$-complex $M$ has one zero-cell, 
$d$ $n$-cells for some $d\geq 0$ and one $2n$-cell. If $d=0$ then $M\simeq S^{2n}$. 
Otherwise, there is a homotopy cofibration 
\[\nameddright{S^{2n-1}}{f}{\bigvee_{i=1}^{d} S^{n}}{h}{M}\] 
where $f$ attaches the top cell to $M$. In~\cite{BT2} it was shown that 
that if $d\geq 2$ then $\Omega h$ has a right homotopy inverse. Therefore 
Proposition~\ref{OmegaChlgy} and Remark~\ref{Hopfalgremark} apply to show that there 
is an isomorphism of Hopf algebras 
\[\hlgy{\Omega M}\cong T(\rhlgy{\bigvee_{i=1}^{d} S^{n-1}})/(R)\] 
where $R=\mbox{Im}(\widetilde{f}_{\ast})$. Written explicitly, let $v_{i}\in\hlgy{\bigvee_{i=1}^{d} S^{n-1}}$ 
be a generator corresponding to the $i^{th}$ wedge summand of $\bigvee_{i=1}^{d} S^{n-1}$. 
The image $R$ of $\widetilde{f}_{\ast}$ is generated by a single 
element $r\in T(v_{1},\ldots,v_{m})$. Therefore there is an isomorphism of Hopf algebras 
\[\hlgy{\Omega M}\cong T(v_{1},\ldots,v_{m})/(r).\] 
A particular example of note is when $M$ is a simply-connected four-manifold.  
\end{example} 

The following example of a connected sum of products of two simply-connected spheres was calculated 
in~\cite{GIPS} using the Adams-Hilton model. 

\begin{example} 
Fix an integer $n\geq 4$. Let $M=\conn_{i=1}^{d} (S^{m_{i}}\times S^{n-m_{i}})$ where 
$m_{i}\geq 2$ for each $1\leq i\leq d$. Then there is a homotopy cofibration 
\[\nameddright{S^{n-1}}{f}{\bigvee_{k=1}^{d} S^{m_{i}}\vee S^{n-m_{i}}}{h}{M}\] 
where $f$ is the sum of the Whitehead products attaching the top sphere to each copy 
of \mbox{$S^{m_{i}}\times S^{n-m_{i}}$}. Iterating Theorem~\ref{connsum} shows that the map $\Omega h$ 
has a right homotopy inverse. Therefore Proposition~\ref{OmegaChlgy} and Remark~\ref{Hopfalgremark} 
imply that there is a Hopf algebra isomorphism 
\[\hlgy{\Omega M}\cong T(\rhlgy{\bigvee_{i=1}^{d} S^{m_{i}-1}\vee S^{n-m_{i}-1}})/(R)\] 
where $R=\mbox{Im}(\widetilde{f}_{\ast})$. Explicitly, let $u_{i}\in\hlgy{S^{m_{i}-1}}$ and 
$v_{i}\in\hlgy{S^{n-m_{i}-1}}$ be generators corresponding to the $i^{th}$ wedge summand 
in $\bigvee_{i=1}^{d} S^{m_{i}}\vee S^{n-m_{i}}$. The image $R$ of $\widetilde{f}_{\ast}$ is 
then generated by the single element $[u_{1},v_{1}]+\cdots+[u_{d},v_{d}]$. Therefore there 
is an isomorphism of Hopf algebras 
\[\hlgy{\Omega M}\cong T(u_{1},v_{1},\ldots,u_{d},v_{d})/([u_{1},v_{1}]+\cdots+[u_{d},v_{d}]).\] 
\end{example} 

\begin{example} 
Let $M$ be an $(n-1)$-connected $(2n+1)$-dimensional Poincar\'{e} Duality 
complex for~$n\geq 2$. By Poincar\'{e} Duality, 
\[H^{m}(M)\cong\left\{\begin{array}{ll} 
        \mathbb{Z} & \mbox{if $m=0$ or $m=2n+1$} \\ 
        \mathbb{Z}^{d} & \mbox{if $m=n$} \\ 
        \mathbb{Z}^{d}\oplus G & \mbox{if $m=n+1$} \\ 
        0 & \mbox{otherwise}\end{array}\right.\] 
for some integer $d\geq 0$ and some finite abelian group $G$. Assume that $d\geq 1$. 
Let $X$ be the $(n+1)$-skeleton of $M$. As in~\cite{BT2}, there is a homotopy 
equivalence $X\simeq(\bigvee_{i=1}^{d}(S^{n}\vee S^{n+1}))\vee\Sigma V$ where~$V$ 
is a wedge of $(n+1)$-dimensional Moore spaces. Therefore there is a homotopy cofibration 
\[\nameddright{S^{2n}}{f}{(\bigvee_{i=1}^{d} S^{n}\vee S^{n+1})\vee\Sigma V}{i}{M}\] 
By~\cite{BT2}, $\Omega i$ has a right homotopy inverse. Thus, by Proposition~\ref{OmegaChlgy} 
and Remark~\ref{Hopfalgremark} there is an isomorphism of Hopf algebras 
\[\hlgy{\Omega M}\cong T(\rhlgy{(\bigvee_{i=1}^{d} S^{n-1}\vee S^{n})\vee V})/(R)\] 
where $R=\mbox{Im}(\widetilde{f})_{\ast}$. As in the previous example, this may be rewritten 
as an isomorphism of Hopf algebras 
\[\hlgy{\Omega M}\cong T(\{u_{1},v_{1},\ldots,u_{d},v_{d}\}\oplus\rhlgy{V})/(r)\] 
where $\vert u_{i}\vert = n-1$, $\vert v_{i}\vert=n$ and $r$ generates the image of $\widetilde{f}_{\ast}$. 
\end{example} 

\begin{remark} 
Proposition~\ref{OmegaChlgy} does not apply in general to an $(n-1)$-connected $(2n+1)$-dimensional 
Poincar\'{e} Duality complex with $d=0$. That is, in the case when $X$ is homotopy 
equivalent to a wedge of Moore spaces. For example, if all the Moore spaces are of the form 
$P^{n+1}(p^{r})$ for a fixed odd prime $p$ and integer $r$, then there is a homotopy cofibration 
\[\nameddright{S^{2n}}{f}{\bigvee_{i=1}^{m} P^{n+1}(p^{r})}{i}{M}.\] 
We will show that $\Omega i$ does not have a right homotopy inverse, implying that one 
of the hypotheses of Proposition~\ref{OmegaChlgy} fails to hold. By Theorem~\ref{PDex} 
there is a homotopy fibration 
\[\nameddright{(\Omega P^{n+1}(p^{r})\ltimes\overline{C})\vee(\bigvee_{i=2}^{m} P^{n+1}(p^{r}))} 
     {h}{M}{q'}{P^{n+1}(p^{r})}\] 
that splits after looping and where $\overline{C}\simeq S^{2n+1}\vee W$ where $W$ is 
a wedge of mod-$p^{r}$ Moore spaces. In particular, $\Omega\overline{C}$ is 
rationally nontrivial (because of the factor $\Omega S^{2n+1}$). However, 
$\Omega(\bigvee_{i=1}^{m} P^{n+1}(p^{r}))$ is rationally trivial, so $\Omega i$ cannot 
have a right homotopy inverse. It would be interesting to calculate $\hlgy{\Omega M}$ 
in this case. 
\end{remark}

\newpage 

\section{A second foundational case} 
\label{sec:iteratedad} 

This section is in preparation for the next. To set things up, suppose 
that there is a space~$M$ with the property that there is a factorization of the inclusion 
\(\namedright{\bigvee_{i=1}^{m}\Sigma X_{i}}{}{\prod_{i=1}^{m}\Sigma X_{i}}\) 
as a composite 
\[\nameddright{\bigvee_{i=1}^{m}\Sigma X_{i}}{v}{M}{w}{\prod_{i=1}^{m}\Sigma X_{i}}\] 
for some maps $v$ and $w$. In addition, suppose that there is a homotopy cofibration 
\[\nameddright{\Sigma A}{f}{M}{}{M'}\] 
with the property that $w\circ f$ is null homotopic. Then $w$ extends to a map 
\[w'\colon\namedright{M'}{}{\prod_{i=1}^{m}\Sigma X_{i}}\] 
and there is a homotopy fibration diagram 
\begin{equation} 
  \label{MEdgrm} 
  \diagram 
         E\rto\dto^{p} & E'\dto^{p'} \\ 
         M\rto\dto^{w} & M'\dto^{w'} \\ 
         \prod_{i=1}^{m}\Sigma X_{i}\rdouble & \prod_{i=1}^{m}\Sigma X_{i}  
  \enddiagram 
\end{equation}  
that defines the spaces $E$ and $E'$ and the maps $p$ and $p'$. The inclusion $w\circ v$ 
of the wedge into the product has a right homotopy inverse after looping, implying that $\Omega w$ 
also has a right homotopy inverse 
\(s\colon\namedright{\prod_{i=1}^{m}\Omega\Sigma X_{i}}{}{\Omega M}\). 
Theorem~\ref{GTcofib} then implies that there is a homotopy cofibration 
\[\nameddright{\prod_{i=1}^{m}\Omega\Sigma X_{i}\ltimes\Sigma A}{\theta}{E}{}{E'}\] 
and a homotopy commutative diagram 
\[\diagram 
       \prod_{i=1}^{m}\Omega\Sigma X_{i}\ltimes\Sigma A\rto^-{\theta}\dto^{\simeq} 
             & E\dto^{p} \\ 
       ((\prod_{i=1}^{m}\Omega\Sigma X_{i})\wedge\Sigma A)\vee\Sigma A\rto^-{[\gamma,f]+f} 
             & M 
  \enddiagram\] 
where $\gamma$ is the composite 
\(\nameddright{\Sigma(\prod_{i=1}^{m}\Omega\Sigma X_{i})}{\Sigma s}{\Sigma\Omega M}{ev}{M}\). 
On the other hand, the suspension of a product splits as a wedge, and the James construction 
lets us further split each of the spaces $\Sigma\Omega\Sigma X_{i}$. In this section we show 
that those splittings can be chosen so that the maps from the wedge summands into $M$ can 
be identified as iterated Whitehead products. 

Recall from Lemma~\ref{SigmaJsplitting} that there is a natural homotopy equivalence 
\[\namedright{\bigvee_{k=1}^{\infty}\Sigma X^{\wedge k}}{\phi}{\Sigma\Omega\Sigma X}\] 
defined as follows. For $k\geq 1$, let $e_{k}$ be the composite 
\[e_{k}\colon\nameddright{X^{\times k}}{E^{\times k}}{(\Omega\Sigma X)^{\times k}} 
     {\mu}{\Omega\Sigma X}\] 
where $\mu$ is the standard loop multiplication. There is a natural homotopy equivalence 
$\Sigma(A\times B)\simeq\Sigma A\vee\Sigma B\vee(\Sigma A\wedge B)$. Iterating 
this we obtain a natural map 
\(\namedright{\Sigma X_{1}\wedge\cdots\wedge X_{k}}{}{\Sigma(X_{1}\times\cdots\times X_{k})}\). 
Let $\phi_{k}$ be the composite 
\[\phi_{k}\colon\nameddright{\Sigma X^{\wedge k}}{}{\Sigma(X^{\times k})}{\Sigma e_{k}} 
     {\Sigma\Omega\Sigma X}.\] 
Let 
\[\phi\colon\namedright{\bigvee_{k=1}^{\infty}\Sigma X^{\wedge k}}{}{\Sigma\Omega\Sigma X}\] 
be the wedge sum of the maps $\phi_{k}$ for $k\geq 1$. As shorthand, this is called the 
$\phi$-decomposition of $\Sigma\Omega\Sigma X$. 

Let $X_{1},\ldots, X_{m}$ be path-connected spaces. For $1\leq j\leq m$, let 
\[t_{j}\colon\namedright{X_{j}}{}{\bigvee_{i=1}^{m} X_{i}}\] 
be the inclusion of the $j^{th}$-wedge summand. 
Applying the James construction gives a map 
\[\llnamedright{\Omega\Sigma X_{j}}{\Omega\Sigma t_{j}}{\Omega\Sigma(\bigvee_{i=1}^{m} X_{i})}.\] 
Multiplying the maps $\Omega\Sigma t_{j}$ together for $1\leq j\leq m$ gives a map 
\begin{equation} 
  \label{Psidef} 
    \Psi\colon\namedright{\prod_{i=1}^{m} \Omega\Sigma X_{i}}{}{\Omega\Sigma(\bigvee_{i=1}^{m} X_{i})}. 
\end{equation} 

As $\Psi$ is not $\Omega\Sigma\psi$ for some map $\psi$, it need not necessarily be the case 
that the decomposition of $\Sigma(\prod_{i=1}^{m}\Omega\Sigma X_{i})$ obtained by combining 
the natural decomposition of the suspension of a product with the $\phi$-decomposition of each 
$\Sigma\Omega\Sigma X_{i}$ is compatible with the $\phi$-decomposition of 
$\Sigma\Omega\Sigma(\bigvee_{i=1}^{m} X_{i})$. However, in Proposition~\ref{Jamescompat} 
we will show that a decomposition of $\Sigma(\prod_{i=1}^{m}\Omega\Sigma X_{i})$ 
may be chosen to be compatible with the $\phi$-decomposition of 
$\Sigma\Omega\Sigma(\bigvee_{i=1}^{m} X_{i})$. 

\begin{proposition} 
   \label{Jamescompat} 
   There is a homotopy equivalence 
   \[\namedright{\displaystyle\bigvee_{k=1}^{\infty}\ \bigvee_{1\leq i_{1}\leq\cdots\leq i_{k}\leq m} 
            \Sigma X_{i_{1}}\wedge\cdots\wedge X_{i_{k}}}{\varepsilon} 
            {\Sigma(\prod_{i=1}^{m}\Omega\Sigma X_{i})}\] 
    satisfying a homotopy commutative diagram 
    \[\diagram 
           \displaystyle\bigvee_{k=1}^{\infty}\ \bigvee_{1\leq i_{1}\leq\cdots\leq i_{k}\leq m} 
                     \Sigma X_{i_{1}}\wedge\cdots\wedge X_{i_{k}}\rto^-{I}\dto^{\varepsilon} 
                 & \bigvee_{k=1}^{\infty}\Sigma(\bigvee_{i=1}^{m} X_{i})^{\wedge k}\dto^{\phi} \\ 
            \Sigma(\prod_{i=1}^{m}\Omega\Sigma X_{i})\rto^-{\Sigma\Psi} 
                 & \Sigma\Omega\Sigma(\bigvee_{i=1}^{m} X_{i}) 
      \enddiagram\] 
    where $I$ is an inclusion of wedge summands. 
\end{proposition} 

\begin{proof} 
First consider the diagram 
\[\diagram 
      X_{i_{1}}\times\cdots\times X_{i_{k}}\rrto^-{t_{i_{1}}\times\cdots\times t_{i_{k}}}\dto^{E\times\cdots\times E} 
            & & (\bigvee_{i=1}^{m} X_{i})^{\times k}\dto^{E^{\times k}}\drto^{e_{k}} & \\ 
      \Omega\Sigma X_{i_{1}}\times\cdots\times\Omega\Sigma X_{i_{k}} 
                \rrto^-{\Omega\Sigma t_{i_{1}}\times\cdots\times\Omega\Sigma t_{i_{k}}} 
            & & \Omega\Sigma(\bigvee_{i=1}^{m} X_{i})^{\times k}\rto^-{m} & \Omega\Sigma(\bigvee_{i=1}^{m} X_{i}). 
  \enddiagram\] 
The left square clearly commutes and the right square commutes by definition of $e_{k}$. Now 
suspend and use the naturality of the map 
\(\namedright{\Sigma A\wedge B}{}{\Sigma(A\times B)}\) 
to obtain a homotopy commutative diagram 
\[\spreaddiagramcolumns{0.5pc}\diagram 
      \Sigma X_{i_{1}}\wedge\cdots\wedge X_{i_{k}}\rrto^-{\Sigma t_{i_{1}}\wedge\cdots\wedge t_{i_{k}}}\dto
            & & \Sigma (\bigvee_{i=1}^{m} X_{i})^{\wedge k}\dto\drto^{\phi_{k}} & \\ 
      \Sigma (\Omega\Sigma X_{i_{1}}\times\cdots\times\Omega\Sigma X_{i_{k}}) 
                \rrto^-{\Sigma(\Omega\Sigma t_{i_{1}}\times\cdots\times\Omega\Sigma t_{i_{k}})} 
            & & \Sigma\Omega\Sigma(\bigvee_{i=1}^{m} X_{i})^{\times k}\rto^-{\Sigma m} 
            & \Sigma \Omega\Sigma(\bigvee_{i=1}^{m} X_{i}). 
  \enddiagram\] 
Observe that the map $\Sigma t_{i_{1}}\wedge\cdots\wedge t_{i_{k}}$ is the inclusion of a wedge 
summand. Doing this for each $1\leq i_{1}\leq\cdots\leq i_{k}\leq m$ then gives a homotopy 
commutative diagram 
\[\diagram 
      \displaystyle\bigvee_{1\leq i_{1}\leq\cdots\leq i_{k}\leq m} \Sigma X_{i_{1}}\wedge\cdots\wedge X_{i_{k}} 
                \rrto^-{I_{k}}\dto^{\varepsilon_{k}}
            & & \Sigma (\bigvee_{i=1}^{m} X_{i})^{\wedge k}\dto\drto^{\phi_{k}} & \\ 
      \Sigma (\Omega\Sigma X_{1}\times\cdots\times\Omega\Sigma X_{m}) 
                \rrto^-{\Sigma(\Omega\Sigma t_{1}\times\cdots\times\Omega\Sigma t_{m})} 
            & & \Sigma\Omega\Sigma(\bigvee_{i=1}^{m} X_{i})^{\times k}\rto^-{\Sigma m} 
            & \Sigma \Omega\Sigma(\bigvee_{i=1}^{m} X_{i}). 
  \enddiagram\] 
where $I_{k}$ is an inclusion of wedge summands. Finally, assembling these diagrams for 
each $k\geq 1$ gives a homotopy commutative diagram 
\[\diagram 
      \displaystyle\bigvee_{k=1}^{\infty}\ \bigvee_{1\leq i_{1}\leq\cdots\leq i_{k}\leq m} 
               \Sigma X_{i_{1}}\wedge\cdots\wedge X_{i_{k}}\rrto^-{I}\dto^{\varepsilon}
            & & \bigvee_{k=1}^{\infty}\Sigma (\bigvee_{i=1}^{m} X_{i})^{\wedge k}\dto\drto^{\phi} & \\ 
      \Sigma (\Omega\Sigma X_{1}\times\cdots\times\Omega\Sigma X_{m}) 
                \rrto^-{\Sigma(\Omega\Sigma t_{1}\times\cdots\times\Omega\Sigma t_{m})} 
            & & \Sigma\Omega\Sigma(\bigvee_{i=1}^{m} X_{i})^{\times k}\rto^-{\Sigma m} 
            & \Sigma \Omega\Sigma(\bigvee_{i=1}^{m} X_{i}).  
  \enddiagram\] 
where $I$ is an inclusion of wedge summands. Observe that the bottom row is $\Sigma\Psi$. 

It remains to show that $\varepsilon$ is a homotopy equivalence. Take homology with 
field coefficiets. For $1\leq i\leq m$, let $V_{i}=\rhlgy{X_{i}}$. By the Bott-Samelson and 
Kunneth Theorems, there is an algebra isomorphism 
\[\hlgy{\Omega\Sigma X_{1}\times\cdots\times\Omega\Sigma X_{m}}\cong 
      T(V_{1})\otimes\cdots\otimes T(V_{m}).\] 
The submodule consisting of elements of tensor length $k$ is 
$\bigvee_{1\leq i_{1}\leq\cdots\leq i_{k}\leq m} V_{i_{1}}\otimes\cdots\otimes V_{i_{k}}$. 
Thus the previous isomorphism implies there is a module isomorphism  
\[\rhlgy{\Omega\Sigma X_{1}\times\cdots\times\Omega\Sigma X_{m}}\cong\bigvee_{k=1}^{\infty}\  
     \bigvee_{1\leq i_{1}\leq\cdots\leq i_{k}\leq m} V_{i_{1}}\otimes\cdots\otimes V_{i_{k}}.\] 
For a fixed sequence $(i_{1},\ldots,i_{k})$ with $1\leq i_{1}\leq\cdots\leq i_{k}\leq m$, 
the composite 
\[\lllnameddright{\Sigma X_{i_{1}}\wedge\cdots\wedge X_{i_{k}}}{} 
     {\Sigma(X_{i_{1}}\times\cdots\times X_{i_{k}})}{\Sigma(E\times\cdots\times E)} 
     {\Sigma\Omega\Sigma X_{i_{1}}\times\cdots\times\Omega\Sigma X_{i_{k}}}\]  
induces the inclusion of the submodule $\Sigma V_{i_{1}}\otimes\cdots\otimes V_{i_{k}}$ 
Therefore $\varepsilon_{k}$ induces the inclusion of the submodule 
$\bigvee_{1\leq i_{1}\leq\cdots\leq i_{k}\leq m} \Sigma V_{i_{1}}\otimes\cdots\otimes V_{i_{k}}$ 
into $\rhlgy{\Sigma(\Omega\Sigma X_{1}\times\cdots\times\Omega\Sigma X_{m}}$, 
implying that $\varepsilon$ induces an isomorphism in homology. As this is true for 
mod-$p$ homology for all primes $p$ and rational homology, $\varepsilon$ incudes 
an isomorphism in integral homology. Hence it is a homotopy equivalence by 
Whitehead's Theorem. 
\end{proof} 

Next is a variation on Proposition~\ref{Jamescompat} that involves half-smashes and 
a generalization of the map~$c$ from Section~\ref{sec:ad}. The maps $b_{k}$ in 
Section~\ref{sec:ad} may be defined more generally as follows. Let 
\[\overline{b}_{1}\colon\namedright{X_{1}\wedge\Sigma Y}{}{X_{1}\ltimes\Sigma Y}\] 
be the inclusion $i$. For $k\geq 2$, define 
\[\overline{b}_{k}\colon\namedright{X_{1}\wedge\cdots\wedge X_{k}\wedge\Sigma Y}{} 
      {(X_{1}\times\cdots\times X_{k})\ltimes\Sigma Y}\] 
recursively by the composite 
\[\llnamedright{X_{1}\wedge X_{2}\wedge\cdots\wedge X_{k}\wedge\Sigma Y} 
      {i}{X_{1}\ltimes(X_{2}\wedge\cdots X_{k}\wedge\Sigma Y)}\] 
\[\hspace{4cm}\stackrel{1\ltimes\overline{b}_{k-1}}{\llarrow} 
     \llnamedright{X_{1}\ltimes((X_{2}\times\cdots\times X_{k})\ltimes\Sigma Y)} 
        {\varphi}{(X_{1}\times X_{2}\times\cdots\times X_{k})\ltimes\Sigma Y}.\] 
Note that the naturality of $i$ and $\varphi$ in all variables implies that $\overline{b}_{k}$ is 
also natural in all variables. Note also that the map $b_{k}$ in Section~\ref{sec:ad} 
is given by taking each $X_{i}$ for $1\leq i\leq k$ equal to a common space $X$. 
Applying the naturality of $\overline{b}_{k}$ to the inclusions 
\(\namedright{X_{j}}{t_{j}}{\bigvee_{i=1}^{m} X_{i}}\) 
then immediately gives the following. 

\begin{lemma} 
   \label{2bdgrm} 
   Let $X_{1},\ldots,X_{m}$ and $Y$ be path-connected spaces. For any 
   $1\leq i_{1}\leq\cdots\leq i_{k}\leq m$ there is a homotopy commutative diagram 
   \[\diagram 
         X_{i_{1}}\wedge\ldots\wedge X_{i_{k}}\wedge\Sigma Y 
               \rrto^-{t_{i_{1}}\wedge\cdots\wedge t_{i_{k}}\wedge 1}\dto^{\overline{b}_{k}} 
             & & (\bigvee_{i=1}^{m} X_{i})^{\wedge k}\wedge\Sigma Y\dto^{b_{k}} \\ 
         (X_{i_{1}}\times\cdots\times X_{i_{k}})\ltimes\Sigma Y 
                \rrto^-{(t_{i_{1}}\times\cdots\times t_{i_{k}})\ltimes 1} 
              & & (\bigvee_{i=1}^{m} X_{i})^{\times k}\ltimes\Sigma Y. 
       \enddiagram\] 
\end{lemma} 
\vspace{-0.9cm}~$\qqed$\bigskip 

In what follows, the $k=0$ case of a smash product 
$X_{i_{1}}\wedge\cdots\wedge X_{i_{k}}\wedge\Sigma A$ refers to $\Sigma A$. 
By Lemma~\ref{cequiv}, there is a homotopy equivalence 
\[\namedright{\bigvee_{k=0}^{\infty}(\textstyle\bigvee_{i=1}^{m} X_{i})^{\wedge k}\wedge\Sigma A} 
      {c}{\Omega\Sigma(\bigvee_{i=1}^{m} X_{i})\ltimes\Sigma A}.\] 

\begin{lemma} 
   \label{cPsi} 
   There is a homotopy commutative diagram 
   \[\diagram 
           \displaystyle\bigvee_{k=0}^{\infty}\ \bigvee_{1\leq i_{1}\leq\cdots\leq i_{k}\leq m} 
                    (X_{i_{1}}\wedge\cdots\wedge X_{i_{k}})\wedge\Sigma A\rto^-{I}\dto^{\varepsilon'} 
              & \displaystyle\bigvee_{k=0}^{\infty}(\textstyle\bigvee_{i=1}^{m} X_{i})^{\wedge k}\wedge\Sigma A\dto^-{c} \\ 
           (\prod_{i=1}^{m}\Omega\Sigma X_{i})\ltimes\Sigma A\rto^-{\Psi\ltimes 1}     
              & \Omega\Sigma(\bigvee_{i=1}^{m} X_{i})\ltimes\Sigma A 
     \enddiagram\] 
   where $\varepsilon'$ is a homotopy equivalence and $I$ is an inclusion of wedge summands. 
\end{lemma} 

\begin{proof} 
The proof is similar to that for Proposition~\ref{Jamescompat}. It begins with the same 
first step, just half-smashed with $\Sigma A$. Consider the diagram 
\[\spreaddiagramcolumns{0.5pc}\diagram 
      (X_{i_{1}}\times\cdots\times X_{i_{k}})\ltimes\Sigma A 
                \rrto^-{(t_{i_{1}}\times\cdots\times t_{i_{k}})\ltimes 1}\dto^{(E\times\cdots\times E)\ltimes 1} 
            & & (\bigvee_{i=1}^{m} X_{i})^{\times k}\ltimes\Sigma A\dto^{E^{\times k}\ltimes 1}\drto^{e_{k}\ltimes 1} & \\ 
      (\Omega\Sigma X_{i_{1}}\times\cdots\times\Omega\Sigma X_{i_{k}})\ltimes\Sigma A 
                    \rrto^-{(\Omega\Sigma t_{i_{1}}\times\cdots\times\Omega\Sigma t_{i_{k}})\ltimes 1} 
            & & \Omega\Sigma(\bigvee_{i=1}^{m} X_{i})^{\times k}\ltimes\Sigma A\rto^-{m\ltimes 1} 
            & \Omega\Sigma(\bigvee_{i=1}^{m} X_{i})\ltimes\Sigma A. 
  \enddiagram\] 
The left square clearly commutes and the right square commutes by definition of $e_{k}$. 

Next, juxtapose the diagram above with that in Lemma~\ref{2bdgrm} (with $Y=A$) to obtain a homotopy 
commutative diagram 
\[\spreaddiagramcolumns{0.5pc}\diagram 
      (X_{i_{1}}\wedge\cdots\wedge X_{i_{k}})\wedge\Sigma A 
                \rrto^-{(t_{i_{1}}\wedge\cdots\wedge t_{i_{k}})\wedge 1}\dto 
            & & (\bigvee_{i=1}^{m} X_{i})^{\wedge k}\wedge\Sigma A\drto^{c_{k}}\dto & \\ 
      (\Omega\Sigma X_{i_{1}}\times\cdots\times\Omega\Sigma X_{i_{k}})\ltimes\Sigma A 
                    \rrto^-{(\Omega\Sigma t_{i_{1}}\times\cdots\times\Omega\Sigma t_{i_{k}})\ltimes 1} 
            & & \Omega\Sigma(\bigvee_{i=1}^{m} X_{i})^{\times k}\ltimes\Sigma A\rto^-{m\ltimes 1} 
            & \Omega\Sigma(\bigvee_{i=1}^{m} X_{i})\ltimes\Sigma A.  
  \enddiagram\] 
Observe that $t_{i_{1}}\wedge\cdots\wedge t_{i_{k}}\wedge 1$ is the inclusion of a wedge summand. 
As in Proposition~\ref{Jamescompat}, a similar diagram exists for each 
$1\leq i_{1}\leq\cdots\leq i_{k}\leq m$, and then all such diagrams for $k\geq 1$ may be assembled 
to give a homotopy commutative diagram 
\small\[\spreaddiagramcolumns{-0.5pc}\diagram 
      \displaystyle\bigvee_{k=1}^{\infty}\ \bigvee_{1\leq i_{1}\leq\cdots\leq i_{k}\leq m} 
               \Sigma X_{i_{1}}\wedge\cdots\wedge X_{i_{k}}\wedge\Sigma A\rrto^-{I}\dto^{\varepsilon'}
            & & \bigvee_{k=1}^{\infty}\Sigma (\bigvee_{i=1}^{m} X_{i})^{\wedge k}\wedge\Sigma A\dto\drto^{c} & \\ 
      (\Omega\Sigma X_{1}\times\cdots\times\Omega\Sigma X_{m})\ltimes\Sigma A 
                 \rrto^-{(\Omega\Sigma t_{1}\times\cdots\times\Omega\Sigma t_{m})\ltimes 1} 
            & & \Omega\Sigma(\bigvee_{i=1}^{m} X_{i})^{\times k}\ltimes\Sigma A\rto^-{m\ltimes 1} 
            & \Omega\Sigma(\bigvee_{i=1}^{m} X_{i})\ltimes\Sigma A  
  \enddiagram\]\normalsize
where $I$ is an inclusion of wedge summands. Observe that the bottom row is $\Psi\ltimes 1$ 
so the homotopy commutativity of the diagram implies that $c\circ I\simeq(\Psi\ltimes 1)\circ\varepsilon'$.  
An argument as in Proposition~\ref{Jamescompat} shows that $\varepsilon'$ is a homotopy 
equivalence. 
\end{proof} 
 
Recall from the setup at the beginning of the section that there is a composite 
\(\nameddright{\bigvee_{i=1}^{m}\Sigma X_{i}}{v}{M}{w}{\prod_{i=1}^{m}\Sigma X_{i}}\) 
that is homotopic to the inclusion of the wedge into the product. For $1\leq k\leq m$, let $v_{k}$ 
be the composite 
\[v_{k}\colon\nameddright{\Sigma X_{k}}{\Sigma t_{k}}{\bigvee_{i=1}^{m}\Sigma X_{i}}{v}{M}.\] 
Recall as well that two maps $f$ and $g$ are congruent if $\Sigma f\simeq\Sigma g$, implying 
that $f_{\ast}=g_{\ast}$. 

\begin{theorem} 
   \label{prodextend} 
   There is a homotopy cofibration 
   \[\nameddright{\displaystyle\bigvee_{k=0}^{\infty}\ \bigvee_{1\leq i_{1}\leq\cdots\leq i_{k}\leq m} 
              (X_{i_{1}}\wedge\cdots\wedge X_{i_{k}})\wedge\Sigma A}{\zeta}{E}{}{E'}\] 
    where the map $\zeta$ is congruent to a map $\zeta'$ satisfying a homotopy commutative diagram 
    \[\diagram 
            \displaystyle\bigvee_{k=0}^{\infty}\ \bigvee_{1\leq i_{1}\leq\cdots\leq i_{k}\leq m} 
              (X_{i_{1}}\wedge\cdots\wedge X_{i_{k}})\wedge\Sigma A\rto^-{\zeta'} 
                    \drto_-(0.6){\bigvee_{k=0}^{\infty}\ \bigvee_{1\leq i_{1}\leq\cdots\leq i_{k}\leq m} 
                    [v_{i_{1}},[v_{i_{2}},[\cdots [v_{i_{k}},f]]\cdots ]\hspace{2cm}} 
                 & E\dto^{p} \\ 
            & M.   
        \enddiagram\] 
\end{theorem} 

\begin{proof} 
The proof is broken into steps. 
\medskip 

\noindent 
\textit{Step 1: setting up}. 
After looping, the inclusion 
\(\namedright{\bigvee_{i=1}^{m}\Sigma X_{i}}{}{\prod_{i=1}^{m}\Sigma X_{i}}\) 
has a right homotopy inverse. A specific choice of a right homotopy 
inverse is given by the map $\Psi$ defined in~(\ref{Psidef}). Let $s$ be the composite 
\[s\colon\nameddright{\prod_{i=1}^{m}\Omega\Sigma X_{i}}{\Psi}{\Omega\Sigma(\bigvee_{i=1}^{m} X_{i})} 
      {\Omega v}{\Omega M}.\] 
Then as $w\circ v$ is homotopic to the inclusion of the wedge into the product, $s$ is a right 
homotopy inverse for $\Omega w$. 

As an intermediate stage, define the space $\overline{E}$ and the map $\overline{p}$ by the 
homotopy fibration 
\[\nameddright{\overline{E}}{\overline{p}}{(\bigvee_{i=1}^{m}\Sigma X_{i})\vee\Sigma A} 
      {q_{1}}{\bigvee_{i=1}^{m}\Sigma X_{i}}\] 
where $q_{1}$ is the pinch map. By Example~\ref{specialGTadexample} there is a lift 
\[\overline{g}\colon\namedright{\Sigma A}{}{\overline{E}}\] 
of the inclusion 
\(\namedright{\Sigma A}{i_{2}}{(\bigvee_{i=1}^{m}\Sigma X_{i})\vee\Sigma A}\) 
and a homotopy equivalence 
\[\llnamedright{\Omega\Sigma(\bigvee_{i=1}^{m}\Sigma X_{i})\ltimes\Sigma A} 
      {\overline{a}\circ(1\ltimes\overline{g})}{\overline{E}}.\] 
Since $w$ extends the inclusion of the wedge into the product and $w\circ f$ is null 
homotopic, there is a homotopy commutative square 
\[\diagram 
       (\bigvee_{i=1}^{m}\Sigma X_{i})\vee\Sigma A\rto^-{v\perp f}\dto^{q_{1}} & M\dto^{w} \\ 
       \bigvee_{i=1}^{m}\Sigma X_{i}\rto & \prod_{i=1}^{m}\Sigma X_{i}. 
  \enddiagram\] 
Let 
\(\alpha\colon\namedright{\overline{E}}{}{E}\) 
be the induced map of fibres and let $g$ be the composite 
\[g\colon\nameddright{\Sigma A}{\overline{g}}{\overline{E}}{\alpha}{E}.\] 
Notice that as $\overline{g}$ is a lift of the identity map on $\Sigma A$ to $\overline{E}$, 
the map $g$ is a lift of $f$ to $E$. Further, we claim that the composite 
\(\nameddright{\Sigma A}{g}{E}{}{E'}\) 
is null homotopic. Since $g$ lifts $f$, the composite 
\(\namedddright{\Sigma A}{g}{E}{}{E'}{p'}{M'}\) 
is homotopic to 
\(\nameddright{\Sigma A}{f}{M}{}{M'}\) 
by~(\ref{MEdgrm}), which is null homotopic since~$M'$ is the homotopy cofibre of $f$. 
Thus 
\(\nameddright{\Sigma A}{g}{E}{}{E'}\) 
lifts to the homotopy fibre of $p'$. But by~(\ref{MEdgrm}), as $\Omega w$ has a right 
homotopy inverse, so does $\Omega w'$, implying that the connecting map for the 
homotopy fibration 
\(\nameddright{E'}{p'}{M'}{w'}{\prod_{i=1}^{m}\Sigma X_{i}}\) 
is null homotopic. Hence the lift of 
\(\nameddright{\Sigma A}{g}{E}{}{E'}\) 
to the homotopy fibre of $p'$ must be null homotopic, implying that the composite 
\(\nameddright{\Sigma A}{g}{E}{}{E'}\) 
is null homotopic. 

With this choice of $g$ and the existence of a right homotopy inverse for $\Omega w$, 
Theorem~\ref{GTcofib} implies that there is a homotopy cofibration 
\begin{equation} 
  \label{thetacof} 
  \nameddright{(\prod_{i=1}^{m}\Omega\Sigma X_{i})\ltimes\Sigma A}{\theta}{E}{}{E'} 
\end{equation}  
where, by definition, $\theta$ is the composite 
\[\nameddright{(\prod_{i=1}^{m}\Omega\Sigma X_{i})\ltimes\Sigma A}{s\ltimes g} 
      {\Omega M\ltimes E}{\Gamma}{E}.\] 
\medskip 
 
\noindent 
\textit{Step 2: the map $\zeta$}. 
Consider the diagram 
\small 
\begin{equation} 
  \label{IGamma} 
  \spreaddiagramcolumns{-0.8pc}\diagram 
      \displaystyle\bigvee_{k=0}^{\infty}\ \bigvee_{1\leq i_{1}\leq\cdots\leq i_{k}\leq m} 
              (X_{i_{1}}\wedge\cdots\wedge X_{i_{k}})\wedge\Sigma A\rto^-{\varepsilon'}\dto^{I} 
           & (\prod_{i=1}^{m}\Omega\Sigma X_{i})\ltimes\Sigma A\dto^-{\Psi\ltimes 1}\drto^{s\ltimes g} & & \\ 
      \bigvee_{k=0}^{\infty} (\bigvee_{i=1}^{m} X_{i})^{\wedge k}\wedge\Sigma A\rto^-{c}
           & \Omega\Sigma(\bigvee_{i=1}^{m} X_{i})\ltimes\Sigma A\rto^-{\Omega v\ltimes g}    
           & \Omega M\ltimes E\rto^-{\Gamma} & E.
   \enddiagram 
 \end{equation} 
 \normalsize 
The left square homotopy commutes by Lemma~\ref{cPsi} and the triangle 
homotopy commutes by definition of $s$. In the upper direction around the diagram,  
$\Gamma\circ(s\ltimes g)$ is the definition of $\theta$ and, by Lemma~\ref{cPsi}, $\varepsilon'$ 
is a homotopy equivalence. So if $\zeta=\theta\circ\varepsilon'$ then by~(\ref{thetacof}) there 
is a homotopy cofibration 
\[\nameddright{\displaystyle\bigvee_{k=0}^{\infty}\ \bigvee_{1\leq i_{1}\leq\cdots\leq i_{k}\leq m} 
              (X_{i_{1}}\wedge\cdots\wedge X_{i_{k}})\wedge\Sigma A}{\zeta}{E}{}{E'}.\]  
\medskip 

\noindent 
\textit{Step 3: the map $\zeta'$ and the congruence with $\zeta$}. 
In the lower row of~(\ref{IGamma}) inserting the homotopy equivalence 
\[\beta\colon\llnamedright{\Omega\Sigma(\bigvee_{i=1}^{m}\Sigma X_{i})\ltimes\Sigma A} 
      {\overline{a}\circ(1\ltimes\overline{g})}{\overline{E}}\]  
and its inverse gives 
$\Gamma\circ(\Omega v\ltimes g)\circ c\simeq\Gamma\circ(\Omega v\ltimes g)\circ\beta^{-1}\circ 
       \overline{a}\circ(1\ltimes\overline{g})\circ c$. 
Notice that the composite $\overline{a}\circ(1\ltimes\overline{g})\circ c$ is the map $d$ 
from Section~\ref{sec:ad}. By Theorem~\ref{dbard}, $d$ is congruent to a map $\mathfrak{d}$  
that satisfies a homotopy commutative diagram 
\begin{equation} 
  \label{adWA} 
  \diagram 
      \bigvee_{k=0}^{\infty} (\bigvee_{i=1}^{m} X_{i})^{\wedge k}\wedge\Sigma A 
             \rto^-{\mathfrak{d}}\drto_{\bigvee_{k=1}^{\infty} ad^{k}(i_{W})(i_{\Sigma A})\ \ } 
         & \overline{E}\dto^{\overline{p}} \\ 
      & (\bigvee_{i=1}^{m}\Sigma X_{i})\vee\Sigma A 
  \enddiagram 
\end{equation} 
where $i_{W}$ and $i_{\Sigma A}$ are the inclusions of $\bigvee_{i=1}^{m}\Sigma X_{i}$ 
and $\Sigma A$ respectively into $(\bigvee_{i=1}^{m}\Sigma X_{i})\vee\Sigma A$. 
Let 
\[\zeta'=\Gamma\circ(\Omega v\ltimes g)\circ\beta^{-1}\circ\mathfrak{d}\circ I.\] 
The congruence between $\mathfrak{d}$ and $d=\overline{a}\circ(1\ltimes\overline{g})\circ c$ 
implies that $\zeta'$ is congruent to 
$\Gamma\circ(\Omega v\ltimes g)\circ\beta^{-1}\circ d\circ I$. 
As $\beta=\overline{a}\circ(1\ltimes\overline{g})$, we obtain a congruence between $\zeta'$ 
and $\Gamma\circ(\Omega v\ltimes g)\circ c\circ I$. The latter is the lower direction around the 
diagram~(\ref{IGamma}), and so is homotopic to the upper direction around that diagram, 
which is $\theta\circ\varepsilon'=\zeta$. Hence $\zeta'$ is congruent to $\zeta$. 
\medskip 

\noindent 
\textit{Step 4: identifying Whitehead products}. 
Finally, consider the diagram 
\[\spreaddiagramcolumns{-0.5pc}\diagram  
      \bigvee_{k=0}^{\infty} (\bigvee_{i=1}^{m} X_{i})^{\wedge k}\wedge\Sigma A 
             \rto^-{\mathfrak{d}}\drto_{\bigvee_{k=1}^{\infty} ad^{k}(i_{W})(i_{\Sigma A})\ \ } 
          & \overline{E}\rto^-{\beta^{-1}}\dto^{\overline{p}} 
          & \Omega\Sigma(\bigvee_{i=1}^{m} X_{i})\ltimes\Sigma A\rto^-{\Omega v\ltimes g}\dto 
          & \Omega M\ltimes E\rto^-{\Gamma}\dto & E\dto^{p} \\ 
      & (\bigvee_{i=1}^{m}\Sigma X_{i})\vee\Sigma A\rdouble 
          & (\bigvee_{i=1}^{m}\Sigma X_{i})\vee\Sigma A\rto^-{v\vee g} 
          & M\vee E\rto^-{1\vee p} & M. 
  \enddiagram\] 
The left triangle homotopy commutes by~(\ref{adWA}), the left of centre square homotopy commutes 
by the homotopy equivalence $\beta$, the right of centre square homotopy commutes by naturality 
and the right square commutes by the definition of $\Gamma$ in Section~\ref{sec:background}. 
The upper direction around the diagram, precomposed with $I$, is the definition of $\zeta'$. 
The lower direction around the diagram, precomposed with $I$, behaves as follows. 
Observe that the restriction of $I$ to the wedge summand 
$X_{i_{1}}\wedge\cdots\wedge X_{i_{k}}\wedge\Sigma A$ is the inclusion 
$t_{i_{1}}\wedge\cdots\wedge t_{i_{k}}\wedge 1$. Thus the restriction of $ad^{k}(i_{W})(i_{\Sigma A})$ 
to $X_{i_{1}}\wedge\cdots\wedge X_{i_{k}}\wedge\Sigma A$ is 
$[\Sigma t_{i_{1}},[\Sigma t_{i_{2}},[\cdots [\Sigma t_{i_{k}},i_{\Sigma A}]]\cdots ]$. 
The naturality of the Whitehead product,  
the definition of $v_{k}$ as $v\circ\Sigma\iota_{k}$, and the fact that $p\circ g\simeq f$ imply that 
\[(v\vee(p\circ g))\circ [\Sigma\iota_{i_{1}},[\Sigma\iota_{i_{2}},[\cdots [\Sigma\iota_{i_{k}},i_{\Sigma A}]]\cdots ] 
      \simeq [v_{i_{1}},[v_{i_{2}},[\cdots [v_{i_{k}},f]]\cdots ].\] 
Thus the lower direction around the diagram is the wedge sum of the iterated Whitehead products  
$[v_{i_{1}},[v_{i_{2}},[\cdots [v_{i_{k}},f]]\cdots ]$ for all $1\leq i_{1}\leq\cdots\leq i_{k}\leq m$ 
and all $k\geq 0$. Hence 
\[p\circ\zeta'\simeq\bigvee_{k=0}^{\infty}\ \bigvee_{1\leq i_{1}\leq\cdots\leq i_{k}\leq m} 
                    [v_{i_{1}},[v_{i_{2}},[\cdots [v_{i_{k}},f]]\cdots ]\] 
as asserted. 
\end{proof}

\newpage 

\section{Polyhedral products and Whitehead products} 
\label{sec:polyprod} 

The main application of Theorem~\ref{prodextend} is to polyhedral products. We 
first recall and formalize the definition in the Introduction. Let $K$ be an abstract simplicial 
complex on the vertex set $[m]=\{1,2,\ldots,m\}$. That is, $K$ is a collection of 
subsets $\sigma\subseteq [m]$ such that for any $\sigma\in K$ all subsets of $\sigma$ 
also belong to $K$. We will usually refer to $K$ as a simplicial complex rather than an 
abstract simplicial complex. A subset $\sigma\in K$ is a \emph{simplex} or \emph{face} 
of $K$. The emptyset $\emptyset$ is assumed to belong to $K$. For $1\leq i\leq m$,
let $(X_{i},A_{i})$ be a pair of pointed $CW$-complexes, where $A_{i}$ is 
a pointed subspace of~$X_{i}$. Let $\uxa=\{(X_{i},A_{i})\}_{i=1}^{m}$ be 
the sequence of $CW$-pairs. For each face $\sigma\in K$, let 
$\uxa^{\sigma}$ be the subspace of $\prod_{i=1}^{m} X_{i}$ defined by
\[\uxa^{\sigma}=\prod_{i=1}^{m} Y_{i}\qquad
       \mbox{where}\qquad Y_{i}=\left\{\begin{array}{ll}
                                             X_{i} & \mbox{if $i\in\sigma$} \\
                                             A_{i} & \mbox{if $i\notin\sigma$}.
                                       \end{array}\right.\]
The \emph{polyhedral product} determined by \uxa\ and $K$ is
\[\uxa^{K}=\bigcup_{\sigma\in K}\uxa^{\sigma}\subseteq\prod_{i=1}^{m} X_{i}.\] 
For example, suppose each $A_{i}$ is a point. If $K$ is a disjoint union
of $m$ points then $(\underline{X},\underline{\ast})^{K}$ is the wedge
$X_{1}\vee\cdots\vee X_{m}$, and if $K$ is the standard $(m-1)$-simplex
then $(\underline{X},\underline{\ast})^{K}$ is the product
$X_{1}\times\cdots\times X_{m}$. 

We aim to apply Theorem~\ref{prodextend} in the context of a homotopy cofibration 
\(\nameddright{\Sigma A}{}{\sux^{K}}{}{\sux^{\overline{K}}}\); 
this will be done in Proposition~\ref{Acofib}.  To prepare some definitions and preliminary 
results are needed. 

The \emph{boundary} of a simplex $\sigma$, 
written $\partial\sigma$, is the simplicial complex consisting of all the proper subsets 
of~$\sigma$. A simplex $\sigma$ is a (minimal) \emph{missing face} of $K$ if $\sigma\notin K$ 
but $\partial\sigma\subseteq K$. The geometric realization of $K$ is written $\vert K\vert$. 
Note that if $\sigma$ is a face of $K$ with $k$ elements then $\vert\sigma\vert\cong\Delta^{k-1}$, 
and $\vert\partial\sigma\vert\cong\partial\Delta^{k-1}$. The \emph{dimension} of $K$, 
written $\dm(\vert K\vert)$, is the dimension of the geometric realization~$\vert K\vert$. 

Given a simplicial complex $K$ on the vertex set $[m]$, let $\mathcal{S}=\{\sigma_{1},\ldots,\sigma_{r}\}$ 
be a subset of the set of missing faces of $K$. Define a new simplicial complex $\overline{K}$ by 
\[\overline{K}=K\cup\mathcal{S}.\] 
In terms of geometric realizations, $\vert\overline{K}\vert$ is obtained from $\vert K\vert$ 
by taking the missing faces indexed by~$\mathcal{S}$ and gluing them to $\vert K\vert$ along 
their boundaries. The naturality of the polyhedral product implies that the simplicial 
inclusion 
\(\namedright{K}{}{\overline{K}}\) 
induces a map 
\(\namedright{\uxa^{K}}{}{\uxa^{\overline{K}}}\). 

We now specialize to pairs of the form $(X_{i},\ast)$ in order to better identify 
certain spaces. By definition of the polyhedral product we have 
\[\ux^{\Delta^{m-1}}=\prod_{i=1}^{m} X_{i}.\] 
The \emph{fat wedge} is the subspace of $\prod_{i=1}^{m} X_{i}$ defined by 
\[FW(X_{1},\ldots,X_{m})=\{(x_{1},\ldots,x_{m})\in\prod_{i=1}^{m} X_{i}\mid 
      \mbox{at least one $x_{i}$ is $\ast$}\}.\] 
The definition of the polyhedral product implies that 
\[\ux^{\partial\Delta^{m-1}}=FW(X_{1},\ldots, X_{m}).\] 
Thus if $\sigma=(i_{1},\ldots,i_{k})\subseteq [m]$ then 
\[\ux^{\sigma}=\prod_{j=1}^{k} X_{i_{j}}\qquad\mbox{and}\qquad 
     \ux^{\partial\sigma}=FW(X_{i_{1}},\ldots, X_{i_{k}}).\] 
Therefore, in our case, for each missing face $\sigma=(i_{1},\ldots,i_{k})\in\mathcal{S}$ 
there is a cofibration 
 \begin{equation} 
   \label{FWcofib} 
   \nameddright{FW(X_{i_{1}},\ldots,X_{i_{k}})}{}{\prod_{j=1}^{k} X_{i_{j}}}{} 
        {X_{i_{1}}\wedge\cdots\wedge X_{i_{k}}}. 
\end{equation} 
We show that an analogue is true for the map of polyhedral products 
\(\namedright{\ux^{K}}{}{\ux^{\overline{K}}}\). 

\begin{remark} 
\label{faceremark} 
It is worth pointing out in what follows that when we write 
$\bigvee_{\sigma\in\mathcal{S}} X_{i_{1}}\wedge\cdots\wedge X_{i_{k}}$ 
we mean $\sigma=(i_{1},\ldots,i_{k})$ and it is understood that the number of 
vertices $k$ may be different for distinct missing faces in $\mathcal{S}$. 
\end{remark} 
 
 \begin{lemma} 
   \label{preLcofib} 
   Suppose that for $1\leq i\leq m$ each space $X_{i}$ is path-connected and each 
   missing face in $\mathcal{S}$ has at least two vertices. Then there is a homotopy cofibration 
   \[\nameddright{\ux^{K}}{}{\ux^{\overline{K}}}{} 
          {\bigvee_{\sigma\in\mathcal{S}} X_{i_{1}}\wedge\cdots\wedge X_{i_{k}}}.\] 
\end{lemma} 

\begin{proof} 
Define the space $C$ by the cofibration 
\[\nameddright{\ux^{K}}{}{\ux^{\overline{K}}}{}{C}.\] 
In general, if $L$ is a simplicial complex on the vertex set $[m]$ then by~\cite{BBCG} 
$\Sigma\ux^{L}$ is homotopy equivalent to 
$\bigvee_{\tau\in L}\Sigma X_{i_{1}}\wedge\cdots\wedge X_{i_{\ell}}$ 
where $\tau=(i_{1},\ldots,i_{\ell})$, and this decomposition is natural with 
respect to simplicial maps 
\(\namedright{L}{}{L'}\). 
In our case, as $\overline{K}$ consists of all the faces of $K$ together with 
the missing faces indexed by $\mathcal{S}$, we obtain 
$\Sigma C\simeq\bigvee_{\sigma\in\mathcal{S}}\Sigma X_{i_{1}}\wedge\cdots\wedge X_{i_{k}}$. 
We claim that this decomposition for $\Sigma C$ desuspends. 

Fix $\sigma=(i_{1},\ldots,i_{k})\in\mathcal{S}$. Consider the diagram 
\[\diagram 
      FW(X_{i_{1}},\ldots,X_{i_{k}})\rto\dto & \prod_{j=1}^{k} X_{i_{j}}\rto\dto 
           & X_{i_{1}}\wedge\cdots\wedge X_{i_{k}}\dto^{g_{\sigma}} \\ 
      \ux^{K}\rto & \ux^{\overline{K}}\rto & C 
  \enddiagram\]  
where the map $g_{\sigma}$ will be defined momentarily. Since $\sigma$ is a missing face 
for $K$ but is a face of $\overline{K}$, the full subcomplexes of $\uxa^{K}$ and~$\uxa^{\overline{K}}$ 
on the vertex set $\{i_{1},\ldots,i_{k}\}$ are $FW(X_{i_{1}},\ldots,X_{i_{k}})$ and $\prod_{j=1}^{k} X_{i_{j}}$ 
respectively. Therefore, by the naturality of the polyhedral product with respect to simplicial 
maps, the left square above commutes. This induces a map of cofibres, which 
gives the right square and defines~$g_{\sigma}$. Notice that the decomposition 
of $\Sigma C$ implies that $\Sigma g_{\sigma}$ is the inclusion of a wedge summand. 
Thus if 
\[g\colon\namedright{\bigvee_{\sigma\in\mathcal{S}} X_{i_{1}}\wedge\cdots\wedge X_{i_{k}}} 
       {}{C}\] 
is the wedge sum of the maps $g_{\sigma}$ for all $\sigma\in\mathcal{S}$, then $\Sigma g$ 
is a homotopy equivalence. This implies that~$g_{\ast}$ induces an isomorphism in homology. 
As each space $X_{i}$ is path-connected and we assume that each missing face in $\mathcal{S}$ 
has at least two vertices, the spaces $X_{i_{1}}\wedge\cdots\wedge X_{i_{k}}$ are 
simply-connected. Hence, by Whitehead's Theorem, $g_{\ast}$ inducing an isomorphism 
in homology implies that $g$ is a homotopy equivalence. 
 \end{proof} 
 
 \begin{remark} 
 \label{preLcofibremark} 
 A useful piece of information to record from the proof of Lemma~\ref{preLcofib} 
 is that if $\sigma\in\mathcal{S}$ then the inclusion of $\sigma$ into $\overline{K}$ 
 induces a homotopy cofibration diagram 
 \[\diagram 
      FW(X_{i_{1}},\ldots,X_{i_{k}})\rto\dto & \prod_{j=1}^{k} X_{i_{j}}\rto\dto 
           & X_{i_{1}}\wedge\cdots\wedge X_{i_{k}}\dto^{g_{\sigma}} \\ 
      \ux^{K}\rto & \ux^{\overline{K}}\rto 
           & \bigvee_{\sigma\in\mathcal{S}} X_{i_{1}}\wedge\cdots\wedge X_{i_{k}}  
  \enddiagram\]  
where $g_{\sigma}$ is the inclusion of a wedge summand. 
 \end{remark} 
 
We now specialize further to pairs of the form $(\Sigma X_{i},\ast)$ in order 
to turn the cofibration in Lemma~\ref{preLcofib} one step to the left.    
The (reduced) \emph{join} of two pointed spaces $A$ and $B$ is 
the quotient space $A\ast B=(A\times I\times B)/\sim$ where $I=[0,1]$ is the 
unit interval and the defining relations are given by $(a,1,b)\sim (a',1,b)$, 
$(a,0,b)\sim (a,0,b')$ and $(\ast,t,\ast)\sim(\ast,0,\ast)$ for all 
$a,a'\in A$, $b,b'\in B$ and $t\in I$. There is a well known homotopy equivalence 
$A\ast B\simeq\Sigma A\wedge B$. 

In the case of pairs $(\Sigma X_{i},\ast)$, for each missing face $\sigma=(i_{1},\ldots,i_{k})\in\mathcal{S}$ 
there is a homotopy cofibration 
\[\nameddright{X_{i_{1}}\ast\cdots\ast X_{i_{k}}}{}{FW(\Sigma X_{i_{1}},\ldots,\Sigma X_{i_{k}})} 
      {}{\prod_{j=1}^{k}\Sigma X_{i_{j}}}\] 
that induces the cofibration in~(\ref{FWcofib}). Let 
\[f\colon\namedright{\bigvee_{\sigma\in\mathcal{S}} X_{i_{1}}\ast\cdots\ast X_{i_{k}}} 
         {}{\sux^{K}}\] 
be the wedge sum of the composites 
\(\nameddright{X_{i_{1}}\ast\cdots\ast X_{i_{k}}}{}{FW(\Sigma X_{i_{1}},\ldots,\Sigma X_{i_{k}})} 
       {}{\sux^{K}}\) 
for all $\sigma=(i_{1},\ldots,i_{k})\in\mathcal{S}$.  
       
\begin{lemma} 
   \label{desusppreLcofib} 
   Suppose that each missing face in $\mathcal{S}$ has at least two vertices. Then there is 
   a homotopy cofibration        
    \[\nameddright{\bigvee_{\sigma\in\mathcal{S}} X_{i_{1}}\ast\cdots\ast X_{i_{k}}} 
      {f}{\sux^{K}}{}{\sux^{\overline{K}}}\] 
    that induces the homotopy cofibration in Lemma~\ref{preLcofib}. 
\end{lemma} 

\begin{proof} 
In general, as $\overline{K}=K\cup S$, the definition of the polyhedral product implies 
that there is a pushout 
\begin{equation} 
  \label{bigpo} 
  \diagram 
      \bigcup_{\sigma\in\mathcal{S}} FW(X_{i_{1}},\ldots,X_{i_{k}})\rto\dto 
            & \bigcup_{\sigma\in\mathcal{S}}\left(\prod_{j=1}^{k} X_{i_{j}}\right)\dto \\ 
      \ux^{K}\rto & \ux^{\overline{K}}.  
  \enddiagram 
\end{equation}  
By Lemma~\ref{preLcofib}, the homotopy cofibre along the bottom row is 
$\bigvee_{\sigma\in\mathcal{S}} X_{i_{1}}\wedge\cdots\wedge X_{i_{k}}$.   
The fact that~(\ref{bigpo}) is a pushout implies that the cofibre of the top row 
is also $\bigvee_{\sigma\in\mathcal{S}} X_{i_{1}}\wedge\cdots\wedge X_{i_{k}}$. 
By Remark~\ref{preLcofibremark}, the restriction of~(\ref{bigpo}) to 
\(\namedright{FW(X_{i_{1}},\ldots, X_{i_{k}})}{}{\prod_{j=1}^{k} X_{i_{j}}}\) 
corresponding to a fixed~$\sigma$ induces the inclusion of the wedge summand 
$X_{i_{1}}\wedge\cdots\wedge X_{i_{k}}$ into the cofibre. In our case each such map 
\(\namedright{FW(\Sigma X_{i_{1}},\ldots,\Sigma X_{i_{k}})}{}{\prod_{j=1}^{k}\Sigma X_{i_{j}}}\) 
is induced by a map 
\(\namedright{X_{i_{1}}\ast\cdots\ast X_{i_{k}}}{}{FW(\Sigma X_{i_{1}},\ldots,\Sigma X_{i_{k}})}\). 
Therefore there is a homotopy cofibration sequence 
\[\namedddright{\bigvee_{\sigma\in S} X_{i_{1}}\ast\cdots\ast X_{i_{k}}}{} 
     {\bigcup_{\sigma\in S} FW(\Sigma X_{i_{1}},\ldots,\Sigma X_{i_{k}})}{} 
     {\bigcup_{\sigma\in S}\bigg(\prod_{j=1}^{k}\Sigma X_{i_{j}}\bigg)} 
     {}{\bigvee_{\sigma\in S}\Sigma X_{i_{1}}\wedge\cdots\wedge\Sigma X_{i_{k}}}.\] 
Hence, as~(\ref{bigpo}) is a pushout, the definition of $f$ implies that there is a 
homotopy cofibration 
\[\nameddright{\bigvee_{\sigma\in\mathcal{S}} X_{i_{1}}\ast\cdots\ast X_{i_{k}}} 
   {f}{\sux^{K}}{}{\sux^{\overline{K}}}.\] 
\end{proof} 

\begin{remark} 
Lemma~\ref{desusppreLcofib} is also proved in~\cite[Theorem 5.1 and Remark 5.2]{IK2} 
as a consequence of a grand organizational scheme for polyhedral products called the 
fat wedge filtration. Our formulation is more elementary as the focus is only on what is 
needed for Lemma~\ref{desusppreLcofib}. 
\end{remark} 

Observe that 
$X_{i_{1}}\ast\cdots\ast X_{i_{k}}\simeq\Sigma^{k-1} X_{i_{1}}\wedge\cdots\wedge X_{i_{k}}$. 
As we assume each missing face $\sigma=(i_{1},\ldots,i_{k})\in\mathcal{S}$ has at least 
two vertices, we have $k\geq 2$ so $\Sigma^{k-1} X_{i_{1}}\wedge\cdots\wedge X_{i_{k}}$ 
is a suspension. Let 
\[A=\bigvee_{\sigma\in\mathcal{S}} \Sigma^{k-2} X_{i_{1}}\wedge\cdots\wedge X_{i_{k}}.\] 
As in Remark~\ref{faceremark}, note that the number $k$ depends on $\sigma$ and 
may be different for distinct elements of $\mathcal{S}$. The homotopy cofibration in 
Lemma~\ref{desusppreLcofib} may now be rewritten as follows. 

\begin{proposition} 
   \label{Acofib} 
   Let $K$ be a simplicial complex on the vertex set $[m]$, let $\mathcal{S}$ be a subset 
   of the missing faces of $K$, and let $\overline{K}=K\cup\mathcal{S}$. Then  
   there is a homotopy cofibration  
   \[\nameddright{\Sigma A}{}{\sux^{K}}{}{\sux^{\overline{K}}}.\] 
\end{proposition} 
\vspace{-1cm}~$\qqed$\bigskip 

The point of Proposition~\ref{Acofib} is to put us in a position to apply Theorem~\ref{prodextend}. 
Let $K$ be a simplicial complex on the vertex set $[m]$, let $\mathcal{S}$ be a subset of 
the missing faces of $K$ and let $\overline{K}=K\cup\mathcal{S}$. Then there is a homotopy 
fibration diagram 
\[\diagram 
        E\rto\dto^{p} & \overline{E}\dto^{\overline{p}} \\ 
        \sux^{K}\rto\dto^{w} & \sux^{\overline{K}}\dto^{\overline{w}} \\ 
        \prod_{i=1}^{m}\Sigma X_{i}\rdouble & \prod_{i=1}^{m}\Sigma X_{i}  
  \enddiagram\] 
where $w$ and $\overline{w}$ are inclusions. By~\cite{DS}, there is are homotopy equivalences 
\[E\simeq\sclxx^{K}\qquad\overline{E}\simeq\sclxx^{\overline{K}}\] 
and, under these equivalences, the maps $p$ and $\overline{p}$ become maps of polyhedral products 
induced by appropriate maps of pairs of spaces. The inclusion of the vertex set into $K$ 
induces a map of polyhedral products 
\[v\colon\namedright{\bigvee_{i=1}^{m}\Sigma X_{i}}{}{\sux^{K}}\] 
with the property that the composite 
\(\nameddright{\bigvee_{i=1}^{m}\Sigma X_{i}}{v}{\sux^{K}}{w}{\prod_{i=1}^{m}\Sigma X_{i}}\)  
is the inclusion of the wedge into the product. For $1\leq k\leq m$, 
let $v_{k}$ be the composite 
\[v_{k}\colon\nameddright{\Sigma X_{k}}{\Sigma t_{k}}{\bigvee_{i=1}^{m}\Sigma X_{i}}{v}{\sux^{K}}.\] 
By Lemma~\ref{Acofib} there is a homotopy cofibration  
\[\nameddright{\Sigma A}{}{\sux^{K}}{}{\sux^{\overline{K}}}.\] 
Thus all the hypotheses of Theorem~\ref{prodextend} apply and we obtain the following, 
which is a restatement of Theorem~\ref{polyWhintro}.  

\begin{theorem} 
   \label{polyWh} 
   There is a homotopy cofibration 
   \[\nameddright{\displaystyle\bigvee_{k=0}^{\infty}\ \bigvee_{1\leq i_{1}\leq\cdots\leq i_{k}\leq m} 
              (X_{i_{1}}\wedge\cdots\wedge X_{i_{k}})\wedge\Sigma A}{\zeta}{\sclxx^{K}} 
              {}{\sclxx^{\overline{K}}}\] 
    where the map $\zeta$ is congruent to a map $\zeta'$ satisfying a homotopy commutative diagram 
    \[\diagram 
            \displaystyle\bigvee_{k=0}^{\infty}\ \bigvee_{1\leq i_{1}\leq\cdots\leq i_{k}\leq m} 
              (X_{i_{1}}\wedge\cdots\wedge X_{i_{k}})\wedge\Sigma A\rto^-{\zeta'} 
                    \drto_-(0.6){\bigvee_{k=1}^{\infty}\ \bigvee_{1\leq i_{1}\leq\cdots\leq i_{k}\leq m} 
                    [v_{i_{1}},[v_{i_{2}},[\cdots [v_{i_{k}},f]]\cdots ]\hspace{2cm}} 
                 & \sclxx^{K}\dto^{p} \\ 
            & \sux^{K}.   
        \enddiagram\] 
\end{theorem} 
\vspace{-0.8cm}~$\qqed$\bigskip 

The value of Theorem~\ref{polyWh} comes from the potential for playing off $\zeta$ 
and $\zeta'$ in order to determine the homotopy type of $\sclxx^{\overline{K}}$ or the 
homotopy class of $p$. 
One way this can be used is explored in the next subsection. 
Before beginning that, it is worth noting that there are many contexts in which one might 
examine the transition in polyhedral products from $\sclxx^{K}$ to $\sclxx^{\overline{K}}$.  
     
\begin{example} 
\label{onefaceexample} 
One way to form a simplicial complex is to start with the vertex set and iteratively add 
one missing face at a time. For example, if $\sigma=(i_{1},\ldots,i_{\ell})$ is a missing face 
of $K$ and $\overline{K}=K\cup\sigma$ then the space $\Sigma A$ in Theorem~\ref{polyWh} 
is $X_{i_{1}}\ast\cdots\ast X_{i_{\ell}}$, and the theorem informs on the homotopy type 
of $\sclxx^{\overline{K}}$. 
\end{example} 

\begin{example} 
\label{skeletonexample} 
The process in Example~\ref{onefaceexample} may be accelerated by building up 
the simplicial complex skeleton-by-skeleton. Let $K$ be a simplicial complex. 
For $0\leq t\leq m-1$ let $K_{t}$ be the full $t$-skeleton of $K$. That is, 
$K_{t}$ is the simplicial complex consisting of all the faces of $K$ of 
dimension~$\leq t$. Notice that if $\sigma\in K_{t}$ then $\partial\sigma\subseteq K_{t-1}$. 
Notice also that $K_{0}$ is the vertex set of $K$. For $1\leq t\leq m-1$, 
let $\mathcal{S}_{t}=\{\sigma_{1},\ldots,\sigma_{r_{t}}\}$ be the set of 
$t$-dimensional faces of $K$. Observe that 
\[K_{t}=K_{t-1}\cup\mathcal{S}_{t}.\] 
Theorem~\ref{polyWh} then gives an approach to analyzing the homotopy type of 
$\sclxx^{K}$ by ``filtering" it via the spaces $\{\sclxx^{K_{t}}\}_{t=0}^{m-1}$. 
\end{example} 

\begin{example} 
\label{MFexample} 
Another curious example is to start with a simplicial complex $K$ and attach 
all of its missing faces simultaneously. That is, if $\mathcal{S}$ is the set of all 
missing faces of $K$, then let $\overline{K}=K\cup\mathcal{S}$. 
\end{example} 
\medskip 

\noindent 
\textbf{Properties when 
\(\namedright{\sclxx^{K}}{}{\sclxx^{\overline{K}}}\) 
is null homotopic}. 
This is related to Example~\ref{MFexample}. Let $K$ be a simplicial complex on the 
vertex set $[m]$. For a subset $I\subseteq [m]$ the \emph{full subcomplex} $K_{I}$ of $K$ 
is the simplicial complex consisting of those faces in $K$ whose vertices all lie in~$I$. 
There is a simplicial inclusion 
\(\namedright{K_{I}}{}{K}\) 
but this does not have a left inverse that is a simplicial map. On the other hand, the 
induced map of polyhedral products 
\(\namedright{\uxa^{K_{I}}}{}{\uxa^{K}}\) 
does have a left inverse constructed via projection maps~\cite{DS}. A missing face 
$\tau$ of $K$ has the property that $\partial\tau\subseteq K$ but $\tau\notin K$. 
If $\tau=(i_{1},\ldots,i_{k})$, let $I=\{i_{1},\ldots,i_{k}\}$. Then $K_{I}=\partial\tau$, 
so $\uxa^{\partial\tau}$ retracts off $\uxa^{K}$. In the case of pairs $\cxx$, 
the polyhedral product $\cxx^{\partial\tau}$ is homotopy equivalent to 
$\Sigma^{k-1} X_{i_{1}}\ast\cdots\ast X_{i_{k}}$~\cite{GT2}. 
In particular, $\cxx^{\partial\tau}$ is not contractible if each of $X_{i_{1}}$ through~$X_{i_{k}}$ 
is not contractible. Therefore if 
\(\namedright{K}{}{L}\) 
is any simplicial inclusion and $K$ and $L$ share a missing face $\tau$ 
then $\cxx^{\partial\tau}$ is a nontrivial retract of both $\cxx^{K}$ and $\cxx^{L}$. 
Hence if the induced map of polyhedral products  
\(\namedright{\cxx^{K}}{}{\cxx^{L}}\) 
is null homotopic then it must be the case that every missing face of $K$ is a face 
of $L$. Consequently, there must be a factorization 
\(\nameddright{K}{}{\overline{K}}{}{L}\) 
where $\overline{K}=K\cup\mathcal{S}$ for $\mathcal{S}$ the set of all missing faces of $K$. 

This lets us focus on when the map 
\(\namedright{\cxx^{K}}{}{\cxx^{\overline{K}}}\) 
is null homotopic. Note that while it is necessary to fill in the missing faces of $K$ to obtain 
such a null homotopy of polyhedral products it may not be sufficient. An example when it 
is sufficient is the following. Take $m=3$ and let $K=\{\{1\},\{2\},\{3\}\}$ be the 
simplicial complex determined by the three vertices. The missing faces of $K$ are 
$\mathcal{S}=\{(1,2),(1,3),(2,3)\}$. Let $\overline{K}=K\cup\mathcal{S}$, so 
$\overline{K}$ is the boundary of $\Delta^{2}$. Then, as in~\cite{GT1},  
\(\namedright{\cxx^{K}}{}{\cxx^{\overline{K}}}\) 
is null homotopic. 

Specialize now to the case when $\cxx$ is of the form $\sclxx$. In Proposition~\ref{KbarK} 
it is shown that the homotopy types of $\Sigma\sclxx^{K}$ and $\sclxx^{\overline{K}}$ are, 
in a precise sense, complementary. 

\begin{proposition} 
   \label{KbarK} 
   Let $K$ be a simplicial complex on the vertex set $[m]$ and let $\mathcal{S}$ 
   be the set of missing faces of $K$. Suppose that $\overline{K}=K\cup\mathcal{S}$ has the 
   property that the map of polyhedral products 
   \(\namedright{\sclxx^{K}}{}{\sclxx^{\overline{K}}}\) 
   is null homotopic. Then there is a homotopy equivalence 
   \[\Sigma\left(\displaystyle\bigvee_{k=0}^{\infty}\ \bigvee_{1\leq i_{1}\leq\cdots\leq i_{k}\leq m} 
              (X_{i_{1}}\wedge\cdots\wedge X_{i_{k}})\wedge\Sigma A\right)\simeq 
              \sclxx^{\overline{K}}\vee\Sigma\sclxx^{K}.\] 
\end{proposition} 

\begin{proof} 
By Theorem~\ref{polyWh} there is a homotopy cofibration 
   \[\nameddright{\displaystyle\bigvee_{k=0}^{\infty}\ \bigvee_{1\leq i_{1}\leq\cdots\leq i_{k}\leq m} 
              (X_{i_{1}}\wedge\cdots\wedge X_{i_{k}})\wedge\Sigma A}{\zeta}{\sclxx^{K}} 
              {}{\sclxx^{\overline{K}}}.\] 
By hypothesis, the right map is null homotopic. The assertion now follows immediately. 
\end{proof} 

One condition implying that the map 
\(\namedright{\sclxx^{K}}{}{\sclxx^{\overline{K}}}\) 
is null homotopic is if the map $\zeta$ in the homotopy cofibration of Theorem~\ref{polyWh} 
has a right homotopy inverse. In that case the congruence in Theorem~\ref{polyWh} allows 
for more to be said. 

\begin{proposition} 
   \label{KbarK2} 
   Let $K$ be a simplicial complex on the vertex set $[m]$, let $\mathcal{S}$ be the set 
   of missing faces of $K$ and let $\overline{K}=K\cup\mathcal{S}$. If the map $\zeta$ 
   in Theorem~\ref{polyWh} has a right homotopy inverse then the map 
   \(\namedright{\sclxx^{K}}{p}{\sux^{K}}\) 
   factors through the wedge sum of Whitehead products 
   $\bigvee_{k=1}^{\infty}\ \bigvee_{1\leq i_{1}\leq\cdots\leq i_{k}\leq m} 
                    [v_{i_{1}},[v_{i_{2}},[\cdots [v_{i_{k}},f]]$ 
   and Proposition~\ref{KbarK} holds. 
\end{proposition} 

\begin{proof} 
Let 
\[s\colon\namedright{\sclxx^{K}}{}
        {\displaystyle\bigvee_{k=0}^{\infty}\ \bigvee_{1\leq i_{1}\leq\cdots\leq i_{k}\leq m} 
        (X_{i_{1}}\wedge\cdots\wedge X_{i_{k}})\wedge\Sigma A}\] 
be a right homotopy inverse of $\zeta$. Consider the composite 
\[\nameddright{\sclxx^{K}}{s}
        {\displaystyle\bigvee_{k=0}^{\infty}\ \bigvee_{1\leq i_{1}\leq\cdots\leq i_{k}\leq m} 
        (X_{i_{1}}\wedge\cdots\wedge X_{i_{k}})\wedge\Sigma A}{\zeta'}{\sclxx^{K}}.\] 
Since $\zeta$ and $\zeta'$ are congruent, they have the same image in homology. 
Therefore, $(\zeta'\circ s)_{\ast}=(\zeta\circ s)_{\ast}$. As $s$ is a right homotopy inverse 
of $\zeta$, the map $(\zeta\circ s)_{\ast}$ is the identity map in homology. Thus 
$\zeta'\circ s$ induces an isomorphism in homology and so is a homotopy equivalence 
by Whitehead's Theorem. Consequently, the homotopy commutative diagram 
involving $\zeta'$ in Theorem~\ref{KbarK} implies that $p$ factors through the 
sum of Whitehead products 
$\bigvee_{k=1}^{\infty}\ \bigvee_{1\leq i_{1}\leq\cdots\leq i_{k}\leq m} 
                    [v_{i_{1}},[v_{i_{2}},[\cdots [v_{i_{k}},f]]$. 
                    
Next, by Theorem~\ref{polyWh} there is a homotopy cofibration 
\[\nameddright{\displaystyle\bigvee_{k=0}^{\infty}\ \bigvee_{1\leq i_{1}\leq\cdots\leq i_{k}\leq m} 
           (X_{i_{1}}\wedge\cdots\wedge X_{i_{k}})\wedge\Sigma A}{\zeta}{\sclxx^{K}} 
           {}{\sclxx^{\overline{K}}}.\] 
The existence of a right homotopy inverse for $\zeta$ implies that the right map in 
this homotopy cofibration is null homotopic. Therefore Proposition~\ref{KbarK} holds as well. 
\end{proof} 

Propositions~\ref{KbarK} and~\ref{KbarK2} raise several interesting questions. 

\begin{problem} 
For which $K$ and $\overline{K}$ is the map 
\(\namedright{\sclxx^{K}}{}{\sclxx^{\overline{K}}}\) 
null homotopic? 
\end{problem} 

\begin{problem} 
For which $K$ and $\overline{K}$ does the map $\zeta$ in Proposition~\ref{KbarK} have 
a right homotopy inverse? 
\end{problem} 

\begin{problem}  
\label{wedgesplitproblem} 
In the homotopy decomposition in Proposition~\ref{KbarK}, 
does each of the wedge summands $\Sigma X_{i_{1}}\wedge\cdots\wedge X_{i_{k}}\wedge\Sigma A$ 
map wholly to one of $\sclxx^{\overline{K}}$ or $\Sigma\sclxx^{K}$, 
or are there cases when there is a nontrivial decomposition 
\[\Sigma X_{i_{1}}\wedge\cdots\wedge X_{i_{k}}\wedge\Sigma A\simeq B\vee C\] 
with $B$ retracting off $\sclxx^{\overline{K}}$ and $C$ retracting off $\Sigma\sclxx^{K}$?  
For which $(i_{1},\ldots,i_{k})$ does 
$\Sigma X_{i_{1}}\wedge\cdots\wedge X_{i_{k}}\wedge\Sigma A$ 
map wholly into $\sclxx^{\overline{K}}$ or into $\Sigma\sclxx^{K}$? 
\end{problem} 

Despite the potential ambiguity involved in the homotopy decomposition in Proposition~\ref{KbarK} 
stated in Problem~\ref{wedgesplitproblem}, there are cases where interesting information can be 
extracted regardless. Suppose that for $1\leq i\leq m$ each space $X_{i}$ is a sphere. 
By definition, the space $A$ is a wedge sum of spaces of the form 
$X_{i_{1}}\wedge\cdots\wedge X_{i_{\ell}}$, and so is homotopy equivalent to a 
wedge of spheres. Therefore each of the spaces 
$X_{i_{1}}\wedge\cdots\wedge X_{i_{k}}\wedge\Sigma A$ is homotopy equivalent to 
a wedge of spheres, and hence 
$\displaystyle\bigvee_{k=0}^{\infty}\ \bigvee_{1\leq i_{1}\leq\cdots\leq i_{k}\leq m} 
              (X_{i_{1}}\wedge\cdots\wedge X_{i_{k}})\wedge\Sigma A$ 
is homotopy equivalent to a wedge of spheres. Any retract of this large wedge is then 
homotopy equivalent to a wedge of spheres. In particular, from the decomposition in 
Proposition~\ref{KbarK} we obtain the following. 

\begin{corollary} 
   \label{KbarKspheres} 
   Let $K$ be a simplicial complex on the vertex set $[m]$ and let $\mathcal{S}$ 
   be the set of missing faces of $K$. Suppose that $\overline{K}=K\cup\mathcal{S}$ has the 
   property that the map of polyhedral products 
   \(\namedright{\sclxx^{K}}{}{\sclxx^{\overline{K}}}\) 
   is null homotopic. If each space $X_{i}$ is a sphere for \mbox{$1\leq i\leq m$}, 
   then $\sclxx^{\overline{K}}$ is homotopy equivalent to a wedge of spheres.~$\qqed$ 
\end{corollary}   

More is true. If the map 
\(\namedright{\sclxx^{K}}{}{\sclxx^{\overline{K}}}\) 
null homotopic then in the homotopy cofibration 
\[\nameddright{\displaystyle\bigvee_{k=0}^{\infty}\ \bigvee_{1\leq i_{1}\leq\cdots\leq i_{k}\leq m} 
           (X_{i_{1}}\wedge\cdots\wedge X_{i_{k}})\wedge\Sigma A}{\zeta}{\sclxx^{K}} 
           {}{\sclxx^{\overline{K}}}\] 
the map $\zeta$ induces an epimorphism in homology. If each $X_{i}$ a sphere for $1\leq i\leq m$ then 
$\displaystyle\bigvee_{k=0}^{\infty}\ \bigvee_{1\leq i_{1}\leq\cdots\leq i_{k}\leq m} 
           (X_{i_{1}}\wedge\cdots\wedge X_{i_{k}})\wedge\Sigma A$ 
is homotopy equivalent to a wedge of spheres so $\zeta_{\ast}$ being an epimorphism 
implies that a subwedge $W$ may be chosen so the composite 
\[W\hookrightarrow\namedright{\displaystyle\bigvee_{k=0}^{\infty}\ \bigvee_{1\leq i_{1}\leq\cdots\leq i_{k}\leq m} 
           (X_{i_{1}}\wedge\cdots\wedge X_{i_{k}})\wedge\Sigma A}{\zeta}{\sclxx^{K}}\] 
induces an isomorphism in homology and so is a homotopy equivalence. Thus $\zeta$ 
has a right homotopy inverse, and now Proposition~\ref{KbarK2} applies. Moreover, 
as $\zeta'$ is congruent to $\zeta$ by Theorem~\ref{polyWh}, they have the same image 
in homology, so the composite 
\[W\hookrightarrow\namedright{\displaystyle\bigvee_{k=0}^{\infty}\ \bigvee_{1\leq i_{1}\leq\cdots\leq i_{k}\leq m} 
           (X_{i_{1}}\wedge\cdots\wedge X_{i_{k}})\wedge\Sigma A}{\zeta'}{\sclxx^{K}}\] 
is a homotopy equivalence and the statement on ``factoring through" a wedge sum of 
Whitehead products in Proposition~\ref{KbarK2} becomes ``is" a wedge sum of Whitehead products. 

\begin{corollary} 
   \label{KbarK2spheres} 
   Let $K$ be a simplicial complex on the vertex set $[m]$ and let $\mathcal{S}$ 
   be the set of missing faces of $K$. Suppose that $\overline{K}=K\cup\mathcal{S}$ has the 
   property that the map of polyhedral products 
   \(\namedright{\sclxx^{K}}{}{\sclxx^{\overline{K}}}\) 
   is null homotopic. If $X_{i}$ is a sphere for $1\leq i\leq m$ then 
   $\sclxx^{K}$ and $\sclxx^{\overline{K}}$ are both 
   homotopy equivalent to wedges of spheres and the map 
   \(\namedright{\sclxx^{K}}{p}{\sux^{K}}\) 
   is a subwedge of the wedge sum of Whitehead products 
   $\bigvee_{k=1}^{\infty}\ \bigvee_{1\leq i_{1}\leq\cdots\leq i_{k}\leq m} 
                    [v_{i_{1}},[v_{i_{2}},[\cdots [v_{i_{k}},f]]$.~$\qqed$ 
\end{corollary} 

Carrying on, the retraction of $S^{1}$ off $\Omega S^{2}$ induces a retraction of the pair 
$(CS^{1},S^{1})$ off the pair $(C\Omega S^{2},\Omega S^{2})$. Hence for any simplicial 
complex $K$ we obtain a retraction of $(\underline{CS^{1}},\underline{S^{1}})^{K}$ 
off $(\underline{C\Omega S^{2}},\underline{\Omega S^{2}})^{K}$. Further, this  retraction 
is natural for maps of simplicial complexes. Writing $(CS^{1},S^{1})$ 
in the more standard way as $(D^{2},S^{1})$, the polyhedral product 
$(\underline{D^{2}},\underline{S^{1}})^{K}$ is the \emph{moment-angle complex} that is 
critical to toric topology, more commonly written as $\zk$. In the context of 
Corollary~\ref{KbarKspheres}, we obtain compatible retractions of $\zk$ and  
$\mathcal{Z}_{\overline{K}}$ off $(\underline{C\Omega S^{2}},\underline{\Omega S^{2}})^{K}$ 
and $(\underline{C\Omega S^{2}},\underline{\Omega S^{2}})^{\overline{K}}$ respectively. 
The compatible retractions implies that as the map 
\(\namedright{(\underline{C\Omega S^{2}},\underline{\Omega S^{2}})^{K}}{} 
      {(\underline{C\Omega S^{2}},\underline{\Omega S^{2}})^{\overline{K}}}\) 
is null homotopic, so is the map 
\(\namedright{\zk}{}{\mathcal{Z}_{\overline{K}}}\). 
As $(\underline{C\Omega S^{2}},\underline{\Omega S^{2}})^{K}$ and 
$(\underline{C\Omega S^{2}},\underline{\Omega S^{2}})^{\overline{K}}$ are homotopy 
equivalent to wedges of spheres so are $\mathcal{Z}_{K}$ and $\mathcal{Z}_{\overline{K}}$. 

\begin{corollary} 
  \label{KbarKzk} 
   Let $K$ be a simplicial complex on the vertex set $[m]$ and let $\mathcal{S}$ 
   be the set of missing faces of $K$. Suppose that $\overline{K}=K\cup\mathcal{S}$ has the 
   property that the map of polyhedral products 
   \(\namedright{(\underline{C\Omega S^{2}},\underline{\Omega S^{2}})^{K}}{} 
      {(\underline{C\Omega S^{2}},\underline{\Omega S^{2}})^{\overline{K}}}\) 
   is null homotopic. Then the map 
   \(\namedright{\zk}{}{\mathcal{Z}_{\overline{K}}}\) 
   is null homotopic, both $\zk$ and $\mathcal{Z}_{\overline{K}}$ are homotopy 
  equivalent to wedges of spheres, and the map 
   \(\namedright{\zk}{}{(\underline{\mathbb{C}P}^{\infty},\underline{\ast})^{K}}\) 
   is a subwedge of the wedge sum of Whitehead products 
   $\bigvee_{k=1}^{\infty}\ \bigvee_{1\leq i_{1}\leq\cdots\leq i_{k}\leq m} 
                    [v_{i_{1}},[v_{i_{2}},[\cdots [v_{i_{k}},f]]$.~$\qqed$ 
\end{corollary} 

Corollary~\ref{KbarKzk} is connected to important problems in toric topology 
and combinatorics. By~\cite{BP} the space $\zk$ is homotopy equivalent to the 
complement of the complex coordinate subspace determined by $K$. A major question 
is combinatorics is to determine for which $K$ these complements of coordinate subspace 
arrangements are homotopy equivalent to a wedge of spheres. A series of 
papers~\cite{GT1,GT2,GW,IK1,IK2} identified families of simplicial complexes $K$ 
for which~$\zk$ is homotopy equivalent to a wedge of spheres, including shifted complexes 
and  those whose Alexander duals are vertex decomposable, shellable or sequentially 
Cohen-Macauley. All of these are subsumed by what~\cite{IK2} calls totally fillable or totally 
homology fillable complexes. Another family of simplicial complexes for which $\zk$ is 
homotopy equivalent to a wedge of spheres is flag complexes whose $1$-skeleton 
is a chordal graph~\cite{GPTW}. Several papers have examined when the map from 
$\zk$ to $(\underline{\mathbb{C}P}^{\infty},\underline{\ast})^{K}$ is described by Whitehead 
products~\cite{AP, GT3, IK3}. 

We end by posing a problem regarding how large might be the family of simplicial 
complexes with the property that $\zk$ is homotopy equivalent to a wedge of spheres. 
Let $\mathcal{F}$ be the collection of simplicial complexes that are either totally fillable 
or flag complexes having a $1$-skeleton that is a chordal graph. 

\begin{problem} 
Are there examples of $K$ and $\overline{K}$ in Corollary~\ref{KbarKzk} 
for which $\mathcal{Z}_{\overline{K}}$ or $\zk$ is homotopy equivalent to a 
wedge of spheres but $\overline{K}$ or $K$ is not in $\mathcal{F}$? 
\end{problem}

\newpage 

\bibliographystyle{amsalpha}

\end{document}